\documentclass[12pt,a4paper,twoside,reqno]{amsart}

\usepackage{amsmath, amssymb, amsthm, enumerate, enumitem, graphicx}
\usepackage[normalem]{ulem}
\usepackage{xr}
\makeatletter
\def\l@section{\@tocline{1}{10pt}{1pc}{}{}}
\def\l@subsection{\@tocline{2}{0pt}{1pc}{4.6em}{}}
\def\l@subsubsection{\@tocline{3}{0pt}{1pc}{7.6em}{}}
\renewcommand{\tocsection}[3]{%
  \indentlabel{\@ifnotempty{#2}{\makebox[2.3em][l]{%
    \ignorespaces#1 #2.\hfill}}}\textbf{#3}}
\renewcommand{\tocsubsection}[3]{%
  \indentlabel{\@ifnotempty{#2}{\hspace*{2.3em}\makebox[2.3em][l]{%
    \ignorespaces#1 #2.\hfill}}}#3}
\renewcommand{\tocsubsubsection}[3]{%
  \indentlabel{\@ifnotempty{#2}{\hspace*{4.6em}\makebox[3em][l]{%
    \ignorespaces#1 #2.\hfill}}}#3}
\makeatother

\newcommand{\MM}{\mathcal{M}}
\newcommand{\CC}{\mathcal{C}}
\newcommand{\IR}{\mathbb{R}}

\newcommand{\IN}{\mathbb{N}}
\newcommand{\IC}{\mathbb{C}}
\newcommand{\IH}{\mathbb{H}}

\newcommand{\IZ}{\mathbb{Z}}

\renewcommand{\SS}{\mathcal{S}}
\newcommand{\eps}{\varepsilon}
\newcommand{\ov}[1]{\overline{#1}}

\newcommand{\td}[1]{\widetilde{#1}}
\DeclareMathOperator{\Klein}{Klein}
\DeclareMathOperator{\thick}{thick}
\DeclareMathOperator{\thin}{thin}
\DeclareMathOperator{\eucl}{eucl}

\DeclareMathOperator{\Int}{Int}

\DeclareMathOperator{\Ric}{Ric}
\DeclareMathOperator{\curv}{curv}
\DeclareMathOperator{\area}{area}

\DeclareMathOperator{\dist}{dist}

\DeclareMathOperator{\diam}{diam}

\DeclareMathOperator{\vol}{vol}

\DeclareMathOperator{\Stab}{Stab}

\DeclareMathOperator{\Rm}{Rm}
\newcommand{\cangle}{\widetilde{\sphericalangle}}

\newcommand{\dotcup}{\ensuremath{\mathaccent\cdot\cup}}
\newcommand{\EMPTY}[1]{}

\newtheorem{Theorem}{Theorem}[section]
\newtheorem{Lemma}[Theorem]{Lemma}

\newtheorem{Proposition}[Theorem]{Proposition}
\newtheorem{Definition}[Theorem]{Definition}

\numberwithin{equation}{section}
\newtheorem*{Claim}{Claim}
\newtheorem*{Claim0}{Claim 0}
\newtheorem*{Claim1}{Claim 1}
\newtheorem*{Claim2}{Claim 2}
\newtheorem*{Claim3}{Claim 3}
\newtheorem*{Claim4}{Claim 4}
\newtheorem*{Claim5}{Claim 5}

\newtheorem*{Result}{Long-time picture after Perelman}
\newtheorem*{Result1}{First step}
\newtheorem*{Result2}{Second step}
\newtheorem*{Result3}{Third step}
\newtheorem*{Result4}{Behavior of the flow on a large time-interval}

\title[Long-time behavior of 3d Ricci flow --- D]{Long-time behavior of 3 dimensional Ricci flow\\D: Proof of the main results}
\author{Richard H Bamler}
\address{UC Berkeley, Department of Mathematics, 970 Evans Hall, Berkeley, CA 94720, USA}
\email{rbamler@math.berkeley.edu}
\date{\today}

\begin{document}
\begin{abstract}
This is the fourth and last part of a series of papers on the long-time behavior of 3 dimensional Ricci flows with surgery.
In this paper, we prove our main two results.
The first result states that if the surgeries are performed correctly, then the flow becomes non-singular eventually and the curvature is bounded by $C t^{-1}$.
The second result provides a qualitative description of the geometry as $t \to \infty$. 
\end{abstract}

\maketitle
\tableofcontents

\section{Introduction and outline of the proof} \label{sec:IntroductionpartD}
In this paper we will prove the main results of this series of papers, namely Theorems \ref{Thm:LT0-main-1} and \ref{Thm:geombehavior} from \cite{Bamler-LT-Introduction}.
In our proofs we will make use of the three previous papers, which we will refer to as \cite{Bamler-LT-Perelman}, \cite{Bamler-LT-simpcx} and \cite{Bamler-LT-topology}.
For a precise statement of the main results, historical remarks and acknowledgements see \cite{Bamler-LT-Introduction}.

Our main interest will focus on the proof of Theorem \ref{Thm:LT0-main-1} in \cite{Bamler-LT-Introduction}, which states that for every Ricci flow with surgery $\MM$ that is performed by sufficiently precise cutoff, the surgeries stop to occur eventually and we have a curvature bound of the form $|{\Rm_t}| < C t^{-1}$ for large $t$.
The geometric characterization of the flow for $t \to \infty$, as described in \cite[Theorem \ref{Thm:geombehavior}]{Bamler-LT-Introduction}, will then be a byproduct of the proof of this theorem.
For more details on this geometric characterization see subsection \ref{subsec:behaviorlargetimesproof}.

We will now present an outline of the proof of \cite[Theorem \ref{Thm:LT0-main-1}]{Bamler-LT-Introduction}.
The key to proving this theorem is to establish the curvature bound $|{\Rm_t}| < C t^{-1}$ for large times $t$.
The fact that surgeries stop to occur eventually follows from this bound, since surgeries can only arise at points where the curvature goes to infinity.
For simplicity, we will only consider Ricci flows \emph{without} surgery in this outline, i.e. we will consider families of metrics $(g_t)_{t \in [0, \infty)}$ on a closed, orientable $3$-manifold $M$ that satisfy the evolution equation
\[ \partial_t g_t = - 2 \Ric (g_t) \]
and we will explain how we obtain the curvature bound $|{\Rm_t}| = |{\Rm} (g_t)| < C t^{-1}$ (for large $t$).
In other words, we will outline the proof of \cite[Corollary \ref{Cor:nonsingflow}]{Bamler-LT-Introduction}.
The case in which there are surgeries follows using the same arguments, but the existence of surgeries adds several less interesting technical difficulties.

Consider the rescaled and reparameterized flow $(\td{g}_t)_{t \in (-\infty, \infty)}$ with $\td{g}_t = e^{-t} g_{e^{t}}$, which satisfies the flow equation
\[ \partial_t \td{g}_t = - 2 \Ric (\td{g}_t) -  \td{g}_t. \]
Then the curvature bound $|{\Rm}(g_t)| < C t^{-1}$ for $g_t$ is equivalent to the bound
\begin{equation} \label{eq:mainrescaledbound}
|{\Rm} (\td{g}_t)| <C \qquad \text{for large $t$.}
\end{equation}
For clarity we will only consider the metrics $\td{g}_t$ instead of $g_t$ for the rest of this subsection.
We will however work with the metric $g_t$ in the main part of this paper.

In \cite[Proposition \ref{Prop:thickthindec}]{Bamler-LT-Perelman}, we recalled Perelman's result on the long-time behavior of Ricci flows (with surgery) (cf \cite[7.3]{PerelmanII}), which, in the language of the rescaled flow $(\td{g}_t)$, can be summarized as follows:

\begin{Result} 
For every sufficiently large time $t$ there is a decomposition $M = M_{\textnormal{thick}} (t) \dotcup M_{\textnormal{thin}} (t)$ (of $M$ into its ``thick part'' and ``thin part'') such that:
\begin{enumerate}[label=(\alph*)]
\item The components of $M_{\textnormal{thick}} (t)$ are diffeomorphic to hyperbolic manifolds and the metric $\td{g}_t$ on each component of $M_{\textnormal{thick}} (t)$ is sufficiently close to a hyperbolic metric of sectional curvature $-\frac14$ whose cusps are truncated along embedded $2$-tori sufficiently far away from a base point.
Those $2$-tori correspond to embedded $2$-tori between $M_{\textnormal{thin}} (t)$ and $M_{\textnormal{thick}} (t)$, which are incompressible in $M$ (i.e. $\pi_1$-injective).
\item If we set
\[ \rho_1 (x, t) = \sup \big\{ r \in (0,1] \;\; : \;\; \sec_t \geq - r^{-2} \quad \text{on} \quad B(x, t, r) \big\}, \]
then for all $x \in M_{\textnormal{thin}} (t)$
\[ \vol B(x, t, \rho_1(x,t)) < w(t) \rho_1^3(x,t). \]
Here $w(t)$ is a small constant with $\lim_{t \to \infty} w(t) = 0$.
\item There are functions $\ov{r}, K : (0,1) \to (0, \infty)$ such that for all $w' \in (w,1)$, $r < \ov{r}(w')$ and $x \in M_{\textnormal{thin}} (t)$ we have:
If $\vol B(x,r) > w' r^3$ and $r < \rho_1(x, t)$, then $|{\Rm_t}|, r |{\nabla\Rm_t}|, r^2 |{\nabla^2 \Rm_t}| < K(w') r^{-2}$ on $B(x,r)$.
\end{enumerate}
\end{Result}

In short, Perelman's result states that the metric on the thick part converges to a hyperbolic metric (see part (a)), while the metric on the thin part collapses locally at scale $\rho_1 (x,t)$ with a lower sectional curvature bound (see part (b)).
Part (c) is a technical statement.
So the desired curvature bound (\ref{eq:mainrescaledbound}) holds on $M_{\textnormal{thick}} (t)$ and it remains to study $M_{\textnormal{thin}} (t)$ and show that the curvature bound holds there as well.

In order to understand the behavior of the metric on the thin part, we first have to analyze the local collapsing behavior there.
This analysis is mainly due to \cite{MorganTian}, \cite{KLcollapse}, \cite{BBMP2}, \cite{CG}, \cite{ShioyaYamaguchi} and \cite{Faessler}, but it will require some work to convert the results of these papers into a form that is suitable for our purposes.
This conversion is carried out in subsection \ref{subsec:MorganTian}, Proposition \ref{Prop:MorganTianMain}, which is quite technical, describes its outcome.
For the purpose of this exposition, we will now present a very simplified version of this description, which does not cover several---less instructive---special cases.
For a more detailed explanation of Proposition \ref{Prop:MorganTianMain}, we refer to the beginning of subsection \ref{subsec:MorganTian}.

To understand the analysis on the thin part, first note that there are essentially three ways in which a collapse with lower sectional curvature bound can occur:
Namely, the collapse can occur along $S^2$, $T^2$ or $S^1$-fibers.
An example for a collapse along an $S^2$-fiber would be the standard geometry on $S^2 \times \IR$, where the metric on first factor is chosen small.
Similarly, an example for a collapse along a $T^2$-fiber would be $T^2 \times \IR$ with a small metric on the first factor.
In these two cases the metric collapses to a ($1$-dimensional) line and the first factor in the product determines the fiber along which the collapse occurs.
A collapse along an $S^1$-fiber can be observed in the model $S^1 \times \IR^2$, where the metric on the first factor is again small.
This example would be collapsed to $\IR^2$.
Note that the $\IR^2$-factor can be replaced by any Riemannian surface $\Sigma^2$.
The product $S^1 \times \Sigma$ would then be collapsed to $\Sigma^2$.
Even more generally, we can replace the product $S^1 \times \Sigma$ by any $S^1$-fibration over $\Sigma$ and choose a metric for which the $S^1$-fibers have small length.

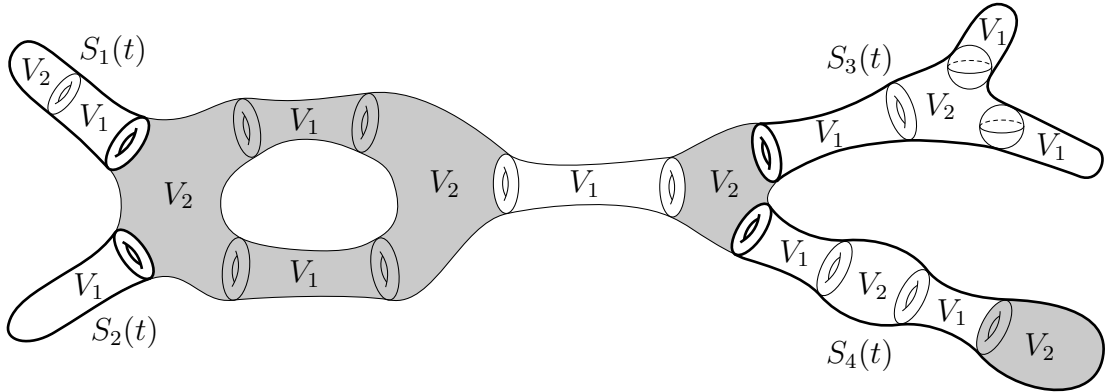
\begin{figure}[t] 
\begin{center}
%\begin{picture}(0,0)%
%\hspace{16mm}\includegraphics[width=11cm]{expandingsolidtorus}%
%\end{picture}%
\setlength{\unitlength}{2863sp}%
\begingroup\makeatletter\ifx\SetFigFont\undefined%
\gdef\SetFigFont#1#2#3#4#5{%
  \reset@font\fontsize{#1}{#2pt}%
  \fontfamily{#3}\fontseries{#4}\fontshape{#5}%
  \selectfont}%
\fi\endgroup%
\begin{picture}(4500,3800)(3500,0)
\hspace{16mm}\includegraphics[width=14.5cm]{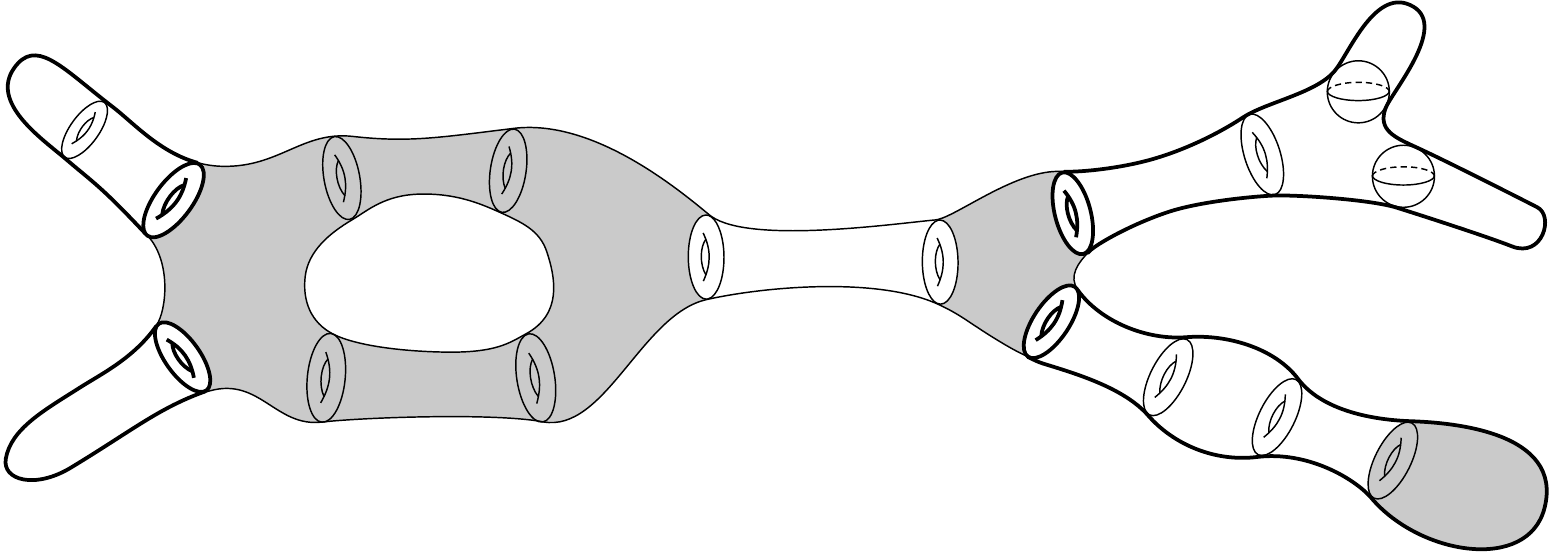}%
\put(-9300,2700){\makebox(0,0)[lb]{\smash{{\SetFigFont{12}{14.4}{\familydefault}{\mddefault}{\updefault}$V_2$}}}}
\put(-8800,2300){\makebox(0,0)[lb]{\smash{{\SetFigFont{12}{14.4}{\familydefault}{\mddefault}{\updefault}$V_1$}}}}
\put(-8850,845){\makebox(0,0)[lb]{\smash{{\SetFigFont{12}{14.4}{\familydefault}{\mddefault}{\updefault}$V_1$}}}}
\put(-8080,1630){\makebox(0,0)[lb]{\smash{{\SetFigFont{12}{14.4}{\familydefault}{\mddefault}{\updefault}$V_2$}}}}
\put(-7000,2270){\makebox(0,0)[lb]{\smash{{\SetFigFont{12}{14.4}{\familydefault}{\mddefault}{\updefault}$V_1$}}}}
\put(-7000,950){\makebox(0,0)[lb]{\smash{{\SetFigFont{12}{14.4}{\familydefault}{\mddefault}{\updefault}$V_1$}}}}
\put(-5800,1730){\makebox(0,0)[lb]{\smash{{\SetFigFont{12}{14.4}{\familydefault}{\mddefault}{\updefault}$V_2$}}}}
\put(-4600,1730){\makebox(0,0)[lb]{\smash{{\SetFigFont{12}{14.4}{\familydefault}{\mddefault}{\updefault}$V_1$}}}}
\put(-3400,1730){\makebox(0,0)[lb]{\smash{{\SetFigFont{12}{14.4}{\familydefault}{\mddefault}{\updefault}$V_2$}}}}
\put(-2500,2160){\makebox(0,0)[lb]{\smash{{\SetFigFont{12}{14.4}{\familydefault}{\mddefault}{\updefault}$V_1$}}}}
\put(-1550,2400){\makebox(0,0)[lb]{\smash{{\SetFigFont{12}{14.4}{\familydefault}{\mddefault}{\updefault}$V_2$}}}}
\put(-1100,3040){\makebox(0,0)[lb]{\smash{{\SetFigFont{12}{14.4}{\familydefault}{\mddefault}{\updefault}$V_1$}}}}
\put(-600,2050){\makebox(0,0)[lb]{\smash{{\SetFigFont{12}{14.4}{\familydefault}{\mddefault}{\updefault}$V_1$}}}}
\put(-2800,1130){\makebox(0,0)[lb]{\smash{{\SetFigFont{12}{14.4}{\familydefault}{\mddefault}{\updefault}$V_1$}}}}
\put(-2130,850){\makebox(0,0)[lb]{\smash{{\SetFigFont{12}{14.4}{\familydefault}{\mddefault}{\updefault}$V_2$}}}}
\put(-1450,620){\makebox(0,0)[lb]{\smash{{\SetFigFont{12}{14.4}{\familydefault}{\mddefault}{\updefault}$V_1$}}}}
\put(-700,320){\makebox(0,0)[lb]{\smash{{\SetFigFont{12}{14.4}{\familydefault}{\mddefault}{\updefault}$V_2$}}}}
\put(-8800,2880){\makebox(0,0)[lb]{\smash{{\SetFigFont{12}{14.4}{\familydefault}{\mddefault}{\updefault}$S_1(t)$}}}}
\put(-8700,450){\makebox(0,0)[lb]{\smash{{\SetFigFont{12}{14.4}{\familydefault}{\mddefault}{\updefault}$S_2(t)$}}}}
\put(-2400,2800){\makebox(0,0)[lb]{\smash{{\SetFigFont{12}{14.4}{\familydefault}{\mddefault}{\updefault}$S_3(t)$}}}}
\put(-2400,250){\makebox(0,0)[lb]{\smash{{\SetFigFont{12}{14.4}{\familydefault}{\mddefault}{\updefault}$S_4(t)$}}}}
\end{picture}%
\caption{A decomposition of $M$ into $V_1(t), V_2(t)$ for the case in which $M_{\textnormal{thin}} (t) = M$.
The ``$(t)$'' is omitted for space reasons. 
The good components of $V_1 (t), V_2(t)$, i.e. the components of $M_{\textnormal{good}}(t)$, are colored in gray.
The solid tori $S_1(t), \ldots, S_4(t)$, which are highlighted in bold, cover the bad part $M_{\textnormal{bad}} (t)$, except for one component of $V_1(t)$, which is diffeomorphic to $T^2 \times I$.
\label{fig:V1V2goodbadSi}}
\end{center}
\end{figure}
The analysis of $M_{\textnormal{thin}} (t)$ roughly implies the following picture (compare with Figure \ref{fig:V1V2goodbadSi}, for more details see Proposition \ref{Prop:MorganTianMain} and the explanation before that):
For large times $t$ there is a decomposition
\begin{equation} \label{eq:MthinV1V2Intro}
M_{\textnormal{thin}} (t) = V_1 (t) \cup V_2 (t)
\end{equation}
of subsets $V_1(t), V_2(t) \subset M$,\footnote{In the main part of this paper we usually fix $t$ and write $V_1, V_2$ instead of $V_1 (t), V_2(t)$.} whose components intersect in embedded $2$-spheres or $2$-tori.
The collapse on $V_1 (t)$ locally occurs along $S^2$ or $T^2$-fibers and can be locally modeled on geometries like $T^2 \times \IR$ or $S^2 \times \IR$, as discussed earlier.
So the components of $V_1(t)$ are diffeomorphic to $T^2 \times I$, $S^2 \times I$ or to a few other standard topologies.
The collapse on $V_2 (t)$ locally occurs along $S^1$-fibers and can be locally modeled on $S^1 \times \IR^2$.
The collapsing fibers induce an $S^1$-fibration (or, more generally, a Seifert fibration) on $V_2 (t)$, apart from a few exceptions.

Using elementary topological arguments, it is possible to reorganize the decomposition (\ref{eq:MthinV1V2Intro}) into a geometric decomposition of $M$ (see \cite[Definition \ref{Def:geomdec}]{Bamler-LT-topology}) and hence prove the Geometrization Conjecture.
In general, or at least a priori, this reorganization modifies the decomposition (\ref{eq:MthinV1V2Intro}) substantially.
For example, the decomposition (\ref{eq:MthinV1V2Intro}) is in general far more complex than the geometric decomposition of $M$ and the $2$-tori between the components of $V_1(t)$ and $V_2(t)$ are in general not incompressible in $M$, as are those of the geometric decomposition.

In a next step, we analyze the decomposition (\ref{eq:MthinV1V2Intro}) from a geometric and a topological point of view (see subsections \ref{subsec:geometricconsequences}, \ref{subsec:topimplications}).
Call a component of $V_1 (t)$ \emph{good} if it locally collapses along incompressible $T^2$-fibers.
Components of $V_1(t)$ that collapse along $S^2$-fibers or compressible $T^2$-fibers are called \emph{bad}.
Similarly, we call a component of $V_2(t)$ \emph{good} if the collapsing $S^1$-fibers are incompressible in $M$, otherwise we call it \emph{bad}.
Let $M_{\textnormal{good}} (t)$ be the union of all good components of $V_1 (t)$ or $V_2 (t)$ and $M_{\textnormal{bad}} (t)$ the union of all bad components.
So we obtain a decomposition (see again Figure \ref{fig:V1V2goodbadSi})
\begin{equation} \label{eq:goodbaddec}
 M_{\textnormal{thin}} (t) = M_{\textnormal{good}} (t) \cup M_{\textnormal{bad}} (t).
\end{equation}

In subsection \ref{subsec:geometricconsequences}, we will learn that good components of $V_1(t)$ or $V_2(t)$ become locally non-collapsed at scale $\rho_1(x,t)$ when we pass to the universal cover (see Lemma \ref{Lem:unwrapfibration}).
By this we mean the following: 
Let $x \in M_{\textnormal{good}} (t)$ be a point that is located in a good component of $V_1 (t)$ or $V_2(t)$, and choose a lift $\td{x} \in \td{M}$ of $x$ in the universal cover of $M$.
Then
\begin{equation} \label{eq:noncollapsednessintro}
 \vol_t B^{\td{M}} (\td{x},t,\rho_1 (x,t)) > w_1 \rho^3_1 (x,t).
\end{equation}
Here the left-hand side denotes the volume of the $\rho_1(x,t)$-ball around $\td{x}$ in the universal cover of $M$ (not the universal cover of $B(x, t, \rho_1(x,t))$!) and $w_1 > 0$ is a universal constant.

In subsection \ref{subsec:topimplications}, we analyze the topology of the decomposition (\ref{eq:goodbaddec}).
We will find that there are embedded, pairwise disjoint solid tori $S_1(t), \ldots, S_{m(t)} (t) \subset M_{\textnormal{thin}} (t)$, $S_i(t) \approx S^1 \times D^2$ such that the following holds (compare again with Figure \ref{fig:V1V2goodbadSi}):
\begin{enumerate}[label=(\alph*)]
\item each $S_i(t)$ is a union of components of $V_1 (t)$ and $V_2(t)$\label{list:topologygoodbad}
\item each $S_i(t)$ is adjacent to a component of $V_2 (t)$ and it contains a component of $V_1 (t)$ that is adjacent to its boundary $\partial S_i(t)$
\item each $S_i(t)$ is incompressible in $M$, i.e. the induced map $\pi_1 (S_i(t)) \cong \IZ \to \pi_1 (M)$ is injective
\item $M_{\textnormal{bad}} (t) \setminus (S_1(t) \cup \ldots \cup S_{m(t)} (t) )$ is a union of components that are diffeomorphic to $T^2 \times I$ or its twofold quotient $\Klein^2 \td\times I$, where the $T^2$-factor is compressible in $M$.
Each of these components is adjacent to $M_{\textnormal{good}} (t)$ on both sides.
\end{enumerate}

With this characterization at hand we are now able to apply our results from \cite{Bamler-LT-Perelman}, and analyze the behavior of the flow $(\td{g}_t)$ more precisely.
This analysis is explained in section \ref{sec:mainargument}, which also contains the proofs of the two main theorems, \cite[Theorems \ref{Thm:LT0-main-1}, \ref{Thm:geombehavior}]{Bamler-LT-Introduction}.
Section \ref{sec:Preparations} contains technical results that will be used in the course of this analysis.

In the remainder of this outline, we explain our strategy of proof in section \ref{sec:mainargument}.
In subsection \ref{subsec:firstcurvbounds} we first analyze the geometry of $\td{g}_t$ for a fixed, large time $t$.
Our analysis is organized in three steps.
In the ``first step'' (see Lemma \ref{Lem:firstcurvboundstep1}), we apply \cite[Proposition \ref{Prop:curvcontrolgood}]{Bamler-LT-Perelman} (``bounded curvature around good points'') to all points of $M_{\textnormal{good}} (t)$ and make use of the non-collapsedness in the universal cover, (\ref{eq:noncollapsednessintro}), to deduce that
\begin{equation} \label{eq:curvboundonMgood}
 |{\Rm}| < K \qquad \text{on} \qquad \big( M_{\textnormal{thick}} (t) \cup M_{\textnormal{good}} (t) \big) \times [t - \tau, t] 
\end{equation}
for some universal constants $K < \infty$ and $\tau > 0$.
Next, but still within the first step, we consider the components mentioned in item (d) of the list above, which are diffeomorphic to $T^2 \times I$ or $\Klein^2 \td\times I$.
The fact that these components are adjacent to $M_{\textnormal{good}} (t)$---a region on which the curvature is bounded on a time-interval of uniform size---can be used to localize the arguments leading to the curvature bound in (\ref{eq:curvboundonMgood}).
Using \cite[Proposition \ref{Prop:curvboundinbetween}]{Bamler-LT-Perelman} (``Curvature control at points that are good relative to regions whose boundary is geometrically controlled''), we find that the curvature is also bounded on these components, by a uniform constant and on a uniform time-interval.
So we obtain that
\begin{equation} \label{eq:curvboundoutsideSiinintro}
 |{\Rm}| < K \qquad \text{on} \qquad \big( M \setminus (S_1 (t) \cup \ldots \cup S_{m(t)} (t) ) \big) \times [t- \tau, t].
\end{equation}
Moreover, for large times $t$, we obtain a curvature bound on each $S_i(t)$ that depends on the distance to $\partial S_i (t)$, using \cite[Proposition \ref{Prop:curvcontrolincompressiblecollapse}]{Bamler-LT-Perelman} (``Bounded curvature at bounded distance from sufficiently collapsed and good regions'').
We will use this bound to conclude that there are two possibilities for each solid torus $S_i (t)$:
Either the diameter of $S_i(t)$ is controlled, in which case we obtain a curvature bound on all of $S_i(t)$ in terms of its diameter, or the diameter is uncontrolled, i.e. large, and we can use a distance dependent curvature bound to understand the geometry of a ``long'' collar neighborhood $P_i(t) \subset S_i(t)$ of $\partial S_i (t)$ (compare with Figure \ref{fig:SiStep1Intro}):

\begin{figure}[t] 
\begin{center}
%\begin{picture}(0,0)%
%\hspace{16mm}\includegraphics[width=11cm]{expandingsolidtorus}%
%\end{picture}%
\setlength{\unitlength}{2863sp}%
\begingroup\makeatletter\ifx\SetFigFont\undefined%
\gdef\SetFigFont#1#2#3#4#5{%
  \reset@font\fontsize{#1}{#2pt}%
  \fontfamily{#3}\fontseries{#4}\fontshape{#5}%
  \selectfont}%
\fi\endgroup%
\begin{picture}(4500,3800)(3500,0)
\hspace{16mm}\includegraphics[width=14.5cm]{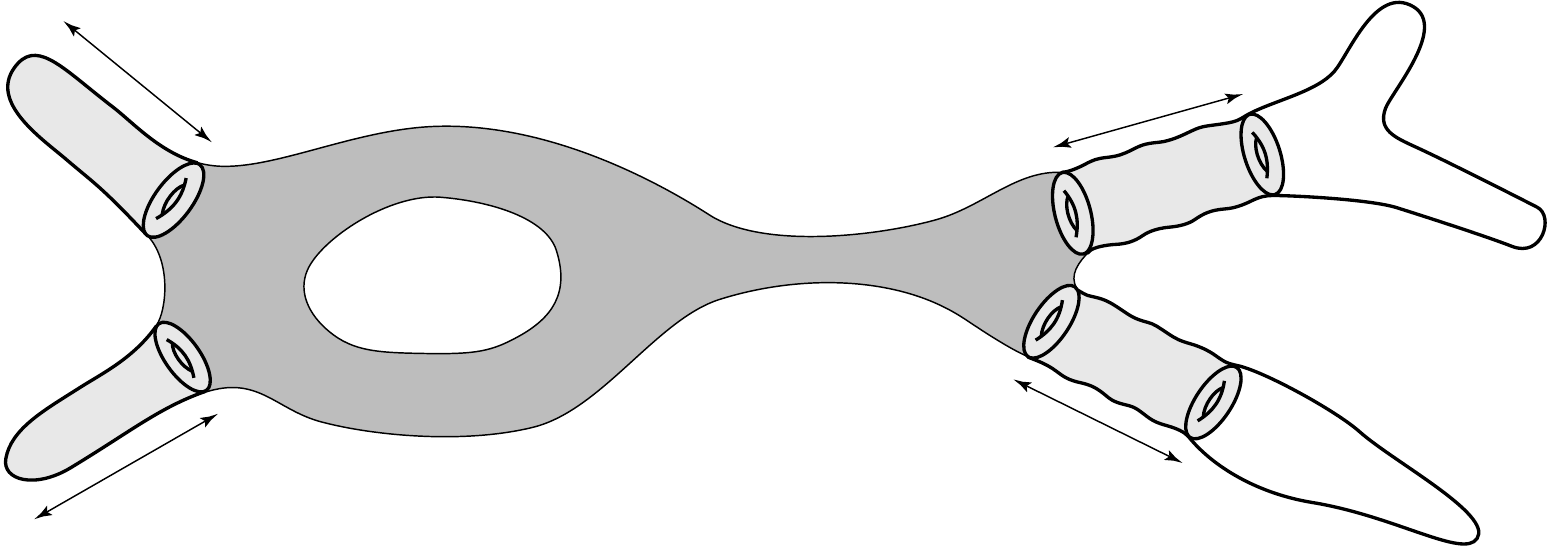}%
\put(-2590,2200){\makebox(0,0)[lb]{\smash{{\SetFigFont{12}{14.4}{\familydefault}{\mddefault}{\updefault}$P_3(t)$}}}}
\put(-2830,1070){\makebox(0,0)[lb]{\smash{{\SetFigFont{12}{14.4}{\familydefault}{\mddefault}{\updefault}$P_4(t)$}}}}
\put(-9400,1980){\makebox(0,0)[lb]{\smash{{\SetFigFont{12}{14.4}{\familydefault}{\mddefault}{\updefault}$S_1(t)$}}}}
\put(-8500,2980){\makebox(0,0)[lb]{\smash{{\SetFigFont{12}{14.4}{\familydefault}{\mddefault}{\updefault}$\leq D$}}}}
\put(-8600,380){\makebox(0,0)[lb]{\smash{{\SetFigFont{12}{14.4}{\familydefault}{\mddefault}{\updefault}$\leq D$}}}}
\put(-2900,2780){\makebox(0,0)[lb]{\smash{{\SetFigFont{12}{14.4}{\familydefault}{\mddefault}{\updefault}$\Delta(D)$}}}}
\put(-3400,560){\makebox(0,0)[lb]{\smash{{\SetFigFont{12}{14.4}{\familydefault}{\mddefault}{\updefault}$\Delta(D)$}}}}
\put(-9400,1270){\makebox(0,0)[lb]{\smash{{\SetFigFont{12}{14.4}{\familydefault}{\mddefault}{\updefault}$S_2(t)$}}}}
\put(-1800,1850){\makebox(0,0)[lb]{\smash{{\SetFigFont{12}{14.4}{\familydefault}{\mddefault}{\updefault}$S_3(t)$}}}}
\put(-1600,1100){\makebox(0,0)[lb]{\smash{{\SetFigFont{12}{14.4}{\familydefault}{\mddefault}{\updefault}$S_4(t)$}}}}
\end{picture}%
\caption{On the dark gray part, we have a curvature bound of $|{\Rm}| < K$ for a universal $K$.
The diameter of the solid tori $S_1 (t), S_2(t)$ (light gray and bold) is bounded by $D$ and the curvature on $S_1(t), S_2(t)$ satisfies the bound $|{\Rm}| < K'(D)$.
The solid tori $S_3(t), S_4(t)$ (highlighted in bold) are both larger than $D$ in diameter and we can only bound the geometry of their collars $P_3 (t) \subset S_3(t), P_4(t) \subset S_4(t)$ (light gray), which have ``length'' $\Delta (D)$.\newline
In the ``first step'', we prove that $S_3(t), S_4(t)$ must also have been large at time $t-\tau$ and in the ``second step'', we extend the curvature bound $|{\Rm}| < K$ from the dark gray area to $P_3(t), P_4(t)$.
\label{fig:SiStep1Intro}}
\end{center}
\end{figure}
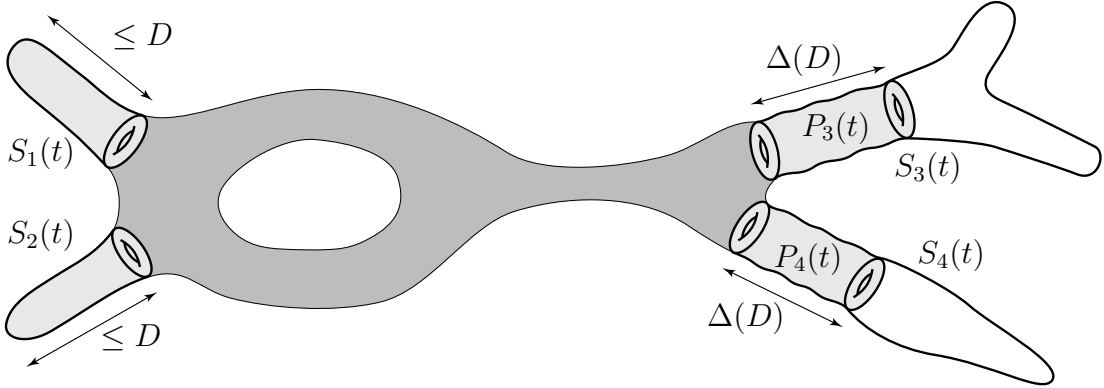
\begin{Result1}[see Lemma \ref{Lem:firstcurvboundstep1}\footnote{Note that for clarity we have altered some of the notation from Lemma \ref{Lem:firstcurvboundstep1}.}]
The curvature bound (\ref{eq:curvboundoutsideSiinintro}) holds, outside the solid tori $S_i (t)$.
Moreover, there are functions $K'(D), \Delta (D)$ with $\Delta (D) \to \infty$ as $D \to \infty$ such that for every $D < \infty$ and sufficiently large time $t$ and for each  solid torus $S_i (t)$ there are two possibilities:
\begin{enumerate}[label=(\alph*)]
\item $\diam_t S_i (t) \leq D$ and $|{\Rm}| < K'(D)$ on $S_i (t)$.
\item $\diam_t S_i (t) > D$ and there is a collar neighborhood $P_i (t) \subset S_i (t)$ of $\partial S_i (t)$ that is diffeomorphic to $T^2 \times I$, in such a way that the diameter of the $T^2$-fibers are bounded by a universal constant (we may assume this constant to be $1$) and the two boundary components of $P_i (t)$ have distance $\Delta(D)$ from one another.

Furthermore, we have a lower bound on the diameter of $S_i (t)$ at earlier times:
\begin{equation} \label{eq:noexpandingintro}
 \diam_{t'} S_i (t) > \Delta (D) \qquad \text{for all} \qquad t' \in [t- \tau, t].
\end{equation}
In other words, the solid tori $S_i (t)$ cannot grow too fast on a small time-interval.
\end{enumerate}
\end{Result1}
Note that the bound (\ref{eq:noexpandingintro}) is a non-trivial result, which will become particularly important later on.
It is a consequence of \cite[Proposition \ref{Prop:slowdiamgrowth}]{Bamler-LT-Perelman} (``Controlled diameter growth of regions whose boundary is sufficiently collapsed and good'').
The result of the first step implies that, in view of our desired curvature bound (\ref{eq:mainrescaledbound}), it suffices to focus our attention on those solid tori $S_i(t)$ of large diameter.
The largeness hereby is regulated by the parameter $D$, which we can choose at our own liking.
It is however important to choose $D$ uniformly, since the constant $K'(D)$ in part (a) may deteriorate with it. 

In the ``second step'' (see Proposition \ref{Prop:firstcurvboundstep2}), we use the fact that the solid tori $S_i(t)$ cannot expand too fast, (\ref{eq:noexpandingintro}), to localize an argument similar to the one leading to the curvature bound (\ref{eq:curvboundonMgood}).
The idea hereby is the following:
The cross-sectional $2$-tori of the collar neighborhoods $P_i (t)$ are compressible within $M$ (or even within $S_i (t)$) and hence regions within $P_i (t)$ stay collapsed when we pass to the universal cover $\td{M}$.
However, any compressing disk for these $2$-tori have to intersect $S_i (t) \setminus P_i (t)$.
So in a local cover (around points that are sufficiently deep inside $P_i (t)$) we obtain a non-collapsing result similar to (\ref{eq:noncollapsednessintro}).
The non-expanding result (\ref{eq:noexpandingintro}) ensures that a similar behavior can be observed at all times of the time-interval $[t - \tau, t]$.
This puts us in a position to apply \cite[Proposition \ref{Prop:curvboundnotnullinarea}]{Bamler-LT-Perelman} (``Curvature control in large regions that are locally good everywhere''), and deduce a uniform curvature bound at most points of $P_i (t)$.
This additional, better curvature bound, can be viewed to fit into a dichotomy, as the following summary shows (note that for the following summary we may have to adjust $S_i (t)$, $P_i(t)$) (see again Figure \ref{fig:SiStep1Intro}):

\begin{Result2}[see Proposition \ref{Prop:firstcurvboundstep2}]
There are constants $K < \infty$, $\tau > 0$ and functions $K'(D), L'(D)$ with $\Delta(D) \to \infty$ as $D \to \infty$  such that for every $D < \infty$ and sufficiently large time $t$ we still have
\[ |{\Rm}| < K \qquad \text{on} \qquad \big( M \setminus (S_1 (t) \cup \ldots \cup S_{m(t)} (t) ) \big) \times [t- \tau, t]. \]
Moreover for each solid torus $S_i (t)$ there are two possibilities:
\begin{enumerate}[label=(\alph*)]
\item $\diam_t S_i (t) \leq D$ and $|{\Rm_t}| < K'(D)$ on $S_i (t)$.
\item $\diam_t S_i (t) > D$ and there is a collar neighborhood $P_i (t) \subset S_i (t)$, diffeomorphic to $T^2 \times I$, whose boundary components have distance $\Delta (D)$ from one another and whose cross-sectional $2$-tori have diameter smaller than some universal constant (e.g. $1$).

Moreover, $|{\Rm}| < K$ on $P_i (t) \times [t-\tau, t]$.
\end{enumerate}
\end{Result2}

In the ``third'' step (see Proposition \ref{Prop:firstcurvboundstep3}) we analyze the collar neighborhoods $P_i (t)$ more closely.
Note that in the previous statement, the ``thickness'' of each $P_i(t)$, i.e. the diameters of the cross-sectional $2$-tori, are bounded by $1$.
We can replace this constant by any arbitrarily small constant, but such a replacement would necessitate a change of the constants $K, K', \tau$, a change that we have to avoid.
In order to obtain more control on this ``thickness'', we instead proceed as follows:
Assume for the moment that $M$ is not diffeomorphic to a quotient of a $2$-torus bundle over a circle and consider the simplicial complex $V$ and the continuous map $f_0 : V \to M$ as constructed in \cite{Bamler-LT-topology}.
By \cite[Proposition \ref{Prop:areaevolutioninMM}]{Bamler-LT-simpcx}, there is a constant $A_0 < \infty$, which only depends on the topology of $V$, and a time-dependent family of maps $f_t : V \to M$ such that
\[ \area_t f_t < A_0 \qquad \text{for all} \qquad t \geq 0. \]
Based on this map---and using \cite[Proposition \ref{Prop:maincombinatorialresult}(a)]{Bamler-LT-topology}---it is then possible to show the existence of a ``compressing planar domain'' $f_{i,t} : \Sigma_{i,t} \to M$ of bounded area for each solid torus $S_i (t)$.
By this we mean the following: $\Sigma_{i,t} \subset \IR^2$ is a compact smooth domain (e.g. a disk or an annulus) and $f_{i,t} : \Sigma_{i,t} \to M$ is a smooth map with $f_{i,t} (\partial \Sigma_{i,t}) \subset \partial S_i (t)$ such that $f_{i,t}$ restricted to only the outer boundary circle of $\Sigma_{i,t}$ is non-contractible in $\partial S_i (t)$.
Moreover, the time-$t$ area of $f_{i,t}$ is bounded by some constant $A_1 < \infty$ that only depends on $A_0$

The existence of these ``compressing planar domains'' of bounded area give us a way to bound the ``thickness'' of  a long section somewhere inside the collar neighborhoods $P_i (t)$ by a constant $\delta(D) \approx A_1 / \Delta(D)$.
This constant becomes arbitrarily small for large $D$.
So after reselecting $P_i(t)$ to be that section of small ``thickness'', and adjusting $S_i(t)$ and $\Delta(D)$, we obtain (compare with Figure \ref{fig:SiStep3Intro}):

\begin{figure}[t] 
\begin{center}
%\begin{picture}(0,0)%
%\hspace{16mm}\includegraphics[width=11cm]{expandingsolidtorus}%
%\end{picture}%
\setlength{\unitlength}{2863sp}%
\begingroup\makeatletter\ifx\SetFigFont\undefined%
\gdef\SetFigFont#1#2#3#4#5{%
  \reset@font\fontsize{#1}{#2pt}%
  \fontfamily{#3}\fontseries{#4}\fontshape{#5}%
  \selectfont}%
\fi\endgroup%
\begin{picture}(4500,3800)(3500,0)
\hspace{16mm}\includegraphics[width=14.5cm]{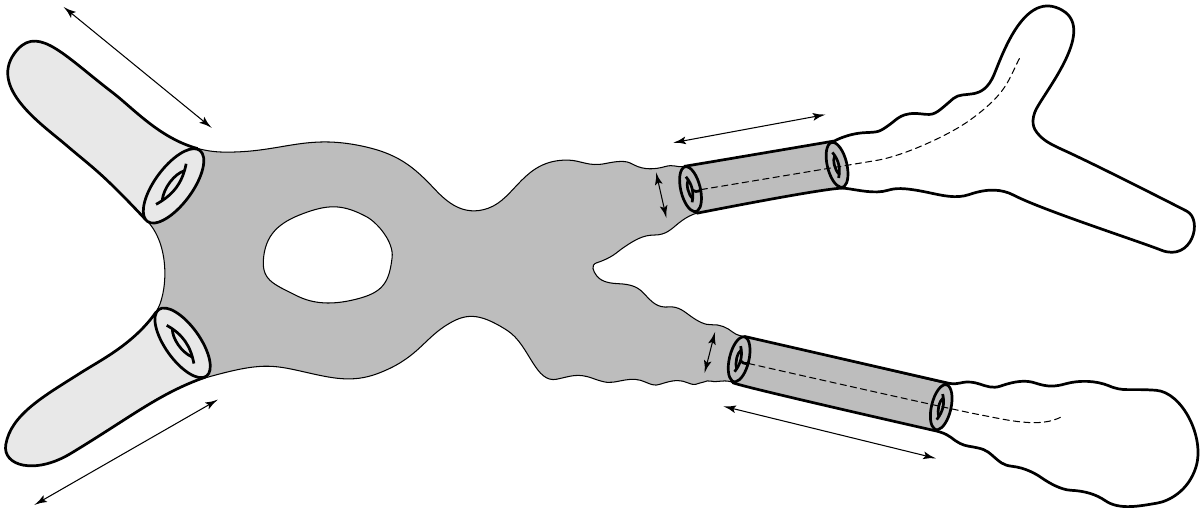}%
\put(-3600,2150){\makebox(0,0)[lb]{\smash{{\SetFigFont{12}{14.4}{\familydefault}{\mddefault}{\updefault}$P_3(t)$}}}}
\put(-5200,2350){\makebox(0,0)[lb]{\smash{{\SetFigFont{12}{14.4}{\familydefault}{\mddefault}{\updefault}$< \delta(D)$}}}}
\put(-4830,1200){\makebox(0,0)[lb]{\smash{{\SetFigFont{12}{14.4}{\familydefault}{\mddefault}{\updefault}$< \delta(D)$}}}}
\put(-3030,1330){\makebox(0,0)[lb]{\smash{{\SetFigFont{12}{14.4}{\familydefault}{\mddefault}{\updefault}$P_4(t)$}}}}
\put(-9400,2500){\makebox(0,0)[lb]{\smash{{\SetFigFont{12}{14.4}{\familydefault}{\mddefault}{\updefault}$S_1(t)$}}}}
\put(-8300,3580){\makebox(0,0)[lb]{\smash{{\SetFigFont{12}{14.4}{\familydefault}{\mddefault}{\updefault}$\leq D$}}}}
\put(-8390,270){\makebox(0,0)[lb]{\smash{{\SetFigFont{12}{14.4}{\familydefault}{\mddefault}{\updefault}$\leq D$}}}}
\put(-3900,3150){\makebox(0,0)[lb]{\smash{{\SetFigFont{12}{14.4}{\familydefault}{\mddefault}{\updefault}$\Delta(D)$}}}}
\put(-3400,330){\makebox(0,0)[lb]{\smash{{\SetFigFont{12}{14.4}{\familydefault}{\mddefault}{\updefault}$\Delta(D)$}}}}
\put(-9400,1220){\makebox(0,0)[lb]{\smash{{\SetFigFont{12}{14.4}{\familydefault}{\mddefault}{\updefault}$S_2(t)$}}}}
\put(-2300,2150){\makebox(0,0)[lb]{\smash{{\SetFigFont{12}{14.4}{\familydefault}{\mddefault}{\updefault}$S_3(t)$}}}}
\put(-1700,1200){\makebox(0,0)[lb]{\smash{{\SetFigFont{12}{14.4}{\familydefault}{\mddefault}{\updefault}$S_4(t)$}}}}
\put(-2160,2650){\makebox(0,0)[lb]{\smash{{\SetFigFont{12}{14.4}{\familydefault}{\mddefault}{\updefault}$f_{3,t}(\Sigma_{i,t})$}}}}
\put(-1190,450){\makebox(0,0)[lb]{\smash{{\SetFigFont{12}{14.4}{\familydefault}{\mddefault}{\updefault}$f_{4,t}(\Sigma_{i,t})$}}}}
\end{picture}%
\caption{In the ``third step'', we reselect the large solid tori $S_3 (t)$, $S_4(t)$ and their collars $P_3 (t)$, $P_4(t)$ such that the thickness of these collars is bounded by $\delta(D)$.
(The wavy lines next to $P_3(t), P_4(t)$ indicate where these collars used to be located in the ``second step''.)
Additionally, we find ``compressing planar domains'' $f_{3,t}, f_{4,t} : \Sigma_{i,t} \to M$ inside $S_3 (t), S_4(t)$ whose area is bounded by $A_1$.
On the dark gray region we have the curvature bound $|{\Rm}| < K$ while in the light gray region (i.e. on the small solid tori $S_1(t), S_2(t)$) we only have a curvature bound of the form $|{\Rm}| < K'(D)$.
\label{fig:SiStep3Intro}}
\end{center}
\end{figure}
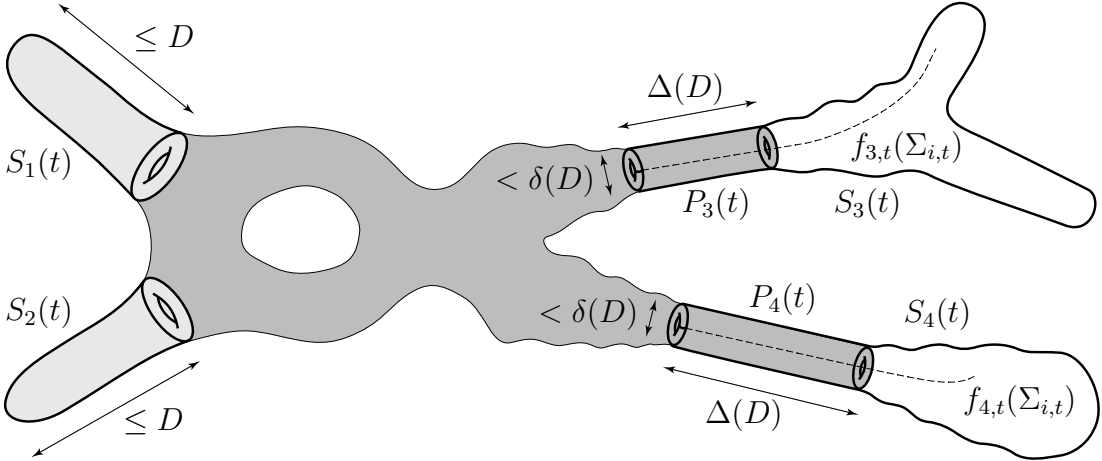
\begin{Result3}[see Proposition \ref{Prop:firstcurvboundstep3}\footnote{In the actual phrasing of Proposition \ref{Prop:firstcurvboundstep3}, the manifold is allowed to be a quotient of a torus bundle.
Instead, the existence of ``compressing planar domains'' of bounded area is assumed and not asserted.}]
The same conclusions as in the ``second step'' hold and if $M$ is not a quotient of a $2$-torus bundle over a circle, then we have in addition:

Whenever $\diam_t S_i (t) > D$, then the thickness of each $P_i(t)$ is bounded by $\delta(D)$, where $\delta(D) \to 0$ as $D \to \infty$.
Moreover for all such $S_i(t)$ there is a ``compressing planar domain'' $f_{i,t} : \Sigma_{i,t} \to M$ of area $< A_1$.
\end{Result3}

The result of the third step is a very useful characterization of the geometry at a single, and late, time-slice.
In order to prove the desired curvature bound (\ref{eq:mainrescaledbound}), it now remains to show that those solid tori $S_i (t)$ of large diameter ($>D$) cannot occur for large times $t$ and for an appropriate and uniform choice of $D$.
This fact will follow from a contradiction argument, which will be carried out in subsections \ref{subsec:lateandlongtimei} and \ref{subsec:finalargument}.

In subsection \ref{subsec:lateandlongtimei} we first analyze the evolution of such large solid tori $S_i(t)$ when going backwards in time.
More specifically, we will fix some large number $\Delta T < \infty$ and consider the flow on the time-interval $[t-\Delta T, t]$ for $t \gg \Delta T$.
We can then apply the result of the ``third step'' at every time $t' \in [t- \Delta T, t]$ and obtain different collections of solid tori $S_1 (t'), \ldots, S_{m(t')} (t') \subset M$.
These solid tori, for different $t'$, will then be related to one another.
Their ``long and thin'' collars $P_i (t')$ are very useful in this process in order to keep track of each $S_i(t')$ in time and to show that sufficiently large solid tori $S_i (t)$ persist when going backwards in time up to time $t- \Delta T$.
Using some technical Lemmas it is moreover possible to modify the ``planar domains'' $f_{i,t - \Delta T} : \Sigma_{i,t - \Delta T} \to M$, obtained in the ``third step'' at time $t - \Delta T$, and construct ``generalized compressing disks'' $h_{i, t - \Delta T} : D^2 \to M$ for the solid tori $S_i (t - \Delta T)$ whose area is bounded and whose boundary loop is small and sufficiently regular.
By this we mean the following: $h_{i, t- \Delta T} (\partial D^2) \subset P_i ( t- \Delta T)$ is non-contractible in $P_i (t - \Delta T)$ and its area is bounded by some constant $A_2$ that only depends on $A_1$.
Moreover, the boundary loop $\gamma_{i, t- \Delta T} = h_{i, t- \Delta T} |_{\partial D^2}$ can be assumed to be arbitrarily small and its geodesic curvatures are bounded by some universal constant.
By the correct choice of $P_i (t - \Delta T)$, we can guarantee that the curvature around $\gamma_{i, t- \Delta T} (S^1)$ is bounded by some universal constant.
So the geodesic characterizations of $\gamma_{i,t-\Delta T}$ still hold on the time-interval $[t-\Delta T, t]$, where the bound on the geodesic curvature now has to depend on $\Delta T$.
After possibly readjusting the $S_i (t)$, we can then summarize our result as follows:

\begin{Result4}[see Proposition \ref{Prop:structontimeinterval}]
Assume that $M$ is not a quotient of a $2$-torus bundle over a circle.

Then there are constants $A_2, K < \infty$ and for every $\delta > 0$ and $\Delta T < \infty$ there are constants $D(\delta, \Delta T), \Gamma (\Delta T) < \infty$ such that for sufficiently large $t$ the following holds:
\[ |{\Rm_t}| < K \qquad \text{on} \qquad M \setminus (S_1 (t) \cup \ldots \cup S_{m(t)} (t) ). \]
Moreover, for each $S_i (t)$ there are two possibilities:
\begin{enumerate}[label=(\alph*)]
\item $\diam_t S_i (t) \leq D(\delta, \Delta T)$ and $|{\Rm_t}| < K'(D( \delta, \Delta T))$ on $S_i (t)$.
Here $K'$ is the constant from the ``second step''.
\item $\diam_t S_i (t) > D (\delta, \Delta T)$ and the following holds:
$\diam_{t'} \partial S_i (t) < \delta$ for all $t' \in [t-\Delta T, t]$.
And there is a smooth map $h_{i, t-\Delta T} : D^2 \to M$ such that $h_{i, t- \Delta T} (\partial D^2) \subset P_i (t)$, $\gamma_{i, t- \Delta T} = h_{i, t-\Delta T} |_{\partial D^2}$ is non-contractible in $P_i (t)$, has time-$t'$ length $< \delta$ and time-$t'$ geodesic curvature $< \Gamma( \Delta T)$ for all $t' \in [t-\Delta T, t]$ and $\area_{t-\Delta T} h_{i, t-\Delta T} < A_2$.
\end{enumerate}
\end{Result4}

Finally, in subsection \ref{subsec:finalargument} we show that the second possibility cannot occur, for the right choices of $\delta$ and $\Delta T$.
This is done using a minimal disk argument, which is due to Hamilton (compare with \cite[Proposition \ref{Prop:evolminsurfgeneral}]{Bamler-LT-simpcx}) and which we will briefly describe as follows:
At every time $t' \in [t - \Delta T, t]$, we can find the minimal disk in $(M, \td{g}_t)$ that is bounded by the loop $\gamma_{i, t- \Delta T}$.
Denote the time-$t'$ area of this disk by $A(t')$.
Then $A(t-\Delta T) < A_2$ and $A(t')$ satisfies a differential inequality, which forces $A(t')$ to become zero somewhere on $[t-\Delta T, t]$.
This is however impossible.
The bound on the length and the geodesic curvature of $\gamma_{i, t-\Delta T}$ at all times of $[t-\Delta T, t]$ is essential in this argument in order to keep certain error terms controlled.
Hence, for the right choices of $\delta$ and $\Delta T$ and for large enough $t$, only case (a) in the previous statement can occur and we have $|{\Rm_t}| < \max \{ K, K' (D(\delta, \Delta T)) \}$ on all of $M$.
This finishes the proof of the desired curvature bound (\ref{eq:mainrescaledbound}) in the case in which $M$ is not a quotient of a $2$-torus bundle over a circle.

In the case in which $M$ is a quotient of a $2$-torus bundle over a circle, we have to use a different argument (which resembles an argument used in \cite{Bamler-certain-topologies}).
Using part (b) of \cite[Proposition \ref{Prop:maincombinatorialresult}]{Bamler-LT-topology} together with \cite{Bamler-LT-simpcx}, we find a simplicial complex $V$ and a sequence of time-dependent maps $f_{1,t}, f_{2,t}, \ldots : V \to M$, with the same domain, such that:
For all $n \geq 1$ the image of $f_n$ has to intersect every incompressible loop $\sigma \subset M$ at least $n$ times.
And for sufficiently large $t$ (depending on $n$) the time-$t$ area of $f_{n,t}$ is bounded by a constant $A_0$ that only depends on $V$ (and not on $n$!).
Thus, the image of each $f_n$ has to intersect each $S^1$-fiber in $V_2(t) \cap M_{\textnormal{good}} (t)$ (see (\ref{eq:MthinV1V2Intro}) and (\ref{eq:goodbaddec})) at least $n$ times.
So the $2$-dimensional space towards which the components of $V_2(t) \cap M_{\textnormal{good}} (t)$ locally collapse---at scale $\rho_1 (x,t) \approx 1$---have to have area $< A_0 / n$ for some arbitrary $n$.
This is impossible for large enough $n$.
So $V_2(t) \cap M_{\textnormal{good}} (t) = \emptyset$.
Using our knowledge on the topology of the decomposition (\ref{eq:goodbaddec}), compare also with the list (a)--(d) on page \pageref{list:topologygoodbad}, we conclude that $V_2 (t) = \emptyset$.
This implies that $M = M_{\textnormal{thin}}(t) = V_1 (t) = M_{\textnormal{good}} (t)$ and hence (by (\ref{eq:curvboundonMgood})) we obtain the desired curvature bound on all of $M$.

Upon first reading we recommend to consider the case in which $\MM$ is non-singular.
The proof in the general case follows along the lines, but the existence of surgeries adds a number of technical difficulties.

\section{The analysis of the collapsed part and consequences} \label{sec:thinpart}
Based on property (e) of \cite[Proposition \ref{Prop:thickthindec}]{Bamler-LT-Perelman} we can analyze the thin part $\MM_{\thin}(t)$ for large times $t$ and recover its graph structure geometrically.
More specifically, we can decompose the thin part into pieces on which the collapse can be approximated by certain models.
We describe this decomposition in the first subsection.
After that we establish important geometric and topological consequences of it.

\subsection{Analysis of the collapse} \label{subsec:MorganTian}
The following result, Proposition \ref{Prop:MorganTianMain}, follows from the work of Morgan and Tian (\cite{MorganTian}).
We have altered its phrasing to include more geometric information.
After stating the proposition, we will explain how each of its assertions follows from the work of Morgan and Tian.
Similar results have also been obtained in \cite{KLcollapse}, \cite{BBMP2}, \cite{ShioyaYamaguchi}, \cite{CG} and \cite{Faessler}.

We first summarize the content of Proposition \ref{Prop:MorganTianMain}.
This summary follows \cite[sec 6]{Bamler-certain-topologies}.
Consider a Riemannian $3$-manifold $(M,g)$ with boundary.
As already mentioned in section \ref{sec:IntroductionpartD}, we define the scalar function $\rho_1$ on $M$ as follows.
\begin{equation} \label{eq:defofrho1inpartD}
 \rho_1 (x) = \sup \big\{ r \in (0,1] \;\; : \;\; \sec \geq - r^{-2} \; \text{on} \; B(x, r) \big\},
\end{equation}
The function $\rho_1(x)$ gives us a local scale at which we observe a collapse with lower sectional curvature bound.

We will first impose assumptions on $(M,g)$ that are satisfied by the rescaled metric on the thin part $(\MM_{\textnormal{thin}} (t), t^{-1} g(t))$, as defined in \cite[Proposition \ref{Prop:thickthindec}]{Bamler-LT-Perelman}.
The main assumption is that $(M,g)$ is locally collapsed at scale $\rho_1(x)$, i.e. for some small $w_0 > 0$ and for all $x \in M$ for which $B(x, \rho_1(x)) \subset \Int M$ we have
\[ \vol B(x, \rho_1(x)) < w_0 \rho_1^3 (x) . \]
Furthermore, we assume that the curvature of $(M,g)$ is bounded if we pass to smaller scales on which $(M,g)$ is non-collapsed.
And finally, we impose geometric conditions on collar neighborhoods of the boundary components of $(M,g)$, which are natural to the setting of \cite[Proposition \ref{Prop:thickthindec}]{Bamler-LT-Perelman}.

The conclusions of Proposition \ref{Prop:MorganTianMain} help us understand both the global topological structure of the collapse on $(M,g)$ as well as its approximate local geometric properties.
Before explaining these conclusions, it is helpful to consider the case in which $(M,g)$ is collapsed with a \emph{global} lower bound on the sectional curvature.
In this case $(M,g)$ is collapsed to either a point, a $1$-dimensional or a $2$-dimensional space.
The following examples illustrate different collapsing behaviors in this setting:
\begin{enumerate}[label=(\arabic*), start=0]
\item In the case in which $(M,g)$ is collapsed to a point, $M$ has to be closed and we speak of a \emph{total collapse}.
Examples for such a behavior would be a small $3$-sphere, a small $3$-torus or a small nilmanifold.
\item A collapse to a $1$-dimensional space generically occurs along $2$-dimensional fibers, which can be either spheres or tori.
For example, the Cartesian products $S^2 \times \IR$ (collapse along spheres) and $T^2 \times \IR$ (collapse along tori) with small first factor are each collapsed to a line.
The $\IZ_2$ quotients of these examples, $\IR P^2 \td\times \IR$ and $\Klein \td\times \IR$, are each collapsed along spheres or tori to a ray.
Note that in the latter two examples, $M$ is only fibered by spheres or tori on a generic subset, away from an embedded $\IR P^2$ or Klein bottle, where the fibration degenerates.

Such a fibration by spheres or tori does not always degenerate along an embedded hypersurface, as the next example illustrates:
Consider a $2$-dimensional, rotationally symmetric surface of positive curvature that has only one end and that is asymptotic to a thin cylinder.
The Cartesian product of this surface with a small $S^1$-factor is collapsed along tori to a ray.
These tori are products of concentric circles around the tip of the surface with the $S^1$-factor, and they degenerate to a circle over the tip of the surface.
Note that in this example $M$ is diffeomorphic to an open solid torus $S^1 \times B^2$.
In a similar way we can construct metrics on $B^3$ that are collapsed to a ray along $2$-spheres, which spatially degenerate to a point.
\item A collapse to a $2$-dimensional space generically occurs along $S^1$-fibers.
Basic examples for such a collapse would be Cartesian products $S^1 \times \IR^2$, with small $S^1$-factor, or $S^1 \times \Sigma^2$, where the $S^1$-factor is small and $\Sigma^2$ is a Riemannian surface whose curvature is bounded from below.
More generally, we can construct collapsing metrics on $S^1$-fibrations over such surfaces $\Sigma^2$.

Note that, similarly as in the previous case, $M$ might only be fibered by $S^1$-fibers on a generic subset of $M$.
For example, if $(M,g)$ is the quotient of $S^1 \times \IR^2$ by a cyclic subgroup that acts as non-trivial rotations around $(0, 0)$ on $\IR^2$ and as rotations on $S^1$, then $M$ only possesses such a fibration away from the quotient of $S^1 \times \{ (0,0) \}$, which is a singular fiber.
In this example $(M,g)$ is collapsed to a cone and the tip of this cone corresponds to the quotient of $S^1 \times \{ (0,0) \}$.

We point out another example in which the fibration on $M$ is degenerate.
Consider again a thin $2$-dimensional, rotationally symmetric surface of positive curvature that is asymptotic to a thin cylinder and take a Cartesian product with $\IR$.
This space is collapsed along $S^1$-fibers to a half plane.
The $S^1$-fibers correspond to concentric circles on the surface.
Hence the $S^1$-fibration only exists away from an embedded line or an embedded solid cylinder.
\end{enumerate}

In the setting of Proposition \ref{Prop:MorganTianMain}, $(M,g)$ is only \emph{locally} collapsed at scale $\rho_1(x)$ around every point $x \in M$.
So the examples given above for the case of the global collapse now only serve as models for these local collapses.
One of the main difficulties in the proof of Proposition \ref{Prop:MorganTianMain} is to understand the transition between those different models, which describe the metric at different scales, and to patch together the induced topological structures on their overlaps.
For example it is possible that a region that is modeled on $T^2 \times I$ is adjacent to a region that collapses along $S^1$-fibers towards a surface $\Sigma^2$.
The change from the first collapsing type to the second can be understood as follows:
The region that collapses along an $S^1$-fibration has a boundary component that is diffeomorphic to a $2$-torus; one $S^1$-direction in this torus corresponds to the collapsing $S^1$-fiber and the other direction corresponds to a boundary circle of the base $\Sigma^2$.
The $S^1$-fiber direction is very small and the boundary circle of $\Sigma^2$ is large enough such that $\Sigma^2$ is sufficiently non-collapsed and we can observe a collapse to a $2$-dimensional space in the interchanging region.
On the other hand, the boundary circle is also small enough such that the same region exhibits a collapse along $T^2$-fibers towards a $1$-dimensional space at the same time.
(So $\Sigma^2$ looks cusp-like around its boundary.)
The precise definition of what we mean by a ``collapse to a $2$-dimensional space'' or a ``collapse to a $1$-dimensional space'' includes a sufficient amount of play such that the collapse in the interchanging region can be modeled by $T^2 \times I$ as well as by an $S^1$-fibration.

Our analysis of $(M,g)$ depends on a parameter $\mu > 0$, which can be chosen arbitrarily small and which governs how well $(M,g)$ is approximated by these local models.
Based on this parameter, we choose scales $0 < s_2(\mu) < s_1(\mu)$.
We will expect to observe a collapse to a $2$-dimensional space at scale $s_2 \rho_1 (x)$, and a collapse to a $1$-dimensional space at scale $s_1 \rho_1 (x)$.
This choice produces the desired play, since a region that is collapsed to a $1$-dimensional space can be collapsed to a $2$-dimensional space at a smaller scale. 

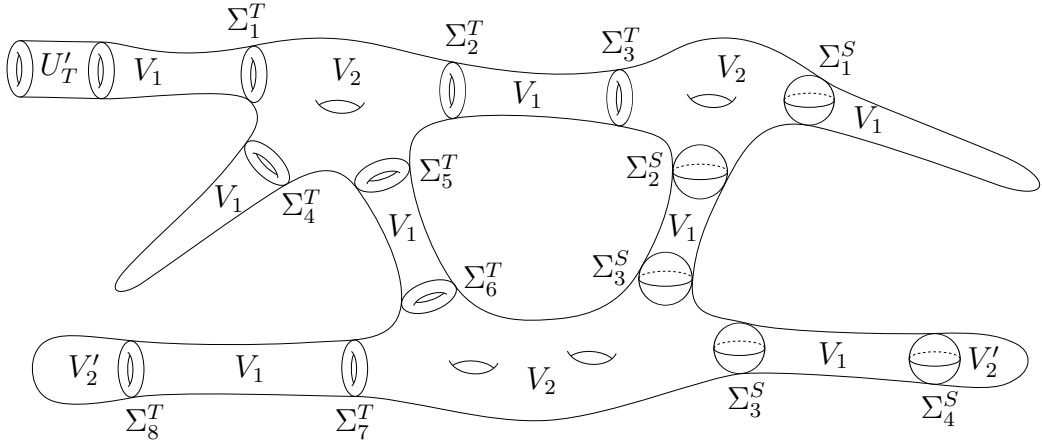
\begin{figure}[t] 
\begin{center}
%\begin{picture}(0,0)%
%\hspace{16mm}\includegraphics[width=11cm]{expandingsolidtorus}%
%\end{picture}%
\setlength{\unitlength}{2863sp}%
\begingroup\makeatletter\ifx\SetFigFont\undefined%
\gdef\SetFigFont#1#2#3#4#5{%
  \reset@font\fontsize{#1}{#2pt}%
  \fontfamily{#3}\fontseries{#4}\fontshape{#5}%
  \selectfont}%
\fi\endgroup%
\begin{picture}(4500,3800)(3500,0)
\hspace{18mm}\includegraphics[width=13.8cm]{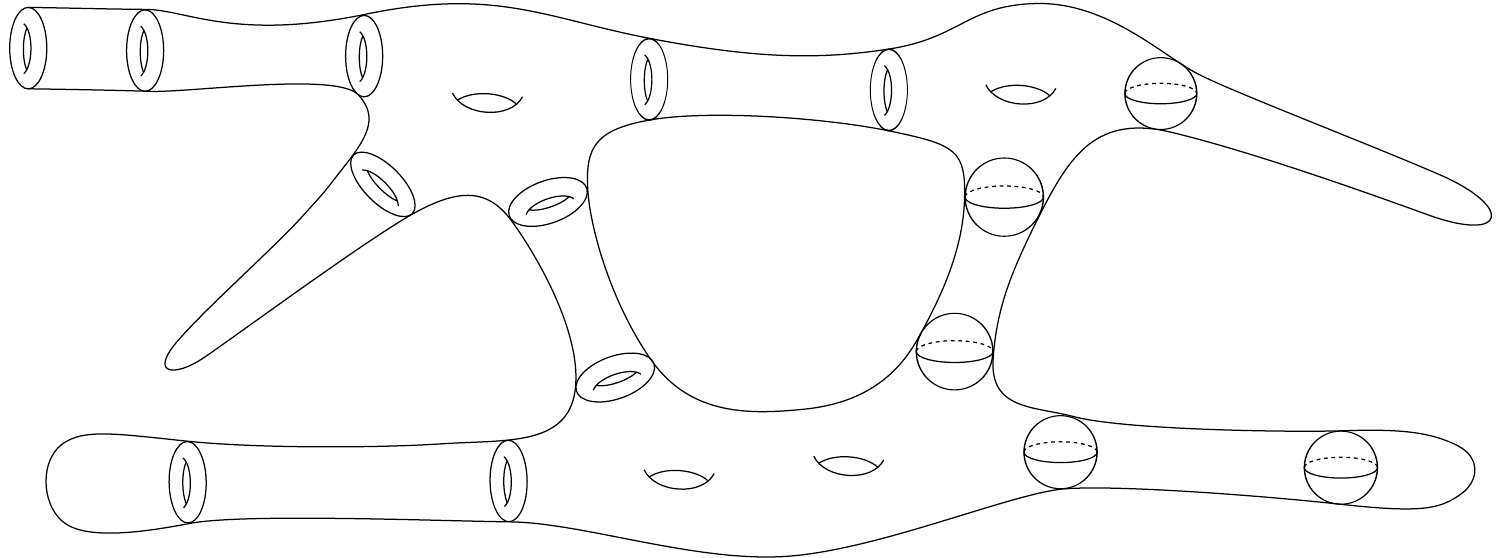}%
\put(-8650,2990){\makebox(0,0)[lb]{\smash{{\SetFigFont{12}{14.4}{\familydefault}{\mddefault}{\updefault}$U'_T$}}}}
\put(-7850,2950){\makebox(0,0)[lb]{\smash{{\SetFigFont{12}{14.4}{\familydefault}{\mddefault}{\updefault}$V_1$}}}}
\put(-6150,2950){\makebox(0,0)[lb]{\smash{{\SetFigFont{12}{14.4}{\familydefault}{\mddefault}{\updefault}$V_2$}}}}
\put(-4600,2750){\makebox(0,0)[lb]{\smash{{\SetFigFont{12}{14.4}{\familydefault}{\mddefault}{\updefault}$V_1$}}}}
\put(-1680,2530){\makebox(0,0)[lb]{\smash{{\SetFigFont{12}{14.4}{\familydefault}{\mddefault}{\updefault}$V_1$}}}}
\put(-2850,2970){\makebox(0,0)[lb]{\smash{{\SetFigFont{12}{14.4}{\familydefault}{\mddefault}{\updefault}$V_2$}}}}
\put(-3270,1600){\makebox(0,0)[lb]{\smash{{\SetFigFont{12}{14.4}{\familydefault}{\mddefault}{\updefault}$V_1$}}}}
\put(-1980,3060){\makebox(0,0)[lb]{\smash{{\SetFigFont{12}{14.4}{\familydefault}{\mddefault}{\updefault}$\Sigma^S_1$}}}}
\put(-7050,3380){\makebox(0,0)[lb]{\smash{{\SetFigFont{12}{14.4}{\familydefault}{\mddefault}{\updefault}$\Sigma^T_1$}}}}
\put(-5180,3240){\makebox(0,0)[lb]{\smash{{\SetFigFont{12}{14.4}{\familydefault}{\mddefault}{\updefault}$\Sigma^T_2$}}}}
\put(-3840,3180){\makebox(0,0)[lb]{\smash{{\SetFigFont{12}{14.4}{\familydefault}{\mddefault}{\updefault}$\Sigma^T_3$}}}}
\put(-5670,1600){\makebox(0,0)[lb]{\smash{{\SetFigFont{12}{14.4}{\familydefault}{\mddefault}{\updefault}$V_1$}}}}
\put(-700,450){\makebox(0,0)[lb]{\smash{{\SetFigFont{12}{14.4}{\familydefault}{\mddefault}{\updefault}$V'_2$}}}}
\put(-1980,500){\makebox(0,0)[lb]{\smash{{\SetFigFont{12}{14.4}{\familydefault}{\mddefault}{\updefault}$V_1$}}}}
\put(-2760,110){\makebox(0,0)[lb]{\smash{{\SetFigFont{12}{14.4}{\familydefault}{\mddefault}{\updefault}$\Sigma_3^S$}}}}
\put(-1100,30){\makebox(0,0)[lb]{\smash{{\SetFigFont{12}{14.4}{\familydefault}{\mddefault}{\updefault}$\Sigma_4^S$}}}}
\put(-7160,1850){\makebox(0,0)[lb]{\smash{{\SetFigFont{12}{14.4}{\familydefault}{\mddefault}{\updefault}$V_1$}}}}
\put(-6580,1780){\makebox(0,0)[lb]{\smash{{\SetFigFont{12}{14.4}{\familydefault}{\mddefault}{\updefault}$\Sigma^T_4$}}}}
\put(-5400,2080){\makebox(0,0)[lb]{\smash{{\SetFigFont{12}{14.4}{\familydefault}{\mddefault}{\updefault}$\Sigma^T_5$}}}}
\put(-5020,1150){\makebox(0,0)[lb]{\smash{{\SetFigFont{12}{14.4}{\familydefault}{\mddefault}{\updefault}$\Sigma^T_6$}}}}
\put(-3630,2100){\makebox(0,0)[lb]{\smash{{\SetFigFont{12}{14.4}{\familydefault}{\mddefault}{\updefault}$\Sigma_2^S$}}}}
\put(-3920,1250){\makebox(0,0)[lb]{\smash{{\SetFigFont{12}{14.4}{\familydefault}{\mddefault}{\updefault}$\Sigma_3^S$}}}}
\put(-4480,280){\makebox(0,0)[lb]{\smash{{\SetFigFont{12}{14.4}{\familydefault}{\mddefault}{\updefault}$V_2$}}}}
\put(-6980,380){\makebox(0,0)[lb]{\smash{{\SetFigFont{12}{14.4}{\familydefault}{\mddefault}{\updefault}$V_1$}}}}
\put(-8410,380){\makebox(0,0)[lb]{\smash{{\SetFigFont{12}{14.4}{\familydefault}{\mddefault}{\updefault}$V'_2$}}}}
\put(-6120,-80){\makebox(0,0)[lb]{\smash{{\SetFigFont{12}{14.4}{\familydefault}{\mddefault}{\updefault}$\Sigma_7^T$}}}}
\put(-7920,-80){\makebox(0,0)[lb]{\smash{{\SetFigFont{12}{14.4}{\familydefault}{\mddefault}{\updefault}$\Sigma_8^T$}}}}%
\end{picture}%
\caption{A decomposition of $M$ into $V_1, V_2$ and $V'_2$ along embedded $2$-tori $\Sigma_1^T, \ldots, \Sigma_8^T$ and embedded $2$-spheres $\Sigma_1^S, \ldots, \Sigma_4^S$.
On the collar neighborhood $U'_T$ around the boundary torus of $M$, we impose several geometric conditions.
\label{fig:V1V2}}
\end{center}
\end{figure}
We will now outline the assertions of our proposition more precisely.
If $M$ is locally collapsed to a point, then $M$ must be closed and the collapse must be global.
This case is very well understood.
So assume that $M$ is not collapsed to a point.
In this case we decompose $M$ into three subsets $V_1, V_2$ and $V'_2$ (see Figure \ref{fig:V1V2}).
\begin{equation} \label{eq:V1V2decinexplanation} 
M = V_1 \cup V_2 \cup V'_2
\end{equation}
The subset $V_1$ roughly consists of all points around which we observe a local collapse to a $1$-dimensional space at scale $s_1 \rho_1$, e.g. $V_1$ contains all points whose local models are mentioned in (1) of the preceding list.
On the subset $V_2$ we observe local collapses to $2$-dimensional spaces and the geometry is locally modeled at scale $s_2 \rho_1$, for example, by spaces mentioned in (2) of the preceding list.
The subset $V'_2$ has the following properties:
On a neighborhood around each point $x \in V'_2$ we observe a local collapse to a half-open interval at scale $s_1 \rho_1(x)$, but there is a scale $r \ll \rho_1(x)$ at which we observe a collapse to a $2$-dimensional space.
The subsets $V_1, V_2, V'_2$ are separated from one another by embedded $2$-tori, which are denoted by $\Sigma^T_i$, and embedded $2$-spheres, which are denoted by $\Sigma^S_i$.
Note that the decomposition (\ref{eq:V1V2decinexplanation}) is not unique, even for fixed $\mu$.

\begin{figure}[t] 
\begin{center}
%\begin{picture}(0,0)%
%\hspace{16mm}\includegraphics[width=11cm]{expandingsolidtorus}%
%\end{picture}%
\setlength{\unitlength}{2863sp}%
\begingroup\makeatletter\ifx\SetFigFont\undefined%
\gdef\SetFigFont#1#2#3#4#5{%
  \reset@font\fontsize{#1}{#2pt}%
  \fontfamily{#3}\fontseries{#4}\fontshape{#5}%
  \selectfont}%
\fi\endgroup%
\begin{picture}(4500,3800)(3500,0)
\hspace{23mm}\includegraphics[width=13.8cm]{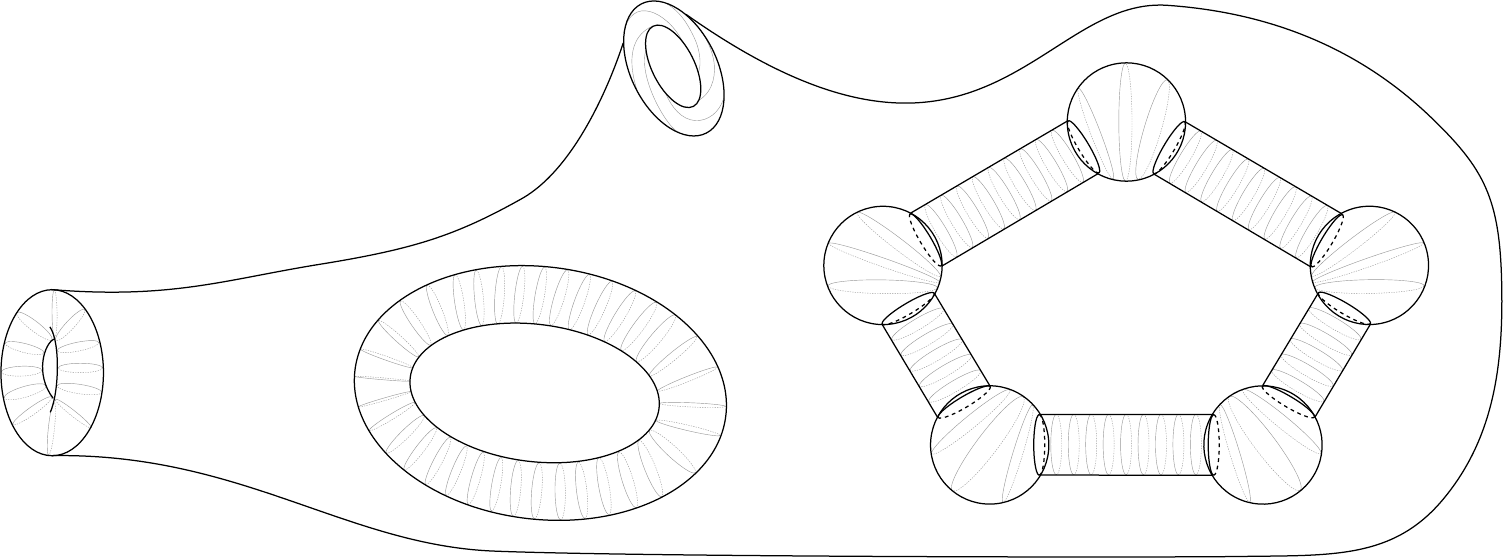}%
\put(-9330,1400){\makebox(0,0)[lb]{\smash{{\SetFigFont{12}{14.4}{\familydefault}{\mddefault}{\updefault}$\Sigma^T_1$}}}}
\put(-5480,2000){\makebox(0,0)[lb]{\smash{{\SetFigFont{12}{14.4}{\familydefault}{\mddefault}{\updefault}$V_{2, \textnormal{reg}}$}}}}
\put(-2840,1790){\makebox(0,0)[lb]{\smash{{\SetFigFont{12}{14.4}{\familydefault}{\mddefault}{\updefault}$V_{2, \partial}$}}}}
\put(-1980,1790){\makebox(0,0)[lb]{\smash{{\SetFigFont{12}{14.4}{\familydefault}{\mddefault}{\updefault}$V_{2, \partial}$}}}}
\put(-3100,1270){\makebox(0,0)[lb]{\smash{{\SetFigFont{12}{14.4}{\familydefault}{\mddefault}{\updefault}$V_{2, \partial}$}}}}
\put(-1780,1270){\makebox(0,0)[lb]{\smash{{\SetFigFont{12}{14.4}{\familydefault}{\mddefault}{\updefault}$V_{2, \partial}$}}}}
\put(-2480,1000){\makebox(0,0)[lb]{\smash{{\SetFigFont{12}{14.4}{\familydefault}{\mddefault}{\updefault}$V_{2, \partial}$}}}}
\put(-3300,2490){\makebox(0,0)[lb]{\smash{{\SetFigFont{12}{14.4}{\familydefault}{\mddefault}{\updefault}$\Xi_1^A$}}}}
\put(-1420,2420){\makebox(0,0)[lb]{\smash{{\SetFigFont{12}{14.4}{\familydefault}{\mddefault}{\updefault}$\Xi_5^A$}}}}
\put(-3970,950){\makebox(0,0)[lb]{\smash{{\SetFigFont{12}{14.4}{\familydefault}{\mddefault}{\updefault}$\Xi_2^A$}}}}
\put(-850,970){\makebox(0,0)[lb]{\smash{{\SetFigFont{12}{14.4}{\familydefault}{\mddefault}{\updefault}$\Xi_4^A$}}}}
\put(-2380,200){\makebox(0,0)[lb]{\smash{{\SetFigFont{12}{14.4}{\familydefault}{\mddefault}{\updefault}$\Xi_3^A$}}}}
\put(-7980,1100){\makebox(0,0)[lb]{\smash{{\SetFigFont{12}{14.4}{\familydefault}{\mddefault}{\updefault}$V_{2, \textnormal{reg}}$}}}}
\put(-4880,3340){\makebox(0,0)[lb]{\smash{{\SetFigFont{12}{14.4}{\familydefault}{\mddefault}{\updefault}$V_{2, \textnormal{cone}}$}}}}
\put(-4600,2600){\makebox(0,0)[lb]{\smash{{\SetFigFont{12}{14.4}{\familydefault}{\mddefault}{\updefault}$\Xi^T_1$}}}}
\put(-5650,750){\makebox(0,0)[lb]{\smash{{\SetFigFont{12}{14.4}{\familydefault}{\mddefault}{\updefault}$V_{2, \partial}$}}}}
\put(-4980,130){\makebox(0,0)[lb]{\smash{{\SetFigFont{12}{14.4}{\familydefault}{\mddefault}{\updefault}$\Xi^O_1$}}}}
\put(-4030,2230){\makebox(0,0)[lb]{\smash{{\SetFigFont{12}{14.4}{\familydefault}{\mddefault}{\updefault}$\Sigma_2^S$}}}}
\put(-4410,1900){\makebox(0,0)[lb]{\smash{{\SetFigFont{12}{14.4}{\familydefault}{\mddefault}{\updefault}$\Xi_2^E$}}}}
\put(-3830,500){\makebox(0,0)[lb]{\smash{{\SetFigFont{12}{14.4}{\familydefault}{\mddefault}{\updefault}$\Sigma_3^S$}}}}
\put(-3560,110){\makebox(0,0)[lb]{\smash{{\SetFigFont{12}{14.4}{\familydefault}{\mddefault}{\updefault}$\Xi_3^E$}}}}
\put(-1020,500){\makebox(0,0)[lb]{\smash{{\SetFigFont{12}{14.4}{\familydefault}{\mddefault}{\updefault}$\Sigma_4^S$}}}}
\put(-1280,130){\makebox(0,0)[lb]{\smash{{\SetFigFont{12}{14.4}{\familydefault}{\mddefault}{\updefault}$\Xi_4^E$}}}}
\put(-780,2180){\makebox(0,0)[lb]{\smash{{\SetFigFont{12}{14.4}{\familydefault}{\mddefault}{\updefault}$\Sigma_5^S$}}}}
\put(-380,1680){\makebox(0,0)[lb]{\smash{{\SetFigFont{12}{14.4}{\familydefault}{\mddefault}{\updefault}$\Xi_5^E$}}}}
\put(-2330,3030){\makebox(0,0)[lb]{\smash{{\SetFigFont{12}{14.4}{\familydefault}{\mddefault}{\updefault}$\Xi^E_1$}}}}
\put(-1900,2860){\makebox(0,0)[lb]{\smash{{\SetFigFont{12}{14.4}{\familydefault}{\mddefault}{\updefault}$\Sigma^S_1$}}}}
\end{picture}%
\caption{An example for a component of $V_2$.
This component has 6 boundary components: one toroidal component $\Sigma^T_1$ and 5 spherical components $\Sigma_1^S, \ldots, \Sigma_5^S$.
The spherical boundary components are connected by 5 components of $V_{2, \partial}$, which are diffeomorphic to solid cylinders $I \times D^2$.
Their annular boundary parts are denoted by $\Xi^A_1, \ldots, \Xi^A_5$.
The boundary circles of each of these annuli lie in the $\Sigma^S_i$ and bound equatorial annuli $\Xi^E_1 \subset \Sigma^S_1, \ldots, \Xi^E_5 \subset \Sigma^S_5$ within these boundary spheres.
The subset $V_{2, \partial}$ also contains a component that is diffeomorphic to a solid torus and bounded by a $2$-torus $\Xi^O_1$.
The subset $V_{2, \textnormal{cone}}$ consists of a single solid torus, which is bounded by a $2$-torus $\Xi_1^T$.
The closure of the complement of $V_{2, \textnormal{cone}} \cup V_{2, \partial}$ is denoted by $V_{2, \textnormal{reg}}$ and carries an $S^1$-fibration.
Thin gray circles illustrate the behavior of this fibration on the boundary components of $V_{2, \textnormal{reg}}$.
\label{fig:V2}}
\end{center}
\end{figure}
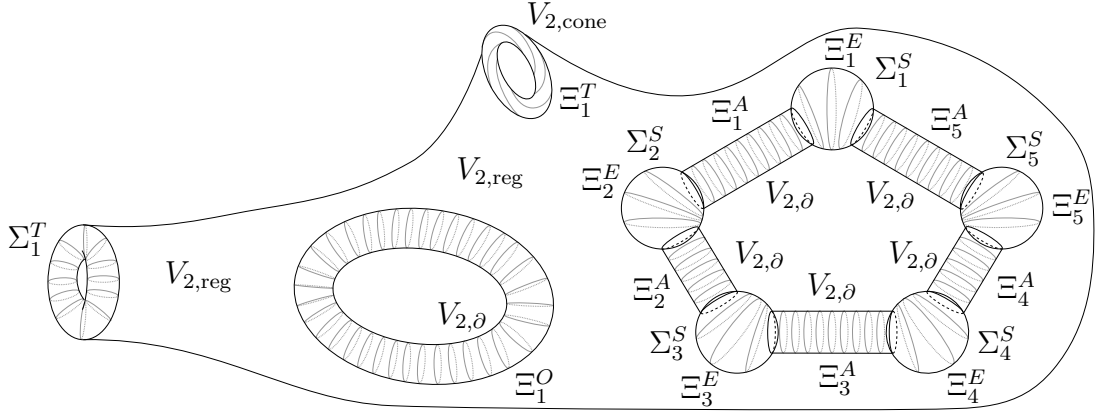
The topology of the components of $V_1$ and $V'_2$ is very controlled and can be classified easily.
In order to understand the topology and local geometry of $V_2$, we decompose $V_2$ into three subsets $V_{2, \textnormal{reg}}, V_{2, \textnormal{cone}}$ and $V_{2, \partial}$ (see Figure \ref{fig:V2}).
Roughly speaking, $V_{2, \textnormal{reg}}$ is the set of all points where the collapse is modeled on the example $S^1 \times \IR^2$ from (2) of the preceding list.
Hence this subset admits an $S^1$-fibration.
The set $V_{2, \textnormal{cone}}$ consists of approximately all points whose local model is a finite quotient of $S^1 \times B^2$ as described in (2) of the above list.
Around the points of this subset, the manifold is collapsed to a cone.
Note that since a cone is regular away from its tip, the components of $V_{2, \textnormal{cone}}$ have bounded diameter and are adjacent to $V_{2, \textnormal{reg}}$.
It can moreover be shown that the components of $V_{2, \textnormal{cone}}$ are diffeomorphic to a solid torus $S^1 \times D^2$ and hence bounded by $2$-tori, which we will denote by $\Xi_i^T$.

The set $V_{2, \partial}$ consists of all points whose neighborhoods are collapsed towards a $2$-dimensional space with boundary.
An example for a local model around such points would be the one involving the surface that is asymptotic to a thin cylinder in (2) of the preceding list (recall that this model is collapsed to a half plane).
It is possible to choose $V_{2, \partial}$ such that its components are either diffeomorphic to a solid cylinder $I \times D^2$ or a solid torus $S^1 \times D^2$ in such a way that the boundary circles of the $D^2$-factors correspond to the $S^1$-fibers of $V_{2, \textnormal{reg}}$.
The components that are diffeomorphic to a solid torus are bounded by $2$-tori, which we denote by $\Xi^O_i$.
Each component that is diffeomorphic to a solid cylinder is positioned within $V_2$ in such a way that its two diskal boundary parts are contained in spherical boundary components $\Sigma^S_i$ of $V_2$.
That means that if we denote their annular boundary components by $\Xi^A_i$, then the boundary circles of each $\Xi^A_i$ lie in the spherical boundary components of $V_2$.
Each spherical boundary component $\Sigma_{i'}^S$ of $V_2$ contains exactly two diskal boundary parts of components of $V_{2, \partial}$ or, in other words, two boundary circles of the annuli $\Xi^A_i$.
These two boundary circles bound an annulus within $\Sigma_{i'}^S$, denoted by $\Xi^E_i$.
So the spherical boundary components $\Sigma_i^S$ and the components of $V_{2, \partial}$ that are diffeomorphic to a solid cylinder, or the annuli $\Xi_i^A$, form chains, whose outer boundaries are homeomorphic to a $2$-torus.
The $S^1$-fibration on $V_{2, \textnormal{reg}}$ restricts to an $S^1$-fibration on these $2$-tori and to the standard $S^1$-fibrations of the annuli $\Xi^A_i$.
Summarizing our discussion, we conclude that the boundary of $V_{2, \textnormal{reg}}$ consists of the $2$-tori $\Sigma_i^T$ that are contained in $V_2$, the $2$-tori $\Xi^T_i$ and $\Xi^O_i$ and the union of the annuli $\Xi^A_i$ and $\Xi^E_i$.

We now state our precise result in Proposition \ref{Prop:MorganTianMain}.
The proposition is structured as follows:
After stating the assumptions (i)--(iii), we explain what happens in the case in which the manifold is collapsed to a point.
If this case does not occur, then the proposition asserts the decomposition of $M$ into subsets $V_1, V_2, V'_2$.
The topological structure of this decomposition is explained in assertions (a1)--(a4).
Next, we explain the decomposition of $V_2$ into $V_{2, \textnormal{reg}}, V_{2, \textnormal{cone}}, V_{2, \partial}$ and list its topological properties in (b1)--(b4).
Finally, in (c1)--(c3), we describe the geometric properties of local collapse in the different subsets $V_1, V_2, V_{2, \textnormal{reg}}$ and $V_{2, \textnormal{cone}}$.
We remark that Proposition \ref{Prop:MorganTianMain} is similar to \cite[Proposition 6.1]{Bamler-certain-topologies}.

\begin{Proposition} \label{Prop:MorganTianMain}
For every two continuous functions $\ov{r}, K : (0, 1) \to (0, \infty)$ and every $\mu > 0$ there are constants $w_0 = w_0(\mu, \ov{r}, K) > 0$ and $0 < s_2 (\mu, \ov{r}, K) < s_1(\mu, \ov{r}, K) < \frac1{10}$, monotonically increasing in $\mu$, such that:

Let $(M,g)$ be a compact manifold with boundary such that:
\begin{enumerate}[label=(\textit{\roman*})]
\item Each component $T$ of $\partial M$ is an embedded torus and for each such $T$ there is a closed subset $U'_T \subset M$ that is diffeomorphic to $T^2 \times I$ such that $T \subset \partial U'_T$ and such that the boundary components of $U'_T$ have distance of at least $2$.
Moreover, there is a fibration $p_T : U'_T \to I$ such that the $T^2$-fiber through every $x \in U'_T$ has diameter $< w_0 \rho_1(x)$.
\item For all $x \in M$ we have (with $\rho_1 (x)$ defined as in (\ref{eq:defofrho1inpartD})):
Whenever $B(x, \rho_1(x)) \subset \Int M$, then
\[ \vol B(x, \rho_1(x) ) < w_0 \rho_1^3(x). \]
\item For all $w \in (w_0, 1)$, $r < \ov{r}(w)$ and $x \in M$ we have: if $B(x,r) \subset \Int M$ and $\vol B(x,r) > w r^3$ and $r < \rho_1 (x)$, then $|{\Rm}|, r |\nabla \Rm |, r^2 | \nabla^2 \Rm | < K(w) r^{-2}$ on $B(x,r)$.
\end{enumerate}

Then $M$ is closed and $\diam M < \mu \rho_1(x)$ for all $x \in M$ and $M$ is diffeomorphic to an infra-nilmanifold or to a manifold that also carries a metric of non-negative sectional curvature, or the following holds:

There are finitely many embedded $2$-tori $\Sigma^T_i$ and $2$-spheres  $\Sigma^S_i \subset \Int M$ that are pairwise disjoint as well as closed subsets $V_1, V_2, V'_2 \subset M$ such that (see again Figure \ref{fig:V1V2}):
\begin{enumerate}[label=(a\textit{\arabic*})]
\item $M = V_1 \cup V_2 \cup V'_2$, the interiors of the sets $V_1, V_2$ and $V'_2$ are pairwise disjoint and $\partial V_1 \cup \partial V_2 \cup \partial V'_2 = \partial M \cup \bigcup_i \Sigma^T_i \cup \bigcup_i \Sigma^S_i$.
Obviously, no two components of the same set share a common boundary component.
\item $\partial V_1 = \partial M \cup \bigcup_i \Sigma^T_i \cup \bigcup_i \Sigma^S_i$.
In particular, $V_2 \cap V_2' = \emptyset$ and $V_2 \cup V_2'$ is disjoint from $\partial M$.
\item $V_1$ consists of components diffeomorphic to one of the following manifolds:
\[ T^2 \times I, \; S^2 \times I, \; \Klein^2 \widetilde{\times} I, \; \IR P^2 \widetilde{\times} I, \;  S^1 \times D^2, \; D^3, \]
a $T^2$-bundle over $S^1$, $S^1 \times S^2$ or the union of two (possibly different) components listed above along their $T^2$- or $S^2$-boundary.
\item Every component of $V_2'$ has exactly one boundary component and this component borders $V_1$ on the other side.
Moreover, every component of $V'_2$ is diffeomorphic to one of the following manifolds:
\[ S^1 \times D^2, \; D^3, \; \Klein^2 \widetilde{\times} I, \; \IR P^2 \td\times I. \]
\end{enumerate}

We can further characterize the components of $V_2$ (see Figure \ref{fig:V2}):
In $\Int V_2$ we find embedded $2$-tori $\Xi^T_i$ and $\Xi^O_i$ that are pairwise disjoint.
Furthermore, there are embedded closed annuli $\Xi^A_i \subset V_2$ whose interior is disjoint from the $\Xi^T_i$, $\Xi^O_i$ and $\partial V_2$ and whose boundary components lie in the components of $\partial V_2$ that are spheres.
Each spherical component of $\partial V_2$ contains exactly two such boundary components, which separate the sphere into two (polar) disks and one (equatorial) annulus $\Xi^E_i$.
We also find closed subsets $V_{2,\textnormal{reg}}, \linebreak[1] V_{2, \textnormal{cone}}, \linebreak[1] V_{2, \partial} \subset V_2$ such that
\begin{enumerate}[label=(b\textit{\arabic*})]
\item  $V_{2, \textnormal{reg}} \cup V_{2, \textnormal{cone}} \cup V_{2, \partial} = V_2$ and the interiors of these subsets are pairwise disjoint.
Moreover, $\partial V_{2, \textnormal{reg}}$ is the union of $\bigcup_i \Xi^T_i \cup \bigcup \Xi^O_i \cup \bigcup_i \Xi^A_i \bigcup_i \Xi^E_i$ with the components of $\partial V_2$ that are diffeomorphic to tori.
\item $V_{2, \textnormal{reg}}$ carries an $S^1$-fibration that is compatible with its boundary components and all its annular regions.
\item The components of $V_{2, \textnormal{cone}}$ are diffeomorphic to solid tori ($\approx S^1 \times D^2$) and each of these components is bounded by one of the $\Xi^T_i$ such that the fibers of $V_{2, \textnormal{reg}}$ on each $\Xi^T_i$ are not nullhomotopic inside $V_{2, \textnormal{cone}}$.
\item The components of $V_{2, \partial}$ are diffeomorphic to solid tori ($\approx S^1 \times D^2$) or solid cylinders ($\approx I \times D^2$).
The solid tori are bounded by the $\Xi^O_i$ such that the $S^1$-fibers of $V_{2, \textnormal{reg}}$ on the $\Xi^O_i$ are nullhomotopic inside $V_{2, \partial}$.
The diskal boundary components of each solid cylinder of $V_{2, \partial}$ are polar disks on spherical components of $\partial V_2$ and the annular boundary component is one of the $\Xi^A_i$.
Every polar disk and every $\Xi^A_i$ bounds such a component on exactly one side.
\end{enumerate}

We now explain the geometric properties of this decomposition:
\begin{enumerate}[label=(c\textit{\arabic*})]
\item If $\CC$ is a component of $V_1$, then there is a closed subset $U \subset \CC$ with smooth boundary, as well as a Riemannian $1$-manifold $J$ whose diameter is larger than $s_1 \rho_1(x)$ for each $x \in \CC$ and a fibration $p : U \to J$ such that
\begin{enumerate}
\item[({$\alpha$})] If $\CC \approx T^2 \times I$ or $S^2 \times I$, then $U = \CC$ and $J$ is a closed interval. \\
If $\CC \approx S^1 \times D^2, \Klein^2 \td\times I, D^3$ or $\IR P^3 \setminus B^3$, then $U \approx T^2 \times I$ (in the first two cases) or $U \approx S^2 \times I$ (in the latter two cases), $\partial \CC \subset \partial U$, $J$ is a closed interval and for all $x \in \CC \setminus U$ we have $\diam \CC \setminus U < \mu s_1 \rho_1(x)$. \\
If $\CC$ is the union of two such components as listed in (a3), then $U \approx T^2 \times I$ or $S^2 \times I$ depending on whether these two components have toroidal or spherical boundary and $\CC \setminus \Int U$ is diffeomorphic to the disjoint union of these two components.
Moreover for all $x \in \CC \setminus U$, the diameter of the component of $\CC \setminus U$ in which $x$ lies, has diameter $< \mu s_1 \rho_1(x)$. \\
If $\CC$ is diffeomorphic to a $T^2$-bundle over $S^1$ or to $S^1 \times S^2$, then $J$ is a circle and $U = \CC$.
\item[({$\beta$})] If $U$ is diffeomorphic to $T^2 \times I$, $S^2 \times I$ or $S^1 \times S^2$, then $p$ corresponds to the projection onto the interval or the circle factor.
\item[({$\gamma$})] $p$ is $1$-Lipschitz.
\item[({$\delta$})] For every $x \in U$, the fiber of $p$ through $x$ has diameter less than $\mu s_1 \rho_1(x)$.
\end{enumerate}
\item For every $x \in V_{2, \textnormal{reg}}$, the ball $(B(x, s_2 \rho_1(x)), s_2^{-2} \rho_1^{-2} (x) g, x)$ is $\mu$-close to a standard $2$-dimensional Euclidean ball $(B = B_1(0), g_{\eucl}, \ov{x} = 0)$.

Moreover, there is an open subset $U$ with $B(x,\frac12 s_2 \rho_1(x)) \subset U \subset \linebreak[1] B(x, \linebreak[0] s_2 \rho_1 (x))$ and a smooth map $p : U \to \IR^2$ such that
\begin{enumerate}
\item[({$\alpha$})] There are smooth vector fields $X_1, X_2$ on $U$ such that $dp (X_i) = \frac{\partial}{\partial x_i}$ and $X_1, X_2$ are almost orthonormal, i.e. $| \langle X_i, X_j \rangle - \delta_{ij} | < \mu$ for all $i,j = 1,2$.
\item[({$\beta$})] $U$ is diffeomorphic to $S^1 \times B^2$ such that $p : U \to p(U)$ corresponds to the projection onto $B^2$ and the $S^1$-fibers are isotopic in $U$ to the $S^1$-fibers of the fibration on $V_{2, \textnormal{reg}}$.
\item[({$\gamma$})] The $S^1$-fibers of $p$ and the $S^1$-fibers of $V_{2, \textnormal{reg}}$ on $U$ enclose an angle $< \mu$ with each other and an angle $\in (\frac\pi{2} - \mu, \frac\pi{2} + \mu)$ with $X_1$ and $X_2$.
\item[({$\delta$})] The $S^1$-fiber of the fibration on $V_{2, \textnormal{reg}}$ that passes through $x$ is isotopic in $U$ to the $S^1$-fibers of $p$.
\item[({$\varepsilon$})] The $S^1$-fibers of $p$ as well as the $S^1$-fibers of $V_{2, \textnormal{reg}}$ on $U$ have diameter less than $\min \{ (\vol U)^{1/3}, \mu s_2 \rho_1 (x) \}$.
\end{enumerate}
\item For every $x \in V_{2, \textnormal{cone}}$, the ball $B(x, \mu \rho_1(x))$ covers the component of $V_{2, \textnormal{cone}}$ in which $x$ lies.
\item For every $x \in V_{2, \partial}$ there is an $x' \in V_{2, \textnormal{reg}}$ with $\dist(x,x') < \mu \rho_1(x)$.
\item For every $x \in V'_2$, the ball $B(x, \mu s_1 \rho_1(x))$ covers the component of $V'_2$ in which $x$ lies.
\end{enumerate}
\end{Proposition}

\begin{proof}
We follow the lines of the proof of \cite[Proposition 6.1]{Bamler-certain-topologies}.
Our proposition is a consequence of the arguments used for the proof of Theorem 0.2 in Morgan-Tian (\cite{MorganTian}).
In the following, we will point out the intermediate steps in this proof that imply the assertions of our proposition and we will explain how some of its arguments have to be modified slightly to fit our setting.

First note that our proposition and Theorem 0.2 in \cite{MorganTian} use different philosophies:
Our proposition asserts that there is a small $w_0 > 0$ with the property that \emph{every} ``$w_0$-collapsed'' manifold $(M,g)$ satisfies the desired topological and geometric assertions while Theorem 0.2 claims that whenever we have a sequence of manifolds $(M_n, g_n)$ that are ``$w_n$-collapsed'' with $\lim_{n \to \infty} w_n = 0$, then these assertions hold for sufficiently large $n$.
These two philosophies are equivalent, similarly as the $\varepsilon$-$\delta$-criterion for continuity is in general equivalent to the sequence criterion.
Under this equivalence, assumption (ii) of our proposition implies assumption 1. of Theorem 0.2, which reads
\begin{quote} \it
1. For each point $x \in M_n$ there exists a radius $\rho = \rho_n(x)$ such that the ball $B_{g_n} (x,\rho)$ has volume at most $w_n \rho^3$ and all the sectional curvatures of the restriction of $g_n$ to this ball are all at least $- \rho^{-2}$.
\end{quote}
Except for the higher derivative bounds, which are not really needed in the proof of Theorem 0.2, assumption (iii) of our proposition implies assumption 3. of Theorem 0.2, which reads
\begin{quote} \it
3. For every $w' > 0$ there exist $\ov{r} = \ov{r} (w') > 0$ and constants $K_m = K_m (w') < \infty$ for $m = 0,1, 2, \ldots$, such that for all $n$ sufficiently large, and any $0 < r \leq \ov{r}$, if the ball $B_{g_n} (x,r)$ has volume at least $w' r^3$ and sectional curvatures at least $- r^{-2}$, then the curvature and its $m^{th}$-order covariant derivatives at $x$, $m = 1,2, \ldots$, are bounded by $K_0 r^{-2}$ and $K_m r^{-m-2}$, respectively.
\end{quote}
Lastly, assumption (i) of our proposition translates to the following condition:
\begin{quote} \it
Each component $T$ of $\partial M_n$ is an incompressible torus [\ldots] such that the $T^2$-fiber through every $x \in U'_T$ has diameter $< w_n \rho_1(x)$.
\end{quote}
This condition does not imply assumption 2. of Theorem 0.2:
\begin{quote} \it
2. Each component of the boundary of $M_n$ is an incompressible torus of diameter at most $w_n$ and with a topologically trivial collar containing all points within distance $1$ of the boundary on which the sectional curvatures are between $- 5/16$ and $-3/16$.
\end{quote}
It will become evident later, however, that either condition is sufficient for our purposes.

Next, Morgan and Tian make the following simplifying assumption:
\begin{quote} \it
Assumption 1. For each $n$, no connected, closed component of $M_n$ admits a Riemannian metric of non-negative sectional curvature.
\end{quote}
In our proposition we don't want to make this assumption.
So we have to find alternative arguments whenever this assumption is used.
Assumption 1 is essentially used at two places in the proof of Theorem 0.2.
Firstly, it is used to rule out certain topologies in the description of the geometric decomposition of $(M,g)$.
This issue can be resolved by adding these topologies to the list of possible topologies, e.g. in assertion (a3).
Secondly, it is used in the proof of \cite[Lemma 1.5]{MorganTian} to show that the function $\rho_n (x)$ can be rechosen to be sufficiently regular and $\leq \frac12 \diam M$.
The regularity assumption is automatically satisfied by our choice $\rho_1(x)$ and by any multiple $\lambda \rho_1 (x)$ for $0 < \lambda \leq 1$.
We will now argue that, nevertheless, we can add the simplifying assumption that 
\[ \rho_1 (x) \leq \max \{ C, \mu^{-1} \} \diam M \qquad \text{for all} \qquad x \in M \]
for some universal constant $C < \infty$.
This bound will be enough for our purposes, because we can choose $\lambda = \frac12 \min \{ C^{-1}, \mu \}$ to ensure that the function $\rho_n(x)$ in \cite{MorganTian} is bounded by $\frac12 \diam M$.

So assume for the moment that $\rho_1 (x) > \max \{ C, \mu^{-1} \} \diam M$ for some $x \in M$ and some constant $C < \infty$ that we will determine later.
Since $M \subset B(x,\rho_1(x))$, this inequality holds for all $x \in M$ and it implies 
\[ \diam M < \min \{ C^{-1}, \mu \} \rho_1(x) \leq \mu \rho_1(x). \]
By condition (i), $M$ must be closed.
It now follows from \cite[Corollary 0.13]{FY} that we can choose $C$ uniformly such that the lower sectional curvature bound of $- \rho_1^{-2} (x)$ on $(M,g)$ together with the diameter bound imply that $M$ either supports a metric of non-negative sectional curvature or is infranil.
This implies that the assertion in the paragraph immediately after condition (iii) is satisfied and we are done.
So we may assume from now on that $\rho_1 (x) \leq \max \{ C, \mu^{-1} \} \diam M$ for all $x \in M$ or, equivalently, that the function $\rho_n (x)$ in \cite{MorganTian}, being equal to $\lambda \rho_1 (x)$, is bounded from above by $\frac12 \diam M$.

Next, we have to construct the sets $V_1, V_2, V'_2$ as well as $V_{2, \textnormal{reg}}, V_{2, \textnormal{cone}}, V_{2, \partial}$.
These sets will arise from the construction of the sets $V_{n,1}$ and $V_{n,2}$ in \cite{MorganTian}.
Note that the construction of $V_{n,1}$ and $V_{n,2}$ is carried out in several steps.
In the following we provide an overview over this construction and point out the necessary changes for the proof of our proposition.

In \cite[Proposition 5.2]{MorganTian}, Morgan-Tian define $X_{n,1} \subset M_n$ to be the set of all points at which $(M_n, g_n)$ is locally collapsed to an open interval.
The statement of Proposition 5.2 is that $X_{n,1}$ can be extended to a subset $X_{n,1} \subset U_{n,1} \subset M_n$ such that the components of $U_{n,1}$ are diffeomorphic to $S^2 \times (0,1)$, $T^2 \times (0,1)$ or a $2$-torus bundle over the circle such that the ends of $U_{n,1}$ are geometrically controlled.
It follows from the proof of this proposition that all points $x \in U_{n,1}$ satisfy the geometric characterization of assertion (c1) in our proposition.
Note that in our setting, due to the lack of Assumption 1, we have to include $2$-sphere bundles over the circle to the list of possible topologies of $U_{n,1}$.

Next, Morgan-Tian analyze the components of $A \subset M_n \setminus U_{n,1}$.
In \cite[Lemma 5.3]{MorganTian} they conclude that for each such $A$ there are three possibilities:
\begin{enumerate}[label=(\arabic*)]
\item $(M_n, g_n)$ is locally collapsed in $A$ to a $2$-dimensional space of area bounded from below.
\item $(M_n, g_n)$ is globally collapsed in $A$ to a half-open interval such that one of its endpoints corresponds to a point in $A$.
In this case $A$ is diffeomorphic to $T^2 \times I$ and adjacent to the boundary of $M_n$ or $A$ is diffeomorphic to $S^1 \times D^2$, $\Klein^2 \td\times I$, $D^2$ or $\IR P^2 \td\times I$.
\item $A$ is ``a component which is close to an interval but which expands to be close to a standard 2-dimensional ball'' (compare with \cite[Definition 5.4]{MorganTian}).
This means roughly that after decreasing $\rho_n(x)$ by a small factor, $A$ satisfies the conditions of (1).
\end{enumerate}
At this point we need to recall that in our setting we are using a different characterization of the metric around the boundary of $M$.
So we have to be careful with arguments that involve points close to the boundary of $M$.
It can however be shown that, for sufficiently small $w_0$, every component $A$ that is adjacent to $\partial M$ satisfies (2).
Using their previous conclusion, Morgan-Tian define $U'_{n,1}$ to be the union of $U_{n,1}$ with all such components $A$ that satisfy (2).
Note that, again, all points $x \in U'_{n,1}$ satisfy the geometric characterization of assertion (c1) in our proposition, since the important part of this assertion involves points $y$ that are sufficiently far away from the endpoints of the interval towards which we observe the local collapse.
We will later choose $V_1$ to be a subset of $U'_{n,1}$.
This will then establish assertion (c1).

After constructing $U'_{n,1}$, Morgan-Tian remove a small bit of each open end of $U'_{n,1}$ and call the new (closed) subset $W_{n,2}$ and the closure of its complement $W_{n,1}$ (see \cite[subsec 5.3]{MorganTian}).
The reason for doing this is that this way the ends of $W_{n,2}$ are equipped with fibrations by $2$-tori or $2$-spheres that are compatible with the boundary components of $W_{n,2}$ and the fibrations of the adjacent components of $W_{n,1}$.
For every component $A \subset M_n \setminus U'_{n,1}$, Morgan-Tian denote the corresponding component of $W_{n,2}$ by $\widehat{A} \supset A$.
Note that the change between $A$ and $\widehat{A}$ is generally negligible.
So if $A$ belongs to case (1) in the preceding list, then we will still interpret $\widehat{A}$ to be locally collapsed to a $2$-dimensional space; analogously for case (3).
We will later choose the subset $V_2 \subset M$ such that for each of its components $\mathcal{C} \subset V_2$ there is some component $A$ from case (1) such that $A \subset \mathcal{C} \subset \widehat{A}$.
The same is true for $V'_2$, with case (3) instead of case (1).

Next, Morgan-Tian analyze the geometry of the subset $W_{n,2}$ in \cite[subsec 5.4]{MorganTian}.
In order to do this, they use the following intuition:
Around every point $x \in W_{n,2}$ the Riemannian manifold $(M_n, g_n)$ is locally collapsed to some $2$-dimensional Alexandrov space $(X,d)$, which depends on $x$.
Every point $y \in X$ satisfies one of the following characterizations, which depend on certain parameters (compare with \cite[Theorem 3.22]{MorganTian}):
\begin{enumerate}[label=(\arabic*)]
\item $y$ is regular, i.e. after enlarging $(X,d)$ by some uniform factor, the geometry around $y$ is close to a $2$-dimensional Euclidean ball.
\item $y$ is conical, i.e. after rescaling, the geometry around $y$ looks like a subset of a $2$-dimensional cone.
\item $y$ is close to a regular boundary point, i.e. the local geometry around $y$ is close to a half plane.
\item $y$ is close to a corner, i.e. the local geometry around $y$ is close to a $2$-dimensional sector.
\end{enumerate}
Based on this classification of the points of the spaces that $W_{n,2}$ is locally collapsed to, Morgan-Tian derive an induced classification of the points of $W_{n,2}$.
As a result, they obtain a covering of $W_{n,2}$ by subsets $U_{2, \textnormal{generic}}$ (case (1)), finitely many ``$\varepsilon'$-solid tori near interior cone points'' (case (2)), $U_{\textnormal{cyl}}$, being the union of ``$\varepsilon'$-solid cylinders near flat $2$-dimensional boundary points''  (case (3)) and finitely many ``$3$-balls near $2$-dimensional boundary corners'' (case (4)).
The subset $U_{2, \textnormal{generic}}$ carries an $S^1$-fibration along which the collapse occurs.
For the exact statements see \cite[Lemmas 5.7, 5.9]{MorganTian}.
Then Morgan-Tian define the subsets $W'_{n,1}, W'_{n,2}$ by removing the ``$3$-balls near $2$-dimensional boundary corners'' from $W_{n,2}$ and adding their closures to $W_{n,1}$.

Eventually, in \cite[subsec 5.5]{MorganTian} Morgan-Tian construct the subset $V_{n,1}$.
In this construction, they first slightly deform the boundary between $W'_{n,1}$ and $W'_{n,2}$ such that the $S^1$-fibration on $W'_{n,2} \cap U_{2, \textnormal{generic}}$ is compatible with each boundary component.
After redefining $W'_{n,1}$ in that way, they set $V_{n,1} := W'_{n,1}$.
For our proposition, we define $V_1 \subset M$ to be the union of this new subset $W'_{n,1}$ minus the components that were added as deformations of ``$3$-balls near $2$-dimensional boundary corners'' when we passed from $W_{n,1}$ to $W'_{n,1}$.
We define the subsets $V_2, V'_2$ to be the unions of components in the closure of $M \setminus V_1$, depending on whether the corresponding component $A$ belonged to case (1) or (3) in the list before the previous list.
The surfaces $\Sigma^T_i$ and $\Sigma^S_i$ are defined to be the boundary components of $V_2 \cup V'_2$.
Assertions (a1)--(a3) follow immediately.
We will discuss the topology of the components of $V'_2$, as asserted in (a4), at the end of this proof.
For now we just note that each such component has exactly one boundary component.
Assertion (c5) follows from the construction process.
Note that in order to get the bound $\mu s_1 \rho_1(x)$ in this assertion, as opposed to a bound of the form $\frac1{10} \rho_1(x)$, we need to replace some of the numeric constants in Morgan-Tian's work, e.g. the distance of $1/25$ from the endpoints of $J$ (see the beginning of \cite[subsec 5.2]{MorganTian}), by constants depending on $\mu$.

Next, we need to construct the subsets $V_{2, \textnormal{reg}}, V_{2, \textnormal{cone}}$ and $V_{2, \partial}$.
For this we look at the construction of $V_{2,n}$ in \cite{MorganTian}.
The subset $V_{2,n}$ arises from $W'_{2,n}$ by removing deformations of certain ``$\varepsilon'$-solid tori near interior cone points'' and ``$\varepsilon'$-solid cylinders near flat $2$-dimensional boundary points''.
For our proposition we denote by $V_{2, \textnormal{cone}}$ the union of all these deformations of ``$\varepsilon'$-solid cylinders near flat $2$-dimensional boundary points'' within $V_2$ and by $V_{2, \partial}$ the union of all these deformations of ``$\varepsilon'$-solid cylinders near flat $2$-dimensional boundary points'' together with the deformations of ``$3$-balls near $2$-dimensional boundary corners''.
So the components of $V_{2, \textnormal{cone}}$ are solid tori; we denote their boundaries by $\Xi^T_i$.
Note that for each deformed ``$3$-balls near $2$-dimensional boundary corners'' and every spherical boundary component of $V_2$ there are exactly two diskal boundary components of deformed ``$\varepsilon'$-solid cylinders near flat $2$-dimensional boundary points'' that are contained in the boundary of this deformed $3$-ball or spherical boundary component.
So the deformations of the ``$\varepsilon'$-solid cylinders near flat $2$-dimensional boundary points'' and the ``$3$-balls near $2$-dimensional boundary corners'' form chains, which may or may not close up.
Chains that do not close up are homeomorphic to solid cylinders $\approx I \times D^2$ whose diskal boundary components are contained in spherical boundary components of $V_2$.
After smoothing out the corners equivariantly with respect to the adjacent $S^1$-fibration, the boundaries of these solid cylinders are smooth annuli; we denote them by $\Xi^A_i$.
Note that $\partial \Xi^A_i \subset \partial V_2$ and every spherical boundary component of $V_2$ contains exactly two circles of $\bigcup_i \partial \Xi^A_i$, which enclose an annulus $\Xi^E_i \subset \partial V_2$.
A chain that does close up is homeomorphic to a solid torus $\approx S^1 \times D^2$ and after smoothing equivariantly, its boundary $2$-torus is denoted by $\Xi^O_i$.
This establishes assertions (b1)--(b4).

Assertions (c3) and (c4) follow from the construction process and assertion (c2) follows from the construction process together with the statement and proof of \cite[Proposition 4.4]{MorganTian}.
Observe that the diameter bound on the fibers in $U$ in part ($\varepsilon$) of assertion (c2) follows from the fact around every such fiber of diameter $d$, we can find neighborhood inside $B(x, s_2 \rho_1(x))$ that is close to $S^1(d) \times B^2(10 d)$.

We now make a remark on the choice of the parameters $\mu$, $w_0$, $s_1$, $s_2$:
The constants in \cite{MorganTian} that determine the preciseness of the collapse or the closeness with respect to the Gromov-Hausdorff distance are mainly assumed to be fixed during the construction process of the subsets $V_{1,n}$ and $V_{2,n}$.
This is due to the fact that the purpose of \cite{MorganTian} was to establish a purely topological theorem.
Our proposition, however, also contains a geometric characterization of the collapse, as presented in assertions (c1)--(c5).
These assertions involve a degree of preciseness $\mu$, which can be chosen arbitrarily in the beginning of our proposition.
Our geometric characterization is more or less a byproduct of the proof in \cite{MorganTian} and the Lemmas and Propositions asserting the desired geometric statements, which can mainly be found in section 4 of \cite{MorganTian}, do allow the choice of arbitrarily small preciseness parameters.
Allowing these parameters to depend on $\mu$ will however entail a $\mu$-dependence of the collapsing degree $w_0$.
The constant $s_1$ has to be chosen much smaller than $\mu$, because in assertion (c1) we also want to describe the cases in which $\diam M \ll \rho_1(x)$, but still $\diam M > \mu \rho_1(x)$.
Likewise, $s_2 \ll s_1$, since at the boundary points of $V_2$, we need to be able to observe a local collapse towards a $1$-dimensional space at scale $s_1 \rho_1 (x)$, but a collapse towards a $2$-dimensional space at scale $s_2 \rho_1 (x)$.

Finally, we establish assertion (a4).
Consider a component $A$ of $V'_2$.
Let $x \in A$ be a base point and choose a collar $P \subset V_1$, $P \approx T^2 \times I$ or $P \approx S^2 \times I$ that is adjacent to $A$, i.e. $S = P \cup A \approx A$.
By assertion (c1) we may pick $P$ such that its diameter is $> \frac12 s_1 \rho_1 (x)$ and such that it satisfies the geometric characterizations of (c1)($\alpha$)--($\delta$).
So $P$ is collapsed to an interval of length $> \frac12 s_1 \rho_1 (x)$ along $T^2$ or $S^2$-fibers of diameter $< \mu s_1 \rho_1 (x)$.
Consider the rescaled metric $g' = \mu^{-1} s_1^{-1} \rho_1^{-1} (x) g$.
Its sectional curvatures are bounded from below by $- \mu s_1 > - \mu$ on $S$.
With respect to $g'$, the collar $P$ is collapsed to an interval of length $> \frac12 \mu^{-1/2}$, along fibers of diameter $< \mu^{1/2}$.
Moreover, by assertion (c5) we have $\diam_{g'} A < \mu^{1/2}$.
We now derive the following topological characterization of $A$ for sufficiently small $\mu$:
If $\pi : \widehat{A} \to A$ is a cover such that the restriction $\pi^{-1} (\partial A) \to \partial A$ is a disjoint union of covers of degree $\leq 6$, then $\partial \widehat{A}$ consists of at most two boundary components.
In fact, this cover induces a Riemannian cover $\pi' : (\widehat{S}, \widehat{g}') \to (S = P \cup A, \widehat{g}')$ and $\pi^{\prime -1} (P)$ consists of components that are collapsed to intervals of size $> \frac12 \mu^{-1/2}$ along fibers of size $< 6 \mu^{1/2}$.
The number of these intervals is equal to the number of boundary components of $\widehat{A}$.
So if $\widehat{A}$ had more than two boundary components, then the space $(\widehat{S}, \widehat{g}')$ would be collapsed to a space containing at least three intervals that intersect a subset of diameter $< 4 \mu^{1/2}$.
For small enough $\mu$ such a scenario is however impossible by Toponogov's Theorem.
So $\widehat{A}$ has at most two boundary components.

By \cite[Lemma 5.9]{MorganTian} and our previous remarks, we have a topological decomposition $A = V'_{2, \textnormal{reg}} \cup V'_{2, \textnormal{cone}} \cup V'_{2, \partial}$ with the same properties as described in assertions (b1)--(b4) (note that this fact can also be derived by applying our Proposition to a rescaled metric in which $A$ is locally collapsed to a $2$-dimensional Alexandrov space and $P$ is locally collapsed to a line).
The $S^1$-fibration on $V'_{2, \textnormal{reg}}$ can be extended to a Seifert fibration on $V'_{2, \textnormal{reg}} \cup V'_{2, \textnormal{cone}}$.
Let $\Sigma$ be the base-orbifold of this Seifert fibration and let $C_0 \subset \partial \Sigma$ be the boundary circle that corresponds to the boundary component of $V'_{2, \textnormal{reg}} \cup V'_{2, \textnormal{cone}}$ that intersects $\partial A$.
The subset $V'_{2, \partial}$ consists of solid tori and, if $P \approx S^2 \times I$, also one solid cylinder.
The solid tori correspond to the boundary circles of $\partial \Sigma \setminus C_0$ and the solid cylinder corresponds to a subinterval of $C_0$.
Let us now distinguish the cases $P \approx S^2 \times I$ and $P \approx T^2 \times I$:
\begin{enumerate}[label=(\arabic*)]
\item If $P \approx S^2 \times I$, then the topological characterization of $A$ implies that $\Sigma$ has no orbifold covers $\widehat{\Sigma} \to \Sigma$ of degree $> 2$.
So since $\partial \Sigma \neq \emptyset$, $\Sigma$ must be homeomorphic to a disk.
Moreover, $\Sigma$ has no orbifold singularity or exactly one singularity of degree $2$.
So $A$ is homeomorphic to the union of a solid torus, $V'_{2, \textnormal{reg}} \cup V'_{2, \textnormal{cone}} \approx S^1 \times D^2$, with a solid cylinder, $V'_{2, \partial} \approx I \times D^2$, in such a way that the $S^1$-fibers on the cylinder wrap once or twice around the $S^1$-factor of the solid torus.
In the first case $A \approx D^3$ and in the second case $A \approx \IR P^2 \td\times I$.
\item If $P \approx T^2 \times I$, then the topological characterization of $A$ implies the following:
Assume that $\pi : \widehat{\Sigma} \to \Sigma$ is an orbifold cover such that the restricted cover $\pi^{-1} (C_0) \to C_0$ consists only of circle covers of degree $\leq 6$.
Then $\pi^{-1} (C_0)$ contains at most two components.

Let $D_6$ be an orbifold that is homeomorphic to a disk and that has a single orbifold singularity of degree $6$.
Attach $D_6$ to $\Sigma$ along $\partial D_6$ and $C_0$ and denote the resulting orbifold by $\Sigma_0$.
Then the previous conclusions on $\Sigma$ imply that $\Sigma_0$ can only have finite orbifold covers.
So $\Sigma_0$ can either be bad or elliptic.
If $\Sigma_0$ is bad, then it is homeomorphic to a sphere and has at most two orbifold singularities.
In this case $\Sigma$ is homeomorphic to a disk and has at most one orbifold singularity, which implies $A \approx S^1 \times D^2$.
If $\Sigma_0$ is elliptic, then it must be homeomorphic to a sphere or a disk.
If it is homeomorphic to a sphere, then it can have at most $3$ orbifold singularities and if it has exactly $3$ singularities, then two of those singularities need to have degree $2$.
So in this case $\Sigma$ has at least $2$ singularities.
If $\Sigma$ has $\leq 1$ singularities, then $A \approx S^1 \times D^2$ as before.
If it has exactly two singularities, then those singularities must have degree $2$ and $\Sigma$ is a quotient of an annulus and $A \approx \Klein^2 \td\times I$.
If $\Sigma_0$ is homeomorphic to a disk, then it can only have one orbifold singularity.
So in this case $\Sigma$ is a smooth annulus and $V'_{2, \textnormal{reg}} \cup V'_{2, \textnormal{cone}} \approx T^2 \times I$ and it follows that $A \approx V'_{2, \partial} \approx S^1 \times D^2$. 
\end{enumerate}
\end{proof}

\subsection{Geometric consequences} \label{subsec:geometricconsequences}
We now identify parts in the decomposition of Proposition \ref{Prop:MorganTianMain} that become non-collapsed when we pass to the universal cover or to a local cover.

\begin{Lemma} \label{Lem:unwrapfibration}
There is a constant $\mu_1 > 0$ such that:
Assume that we're in the situation of Proposition \ref{Prop:MorganTianMain} and assume $\mu \leq \mu_1$.
Then there is a constant $w_1 = w_1 (\mu) > 0$, which only depends on $s_2(\mu, \ov{r}, K)$, such that the following holds:
Consider a subset $N \subset M$ and a point $x \in N$ such that $B(x,\rho_1(x)) \subset N$.
Assume that we are in one of the following cases:
\begin{enumerate}[label=(\roman*)]
\item $x \in \CC \subset N$ where $\CC$ is a component of $V_2$ with the property that the $S^1$-fiber of $\CC \cap V_{2, \textnormal{reg}}$ is incompressible in $N$.
\item $x \in \CC \subset N$ where $\CC$ is a component of $V_1$ that is diffeomorphic to $T^2 \times I$, $\Klein^2 \td\times I$, a $T^2$-bundle over $S^1$ or the union of two copies of $\Klein^2 \td\times I$ along their boundary and in all of these cases the generic $T^2$-fiber is incompressible in $N$.
\item $x \in \CC$ where $\CC$ is a component of $V_1$ that is diffeomorphic to $S^1 \times D^2$ or to a union of two (possibly different) copies of $S^1 \times D^2$ or $\Klein^2 \td\times I$.
Let $U \subset \CC$ be a subset as described in Proposition \ref{Prop:MorganTianMain}(c1).
Then we assume that $U \subset N$ and that the $T^2$-fiber of $U$ is incompressible in $N$.
\item $x \in \CC \subset N$ where $\CC$ is a component of $V'_2$ that is diffeomorphic to $\Klein^2 \td\times I$ and whose generic $T^2$-fiber is incompressible in $N$.
\item We are in the case $\diam M < \mu \rho_1(y)$ for all $y \in M$ as mentioned in the beginning of Proposition \ref{Prop:MorganTianMain}, $N = M$ and $M$ is either an infra-nilmanifold or a quotient of $T^3$.
\end{enumerate}
Now consider the universal cover $\td{N}$ of $N$ and choose a lift $\td{x} \in \td{N}$ of $x$.
Then we claim that
\[ \vol B(\widetilde{x}, \rho_1(x)) > w_1 \rho_1^3(x). \]
\end{Lemma}

In other words, $x$ is $w_0$-good at scale $1$ relative to $N$ in the sense of \cite[Definition \ref{Def:goodness}]{Bamler-LT-Perelman}.

The proof of the Lemma follows the lines of \cite[sec 7]{Bamler-certain-topologies}, where a special case is established.
The proof makes use of comparison geometry.
For any three points $x_0, x_1, x_2$ in a metric space $(X, d)$ we can construct a triangle $\triangle \ov{x}_0 \ov{x}_1 \ov{x}_2 \subset \IH^2$ in the hyperbolic plane with the property that $\dist(\ov{x}_i, \ov{x}_j) = d(x_i, x_j)$ for all $i,j = 0,1,2$.
Its angles do not depend on the choice of the $\ov{x}_i$ and are called \emph{comparison angles}.
We will write $\cangle x_1 x_0 x_2 := \sphericalangle \ov{x}_1 \ov{x}_0 \ov{x}_2$.
Note that this construction can be carried out in any model space of constant curvature, but in this paper we will only be interested in the model space of constant curvature $-1$.
Using this notion, we define the notion of strainers as follows:

\begin{Definition}[$(m,\delta)$-strainer]
Let $\delta > 0$ and $m \geq 1$.
A $2m$-tuple $(a_1, b_1, \ldots, \linebreak[1] a_m, \linebreak[1] b_m)$ of points in a metric space $(X, d)$ is called an \emph{$(m,\delta)$-strainer around a point $x \in X$} if
\[ \cangle a_i x b_j, \; \cangle a_i x a_j, \; \cangle b_i x b_j > \tfrac{\pi}2 - \delta \qquad  \text{for all} \quad i \neq j, \quad i, j = 1, \ldots, m \]
and
\[ \cangle a_i x b_i > \pi - \delta \qquad \text{for all} \quad i = 1, \ldots, m. \]
The strainer is said to have \emph{size $r$} if $d(x, a_i) = d(x, b_i) = r$ for all $i = 1, \ldots, m$ or \emph{size $\geq r$} if $d(x, a_i), d(x,b_i) \geq r$ for all  $i = 1, \ldots, m$.
\end{Definition}

We will also need the following

\begin{Definition}[$(m+\frac12, \delta)$-strainer]
Let $\delta > 0$ and $m \geq 1$.
A $2m+1$-tuple $(a_1, b_1, \ldots, a_m, b_m, a_{m+1})$ of points in a metric space $(X, d)$ is called an \emph{$(m+\frac12,\delta)$-strainer around a point $x \in X$} if
\begin{alignat*}{3}
 \cangle a_i x b_j &> \tfrac{\pi}2 - \delta &\qquad \text{for all} \quad i &\neq j, \quad & i &= 1, \ldots, m+1, \quad j = 1, \ldots, m, \\
 \cangle a_i x a_j &> \tfrac{\pi}2 - \delta & \qquad \text{for all} \quad i &\neq j, \quad & i, j &= 1, \ldots, m+1, \\
 \cangle b_i x b_j &> \tfrac{\pi}2 - \delta &\qquad \text{for all} \quad i &\neq j, \quad & i, j &= 1, \ldots, m
\end{alignat*}
and
\[ \cangle a_i x b_i > \pi - \delta \qquad \text{for all} \quad i = 1, \ldots, m. \]
The strainer is said to have \emph{size $r$} if $d(x, a_i) = d(x, b_i) = r$ for all $i = 1, \ldots, m$ or $m+1$ or \emph{size $\geq r$} if $d(x, a_i), d(x,b_i) \geq r$ for all  $i = 1, \ldots, m$ or $m+1$.
\end{Definition}

In the following proofs (for each of the cases (i)--(v)) we will denote by $\delta_k(\mu_1)$ a positive constant that depends on $\mu_1 > 0$ and goes to zero as $\mu_1$ goes to zero.
In the end of each proof we will choose $\mu_1$ small enough so that all constants $\delta_k$ are sufficiently small.

\begin{proof}[Proof in case (i).]
We have either $x \in V_{2, \textnormal{reg}}$ or $x \in V_{2, \textnormal{cone}} \cup V_{2, \partial}$.
In the second case, there is an $x' \in B(x, \mu \rho_1(x)) \cap \CC \cap V_{2, \textnormal{reg}}$ and $\frac12 \rho_1(x) < \rho_1(x') < 2 \rho_1(x)$ (here we have assumed that $\mu < \frac12$).
So $B(x', \frac14 \rho_1(x')) \linebreak[1] \subset B(x, \linebreak[1] \rho_1(x)) \linebreak[1] \subset N$.
Let $\td{x}' \in \td{N}$ be a lift of $x'$ such that $\dist(x, x') = \dist(\td{x}, \td{x}')$.
Then $B(\td{x}', \frac14 \rho_1(x')) \subset B(\td{x}, \rho_1(x))$ and hence $\vol B(\td{x}, \rho_1(x)) \geq \vol B(\td{x}', \frac14 \rho_1(x'))$.
So if we relax the assumption $B(x, \rho_1(x)) \subset N$ to $B(x, \frac14 \rho_1(x)) \subset N$, then we can replace $x$ by $x'$.
This shows that under this relaxed assumption, we can assume without loss of generality that $x \in V_{2, \textnormal{reg}}$.

Consider now the subset $U$ with $B(x, \frac12 s_2 \rho_1(x)) \subset U \subset B(x, s_2 \rho_1(x)) \subset N$ and the map $p : U \to \IR^2$ from Proposition \ref{Prop:MorganTianMain}(c2).
For the rest of the proof of case (i), we will only work with the metric $g' = s_2^{-2} \rho_1^{-2}(x) g$ on $M$ as opposed to $g$, and we will bound the $g'$-volume of the $1$-ball around $x$ in the universal cover $\td{N}$ from below by a universal constant.
This will then imply the Lemma.
Observe that the sectional curvatures of the metric $g'$ are bounded from below by $-1$ on this ball.
In the following paragraphs, we carry out concepts that can also be found in \cite{BBI} or \cite{BGP}.

By the properties of $x$, we can find a $(2, \delta_1(\mu_1))$-strainer $(a_1, b_1, a_2, b_2)$ of size $\frac12$ around $x$ (here $\delta_1(\mu_1)$ is a suitable constant as mentioned above).
Recall that this entails that $\dist (a_i, x) = \dist (b_i, x) = \frac12$ for all $i = 1,2$.
In the universal cover $\widetilde{N}$, we can now choose lifts $\widetilde{x}, \widetilde{a}_i, \widetilde{b}_i$ such that $\dist(\widetilde{a}_i, \widetilde{x}) = \dist(a_i, x) = \frac12$ and $\dist(\widetilde{b}_i, \widetilde{x}) = \dist(b_i, x) = \frac12$.
Since the universal covering projection $\pi : \td{N} \to N$ is $1$-Lipschitz, we obtain furthermore $\dist(\widetilde{a}_i, \widetilde{b}_j) \geq \dist(a_i, b_j)$, $\dist(\widetilde{a}_1, \widetilde{a}_2) \geq \dist(a_1, a_2)$ and $\dist(\widetilde{b}_1, \widetilde{b}_2) \geq \dist(b_1, b_2)$.
So all the comparison angles in the universal cover are at least as large as those on $M$ and hence we conclude that $(\widetilde{a}_1, \widetilde{b}_1, \widetilde{a}_2, \widetilde{b}_2)$ is a $(2, \delta_1(\mu_1))$-strainer around $\td{x}$ of size $\frac12$.

Next, we extend this strainer to a $2\frac12$-strainer around $\td{x}$.
By the property of the map $p$ there is a sequence $\widetilde{x}_n$ of lifts of $x$ in $\td{N}$ that is unbounded, such that the consecutive distances of its members are at most $2 (\vol U)^{1/3}$.
We can assume that $2 (\vol U)^{1/3} < \mu_1$, because otherwise we have a lower bound on $\vol U$ and we are done.
So for sufficiently small $\mu_1$ we can find an $n$ such that with $\widetilde{y} = \td{x}_n \in \td{N}$ we have
\[ | \dist (\td{x}, \td{y}) - 2 \sqrt{\mu_1} | \leq \mu_1 \]
Note that $\td{x}$ and $\td{y}$ both project to $x$ under the universal covering projection $\pi : \td{N} \to N$.
Since $\pi$ is $1$-Lipschitz, it follows that for $i = 1, 2$
\[ \dist(\widetilde{y}, \widetilde{a}_i) \geq \dist (x, a_i) =\tfrac12  \qquad \text{and} \qquad \dist(\widetilde{y}, \widetilde{b}_i) \geq \tfrac12 . \]
So in the triangle $\triangle \td{y} \td{x} \td{a}_i$, the segment $|\td{y} \td{a}_i|$ is the longest, which means that it must be opposite to the largest comparison angle, i.e.
\[ \cangle \td{a}_i \td{x} \td{y} \geq \cangle \td{x} \td{y} \td{a}_i. \]
Since $\dist(\td{x}, \td{y}) \to 0$ as $\mu_1 \to 0$, we find using hyperbolic trigonometry, that
\begin{equation} \label{eq:sumofthreeangles}
 \cangle \td{a}_i \td{x} \td{y} + \cangle \td{x} \td{y} \td{a}_i + \cangle \td{y} \td{a}_i \td{x} > \pi -  \delta_2(\mu_1)
\end{equation}
and
\begin{equation} \label{eq:smallangleoppositetosmallside}
 \cangle \td{y} \td{a}_i \td{x} < \delta_2 (\mu_1).
\end{equation}
The last three inequalities imply
\[ 2 \cangle \td{a}_i \td{x} \td{y} > \pi - 2 \delta_2 (\mu_1). \]
The same is true with $\td{a}_i$ replaced by $\td{b}_i$.
So
\begin{equation} \label{eq:212atx}
 \cangle \widetilde{a}_i \widetilde{x} \widetilde{y} > \tfrac{\pi}2 - \delta_2(\mu_1) \qquad \text{and} \qquad \cangle \widetilde{b}_i \widetilde{x} \widetilde{y} > \tfrac{\pi}2 - \delta_2(\mu_1).
\end{equation}
Hence $(\td{a}_1, \td{b}_1, \td{a}_2, \td{b}_2, \td{y})$ is a $(2\frac12, \delta_2 (\mu_1))$-strainer around $\td{x}$ of size $\geq 2 \sqrt{\mu_1} - \mu_1$.

Since $| {\dist(\widetilde{y}, \widetilde{a}_i) - \dist(\widetilde{x}, \widetilde{a}_i)}| < 2 \sqrt{\mu_1}$ and $|{\dist(\widetilde{y}, \widetilde{b}_i) - \dist(\widetilde{x}, \widetilde{b}_i)}| < 2 \sqrt{\mu_1}$, we conclude that $(\td{a}_1, \td{b}_1, \td{a}_2, \td{b}_2)$ is a $(2, \delta_3(\mu_1))$-strainer around $\td{y}$ of size $\geq \frac12 - 2 \sqrt{\mu_1} -2 \mu_1$.
We now show that symmetrically $(\td{a}_1, \td{b}_1, \td{a}_2, \td{b}_2, \td{x})$ is a $2\frac12$-strainer around $\td{y}$ of arbitrarily good precision:
By comparison geometry
\[ \cangle \td{a}_i \td{x} \td{y} + \cangle \td{y} \td{x} \td{b}_i + \cangle \td{a}_i \td{x} \td{b}_i \leq 2 \pi. \]
Together with (\ref{eq:212atx}) and the strainer inequality at $\td{x}$, this yields
\[ \cangle \td{a}_i \td{x}  \td{y} < \tfrac{\pi}2 + \delta_1(\mu_1) + \delta_2(\mu_1). \]
Combining this bound with (\ref{eq:sumofthreeangles}) and (\ref{eq:smallangleoppositetosmallside}) yields
\[ \cangle \td{x} \td{y} \td{a}_i  > \tfrac{\pi}2 - \delta_1(\mu_1) - 3 \delta_2(\mu_1) = \tfrac{\pi}2 - \delta_4 (\mu_1). \]
The same estimate holds for $\cangle \td{x} \td{y} \td{b}_i$.
So $(\td{a}_1, \td{b}_1, \td{a}_2, \td{b}_2, \td{x})$ is indeed a $(2\frac12, \delta_4 (\mu_1))$-strainer around $\td{y}$.

Let $\widetilde{m}$ be the midpoint of a minimizing segment joining $\widetilde{x}$ and $\widetilde{y}$ (this segment is contained in $\td{N}$ for small enough $\mu_1$).
We will now show that $(\td{a}_1, \td{b}_1, \td{a}_2, \td{b}_2, \td{y}, \td{x})$ is a $3$-strainer around $\td{m}$ of arbitrarily good precision.
Since the distances of $\widetilde{a}_i$ and $\widetilde{b}_i$ to $\widetilde{m}$ differ from the corresponding distances to $\widetilde{x}$ by at most $\sqrt{\mu_1} + \mu_1$, we conclude that $(\widetilde{a}_1, \widetilde{b}_1, \widetilde{a}_2, \widetilde{b}_2)$ is a $(2, \delta_5 (\mu_1))$-strainer of size $\geq \frac12 - \sqrt{\mu_1} - \mu_1$ around $\td{m}$.
It remains to bound comparison angles involving the points $\td{x}$, $\td{y}$:
By the monotonicity of comparison angles, we have
\[ \cangle \td{m} \td{x} \td{a}_i \geq \cangle \td{y} \td{x} \td{a}_i > \tfrac{\pi}2 - \delta_2(\mu_1) \qquad \text{and} \qquad \cangle \td{m} \td{x} \td{b}_i \geq \cangle \td{y} \td{x} \td{b}_i > \tfrac{\pi}2 - \delta_2(\mu_1). \]
Now, if we apply the same argument as in the last paragraph, replacing $\td{y}$ by $\td{m}$, we obtain $\cangle \td{x} \td{m} \td{a}_i, \cangle \td{x} \td{m} \td{b}_i > \frac{\pi}2 - \delta_6(\mu_1)$.
For analogous estimates on the opposing angles, we then interchange the roles of $\td{x}$ and $\td{y}$.
Finally, $\cangle \td{x} \td{m} \td{y} = \pi$ is trivially true.

Set $\td{a}_3 = \td{y}$ and $\td{b}_3 = \td{x}$.
We have shown that $(\td{a}_1, \td{b}_1, \td{a}_2, \td{b}_2, \td{a}_3, \td{b}_3)$ is a $(3, \delta_7(\mu_1))$-strainer around $\td{m}$ of size $\geq \sqrt{\mu_1} - \mu_1 > \frac12 \sqrt{\mu_1}$ (for $\mu_1 < \frac12$) for a suitable $\delta_7(\mu_1)$.
We will now use this fact to estimate the volume of the $\lambda \sqrt{\mu_1}$-ball around $\td{m}$ from below for sufficiently small $\lambda$ and $\mu_1$.
We follow here the ideas of the proof of \cite[Theorem 10.8.18]{BBI}.
Define the function
\begin{multline*}
 f : B(\td{m}, \lambda \sqrt{\mu_1}) \longrightarrow \IR^3 \qquad 
  z \longmapsto (\dist(\td{a}_1, z) - \dist(\td{a}_1, \td{m}), \\ \dist(\td{a}_2, z) - \dist(\td{a}_2, \td{m}), \dist(\td{a}_3, z) - \dist(\td{a}_3, \td{m})).
\end{multline*}
We will show that $f$ is $100$-bilipschitz for sufficiently small $\mu_1$ and $\lambda$.
Obviously, $f$ is $3$-Lipschitz, so it remains to establish the lower bound $\frac1{100}$.
Assume that this was false, i.e. that there were $z_1, z_2 \in B(\td{m}, \lambda \sqrt{\mu_1})$ with $\dist(z_1, z_2) > 100 | f(z_1) - f(z_2) |$.
Then for all $i = 1, 2, 3$
\begin{equation} \label{eq:almostiso}
 \dist(z_1, z_2) > 100 | {\dist(a_i, z_1) - \dist(a_i, z_2)} |.
\end{equation}
By the previous conclusions and the fact that comparison angles can be computed in terms of the distance function, we find that given any $\delta > 0$, we can choose $\lambda > 0$ and $\mu_1 > 0$ sufficiently small, to ensure that $(\td{a}_1, \td{b}_1, \td{a}_2, \td{b}_2, \td{a}_3, \td{b}_3)$ is a $(3, \delta)$-strainer around $z_1$ and around $z_2$.
Now look at the comparison triangle corresponding to the points $z_1, z_2, \td{a}_i$.
By (\ref{eq:almostiso}), it is almost isosceles and hence by elementary hyperbolic trigonometry we conclude for $\lambda$ sufficiently small
\[ \tfrac{9}{10} \tfrac{\pi}2 < \cangle z_2 z_1 \td{a}_i, \; \cangle z_1 z_2 \td{a}_i < \tfrac{11}{10} \tfrac{\pi}2 . \]
Using comparison geometry
\[ \cangle z_1 z_2 \td{b}_i \leq 2 \pi - \cangle \td{a}_i z_2 \td{b}_i - \cangle z_1 z_2 \td{a}_i < \tfrac{11}{10} \tfrac{\pi}2 + \delta. \]
For $\lambda$ sufficiently small, we obtain furthermore by hyperbolic trigonometry
\[  \cangle \td{b}_i z_1 z_2 + \cangle z_1 z_2 \td{b}_i +\cangle z_2 \td{b}_i z_1 > \pi - \delta \qquad \text{and} \qquad \cangle z_2 \td{b}_i z_1 < \delta. \]
So
\[ \cangle \td{b}_i z_1 z_2 > \tfrac{9}{10} \tfrac{\pi}2 - 3 \delta. \]

Now join $z_1$ with $\td{a}_1, \td{b}_1, \td{a}_2, \td{b}_2, \td{a}_3$ by minimizing geodesics.
By comparison geometry, these geodesics enclose angles of at least $\frac{\pi}2 - \delta$ or $\pi - \delta$,  depending on the geodesics, between each other.
So their unit direction vectors at $z_1$ approximate the negative and positive directions of an orthonormal basis of $T_{z_1} \td{N}$.
By the same argument, the minimizing geodesic that connects $z_1$ with $z_2$ however encloses an angle of at least $\frac9{10} \frac{\pi}2 - 3\delta$ with each of these geodesics.
For sufficiently small $\delta$ this contradicts the fact that the tangent space at $z_1$ is $3$-dimensional.
So $f$ is indeed $100$-bilipschitz for sufficiently small $\lambda$ and $\mu_1$.

From the bilipschitz property we can conclude that
\[ \vol B(\td{m}, \lambda \sqrt{\mu_1} ) > c (\lambda \sqrt{\mu_1})^3 \]
for some universal $c > 0$.
Fixing $\mu_1 < \frac14$ and $\lambda < 1$ such that the argument above can be carried out, we obtain
\[ \vol B(\td{x}, 1) > \vol B(\td{m}, \lambda \sqrt{\mu_1}) > c (\lambda \sqrt{\mu_1})^3 = c' > 0. \]
By rescaling, this implies the desired inequality for the metric $g$.
\end{proof}

\begin{proof}[Proof in cases (ii)--(iv).]
By Proposition \ref{Prop:MorganTianMain}(c1), (c5) we know that there is an $x' \in \CC$ (or $x' \in \CC'$ in case (iv) where $\CC'$ is the component of $V_1$ adjacent to $\CC$) with $\dist(x, x') < \frac1{10} \rho_1(x)$ such that there is a subset $U$ with $B(x', \frac14 s_1 \rho_1(x')) \subset U \subset B(x', \frac12 s_1 \rho_1(x'))$ that is diffeomorphic to $T^2 \times I$ and incompressible in $N$ and the ball $(B(x', \frac12 s_1 \rho_1(x')), \linebreak[1] 4 s_1^{-2} \rho_1^{-2}(x') g, \linebreak[1] x')$ is $\delta_1(\mu_1)$-close to $((-1,1), g_{\eucl}, 0)$.
As in the proof in case (i) we can replace $x$ by $x'$ and hence assume without loss of generality that $B(x, \frac14 s_1 \rho_1(x)) \subset U \subset B(x, \frac12 s_1 \rho_1(x))$ and that $(B(x, \frac12 s_1 \rho_1(x)), \linebreak[1] 4 s_1^{-2} \rho_1^{-2}(x) g, \linebreak[1] x)$ is $\delta_1(\mu_1)$-close to $((-1,1), \linebreak[1] g_{\eucl}, \linebreak[1] 0)$.

Choose $q \in \pi_1(N)$ corresponding to a non-trivial simple loop in one of the cross-sectional tori of $U$ and denote by $\widehat{N}$ the covering of $N$ corresponding to the cyclic subgroup generated by $q$, i.e. if we also denote by $q$ the deck-transformation of $\widetilde{N}$ corresponding to $q$, then $\widehat{N} = \td{N} / \langle q \rangle$.
So we have a tower of coverings $\td{N} \to \widehat{N} \to N$.

Consider first the rescaled metric $g' = 4 s_1^{-2} \rho_1^{-2} (x) g$.
With respect to this metric we can construct a $(1, \delta_2(\mu_1))$ strainer $(a_1, b_1)$ around $x$ of size $\frac12$ on $M$ for a suitable $\delta_2(\mu_1)$.
By the same arguments as in case (i), but using the covering $\widehat{N} \to N$, we can find a point $\widehat{m} \in \widehat{N}$ at a distance of $2 \sqrt{\mu_1} + \mu_1$ from a lift $\widehat{x} \in \widehat{N}$ of $x$ and a $(2, \delta_2(\mu_1))$ strainer $(\widehat{a}_1, \widehat{b}_1, \widehat{a}_2, \widehat{b}_2)$ around $\widehat{m}$ of size $\geq \frac12 \sqrt{\mu_1}$.
Connect the points $\widehat{a}_i$ and $\widehat{b}_i$ with $\widehat{m}$ by minimizing geodesics and choose points $\widehat{a}'_i$ and $\widehat{b}'_i$ of distance $\frac12 \sqrt{\mu_1}$ from $\widehat{m}$.
By monotonicity of comparison angles, $(\widehat{a}'_1, \widehat{b}'_1, \widehat{a}'_2, \widehat{b}'_2)$ is a $(2, \delta_2(\mu_1))$-strainer of size $\sqrt{\mu_1}$.

Let $g'' = \mu_1^{-1} g'$.
Then $(\widehat{a}'_1, \widehat{b}'_1, \widehat{a}'_2, \widehat{b}'_2)$ has size $\frac12$ with respect to $g''$.
Using this strainer, the metric $g''$ and the covering $\td{N} \to \widehat{N}$, we can apply the same argument from case (i) again and obtain a $(3, \delta_3(\mu_1))$ strainer $(\td{a}_1, \td{b}_1, \td{a}_2, \td{b}_2, \td{a}_3, \td{b}_3)$ around a point $\td{m}' \in \td{N}$, which is $2\sqrt{\mu_1}+\mu_1$ close to a lift $\td{m}$ of $\widehat{m}$ in $\td{N}$.

As in case (i), for a sufficiently small $\mu_1$ we can deduce a lower volume bound $\vol_{g''} \td{B} (\td{m}', 1) > c'$.
With respect to $g'$, the point $\td{m}'$ is within a distance of $\sqrt{\mu_1} (2\sqrt{\mu_1} + \mu_1) + 2\sqrt{\mu_1} + \mu_1$ from a lift $\td{x}$ of $\widehat{x}$.
Hence for sufficiently small $\mu_1$
\[ \vol_{g'} B (\td{x}, 1) > \vol_{g'} B(\td{m}', \sqrt{\mu_1}) > c' \mu_1^{3/2} = c'' > 0. \]
The desired inequality follows by rescaling.
\end{proof}

\begin{proof}[Proof in case (v).]
In this case, there is a covering $\widehat{M} \to M$ such that $\widehat{M} \approx T^2 \times \IR$ and whose group of deck transformations is isomorphic to $\IZ$.
Let $\widehat{g}$ be the pullback of the Riemannian metric $g$ onto $\widehat{M}$ and let $\widehat{x} \in \widehat{M}$ be a lift of $x$.
Then the rescaled ball $(B(\widehat{x}, \rho_1(x)), \rho_1^{-2} (x) \widehat{g}, \widehat{x})$ inside $\widehat{M}$ is $\delta_1(\mu_1)$-close to $((-1,1), g_{\textnormal{eucl}}, 0)$ and there is a subset $U \subset \widehat{M}$ with $B(\widehat{x}, \frac12 \rho_1(x)) \subset U \subset B(\widehat{x}, \rho_1(x))$ that is diffeomorphic to $T^2 \times I$ and incompressible in $\widehat{M}$.
We can now apply the previous proof.
\end{proof}

\subsection{Topological consequences} \label{subsec:topimplications}
We now discuss the topological structure of the decomposition obtained in Proposition \ref{Prop:MorganTianMain}.
So let in the following $M$ be a compact manifold, possibly with boundary and consider a decomposition $M = V_1 \cup V_2 \cup V'_2$, along with the surfaces $\Sigma_i^T, \Sigma_i^S, \Xi_i^T, \Xi_i^O, \Xi_i^A, \Xi_i^E$ and subsets $V_{2, \text{reg}}, \linebreak[1] V_{2, \text{cone}}, \linebreak[1] V_{2, \partial}$ that satisfy all the topological assertions of Proposition \ref{Prop:MorganTianMain} (a1)--(a4), (b1)--(b4).
For future applications, we will discuss the following three cases:
\begin{description}
\item[case A] $M$ is closed, i.e. $\partial M = \emptyset$ and $M$ is irreducible and not a spherical space form. (See \cite[sec \ref{sec:3dtopology}]{Bamler-LT-topology} for definitions of basic topological terminologies.)
\item[case B] $M$ is irreducible, has a boundary and all its boundary components are tori that are incompressible in $M$,
\item[case C] $M \approx S^1 \times D^2$.
\end{description}

The main result of this subsection will be Proposition \ref{Prop:GGpp}.
We first need to make some preparations.

\begin{Definition} \label{Def:GG}
Let $\mathcal{G} \subset M$ be the union of
\begin{enumerate}[label=(\arabic*)]
\item all components of $V_2$ whose generic $S^1$-fiber is incompressible in $M$,
\item all components of $V_1$ that are diffeomorphic to $T^2 \times I$ or $\Klein^2 \td{\times} I$ and whose generic $T^2$-fiber is incompressible in $M$, or components of $V_1$ that are diffeomorphic to a $T^2$-bundle over $S^1$ or the union of two copies of $\Klein^2 \td{\times} I$ along their common torus boundary.
\item all components of $V'_2$ that are diffeomorphic to $\Klein^2 \td{\times} I$ and whose generic $T^2$-fiber is incompressible in $M$.
\end{enumerate}
We call the components of $V_1$, $V_2$ or $V'_2$ that are contained in $\mathcal{G}$ \emph{good (in $M$)}.
\end{Definition}
By Proposition \ref{Prop:MorganTianMain}(b4) good components of $V_2$ do not contain points of $V_{2, \partial}$.

\begin{Lemma} \label{Lem:bdrygoodisV2}
Consider the cases A--C.
Every component of $V_1$, $V_2$ or $V'_2$ that is contained in $\mathcal{G}$ and shares a boundary component with $\mathcal{G}$ (meaning that it has a boundary component that is contained in $\partial \mathcal{G}$) either belongs to $V_2$ and is not adjacent to $\partial M$ or belongs to $V_1$ and is adjacent to $\partial M$.
\end{Lemma}
\begin{proof}
This follows directly from the definition of $\mathcal{G}$ and Proposition \ref{Prop:MorganTianMain}(a2).
Observe that any component that is adjacent to a good component of $V_1$ is good.
\end{proof}

\begin{Lemma} \label{Lem:SStori}
Consider the cases A or B.
There is a subset $\mathcal{S} \subset M$ that is the disjoint union of finitely many embedded solid tori $\approx S^1 \times D^2$, bounded by some of the $\Sigma^T_i$, such that for any $\Sigma^T_i$ the following statement holds: $\Sigma^T_i$ is compressible in $M$ if and only if $\Sigma^T_i \subset \mathcal{S}$ (i.e. it either bounds a component of $\mathcal{S}$ or it is contained in its interior).

In particular, $M = \mathcal{G} \cup \mathcal{S}$.
\end{Lemma}

Note that an important consequence of this Lemma is that in cases A and B we have $\mathcal{G} \neq \emptyset$.

\begin{proof}
Without loss of generality, we assume in the following that $M$ is connected.

First observe that by the irreducibility of $M$ in case A and the fact that $\partial M \neq \emptyset$ in case B, each sphere $\Sigma^S_i \subset M$ bounds a ball $B_i \subset M$ on exactly one side.
By \cite[Lemma \ref{Lem:coverMbysth}]{Bamler-LT-topology}, any two of those balls are either disjoint or one is contained in the other.
Let $B'_1, \ldots, B'_{m'}$ be a disjoint subcollection of those balls that are maximal with respect to inclusion.

Now, consider all tori $\Sigma^T_i$ that already bound a solid torus $S_i \subset M$.
In case B, we have $S_i \cap \partial M = \emptyset$.
Again by \cite[Lemma \ref{Lem:coverMbysth}]{Bamler-LT-topology}, any two such tori are either disjoint or one is contained in the other.
So there is a unique subcollection of those solid tori that are maximal with respect to inclusion.
Denote the union of those solid tori by $\mathcal{S}$.

We will now show by contradiction that every torus $\Sigma^T_i \subset M \setminus ( \mathcal{S} \cup B'_1 \cup \ldots \cup B'_{m'})$ is incompressible in $M$.
Observe that each such torus does not bound a solid torus in $M$.
For each compressible torus $\Sigma^T_i \subset M \setminus ( \mathcal{S} \cup B'_1 \cup \ldots \cup B'_{m'})$ we choose a compressing disk $D_i \subset M$ (see \cite[Proposition \ref{Prop:incompressibleequiv}]{Bamler-LT-topology} for a definition of a compressing disk and the statement that such a compressing disk always exists).
By \cite[Lemma \ref{Lem:compressingtorus}]{Bamler-LT-topology} and a maximum argument, we can find one such torus $\Sigma^T_j \subset M \setminus ( \mathcal{S} \cup B'_1 \cup \ldots \cup B'_{m'})$ with the following property:
For any other compressible torus $\Sigma^T_i \subset M \setminus ( \mathcal{S} \cup B'_1 \cup \ldots \cup B'_{m'})$ that lies in the same component of $M \setminus \Sigma^T_j$ as $D_j$, the disk $D_i$ lies in the same component of $M \setminus \Sigma^T_i$ as $\Sigma^T_j$.
(E.g. one may choose the torus $\Sigma_j^T$ that has the most other tori $\Sigma_i^T$ on the side opposite to $D_j$.)

Let $\CC$ be the component of $V_1$, $V_2$ or $V'_2$ whose boundary contains $\Sigma^T_j$ and that lies on the same side of $\Sigma^T_j$ as $D_j$.
We first argue that $\CC \not\subset V_1$:
If $\CC$ was a component of $V_1$, then $\CC \approx T^2 \times I$, $\Klein^2 \td\times I$ or $S^1 \times D^2$.
The third case cannot occur, because $\Sigma_j^T \not\subset \mathcal{S}$ and the second case is impossible, because the boundary torus of $\Klein^2 \td\times I$ is incompressible in $\Klein^2 \td\times I$.
So $\CC \approx T^2 \times I$.
Let $\Sigma^T_i$ be the other boundary component of $\CC$.
Since the components of $M \setminus \Sigma_j^T$ and $M \setminus \Sigma_i^T$ are diffeomorphic, and by \cite[Lemma \ref{Lem:compressingtorus}]{Bamler-LT-topology} the compressing disks $D_j$, $D_i$ have to lie on the same sides of their tori, we obtain a contradiction to the choice of $\Sigma_j^T$.
So $\CC \not\subset V_1$.
Moreover $\CC \not\subset V'_2$, because otherwise $\CC \approx S^1 \times D^2$ or $\Klein^2 \td\times I$ by Proposition \ref{Prop:MorganTianMain}(a4), which is impossible by the same reasons.
So $\CC \subset V_2$.

We will now analyze the boundary of $\CC$.
Consider $\Sigma^T_i \subset \partial\CC$ (possibly $\Sigma^T_i = \Sigma^T_j$).
If $\Sigma^T_i$ bounds a solid torus $S_i \subset \mathcal{S}$, then since $\Sigma^T_j \not\subset \mathcal{S}$, $S_i$ lies on the opposite side of $\CC$.
If $\Sigma^T_i$ does not bound a solid torus, then $\Sigma^T_i \subset M \setminus ( \mathcal{S} \cup B'_1 \cup \ldots \cup B'_{m'})$.
So if $\Sigma^T_i$ has compressing disks, then by \cite[Lemma \ref{Lem:compressingtorus}(a)]{Bamler-LT-topology} and again by the choice of $\Sigma^T_j$, these disks can only lie on the same side as $\CC$.
By \cite[Proposition \ref{Prop:incompressibleequiv}]{Bamler-LT-topology}, this implies that then $\Sigma^T_i$ is incompressible in the component of $M \setminus \Sigma^T_i$ that does not contain $\CC$.
For every spherical boundary component $\Sigma^S_i \subset \partial\CC$, the ball $B_i$ lies on the opposite side of $\CC$ and has to be one of the maximal balls $B'_i$ because otherwise $\Sigma^T_j$ would be contained in $B'_1 \cup \ldots \cup B'_{m'}$.

Set now $N = \CC \cap V_{2, \textnormal{reg}}$ and define $\mathcal{S}^*$ to be the closure of the union of $\CC \setminus N$ with the balls $B'_i$ whose boundary lies in $\partial \CC$.
So $N$ carries an $S^1$-fibration and is bounded by some of the tori $\Sigma^T_i$ and $\partial\mathcal{S}^*$.
The set $\mathcal{S}^*$ consists of components of $V_{2, \textnormal{cone}}$ (those are solid tori), components of $V_{2, \partial}$ that are solid tori and chains made out of components of $V_{2, \partial}$ that are solid cylinders and balls $B'_i$.
So (after smoothing out the edges of the chains) all components of $\mathcal{S}^*$ are solid cylinders.
We can hence apply \cite[Lemma \ref{Lem:Seifertfiberincompressible}]{Bamler-LT-topology} to conclude that there are two cases:
In the first case there is a boundary torus $T \subset \partial N$ that bounds a solid torus in $M$ on the same side as $N$.
In the second case every boundary torus of $N$ either bounds a solid torus on the side opposite to $N$ or it is incompressible in $M$.
The second case cannot occur, since $\Sigma^T_j$ is compressible and does not bound a solid torus.
So the first case applies.
Consider the component $T \subset \partial N$ that bounds a solid torus on the same side as $N$.
We find using \cite[Lemma \ref{Lem:coverMbysth}]{Bamler-LT-topology} that $T \not\subset \partial\mathcal{S}^*$, so $T = \Sigma^T_i$ for some $i$ and the solid torus is $S_i$.
But this would imply that $\Sigma^T_j \subset N \subset \CC \subset S_i$ in contradiction to our assumptions.

We have shown so far that every $\Sigma^T_i \subset M \setminus ( \mathcal{S} \cup B'_1 \cup \ldots \cup B'_{m'})$ is incompressible in $M$.

Now assume that there was some $B'_i$ that is not contained in $\mathcal{S}$.
Then $\partial B'_i \cap \mathcal{S} = \emptyset$ by \cite[Lemma \ref{Lem:coverMbysth}]{Bamler-LT-topology}.
By maximality of $B'_i$ and the fact that $M$ is not a spherical space form, $B'_i$ borders a component $\CC$ of $V_2$ on the other side.
Since $\CC$ has a spherical boundary component, $\CC \cap V_{2, \partial} \neq \emptyset$ and hence the $S^1$-fibers on $\CC \cap V_{2, \text{reg}}$ are contractible in $M$.
So every boundary torus of $\CC$ must be compressible and hence be contained in $\mathcal{S} \cup B'_1 \cup \ldots \cup B'_{m'}$ and, in case B, $\partial \CC \cap \partial M = \emptyset$.
Since $\CC \not\subset \mathcal{S} \cup B'_1 \cup \ldots \cup B'_{m'}$, all boundary tori of $\CC$ bound solid tori on the other side.
Define $N$ and $\mathcal{S}^*$ in the same way as above.
Then $N$ carries an $S^1$-fibration and $\mathcal{S}^*$ is a disjoint union of solid tori.
So every boundary component of $N$ bounds a solid torus on the other side.
In particular, by \cite[Lemma \ref{Lem:coverMbysth}]{Bamler-LT-topology}, no boundary component of $N$ bounds a solid torus on the same side.
Hence by \cite[Lemma \ref{Lem:Seifertfiberincompressible}]{Bamler-LT-topology}, the $S^1$-fibers on $N$ are incompressible in $M$ in contradiction to our previous conclusion.

We conclude that $B'_1 \cup \ldots \cup B'_{m'} \subset \mathcal{S}$ and one direction of the first claim follows.
The other direction is clear since $\pi_1(S^1 \times D^2) \cong \IZ$.
It remains to show that $M = \mathcal{G} \cup \mathcal{S}$.
Let $\CC$ be a component of $V_1$, $V_2$ or $V'_2$ whose interior is contained in $M \setminus \mathcal{S}$ and assume that $\CC \not\subset \mathcal{G}$.
Observe that since all $\Sigma^S_i$ are contained in $\mathcal{S}$, $\partial \CC$ only consists of tori.

Consider first the case $\CC \subset V_1$.
If $\CC$ has no boundary, then it must be diffeomorphic to either $S^1 \times S^2$ or the union of $D^3$ and $D^3$, $D^3$ and $\IR P^3 \setminus B^3$, two copies of $\IR P^3 \setminus B^3$ along their boundary, two copies of $S^1 \times D^2$ along their boundary or to the union of $\Klein^2 \td\times I$ and $S^1 \times D^2$ along their boundary.
The first four cases can be excluded immediately, since $M$ is assumed to be irreducible and not a spherical space form and the last two cases are excluded by \cite[Lemma \ref{Lem:coverMbysth}]{Bamler-LT-topology} and \cite[Lemma \ref{Lem:Kleinandsolidtorus}]{Bamler-LT-topology}, respectively.
So $\CC$ has a boundary and hence it is diffeomorphic to $T^2 \times I$, $\Klein^2 \td\times I$ or $S^1 \times D^2$.
The last case cannot occur, since otherwise $\CC \subset \mathcal{S}$.
In the other two cases, the boundary component(s) are compressible in $M$ and hence not contained in $\partial M$.
So they bound a component of $\mathcal{S}$ on the side opposite to $\CC$, i.e. $M = \CC \cup \mathcal{S}$.
But this is again ruled out by \cite[Lemma \ref{Lem:coverMbysth}]{Bamler-LT-topology} and \cite[Lemma \ref{Lem:Kleinandsolidtorus}]{Bamler-LT-topology}.

Similarly, in the case $\CC \subset V'_2$, $\CC$ would be diffeomorphic to $\Klein^2 \td\times I$ or $S^1 \times D^2$.
The second case can not occur since otherwise $\CC \subset \mathcal{S}$ and in the first case, $M$ would be the union of $\Klein^2 \td\times I$ with a solid torus, which is ruled out by \cite[Lemma \ref{Lem:Kleinandsolidtorus}]{Bamler-LT-topology}.

Finally, assume that $\CC \subset V_2$.
Since the generic fiber in $\CC$ is assumed to be compressible in $M$, all boundary tori of $\CC$ are compressible in $M$ and hence $\partial \CC$ is disjoint from $\partial M$.
So $M = \CC \cup \mathcal{S}$, which gives a contradiction already in case B.
In case A, define $N$ and $\mathcal{S}^*$ again in the same way as above.
$N$ carries an $S^1$-fibration, $\mathcal{S}^*$ is a disjoint union of solid tori and $M = N \cup \mathcal{S} \cup \mathcal{S}^*$.
By \cite[Lemma \ref{Lem:Seifertfiberincompressible}]{Bamler-LT-topology} and \cite[Lemma \ref{Lem:coverMbysth}]{Bamler-LT-topology}, we conclude that the $S^1$-fibers of $N$ are in fact incompressible in $M$ and hence $\CC \subset \mathcal{G}$.
\end{proof}

We now focus on the intersection $\mathcal{G} \cap \mathcal{S}$.
\begin{Lemma} \label{Lem:solidtoriinsideSS}
Consider the cases A--C.
In the cases A, B let $\mathcal{S}$ be the set defined in Lemma \ref{Lem:SStori} and in case C let $\mathcal{S} = M$.
Then $\mathcal{G} \cap \mathcal{S} \subset V_2$ and every component $\CC$ of $V_2$ that is contained in $\mathcal{G} \cap \mathcal{S}$ is bounded by tori that bound solid tori inside $\mathcal{S}$.
\end{Lemma}
\begin{proof}
Note first that $\mathcal{G} \cap \mathcal{S}$ cannot contain any components of $V_1$ or $V'_2$, because such components would have at least one incompressible boundary torus.
Let now $\CC$ be a component of $V_2$ that is contained in $\mathcal{G} \cap \mathcal{S}$.
In particular, $\CC \cap V_{2, \partial} = \emptyset$ and hence the boundary of $\CC$ is a disjoint union of tori that are compressible in $\mathcal{S}$.
Consider a component $T \subset \partial \CC$ and let $D \subset \mathcal{S}$ be a compressing disk for $T$.
Then by \cite[Lemma \ref{Lem:compressingtorus}]{Bamler-LT-topology}, either $T$ bounds a solid torus, or $T \cup D$ is contained in an embedded ball.
But the latter case cannot occur, since the $S^1$-fiber direction of $\CC$ in $T$ is incompressible in $M$.
So $T$ bounds a solid torus, which by \cite[Lemma \ref{Lem:coverMbysth}]{Bamler-LT-topology} has to lie inside $\mathcal{S}$.
\end{proof}

\begin{Definition} \label{Def:GGp}
Let the subset $\mathcal{G}' \subset M$ to be the union of $\mathcal{G}$ with
\begin{enumerate}[label=(\arabic*)]
\item all components of $V_1$ that are diffeomorphic to $\Klein^2 \td{\times} I$ and adjacent to $\mathcal{G}$ or $\partial M$,
\item all components of $V_1$ that are diffeomorphic to $T^2 \times I$ and that are adjacent to $\mathcal{G} \cup \partial M$ on both sides,
\item all unions $\CC_1 \cup \CC_2$ where $\CC_1$ is a component of $V_1$ diffeomorphic to $T^2 \times I$ and adjacent to $\mathcal{G}$ or $\partial M$ on one side and adjacent to $\CC_2$, which is a component of $V'_2$ diffeomorphic to $\Klein^2 \td{\times} I$, on the other side.
\end{enumerate}
\end{Definition}

\begin{Lemma} \label{Lem:GGpisGGoutside}
Consider the cases A--C.
Every component of $V_1$, $V_2$ or $V'_2$ that is contained in $\mathcal{G}'$ and meets the boundary $\partial \mathcal{G}'$, already belongs to $\mathcal{G}$ or is adjacent to $\partial M$.
In other words, $\partial \mathcal{G}' \subset \partial \mathcal{G} \cup \partial M$.
In the cases A and B, the second option can be omitted.

In particular, any such component is either contained in $V_2$ if it is not adjacent to $\partial M$ or is contained in $V_1$ if it is adjacent to $\partial M$.
\end{Lemma}
\begin{proof}
This is a direct consequence of the definition of $\mathcal{G}'$ and Lemma \ref{Lem:bdrygoodisV2}.
\end{proof}

We can now state the main result of this subsection.

\begin{Proposition} \label{Prop:GGpp}
In cases A and B let $\mathcal{G}''$ be the union of all components of $\mathcal{G}'$ that share points with $M \setminus \mathcal{S}$.
In case C, let $\mathcal{G}''$ be the component of $\mathcal{G}'$ that is adjacent to $\partial M$, if it exists.
Then the following is true:
\begin{enumerate}[label=(\alph*)]
\item In cases A and B, every connected component of $M$ contains exactly one component of $\mathcal{G}''$.
In case C, $\mathcal{G}''$ is connected and possibly empty.
\item Let $\mathcal{S}''$ be the closure of $M \setminus \mathcal{G}''$.
Then $\mathcal{S}''$ is a disjoint union of finitely many solid tori ($\approx S^1 \times D^2$) each of which is incompressible in $M$.
\item Each component of $V_1$, $V_2$ or $V'_2$ that is contained in $\mathcal{G}''$ and that shares a boundary component with $\mathcal{G}''$ is contained in $\mathcal{G}$.
If such a component is adjacent to $\partial M$, then it is also contained in $V_1$ and does not intersect $\partial\mathcal{G}'' \setminus \partial M$.
If such a component is not adjacent to $\partial M$, then it is contained in $V_2$.
\item $\mathcal{G} \setminus \mathcal{G}'' \subset V_2$.
\item For every component $\CC''$ of $\mathcal{S}''$ there is a component $\CC$ of $V_1$ that is contained in $\CC''$ and adjacent to $\partial \CC''$.
Moreover, $\CC$ is diffeomorphic to either $S^1 \times D^2$ (and hence $\CC'' = \CC$) or $T^2 \times I$.
In the latter case, the component $\CC'$ of $V_2$ or $V'_2$ that is adjacent to $\CC$ and contained in $\CC''$ (i.e. $\partial \CC \setminus \partial \CC'' \subset \CC'$), is not contained in $\mathcal{G}$.
More precisely, if $\CC' \subset V'_2$, then $\CC' \approx S^1 \times D^2$ and if $\CC' \subset V_2$, then the $S^1$-fibers of $\CC' \cap V_{2, \textnormal{reg}}$ are contractible in $\CC''$.
\end{enumerate}
\end{Proposition}

\begin{proof}
For assertion (a) observe that $M = \mathcal{G} \cup \mathcal{S}$ and that every component of $M$ contains exactly one component of $M \setminus \mathcal{S}$.

Assertion (b)--(d) follow from Lemmas \ref{Lem:bdrygoodisV2}, \ref{Lem:solidtoriinsideSS} and \ref{Lem:GGpisGGoutside}:
For assertion (b) observe each component of $\mathcal{C} \subset \mathcal{S}''$, $\mathcal{C} \approx S^1 \times D^2$, is either equal to $M$ (which may happen in case C), and hence trivially incompressible in $M$, or it is adjacent to a Seifert fibration with incompressible generic $S^1$-fibers.
In the second case, this would imply that there is an embedded loop $\gamma \subset\partial \mathcal{C} \subset \mathcal{C}$ with the property that the image of $[\gamma]$ under the sequence $\pi_1 (\partial \mathcal{C} ) \to \pi_1 (\mathcal{C} ) \to \pi_1 (M)$ has infinite order.
So the image of the second map, $\pi_1 (\mathcal{C} ) \cong \IZ \to \pi_1 (M)$, must be infinite and therefore injective, which means that $\mathcal{C}$ is incompressible in $M$.
Finally, for assertion (d) observe that $\mathcal{G} \setminus \mathcal{G}'' \subset \mathcal{S}$.

For assertion (e) observe that $\CC''$ is either adjacent to $\partial M$ or to $\mathcal{G}''$ and hence to $V_2$.
So the component $\CC$ of $V_1$, $V_2$ or $V'_2$ that is adjacent to $\partial \CC''$ inside $\CC''$, is contained in $V_1$.
Since $\CC$ is contained in a solid torus, it cannot be diffeomorphic to $\Klein^2 \td\times I$.
So it is diffeomorphic to $S^1 \times D^2$ or $T^2 \times I$.
The rest follows from the definition of $\mathcal{G}'$.
Observe that in the case $\CC' \subset V_2$, the $S^1$-fibers of $\CC' \cap V_{2, \textnormal{reg}}$ are compressible in $M$, since otherwise $\CC' \subset \mathcal{G}$ and hence (by Definition \ref{Def:GGp}(2)) $\mathcal{C} \subset \mathcal{G}'$, which contradicts the fact that $\mathcal{G}''$ is a union of components of $\mathcal{G}'$.
Since $\CC''$ is incompressible in $M$, we conclude that the $S^1$-fibers of $\CC' \cap V_{2, \textnormal{reg}}$ are even contractible in $\mathcal{S}''$.
\end{proof}

\section{Preparations for the main argument} \label{sec:Preparations}
In this section we list smaller Lemmas that will be used in the main argument in section \ref{sec:mainargument}.

\subsection{Torus structures and torus collars}
We will make use of the following terminology to describe the geometry of collar neighborhoods in an approximate sense.

\begin{Definition} \label{Def:torusstruc}
Let $a > 0$.
A subset $P \subset M$ of a Riemannian manifold $(M, g)$ is called a \emph{torus structure of width $\leq a$} if there is a diffeomorphism $\Phi : T^2 \times [0,1] \to P$ such that $\diam \Phi(T^2 \times \{ s \}) \leq a$ for all $s \in [0,1]$ (here ``$\diam$'' denotes the extrinsic diameter in $(M,g)$).
The \emph{length} of $P$ is the distance between the two boundary components of $P$. \\
If $h, r_0 > 0$, then we say that $P$ is \emph{$h$-precise (at scale $r_0$)} if it has width $\leq h r_0$ and length $> h^{-1} r_0$.
\end{Definition}
Note that every torus structure of width $\leq a$ and length $L_1$ can be shortened to a torus structure of width $\leq a$ and length $L_2$ for any $L_2 < L_1$.

In the proof of Proposition \ref{Prop:firstcurvboundstep2}, we will moreover make use of the following variation of this notion:
\begin{Definition} \label{Def:toruscollars}
Consider a constant $a> 0$, a Riemannian manifold $(M,g)$ and a smoothly embedded solid torus $S \subset M, S \approx S^1 \times D^2$.
We say that $S$ \emph{has torus collars of width $\leq a$ and length up to $b$}, if for every point $x \in \Int S$ with $\dist(x, \partial S) \leq b$ there is a set $P \subset S$ that is diffeomorphic to $T^2 \times I$ such that:
$P$ is bounded by $\partial S$ and another smoothly embedded $2$-torus $T \subset S$ with $x \in T$ and $\diam T \leq a$.
\end{Definition}
So if $P \subset S$ (such that $\partial S \subset \partial P$) is a torus structure of width $\leq a$ and length $b$, then $S$ has torus collars of width $\leq a$ and length up to $b$.

We mention two conclusions, which we will use frequently.
\begin{Lemma} \label{Lem:collardiameter}
Assume that $S$ has torus collars of width $\leq a$ and length up to $b$.
Let $x \in \Int S$ with $\dist(x, \partial S) < b - 2a$ and choose $P \subset S$ according to Definition \ref{Def:toruscollars}.
Then $\dist(x, \partial S) \leq \diam P \leq \dist(x, \partial S) + 4a$.
\end{Lemma}
\begin{proof}
The first inequality is clear.
For the second inequality consider a minimizing geodesic $\gamma$ joining $\partial S$ with $x$.
By minimality, $\gamma \subset S$ and all points of $\gamma$ have distance $< b - 2a$ from $\partial S$.
Let $y \in P \setminus \partial S$ and assume that $\dist(y, \partial S) \leq b$.
So there is an embedded $2$-torus $T' \subset S$ with $y \in T'$ and $\diam T' \leq a$ and a set $P'$ that is diffeomorphic to $T^2 \times I$ and bounded by $\partial S$ and $T'$.

If $T'$ is disjoint from $T$, then $P' \subset P \setminus T$ and $\gamma$ must intersect $T'$ as it connects the two boundary components of $P$, one of which coincides with a boundary component of $P'$ and the other, which is disjoint from $P'$.
We conclude that $T'$ intersects $\gamma$ or $T$.
In the first case, $\dist(y, \gamma) \leq a$ and in the second case $\dist(y, \gamma) \leq \dist(y, x) \leq 2a$.
So in fact we have the strict bound $\dist(y, \partial S) < b$ and the bound $\dist(y, \gamma) \leq 2a$.
The first bound implies that \emph{all} points of $P \setminus \partial S$ have distance less than $b$ from $\partial S$ and the second bound implies that all these points are no more than $2a$ away from $\gamma$.
This implies the diameter bound.
\end{proof}

\begin{Lemma} \label{Lem:2T2timesI}
Consider two subsets $P_1, P_2 \subset M$ of a smooth $3$-manifold that are diffeomorphic to $T^2 \times I$.
Assume that one boundary component, $T_1$, of $P_1$ is contained in the interior of $P_2$ and the other boundary component, $T'_1$, is disjoint from $P_2$.
Assume also that conversely one boundary component, $T_2$, of $P_2$ is contained in the interior of $P_1$ and the other boundary component, $T'_2$ is disjoint from $P_1$.
Then $P_1 \cup P_2$ is diffeomorphic to $T^2 \times I$
\end{Lemma}

\begin{proof}
First observe that $P_1 \setminus P_2$ is a deformation retract of $P_1 \cup P_2$.
So since $T'_1$ is incompressible in $P_1$ and hence also in $P_1 \setminus P_2$, we find that it is also incompressible in $P_1 \cup P_2$.
This implies that $T_1$, being isotopic to $T'_1$, is incompressible in $P_1 \cup P_2$.
So $T_1$ is also incompressible in $P_2$.

By elementary $3$-manifold topology (see e.g. the proof of \cite[Proposition 1.7]{Hat}), this implies that $P_2 \setminus P_1 \approx T^2 \times [0,1)$ and hence $P_1 \cup P_2 = (P_2 \setminus P_1) \cup P_1 \approx T^2 \times I$.
\end{proof}

The next Lemma asserts that under the presence of a curvature bound, we can find a torus structure of small width around a cross-section of small diameter inside a given torus structure.
This fact will be used in the proof of Proposition \ref{Prop:firstcurvboundstep3}.
In the subsequent Lemma \ref{Lem:2loopstorus} we show that such a small cross-section exists if we can find two short loops that represent linearly independent homotopy classes inside the torus structure.

\begin{Lemma} \label{Lem:bettertorusstructure}
For any $K < \infty$, $L < \infty$ and $h > 0$ there is a constant $0 < \td\nu = \td\nu(K, L, h) < 1$ such that: \\
Let $(M,g)$ be a complete Riemannian manifold, consider a torus structure $P' \subset M$ of width $\leq 1$ and assume that $|{\Rm}| < K$ on $P'$.
Let $T \subset P'$  be an embedded $2$-torus that is incompressible in $P'$, that separates the two boundary components of $P'$ from one another, that has distance $\geq \frac12 L + 30$ from the boundary components of $P'$ and that satisfies $\diam T < \td\nu$.

Then there is a torus structure $P \subset P'$ of width $\leq h$ and length $> L$ such that $T \subset P'$ and such that the pair $(P', P)$ is diffeomorphic to $(T^2 \times [-2,2], T^2 \times [-1,1])$.
\end{Lemma}
\begin{proof}
By chopping off the ends of $P'$, we first construct a torus structure $P'_1 \subset P'$ of width $\leq 1$ and length $< L + 100$ such that the boundary tori of $P'_1$ have distance at least $5$ from the boundary tori of $P'$ and such that $T$ has distance of at least $\frac12 L + 20$ from $\partial P'_1$.
Then still $T \subset P'_1$.
Choose points $z_1, z_2 \in \partial P'_1$ in each boundary component of $P'_1$ and let $\gamma \subset M$ be a minimizing geodesic from $z_1$ to $z_2$.
Then $\gamma \subset P'$ and $\gamma$ intersects $T$ in a point $z$.

By the same construction as above, we choose $P'_2 \subset P'_1$ such that the boundary tori of $P'_2$ have distance of at least $5$ from the boundary tori of $P'_1$ and such that $T$ has distance of at least $\frac12 L + 10$ from $\partial P'_2$.
We still have $T \subset P'_2$.
Let now $x \in P'_2$ be an arbitrary point.
Consider minimizing geodesics $\gamma_1, \gamma_2 \subset M$ from $x$ to $z_1$, $z_2$.
Then again $\gamma_1, \gamma_2 \subset P'$ and one of these geodesics have to intersect $T$; without loss of generality assume that this geodesic is $\gamma_1$ and choose a point $x_1 \in \gamma_1 \cap T$.
Let $y_1 \in \gamma$ be a point with $\dist(z_1, y_1) = \dist(z_1, x)$ (we can find such a point since $\dist(z_1,x) < \dist(z_1, z_2)$).
We now apply Toponogov's Theorem using the lower sectional curvature bound $- K$:
Observe that $\dist(z_1, x_1), \dist(z_1, z) < L + 100$ and $\dist(x_1, z) < \td\nu$.
So the comparison angle $\beta = \cangle x_1 z_1 z$ (in the model space of constant sectional curvature $-K$) is bounded by a quantity $\beta_0 = \beta_0 (\td\nu, L, K)$ that goes to zero in $\td\nu$ whenever $L$ and $K$ are kept fixed.
By Toponogov's Theorem, we have $\cangle x z_1 y_1 \leq \beta \leq \beta_0$ and since the comparison triangle $\td{\triangle} x z_1 y_1$ is isosceles and the lengths of the hinges are bounded by $L+100$, we conclude that $\dist(x, y_1) < \beta_1(\td\nu, L, K)$, where $\beta_1(\td\nu, L, K)$ is a quantity that goes to zero in $\td\nu$ if $L$ and $K$ are kept fixed.
This implies in particular that
\[ \dist(z_1, z_2) \leq \dist(z_1, x) + \dist( z_2, x) \leq \dist(z_1, z_2) + 2 \beta_1(\td\nu, L, K). \]
Hence, if $\td\nu$ is small enough depending on $L$ and $K$, then we have the following bound for the comparison angle at $x$:
\begin{equation} \label{eq:cangleatxtorusstruct}
 \cangle z_1 x z_2 > 0.9 \pi.
\end{equation}
For the rest of the proof, fix such a $\td\nu > 0$ for which also $\beta_1(\td\nu, L, K) < 0.1h$.

By (\ref{eq:cangleatxtorusstruct}) the function $p : \Int P'_2 \to \IR$, $p(x) = \dist(z_1, x)$ is regular in a uniform sense and hence we can find a smooth unit vector field $\chi$ on $\Int P'_2$ such that the directional derivative of $p$ is uniformly positive everywhere, i.e.  $\chi \cdot p > c > 0$.
We can moreover choose a smoothing $p'$ of $p$ with $| p- p' | < 0.1h$ and $\chi \cdot p' > 0$ everywhere (compare with the techniques used in \cite{Grove-Shiohama-1977} and \cite[sec 3.3]{Meyer-1989}).
Let $P = (p')^{-1}(I)$ be the preimage of a closed subinterval $I \subset p'(P'_2)$ whose endpoints have distance $3$ from the endpoints of $p' (P'_2)$.
This implies that the preimage $(p')^{-1}(t)$ of every point $t \in I$ is far enough from the boundary of $P'_2$ and hence is compact.
Then in particular $T \subset P$.
So $P \approx \Sigma \times I$, for some connected, closed surface $\Sigma$ and $p'$ is the projection onto the second factor.
Since $T \subset P$, it follows that $\pi_1(\Sigma)$ contains a subgroup isomorphic to $\IZ^2$, which implies that $\Sigma \approx T^2$.

We now estimate the diameter of $(p')^{-1}(t)$ for each $t \in I$.
Let again $x \in P$ and consider as above the geodesics $\gamma_1, \gamma_2$ as well as the point $y_1 \in \gamma$ with $\dist(z_1, y_1) = \dist (z_1, x) = p(x)$.
Additionally, we construct $y_2 \in \gamma$ with $\dist (z_2, y_2) = \dist(z_2, x)$.
Then $\dist(y_1, y_2) \leq 0.2 h$.
In the case in which $\gamma_1$ intersects $T$, we conclude as above that $\dist(x, y_1) \leq 0.1 h$.
Analogously, if $\gamma_2$ intersects $T$, we have $\dist(x, y_2) \leq 0.1 h$ and hence $\dist(x,y_1) \leq 0.3 h$.
Let $y' \in \gamma$ now be a point with $\dist(z_1, y') = p'(x)$.
Then $\dist(y', y_1) = | p (x) - p' (x) | < 0.1 h$ and hence $\dist(y', x) < 0.4 h$.
This implies that $\diam (p')^{-1}(t) < 0.8 h < h$ for all $t \in I$.
So $P$ has width $\leq h$.

Finally, we bound the length of $P$ from below.
Consider points $x_1, x_2 \in \partial P$ in each boundary component and let $y'_1, y'_2 \in\gamma$ be points with $\dist (z_1, y'_1) = p'(x_1)$ and $\dist (z_1, y'_2) = p'(x_2)$.
Then by the last paragraph
\begin{multline*}
 \dist(x_1, x_2) > \dist (y'_1, y'_2) - 2 \cdot 0.4 h = | p'(x_1) - p'(x_2) | - 0.8 h \\
  = \ell( p'(P'_2) ) - 2 \cdot 3 - 0.8 h > \ell( p (P'_2) ) - 6 - h.
\end{multline*}
where $\ell( p (P'_2) )$ denotes the length of the interval $p(P'_2)$.
By assumption $p(P'_2) \geq 2 (\frac12 L + 10) = L + 20$.
So $\dist(x_1, x_2) > L + 14 - h > L$ for $h < 1$.
This implies that $P$ has length $> L$.
\end{proof}

\begin{Lemma} \label{Lem:2loopstorus}
For every $K < \infty$ there is a constant $\td\varepsilon_1 = \td\varepsilon_1(K) > 0$ such that: \\
Let $(M,g)$ a complete Riemannian manifold with boundary that is diffeomorphic to $T^2 \times I$ and $p \in M$ such that $B(p, 1) \subset M \setminus \partial M$.
Assume that $| {\Rm} | < K$ and assume that there are loops $\gamma_1, \gamma_2$ based at $p$ that represent two linearly independent homotopy classes in $\pi_1(M) \cong \IZ^2$.
Now if $m = \max \{ \ell(\gamma_1), \ell(\gamma_2) \} < \td\varepsilon_1$, then there is an embedded incompressible torus $T \subset M$ that separates the two ends of $M$ such that $p \in T$ and $\diam T < 10 m$.
\end{Lemma}
\begin{proof}
By the results of Cheeger, Fukaya and Gromov \cite{CFG}, there are universal constants $\rho = \rho(K) > 0$, $k < \infty$ such that we can find an open neighborhood $B(p, \rho) \subset V \subset M$ and a metric $g'$ on $V$ with $0.9 g < g' < 1.1 g$ with the following properties: 
There is a Lie group $H$ with at most $k$ connected components whose identity component $N$ is nilpotent and that acts isometrically and faithfully on the universal cover $(\td{V}, \td{g}')$.
The fundamental group $\Lambda = \pi_1(V)$ can be embedded into $H$ such the action of $H$ on $(\td{V}, \td{g}')$ restricted to $\Lambda$ is the action by deck transformations.
Moreover, $H$ is generated by $\Lambda$ and $N$.
Lastly, the injectivity radius of $(\td{V}, \td{g}')$ at any lift $\td{p} \in \td{V}$ of $p$ is larger than $\rho$.

Consider the dimension $d$ of the orbit $\td{T}$ of a lift $\td{p}$ under the action of $N$.
Since $V$ has to be non-compact, we have $d \leq 2$.
On the other hand, assuming $\td{\varepsilon}_1 < \rho$, the loops $\gamma_1, \gamma_2$ generate an infinite subgroup in $\Lambda = \pi_1(V)$ that does not have a finite index subgroup isomorphic to $\IZ$.
So $d = 2$.
Since $N \cap \Lambda$ is nilpotent and acts discontinuously on $\td{T}$, we have $N \cap \Lambda \cong \IZ^2$ and all orbits of the $N$-action are $2$ dimensional.
Consider the cover $\widehat\pi : (\widehat{V}, \widehat{g}') \to (V, g')$ corresponding to $N \cap \Lambda$.
Then $\widehat{V} \approx T^2 \times (0,1)$ and $\widehat{V} \to V$ has at most degree two.
The action of $N$ on $(\widehat{V}, \widehat{g}')$ exhibits $(\widehat{V}, \widehat{g}')$ as a warped product of a flat torus over an interval.
We can find loops $\gamma'_1, \gamma'_2$ based at a lift $\widehat{p}$ of $p$ each of which have $\widehat{g}'$-length $< 2 (1.1)^{1/2} m$.
This implies that the $T^2$-orbit $\widehat{T}$ of $\widehat{p}$ under $N$ has $g'$-diameter $< 4 \cdot 1.1^{1/2} m$.
Let $T = \widehat\pi (\widehat{T})$ be the projection of $\widehat{T}$.
Then $\diam_{g} T < 4 \cdot 1.1^{1/2} \cdot 0.9^{-1/2} m < 10 m$ and $\widehat\pi$ restricted to $\widehat{T}$ induces a map $f : T^2 \to M$ with $f(T^2) = T$, which has at most two sheets.

We show that the intersection number of $f$ with the line $\{ \mathop{\text{pt}} \} \times I \subset M \approx T^2 \times I$ is non-zero:
Consider the composition of $f$ with the projection $M \approx T^2 \times I \to T^2$.
This is a smooth and incompressible map of the form $T^2 \to T^2$. 
Hence, its degree is non-zero, which is equal to the sought intersection number.
We conclude that $T \subset M$ is a $2$-torus that separates the two boundary components of $M$.
\end{proof}

The next Lemma will be used in the proof of Lemma \ref{Lem:shortloopingeneralcase}.
\begin{Lemma} \label{Lem:1loop2d}
For every $K < \infty$, there are constants $\td\varepsilon_2 = \td\varepsilon_2 (K) > 0$ and $\Gamma' = \Gamma' (K) < \infty$ such that the following holds: \\
Let $(M,g)$ a $2$-dimensional, orientable Riemannian manifold and $p \in M$ a point such that the $1$-ball around $x$ is relatively compact in $M$.
Assume that $| {\Rm} |  < K$ on $M$ and assume that there is a loop $\gamma : S^1 \to M$ based at $p$ that is non-contractible in $M$ and has length $\ell (\gamma) < \td\varepsilon_2$.
Then there is an embedded loop $\gamma' \subset M$ that is also based at $p$, homotopic to $\gamma$ and that satisfies the following properties:
$\ell(\gamma') < 2 \ell(\gamma)$ and the geodesic curvatures of $\gamma'$ are bounded by $\Gamma'$.
\end{Lemma}

\begin{proof}
Similarly as in the proof of Lemma \ref{Lem:2loopstorus} there is a universal constant $\rho = \rho(K) > 0$ such that we can find a neighborhood $B(p, \rho) \subset V \subset M$ and a metric $g'$ on $V$ with $0.9 g < g' < 1.1 g$ such that the same conditions as before hold.
Note that $g'$ can moreover be chosen such that $|\nabla - \nabla'| < 0.1$ and such that the curvature of $g'$ is bounded by a universal constant $K' = K' (K) < \infty$ (see \cite{CFG}).
Hence, it suffices to construct a loop $\gamma'$ with $\ell_{g'} (\gamma') < 1.5 \ell_{g'} (\gamma)$ and on which the geodesic curvatures with respect to $g'$ are bounded.

As in the proof of Lemma \ref{Lem:2loopstorus} we conclude that either $(V, g')$ is a flat torus (in which case the Lemma is clear), or all orbits under the action of $N$ on the universal cover $(\td{V}, \td{g}')$ of $(V,g')$ are 1 dimensional.
In the latter case, this implies that $\Lambda \subset N$ and $(V, g')$ is a warped product of a circle over an interval $(-a,b)$ with $a, b > 0.9 \rho$ (we assume that $p$ lies in the fiber over $0 \in (-a, b)$).
Let $\varphi : (-a,b) \to (0,\infty)$ be the warping function.
By the curvature bound on $g'$ we have $| \varphi'' \varphi^{-1} | < K'$ on $(-a,b)$ (compare for example with \cite[Chp 3, sec 2.3]{Petersen}).

We now argue that this bound implies
\begin{equation} \label{eq:varphiprimevarphiinverse}
| \varphi' \varphi^{-1} | < C = C (K') \quad \text{on} \quad (- \tfrac12 a, \tfrac12 b).
\end{equation}
Let $x \in (- \frac12 a, \frac 12 b)$.
If $\varphi' (x) = 0$, then there is nothing to show.
Assume next that $\varphi' (x) > 0$.
Choose $x' \in [x - 0.1 \rho, x) \subset (-a, b)$ minimal with the property that $\varphi \leq \varphi (x)$ on $[x', x]$.
Then
\begin{equation} \label{eq:varphiKprimexprimex}
| \varphi '' | < K' \varphi (x) \qquad \text{on} \qquad [x', x].
\end{equation}
If $|x - x'| < 0.1 \rho$, then we must have $\varphi (x') = \varphi(x)$.
So, by the mean value inequality, $\varphi' (x'') = 0$ for some $x'' \in [x',x']$ and hence by (\ref{eq:varphiKprimexprimex}) we get $\varphi' (x) < \frac12 (0.1 \rho)^2 K'  \varphi (x)$.
If $|x - x'| = 0.1 \rho$, then, by the mean value inequality and the fact that $\varphi(x') > 0$, we have
\[ \varphi' (x'') = \frac{\varphi(x) - \varphi(x')}{x - x'} < \frac{\varphi(x)}{ 0.1 \rho}. \]
Using (\ref{eq:varphiKprimexprimex}), we get
\[ \varphi' (x) < \frac{\varphi(x)}{ 0.1 \rho} + \frac12 (0.1\rho)^2 K' \varphi(x). \]
This shows (\ref{eq:varphiprimevarphiinverse}) wherever $\varphi'  > 0$.
By exchanging the roles of $a$ and $b$, this bound follows similarly wherever $\varphi' < 0$.

The bound (\ref{eq:varphiprimevarphiinverse}) implies that the geodesic curvature on the circle $\gamma'$ through $p$ is bounded by $C$ and for sufficiently small $\ell_{\gamma'} (\gamma)$ we have $\ell_{g'} (\gamma') < 1.5 \ell_{g'} (\gamma)$.
\end{proof}

\subsection{Existence of short loops and compressing disks of bounded area}
In this subsection we establish the existence of short geodesic loops on surfaces of large diameter, but controlled area.
The main result of this subsection, Lemma \ref{Lem:shortloopingeneralcase}, will be used in the proof of Propositions  \ref{Prop:firstcurvboundstep3} and \ref{Prop:structontimeinterval}.
In the proof of Proposition  \ref{Prop:firstcurvboundstep3} it will enable us to apply Lemmas \ref{Lem:bettertorusstructure} and \ref{Lem:2loopstorus} and hence to find torus structures of small width.
\begin{Lemma} \label{Lem:annulus}
Let $\Sigma$ be a topological annulus and let $g$ be a symmetric, non-negative definite $2$-tensor on $\Sigma$ (i.e. a possibly degenerate Riemannian metric).
Assume that with respect to $g$ any smooth arc connecting the two boundary components of $\Sigma$ has length $\geq a$ and every embedded, closed loop of non-zero winding number in $\Sigma$ has length $\geq b$.
Then $\area \Sigma \geq a b$.
\end{Lemma}
\begin{proof}
Let $g'$ be an arbitrary metric on $\Sigma$.
If the Lemma is true for $g + \varepsilon g'$ for all $\varepsilon > 0$, then we obtain the result for $g$ by letting $\varepsilon \to 0$.
So we can assume in the following that $g$ is a Riemannian metric.

We can furthermore assume that $\Sigma = A(r_1, r_2) \subset \IC$ with $0 \leq r_1 < r_2 \leq \infty$ and that $g = r^{-2} f^2(r, \theta) g_{\textnormal{eucl}}$ for polar coordinates $(r, \theta)$.
By assumption, we have
\[ \int_{r_1}^{r_2} r^{-1} f(r, \theta) dr \geq a \qquad \text{for all} \qquad \theta \in [0, 2\pi] \]
and
\[ \int_0^{2\pi} f(r, \theta) d\theta \geq b \qquad \text{for all} \qquad r \in (r_1, r_2). \]
Hence
\[  \int_0^{2\pi} \int_{r_1}^{r_2}r^{-1} f(r, \theta) dr d\theta \geq 2 \pi a \]
and
\[ \int_{r_1}^{r_2} \int_0^{2\pi} r^{-1} f(r, \theta) d\theta dr \geq b \int_{r_1}^{r_2} r^{-1} dr = b \log \Big( \frac{r_2}{r_1} \Big). \]
So
\begin{multline*}
 2 \pi a b \log \Big( \frac{r_2}{r_1} \Big) \leq \bigg( \int_0^{2\pi} \int_{r_1}^{r_2} r^{-1} f(r, \theta) dr d \theta \bigg)^2 \\ \leq \bigg( \int_0^{2\pi} \int_{r_1}^{r_2} r^{-1} f^2(r, \theta) dr d\theta \bigg) \bigg( \int_0^{2\pi} \int_{r_1}^{r_2} r^{-1} dr d \theta \bigg)
 = 2 \pi (\area \Sigma) \log \Big( \frac{r_2}{r_1} \Big) \qedhere
\end{multline*}
\end{proof}
\vspace{3mm}

The following lemma is an application of the previous one.

\begin{Lemma} \label{Lem:separtingloopsinmultannulus}
Let $\Sigma \subset \IR^2$ be a compact, smooth domain whose boundary circles are denoted by $C_1, \ldots, C_{m}, C'_1, \ldots, C'_{m'}$ with $m, m' \geq 1$.
Moreover, let $g$ be a symmetric, non-negative definite $2$-tensor on $\Sigma$ (i.e. a degenerate Riemannian metric).

Choose constants $a, b > 0$ and assume that $\area \Sigma \leq ab$ and $\dist (C_i, C'_{i'}) > a$ for any $i = 1, \ldots, m$ and $i' = 1, \ldots, m'$ (both times with respect to $g$).
Then we can find a collection of pairwise disjoint, smoothly embedded loops $\gamma_1, \ldots, \gamma_n \subset \Int \Sigma$ with the property that $\gamma_1 \cup \ldots \cup \gamma_n$ separates $C_1 \cup \ldots \cup C_m$ from $C'_1 \cup \ldots \cup C'_{m'}$ and
\[ \ell(\gamma_1) + \ldots + \ell (\gamma_n) < b. \]
\end{Lemma}

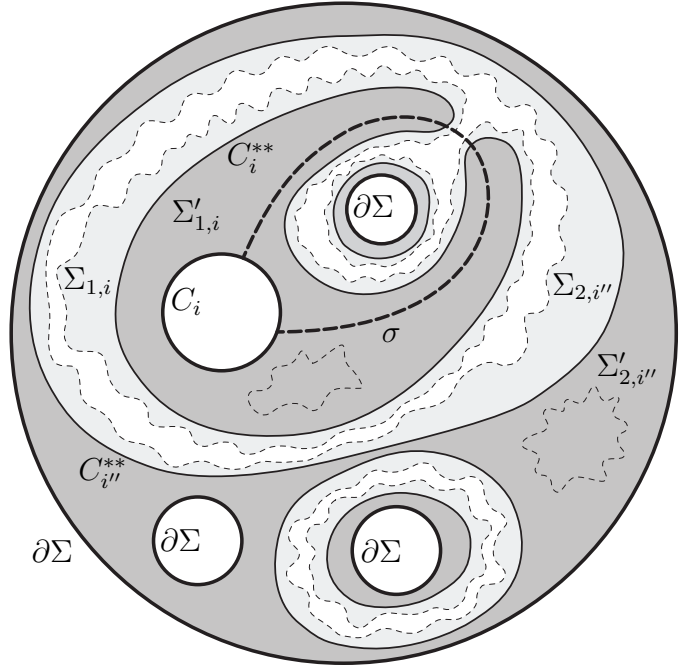
\begin{figure}[t] 
\begin{center}
%\begin{picture}(0,0)%
%\hspace{16mm}\includegraphics[width=11cm]{expandingsolidtorus}%
%\end{picture}%
\setlength{\unitlength}{2863sp}%
\begingroup\makeatletter\ifx\SetFigFont\undefined%
\gdef\SetFigFont#1#2#3#4#5{%
  \reset@font\fontsize{#1}{#2pt}%
  \fontfamily{#3}\fontseries{#4}\fontshape{#5}%
  \selectfont}%
\fi\endgroup%
\begin{picture}(4500,5800)(3500,0)
\hspace{43mm}\includegraphics[width=9cm]{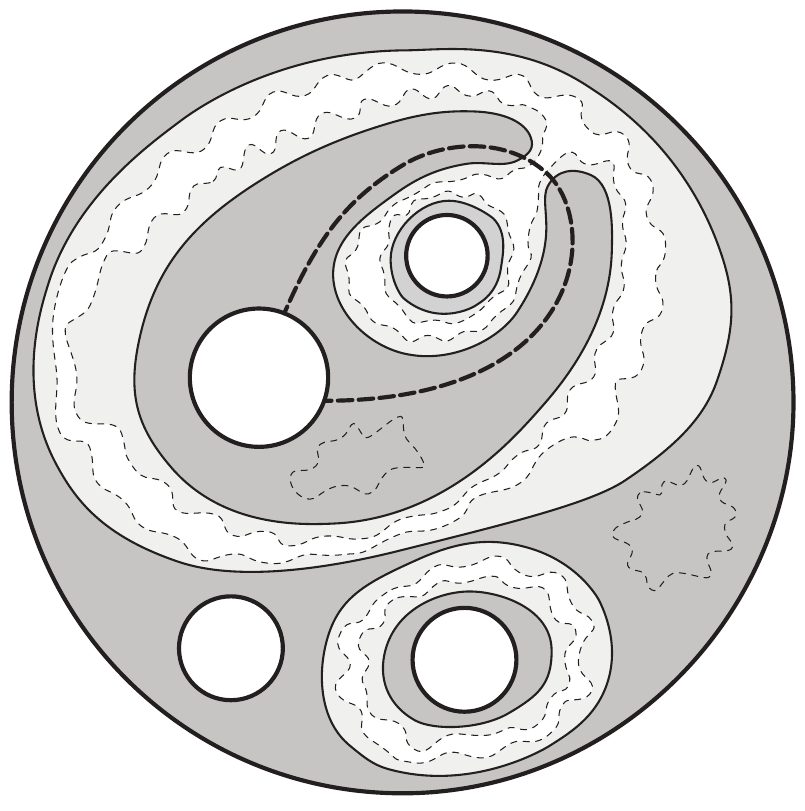}%
\put(-2600,2900){\makebox(0,0)[lb]{\smash{{\SetFigFont{12}{14.4}{\familydefault}{\mddefault}{\updefault}$\sigma$}}}}
\put(-2850,3980){\makebox(0,0)[lb]{\smash{{\SetFigFont{12}{14.4}{\familydefault}{\mddefault}{\updefault}$\partial \Sigma$}}}}
\put(-4400,3940){\makebox(0,0)[lb]{\smash{{\SetFigFont{12}{14.4}{\familydefault}{\mddefault}{\updefault}$\Sigma'_{1,i}$}}}}
\put(-3930,4430){\makebox(0,0)[lb]{\smash{{\SetFigFont{12}{14.4}{\familydefault}{\mddefault}{\updefault}$C^{**}_i$}}}}
\put(-5330,3360){\makebox(0,0)[lb]{\smash{{\SetFigFont{12}{14.4}{\familydefault}{\mddefault}{\updefault}$\Sigma_{1,i}$}}}}
\put(-4410,3180){\makebox(0,0)[lb]{\smash{{\SetFigFont{12}{14.4}{\familydefault}{\mddefault}{\updefault}$C_i$}}}}
\put(-5200,1710){\makebox(0,0)[lb]{\smash{{\SetFigFont{12}{14.4}{\familydefault}{\mddefault}{\updefault}$C^{**}_{i''}$}}}}
\put(-760,2610){\makebox(0,0)[lb]{\smash{{\SetFigFont{12}{14.4}{\familydefault}{\mddefault}{\updefault}$\Sigma'_{2,i''}$}}}}
\put(-1140,3310){\makebox(0,0)[lb]{\smash{{\SetFigFont{12}{14.4}{\familydefault}{\mddefault}{\updefault}$\Sigma_{2,i''}$}}}}
\put(-2780,1000){\makebox(0,0)[lb]{\smash{{\SetFigFont{12}{14.4}{\familydefault}{\mddefault}{\updefault}$\partial \Sigma$}}}}
\put(-4500,1100){\makebox(0,0)[lb]{\smash{{\SetFigFont{12}{14.4}{\familydefault}{\mddefault}{\updefault}$\partial \Sigma$}}}}
\put(-5600,1000){\makebox(0,0)[lb]{\smash{{\SetFigFont{12}{14.4}{\familydefault}{\mddefault}{\updefault}$\partial \Sigma$}}}}%
\end{picture}%
\caption{The planar domain $\Sigma$ and the subsets $\Sigma_{1,i}, \Sigma'_{1,i}, \Sigma_{2,i''}$ and $\Sigma'_{2,i''}$ of $\Sigma$, as used in the proof of Lemma \ref{Lem:separtingloopsinmultannulus}.
The subset $\Sigma'_{1,i}$ is diffeomorphic to a closed annulus and is bounded by $C_i$ and $C^{**}_i$.
The arc $\sigma : [0,1] \to \Sigma$ does not intersect $\Sigma_2$.
Its endpoints are contained in $C_i \subset \partial \Sigma$.
The arc cannot be homotoped into $C_i$, because it surrounds another boundary component of $\Sigma$.
So the boundary circle $C^{**}_{i''} \subset \Sigma'_{2,i''}$ separates the interior of $\Sigma'_{2,i''}$ from at least two boundary circles of $\Sigma$.
\label{fig:planardomain}}
\end{center}
\end{figure}
\begin{proof}
We will proceed by induction on $m + m'$.
For $m + m' = 2$, we are done by Lemma \ref{Lem:annulus}.
So assume without loss of generality that $m' \geq 2$.

Let $a_1$ be the infimum of all $a' \geq 0$ such that there is a smooth arc $\sigma : [0,1] \to \Sigma$ of length $2a'$ that either connects two distinct circles $C_{i_1}, C_{i_2}$ or that has both its endpoints lie in the same boundary circle $C_i$ and cannot be homotoped into $C_i$ while keeping its endpoints fixed.
Pick $\varepsilon > 0$ such that $\dist(C_i, C'_{i'}) > a + 4 \varepsilon$ for all $i = 1, \ldots, m$ and $i' = 1, \ldots, m'$ and such that $3\varepsilon < a_1$ in the case in which $a_1 > 0$.
Moreover, fix such a smooth arc $\sigma : [0,1] \to \Sigma$ with the properties as described above, that has length $2a'$ with $a' < a_1 + \varepsilon$ (compare with Figure \ref{fig:planardomain}).

Next consider the subsets
\begin{alignat*}{1}
\Sigma_1 &= \{ x \in \Sigma \;\; : \;\; \dist (x, C_1 \cup \ldots \cup C_m) < a_1 - \varepsilon \}, \\
\Sigma_2 &= \{ x \in \Sigma \;\; : \;\; \dist (x, C'_1 \cup \ldots \cup C'_{m'}) < a - a_1 + 3 \varepsilon \}.
\end{alignat*}
Then $\Sigma_1 \cap \Sigma_2 = \emptyset$ and $\sigma ([0,1]) \cap \Sigma_2 = \emptyset$.
By definition of $a_1$, the set $\Sigma_1$ is either empty (for $a_1 = 0$) or a disjoint union of connected, half-open domains $\Sigma_{1,1}, \ldots, \Sigma_{1,m}$ such that $C_i \subset \partial \Sigma_{1,i}$ for all $i = 1, \ldots, m$.

Fix some $i \in \{ 1, \ldots, m \}$ for this paragraph.
We claim that every smoothly embedded loop $\alpha : S^1 \to \Sigma_{1,i}$ inside $\Sigma_{1,i}$ can be homotoped inside $\Sigma$ to a multiple of a parameterization of $C_i$.
Fix such a loop $\alpha$.
To show our claim, we will construct a continuous map $H : S^1 \times [0,1] \to \Sigma$ such that $H (S^1 \times \{ 0 \} ) \subset C_i$ and $H(t,1) = \alpha(t)$.
Subdivide $S^1$ into sub-intervals $[t_1, t_2], [t_2, t_3], \ldots, [t_N, t_1] \subset S^1$ by parameters $t_1 = t_{N+1}, t_2, \ldots, t_N \in S^1$ that are arranged counterclockwise on $S^1$ such that $\ell (\alpha |_{[t_j, t_{j+1}]}) < \eps$ for all $j =1, \ldots, N$.
For each $j = 1, \ldots, N$ choose a smooth arc $\gamma^*_j : [0,1] \to \Sigma_{1,i}$ between $C_i$ and $\alpha (t_j)$ of length $\ell (\gamma^*_j) < a_1 - \eps$.
Then the concatenation $\sigma_j$ of $\gamma^*_j, \alpha |_{[t_j, t_{j+1}]}, \gamma^*_{j+1}$ has length less than $2 a_1 - \eps$.
So, by the definition of $a_1$, the arc $\sigma_j$ can be homotoped into $C_i$ while keeping its endpoints fixed.
It follows that we can find a continuous map $\varphi_j : [t_j, t_{j+1}] \times [0,1] \to \Sigma$ that agrees with $\gamma^*_j$ on $\{ t_j \} \times [0,1]$, with $\gamma^*_{j+1}$ on $\{ t_{j+1} \} \times [0,1]$, with $\alpha |_{[t_j, t_{j+1}]}$ on $[t_j, t_{j+1}] \times \{ 1 \}$ and such that $\varphi_j : ([t_j, t_{j+1}] \times \{ 0 \} ) \subset C_i$.
The homotopy $H$ can now be constructed by combining $\varphi_1, \ldots, \varphi_N$, hence proving our claim.

Consider in this paragraph the case $a_1 > 0$.
Let $\psi_1 \in C^\infty_c ( \Sigma_{1} )$ be some cutoff function such that
\[ \psi_1 \equiv 1 \qquad \text{on} \qquad \Sigma_1^* := \{ x \in \Sigma \;\; : \;\; \dist (x, C_1 \cup \ldots \cup C_m) < a_1 - 2\varepsilon \} \subset \Sigma_1. \]
Let $u_1 \in [0,1)$ be a regular value of $\psi_1$ and set
\[ \Sigma_1^{**} = \psi_1^{-1} ( [u_1, 1] ). \]
Then $\Sigma_1^{**}$ is a compact, possibly disconnected, planar domain and $C_1 \cup \ldots \cup C_m \subset \Sigma_1^* \subset \Sigma_1^{**} \subset \Sigma_1$.
For each $i =1 , \ldots, m$ let $\Sigma_{1,i}^{**} \subset \Sigma_1^{**}$ be the component of $\Sigma_1^{**}$ that contains $C_i$.
Note that then $\Sigma_{1,i}^{**} \subset \Sigma_{1,i}$.
The boundary circles of $\Sigma_{1,i}^{**}$ are freely homotopic to a multiple of $C_i$ within $\Sigma$ by our conclusion from the last paragraph.
It follows that there is exactly one boundary circle $C_i^{**} \subset \partial \Sigma_{1,i}^{**} \setminus C_i$ that separates $C_i$ from all the other boundary circles of $\Sigma$.
Therefore, all the boundary circles of $\Sigma^{**}_{1,i}$, other than $C_i, C^{**}_i$, bound closed disks in $D^2$ whose interiors are disjoint from $\Sigma^{**}_{1,i}$.
Let $\Sigma'_{1,i}$ now be the union of $\Sigma^{**}_{1,i}$ with all these disks.
Then $\Sigma'_{1,i}$ is diffeomorphic to a closed annulus, which is bounded by $C_i$ and $C^{**}_i$.
By construction, these boundary circles have distance $\dist (C_i, C^{**}_i) \geq a_1 - 2 \eps$ from each other.
Since each $\Sigma'_{1,i}$ arose from $\Sigma^{**}_{1,i}$ by adding disks that are disjoint from $\partial \Sigma$, we find that the resulting annuli $\Sigma'_{1,1}, \ldots, \Sigma'_{1,m}$ are pairwise disjoint and disjoint from $\Sigma_2$.

Next, construct a cutoff function $\psi_2 \in C^\infty_c (\Sigma_2)$ such that
\[ \psi_2 \equiv 1 \qquad \text{on} \qquad \Sigma_2^* := \{ x \in \Sigma \;\; : \;\; \dist (x, C'_1 \cup \ldots \cup C'_{m'}) < a - a_1 + 2 \eps \} \subset \Sigma_2, \]
let $u_2 \in [0,1)$ be a regular value of $\psi_2$ and set
\[ \Sigma_2^{**} = \psi_2^{-1} ([u_2, 1]). \]
Then
\[ C'_1 \cup \ldots \cup C'_{m'} \subset \Sigma^*_2 \subset \Sigma_2^{**} \subset \Sigma_2  \]
and
\[ \Sigma^{**}_2 \cap \big( \Sigma'_{1,1} \cup \ldots \cup \Sigma'_{1,m} \big) = \emptyset. \]
Let $\Sigma^{**}_{2,1}, \ldots, \Sigma^{**}_{2,m''}$, $m'' \leq m'$ be the components of $\Sigma^{**}_2$ that contain boundary circles of $\Sigma$.
As before, for each $i'' = 1, \ldots, m''$ consider the components of $\Sigma \setminus \Sigma^{**}_{2,i'}$ that are diffeomorphic to open disks, and denote the union of these components with $\Sigma^{**}_{2,i'}$ by $\Sigma'_{2,i''}$.
Then the subsets $\Sigma'_{1,1}, \ldots, \Sigma'_{1,m}, \Sigma'_{2,1}, \ldots, \Sigma'_{2,m''}$ are pairwise disjoint and hence
\begin{equation} \label{eq:areasSigma1plusareasSigma2}
 \area \Sigma'_{1,1} + \ldots + \area \Sigma'_{1,m} + \area \Sigma'_{2,1} + \ldots + \area \Sigma'_{2,m''} < \area \Sigma \leq ab.
\end{equation}
Moreover, the boundary circles of each $\Sigma'_{2,i''}$ that are not contained in $\partial \Sigma$ have distance $> a - a_1 + 2 \eps$ from $\partial \Sigma \cap \Sigma'_{2,i''}$.

Consider now the case in which $a_1 > 0$ and $\area \Sigma'_{1,1} + \ldots + \area \Sigma'_{1,m} < (a_1 - 2\varepsilon) b$.
Since the two boundary circles of each annulus $\Sigma'_{1,i}$ are at least $a_1 - \varepsilon$ apart from one another, we can use Lemma \ref{Lem:annulus} to find a geodesic loop $\gamma_i \subset \Sigma'_{1,i}$ for each $i =1 , \ldots, m$ that separates both boundary circles $C_i, C^{**}_i$ of $\Sigma'_i$ such that
\[ \ell (\gamma_i) \leq \frac{\area \Sigma'_{1,i}}{a_1 - 2\varepsilon} \qquad \text{for all} \qquad i = 1, \ldots, m. \]
Then $\ell(\gamma_1) + \ldots + \ell(\gamma_m) < b$ and $\gamma_1 \cup \ldots \cup \gamma_m$ separates $C_1 \cup \ldots \cup C_m$ from $C'_1 \cup \ldots \cup C'_{m'}$.
So in this case we are done.

Next consider the opposite case.
That is $a_1 = 0$ or $\area \Sigma'_{1,1} + \ldots + \area \Sigma'_{1,m} \geq (a_1 - 2\varepsilon) b$.
Then, by (\ref{eq:areasSigma1plusareasSigma2}), we must have
\[ \area \Sigma'_{2,1} + \ldots + \area \Sigma'_{2,m''} < (a - a_1 + 2\varepsilon) b. \]
In the following we will use the induction hypothesis on each domain $\Sigma'_{2,i''}$.
In order to do this, we first ensure that each $\Sigma'_{2,i''}$ has at most $m + m' -1$ boundary circles.

Fix some $i'' \in \{1, \ldots, m'' \}$.
Note that every boundary circle of $\Sigma'_{2, i''}$ that does not belong to $\partial \Sigma$ separates the interior of $\Sigma'_{2,i''}$ from at least one component of $\partial \Sigma \setminus \Sigma'_{2,i''}$.
Moreover, for every component of $\partial \Sigma \setminus \Sigma'_{2,i''}$ there is exactly one such boundary circle of $\Sigma'_{2, i''}$.
So, in order to prove that $\Sigma'_{2,i''}$ has at most $m + m' - 1$ boundary circles, we just need to show that there is a boundary circle $\Sigma'_{2, i''}$ that separates the interior of $\Sigma'_{2,i''}$ from at least two components of $\partial \Sigma \setminus \Sigma'_{2,i''}$.
To do this, we recall the arc $\sigma : [0,1] \to \Sigma$, whose image is disjoint from $\Sigma_2$ and hence also from $\Sigma'_{2,i''}$.
Let $C^{**}_{i''} \subset \partial \Sigma'_{2,i''} \setminus \partial \Sigma$ be the boundary circle that separates the interior of $\Sigma'_{2,i''}$ from the image of $\sigma$.
If $\sigma$ connects two distinct boundary circles $C_{i_1}, C_{i_2}$, then $C^{**}_{i''}$ does indeed separate the interior of $\Sigma'_{2,i''}$ from at least two components of $\partial \Sigma \setminus \Sigma'_{2,i''}$.
So assume now that we are in the second case, namely that both endpoints of $\sigma$ are contained in the same boundary circle $C_i \subset \partial \Sigma$ and that $\sigma$ cannot be homotoped into $C_i$ while keeping its endpoints fixed.
If $C_i$ was the only boundary circle of $\Sigma$ that is contained in the component of $\Sigma \setminus C^{**}_{i''}$ that does not contain $\Sigma'_{2,i''}$, then this component would be diffeomorphic to a half-open annulus.
This, however, would contradict the fact that $\sigma$ cannot be homotoped into $C_i$.
So there must be more than one boundary component of $\partial \Sigma$ that is separated by $C^{**}_i$ from the interior of $\Sigma'_{2,i''}$.
This finishes the proof that each $\Sigma'_{2,i''}$ has at most $m + m' -1$ boundary circles.

The conclusion from the previous paragraph enables us to use the induction hypothesis on each $\Sigma'_{2,i''}$ and to construct smoothly embedded, pairwise disjoint loops $\gamma_1, \ldots, \gamma_n \subset \Sigma'_{2,1} \cup \ldots \cup \Sigma'_{2,m''}$ such that $(\gamma_1 \cup \ldots \cup \gamma_n) \cap \Sigma'_{2,i''}$ separates $(C'_1 \cup \ldots \cup C'_{m'}) \cap \Sigma'_{2, i''}$ from $\partial \Sigma'_{2,i''} \setminus \partial \Sigma$, for each $i'' = 1, \ldots, m''$, and such that
\[ \ell (\gamma_1) + \ldots + \ell(\gamma_n) < \frac{\area \Sigma'_{2,1}}{a - a_1 + 2\varepsilon} + \ldots +  \frac{\area \Sigma'_{2,m''}}{a - a_1 + 2\varepsilon} < b. \]
This finishes the proof.
\end{proof}

The next Lemma provides a compressing disk of bounded area in a solid torus given a ``compressing planar domain'' of bounded area in a larger solid torus.
Such ``compressing planar domain'' of bounded area are produced by \cite[Proposition \ref{Prop:maincombinatorialresult}(a)]{Bamler-LT-topology}.

\begin{Lemma} \label{Lem:makediskfrommultiannulus}
For every $A, K < \infty$ there is an $\td{h}_0 = \td{h}_0 (A, K) < \infty$ such that the following holds:

Consider a Riemannian manifold $(M, g)$, a smoothly embedded solid torus $S \subset M$, $S \approx S^1 \times D^2$ and a collar neighborhood $P \subset S$, $\partial S \subset \partial P$, $P \approx T^2 \times I$ of $\partial S$ that is an $\td{h}_0$-precise torus structure at scale $1$.
Note that $S' = \Int S \setminus \Int P$ is a solid torus.
Assume also that $|{\Rm}|, |{\nabla\Rm}| < K$ on the $1$-neighborhood around $P$.

Let now $\Sigma \subset \IR^2$ be a compact, smooth domain and $f : \Sigma \to S$ a smooth map of $\area f < A$ such that $f (\partial \Sigma) \subset \partial S$ and such that $f$ restricted to only the exterior circle of $\Sigma$ is non-contractible in $\partial S$.

Then there is a smooth map $f' : D^2 \to M$ such that $f' (\partial D^2) \subset \partial S'$, such that $f' |_{\partial D^2}$ is non-contractible in $\partial S'$ and such that $\area f' < \area f + 1$.
\end{Lemma}

\begin{proof}
We first argue that there are constants $\varepsilon = \varepsilon(K) > 0$ and $C < \infty$ such that the following isoperimetric inequality holds:
Assume that $\td{h}_0 \leq \varepsilon$.
Then for any smooth loop $\gamma : S^1 \to P$ of length $\ell(\gamma) < \varepsilon$ that is contractible in $P$ there is map $h : D^2 \to M$ with $h |_{S^1} = \gamma$ and
\begin{equation} \label{eq:isoperimetricintorusstruc}
 \area h \leq C \ell(\gamma)^2.
\end{equation}

By the results of Cheeger, Fukaya and Gromov (see \cite{CFG}) there are universal constants $\rho = \rho(K) > 0$, $K' = K'(K) < \infty$ such that for every $p \in P$ we can find an open neighborhood $B(p, \rho) \subset V \subset M$ and a metric $g'$ on $V$ with $0.9 g < g' < 1.1 g$ whose curvature is bounded by $K'$ such that the injectivity radius in the universal cover $(\td{V}, \td{g}')$ of $(V, g')$ at every lift $\td{p} \in \td{V}$ of $p$ is larger than $\rho$.
Let $\pi : \td{V} \to V$ be the covering projection.

Now assume that $\varepsilon (K) \leq \min \{ \frac1{10} \rho(K), \frac1{10} K^{\prime -1/2} (K) \}$ and pick $p \in \gamma (S^1)$.
Since $\varepsilon \leq \frac1{10}\rho$, we can find a chunk $P'$ of $P \approx T^2 \times I$ (i.e. $P' \subset P$ corresponds to a subset of the form $T^2 \times I'$ for some subinterval $I' \subset I$), such that $P'$ contains the image of $\gamma$ and such that $P' \subset B(p, \rho) \subset V$.
Then $\gamma$ is also contractible in $P'$ and hence we can lift it to a loop $\td\gamma : S^1 \to \td{V}$ based at a lift $\td{p} \in \td{V}$ of $p$.
Using the exponential map at $\td{p}$ with respect to the metric $g'$, we can then construct a map $\td{h} : D^2 \to \td{V}$ with $\td{h} |_{\partial D^2} = \td\gamma$ and $\area_{g'} \td{h} \leq \frac12 C \ell_{g'} (\td\gamma)^2$, where $C < \infty$ is a universal constant (note that since $\ell_{g'} (\td\gamma) < (1.1)^{1/2} \ell_g (\gamma) < 2\varepsilon \leq \frac1{15} K^{\prime -1/2}$, we have upper and lower bounds on the Jacobian of this exponential map).
Now $h = \pi \circ \td{h}$ satisfies (\ref{eq:isoperimetricintorusstruc}), which proves the claim.

\begin{figure}[t] 
\begin{center}
%\begin{picture}(0,0)%
%\hspace{16mm}\includegraphics[width=11cm]{expandingsolidtorus}%
%\end{picture}%
\setlength{\unitlength}{2863sp}%
\begingroup\makeatletter\ifx\SetFigFont\undefined%
\gdef\SetFigFont#1#2#3#4#5{%
  \reset@font\fontsize{#1}{#2pt}%
  \fontfamily{#3}\fontseries{#4}\fontshape{#5}%
  \selectfont}%
\fi\endgroup%
\begin{picture}(4500,5800)(3500,0)
\hspace{43mm}\includegraphics[width=9cm]{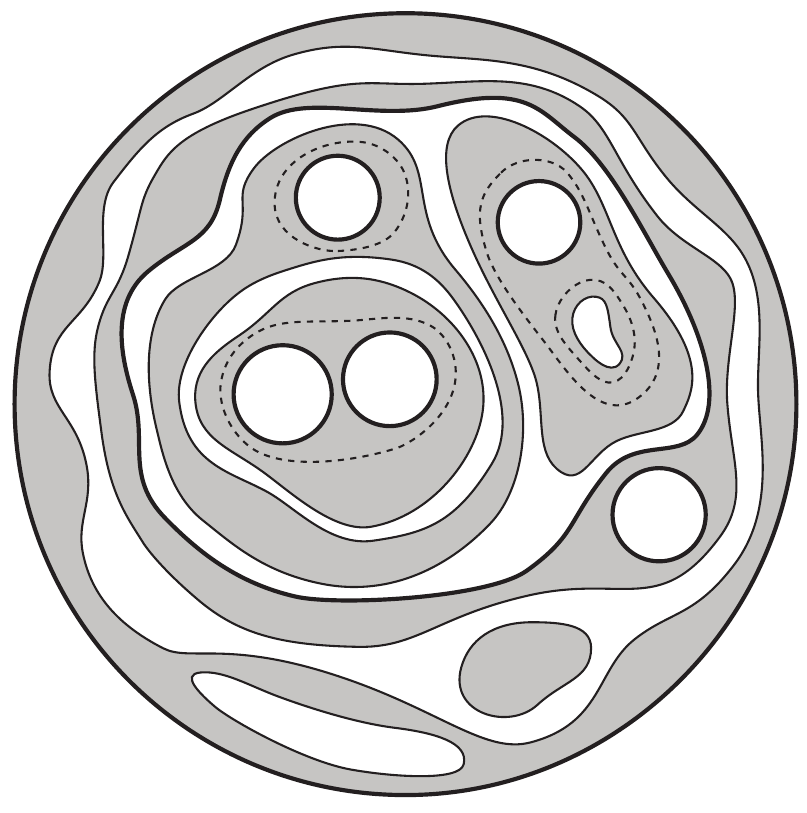}%
\put(-3300,3050){\makebox(0,0)[lb]{\smash{{\SetFigFont{12}{14.4}{\familydefault}{\mddefault}{\updefault}$\partial\Sigma$}}}}
\put(-4090,2950){\makebox(0,0)[lb]{\smash{{\SetFigFont{12}{14.4}{\familydefault}{\mddefault}{\updefault}$\partial\Sigma$}}}}
\put(-1850,2740){\makebox(0,0)[lb]{\smash{{\SetFigFont{12}{14.4}{\familydefault}{\mddefault}{\updefault}$Q_3$}}}}
\put(-3640,4380){\makebox(0,0)[lb]{\smash{{\SetFigFont{12}{14.4}{\familydefault}{\mddefault}{\updefault}$\partial \Sigma$}}}}
\put(-2130,1980){\makebox(0,0)[lb]{\smash{{\SetFigFont{12}{14.4}{\familydefault}{\mddefault}{\updefault}$C_l$}}}}
\put(-1730,1780){\makebox(0,0)[lb]{\smash{{\SetFigFont{12}{14.4}{\familydefault}{\mddefault}{\updefault}$Q_{j_0}$}}}}
\put(-3480,2680){\makebox(0,0)[lb]{\smash{{\SetFigFont{12}{14.4}{\familydefault}{\mddefault}{\updefault}$\gamma_2$}}}}
\put(-4050,4110){\makebox(0,0)[lb]{\smash{{\SetFigFont{12}{14.4}{\familydefault}{\mddefault}{\updefault}$\gamma_1$}}}}
\put(-4100,2010){\makebox(0,0)[lb]{\smash{{\SetFigFont{12}{14.4}{\familydefault}{\mddefault}{\updefault}$Q_1$}}}}
\put(-3340,2210){\makebox(0,0)[lb]{\smash{{\SetFigFont{12}{14.4}{\familydefault}{\mddefault}{\updefault}$Q_2$}}}}
\put(-2480,4700){\makebox(0,0)[lb]{\smash{{\SetFigFont{12}{14.4}{\familydefault}{\mddefault}{\updefault}$\gamma_3$}}}}
\put(-2100,3740){\makebox(0,0)[lb]{\smash{{\SetFigFont{12}{14.4}{\familydefault}{\mddefault}{\updefault}$\gamma_4$}}}}
\put(-1350,2100){\makebox(0,0)[lb]{\smash{{\SetFigFont{12}{14.4}{\familydefault}{\mddefault}{\updefault}$\partial \Sigma$}}}}
\put(-2170,4200){\makebox(0,0)[lb]{\smash{{\SetFigFont{12}{14.4}{\familydefault}{\mddefault}{\updefault}$\partial \Sigma$}}}}
\put(-5550,1000){\makebox(0,0)[lb]{\smash{{\SetFigFont{12}{14.4}{\familydefault}{\mddefault}{\updefault}$C_0$}}}}
\put(-6180,2300){\makebox(0,0)[lb]{\smash{{\SetFigFont{12}{14.4}{\familydefault}{\mddefault}{\updefault}$\partial \Sigma$}}}}%
\end{picture}%
\caption{The domains used in the proof of Lemma \ref{Lem:makediskfrommultiannulus}.
The preimage $\Sigma^* = f^{-1} (P) = Q_1 \cup \ldots \cup Q_p$ is the shaded region, which is bounded by $\partial \Sigma$ and the circles $C_0, \ldots, C_q$.
The boundary circles $\partial \Sigma$ and the circle $C_l$ are highlighted in bold.
The regions $Q_1, \ldots, Q_{n'}$, which are surrounded by $C_l$, are labeled in the picture.
The loops $\gamma_1, \ldots, \gamma_n$ are the dashed loops.
\label{fig:Qjs}}
\end{center}
\end{figure}
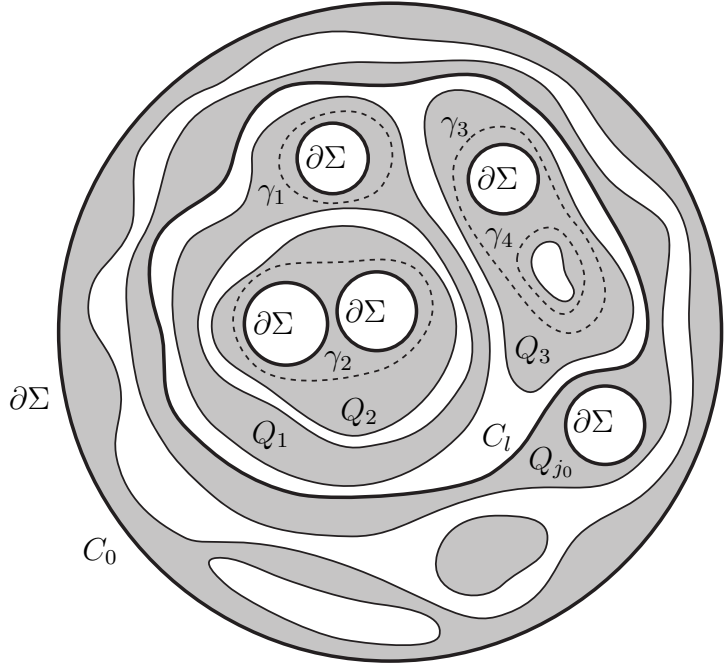
We can now prove the Lemma.
We first choose
\[ \td{h}_0 (A, K) = \min \Big\{ \frac{\varepsilon(K)}{A}, \; \varepsilon (K), \; \frac1 {2\sqrt{C} A} \Big\}. \]
Let $\sigma \subset S \setminus P$ be a loop that generates the fundamental group $\pi_1(S) \cong \IZ$.
Note that by our assumptions, $f$ has non-zero intersection number with $\sigma$.

We first perturb $f$ slightly to make it transversal to $\partial S'$.
This can be done such that we still have $\area f < A$ and that $\area f$ increases by less than $\frac12$.
So without loss of generality, we may assume in the following that $f$ is transversal to $\partial S'$.
We will construct $f'$ such that $\area f' < \area f + \frac12$.

So $\Sigma^* = f^{-1} (P) \subset \Sigma$ is a (possibly disconnected) compact smooth subdomain of $\Sigma$, which contains $\partial \Sigma$.
Note that $f(\partial \Sigma^* \setminus \partial \Sigma) \subset \partial S'$.
Denote the components of $\Sigma^*$ by $Q_1, \ldots, Q_p$.
Let $C_0 \subset \partial\Sigma$ be the outer boundary circle of $\Sigma$ and let $C_1 , \ldots, C_q \subset \partial\Sigma^* \subset \Sigma$ be the boundary circles of $\Sigma^*$ that are not boundary circles of $\Sigma$.
Each such circle $C_l \subset \Sigma \subset \IR^2$ bounds a disk $D_l \subset \IR^2$.
Set $\Sigma'_l = D_l \cap \Sigma$, which is a domain with non-empty interior.
Any two disks $D_{l_1}, D_{l_2}$ are either disjoint or one disk is contained in the other.
The same is true for the domains $\Sigma'_l$.
We can hence pick $\Sigma'_l$ minimal with respect to inclusion such that $f |_{\Sigma'_l}$ has non-zero intersection number with $\sigma$.
Such a $\Sigma'_l$ always exists, since $\Sigma'_0 = \Sigma$.
Let $j_0 \in \{1, \ldots, p \}$ be the index for which $C_l \subset \partial Q_{j_0}$.

We claim that $C_l$ is an \emph{interior} boundary circle of $Q_{j_0}$.
Assume the converse.
The intersection number of $f |_{Q_{j_0}}$ with $\sigma$ is zero, so if $C_l$ is an exterior boundary circle of $Q_{j_0}$, then the intersection number of $f$ restricted to the closure of $\Sigma'_l \setminus Q_{j_0}$ is non-zero.
This closure however is the disjoint union of some $\Sigma'_{l'} \subsetneq \Sigma'_l$.
Hence we can pick a $\Sigma'_{l'} \subsetneq \Sigma'_l$ on which $f$ has non-zero intersection number with $\sigma$.
This contradicts the minimal choice of $\Sigma'_l$ and proves that $C_l$ is an interior boundary circle of $Q_{j_0}$.
A direct consequence of this fact is that $l \neq 0$ and hence $C_l \not\subset \partial \Sigma$.
This implies $f(C_l) \subset \partial S'$.
We fix $l$ for the rest of the proof.

Next we show that for any circle $C_{l'} \subset \Sigma'_l \setminus C_l$ the restriction $f|_{C_{l'}}$ is contractible in $P$.
Note that for any such index $l'$ we have $\Sigma'_{l'} \subsetneq \Sigma'_l$ and hence $f|_{\Sigma'_{l'}}$ has intersection number zero with $\sigma$.
Moreover, every interior boundary circle of $\Sigma'_{l'}$ is also an interior boundary circle of $\Sigma$ (recall that by definition $\Sigma'_{l'} = D_{l'} \cap \Sigma$).
So, using the assumption of the lemma and the fact that $\partial S$ is a deformation retract of $P$, we find that $f$ restricted to any interior boundary circle of $\Sigma'_{l'}$ is contractible in $P$.
So, since $f |_{\Sigma'_{l'}}$ has intersection number zero with $\sigma$, this shows the desired fact.

We now use the facts that $f( C_l ) \subset \partial S'$ and $f( \partial\Sigma'_l \setminus C_l ) \subset \partial S$ (the latter fact is true since $\partial\Sigma'_l \setminus C_l \subset \partial \Sigma$).
We conclude that $f^{-1}(P) = Q_1 \cup \ldots \cup Q_p$ separates $C_l$ from $\partial\Sigma'_l \setminus C_l$.
Hence, after possibly rearranging the $Q_j$, we can find a $p' \in \{ 0, \ldots, p  \}$ such that $Q_1, \ldots, Q_{p'} \subset \Sigma'_l$, such that $Q_1 \cup \ldots \cup Q_{p'}$ is a neighborhood of $\partial \Sigma'_l \setminus C_l$ and such that each $Q_j$, $j =1, \ldots, p'$, contains at least one boundary circle of $\Sigma'_l$.
Observe that $f$ maps the boundary circles of each $Q_j$, $j = 1, \ldots, p'$, that are not contained in $\partial \Sigma'_l \setminus C_l$, and hence are not contained in $\partial \Sigma$, into $\partial S'$.
We now apply Lemma \ref{Lem:separtingloopsinmultannulus} to each $Q_j$ equipped with the pull-back $f^* (g)$ where we group the boundary circles of $Q_j$ into those that are contained in $\partial \Sigma'_l \setminus C_l$ and those that are not.
Doing this, we obtain pairwise disjoint, embedded loops $\gamma_1, \ldots, \gamma_n \subset Q_1 \cup \ldots \cup Q_{p'} \subset \Sigma'_l$ whose union separates $\partial \Sigma'_l \setminus C_l$ from $C_l$ and for which
\[ \ell(f|_{\gamma_1}) + \ldots + \ell(f|_{\gamma_n}) < \frac{\area f|_{Q_1}}{\td{h}_0^{-1}} + \ldots + \frac{\area f|_{Q_1}}{\td{h}_0^{-1}} < A \td{h}_0 . \]

For each $k = 1, \ldots, n$ let $D'_k \subset \IR^2$ be the disk that is bounded by $\gamma_k$.
By the separation property of the $\gamma_k$, the union $\Sigma'_l \cup D'_1 \cup \ldots \cup D'_n$ is equal to the disk $D_l$, i.e. the disks $D'_k$ cover the ``holes'' of $\Sigma'_l$.
Any two disks $D'_{k_1}, D'_{k_2}$ are either disjoint or one is contained in the other.
So after possibly rearranging these disks, we can find an $n' \in \{ 0, \ldots, n \}$ such that the disks $D'_1, \ldots, D'_{n'}$ are pairwise disjoint and such that still $\Sigma'_l \cup D'_1 \cup \ldots \cup D'_{n'} = D_l$.
For each $k = 1, \ldots, n'$, consider the intersection $D'_k \cap Q_{j'}$ of the disk $D'_k$ with the domain $Q_{j'}$ that contains $\gamma_k$ ($j' \in \{ 1, \ldots, p' \}$).
Each interior boundary circle of $D'_k \cap Q_{j'}$ is either an interior boundary circle of $\Sigma$, and hence $f$ restricted to this circle is contractible in $P$ or it is equal to one of the circles $C_{l'}$ and is contained in $\Sigma'_l \setminus C_l$.
By what we have proved earlier, $f |_{C_{l'}}$ is contractible in $P$.
Since $f(D'_k \cap Q_{j'}) \subset P$, it follows that $f |_{\gamma_k}$ is contractible in $P$.

Observe that $\ell (f|_{\gamma_k}) < A \td{h}_0 \leq \varepsilon$.
So we can use the isoperimetric inequality (\ref{eq:isoperimetricintorusstruc}) from the beginning of this proof to construct a map $f'_k : D_k \to M$ with $\area f'_k \leq C \ell( f |_{\gamma_k} )^2$ and $f'_k |_{\gamma_k} = f |_{\gamma_k}$.
Let now $f' : D_l \to M$ be the map that is equal to $f$ on $\Sigma'_l \setminus (D'_1 \cup \ldots \cup D'_{n'})$ and equal to $f'_k$ on each $D'_k$ ($k = 1, \ldots, n'$).
Then
\begin{multline*}
 \area f' < \area f + C \big( \ell(f |_{\gamma_1})^2 + \ldots + \ell ( f |_{\gamma_{n'}})^2 \big) \\
 < \area f + C \big( \ell(f |_{\gamma_1}) + \ldots + \ell ( f |_{\gamma_n}) \big)^2  \leq \area f + C A^2 \td{h}_0^2 \leq \area f + \tfrac12.
\end{multline*}
This proves the desired result (after smoothing $f'$).
\end{proof}

The following Lemma is the main result of this subsection.
It will be used to find short loops in the proofs of Propositions \ref{Prop:firstcurvboundstep3} and \ref{Prop:structontimeinterval}.
In the proof of Proposition \ref{Prop:structontimeinterval} it will also be used to ensure that these loops bound compressing disks of bounded area.

\begin{Lemma} \label{Lem:shortloopingeneralcase}
For every $\alpha > 0$ and every $A, K < \infty$ there are constants $\td{L}_0 = \td{L}_0 (\alpha, A) < \infty$ and $\td\alpha_0 = \td\alpha_0 (A, K) > 0$, $\td\Gamma = \td\Gamma (K) < \infty$ such that:

Let $(M, g)$ be a Riemannian manifold and $S \subset M$, $S \approx S^1 \times D^2$ a smoothly embedded solid torus that is incompressible in $M$.
Let $P \subset S$ be a torus structure of width $\leq 1$ and length $L \geq \td{L}_0$ with $\partial S \subset \partial P$ (i.e. the pair $(S, S \setminus \Int P)$ is diffeomorphic to $(S^1 \times D^2(1), S^1 \times D^2 (\frac12))$).

Consider a compact, smooth domain $\Sigma \subset \IR^2$ and a smooth map $f : \Sigma \to S$ with $f(\partial \Sigma) \subset \partial S$ such that $f$ restricted to the outer boundary circle of $\Sigma$ is non-contractible in $\partial S$ and $f$ restricted to all other boundary circles of $\Sigma$ is contractible in $\partial S$.
Moreover, assume that $\area f < A$.

\begin{enumerate}[label=(\alph*)]
\item Then there is a closed loop $\gamma : S^1 \to P$ that is non-contractible in $P$, but contractible in $S$ and that has length $\ell(\gamma) < \alpha$ and distance of at least $\frac13 L - 2$ from $\partial P$.
\item Assume that additionally $\alpha \leq \td\alpha_0$, that $|{\Rm}|, |{\nabla\Rm}| < K$ on the $1$-neighborhood around $P$, that $P$ has width $\leq \td{L}_0^{-1}$ and that $\pi_2(M) = 0$.
Then $\gamma$ in part (a) can be chosen such that its geodesic curvatures are bounded by $\td\Gamma$ and such that there is a map $h : D^2 \to M$ with $h |_{S^1} = \gamma$ of $\area h < \area f + 1$.
\end{enumerate}

\end{Lemma}

\begin{proof}
We first explain the general setup.
Assume that without loss of generality $\alpha < 0.1$ and set
\begin{equation} \label{eq:constantsinshortlooplem} 
\td{L}_0 (\alpha, A) = \max \Big\{ 12 \frac{A+1}{\alpha} + 3, \; \frac3{\alpha} + 12 \Big \}.
\end{equation}
We divide $P$ into three torus structures $P_1, P_2, P_3$ of width $\leq 1$ (in part (a)) or width $\leq \td{L}_0^{-1}$ (in part (b)) and length $> \frac13 L - 1$ in such a way that: $\partial S \subset \partial P_1$ and $P_i$ shares a boundary with $P_{i+1}$.
Then any point in $P_2$ has distance of at least $\frac13 L - 1$ from $\partial P$.
For later use, we define the solid tori
\[ S_1 = S, \qquad S_2 = \ov{S \setminus P_1}, \qquad S_3 =  \ov{S \setminus  (P_1 \cup P_2)}. \]
Moreover, let $P' \subset P_1$ be a torus structure of length $> \frac13 L - 4$ and width $\leq 1$ if we are in the setting of part (a) or of width $\leq \td{L}^{-1}_0$ if we are in the setting of part (b) such that $\partial S \subset \partial P'$ and such that $\dist (P', S_2) > 2$.
We also choose an embedded loop $\sigma \subset S \setminus P$ that generates $\pi_1(S) \cong \IZ$.

Next, we explain the strategy of the proof.
In the setting of part (b), we obtain by (\ref{eq:constantsinshortlooplem}) that, assuming $\td\alpha_0 \leq \td{h}_0 (A, K)$ (here $\td{h}_0$ is the constant from Lemma \ref{Lem:makediskfrommultiannulus}), the torus structure $P'$ is $\td{h}_0 (A, K)$-precise.
Hence Lemma \ref{Lem:makediskfrommultiannulus} yields a map $f' : D^2 \to M$ of $\area f' < \area f +1$ such that $f' |_{\partial D^2}$ parameterizes a non-contractible loop in $P'$.
Without loss of generality, we may assume that $f'$ is in fact an area minimizing map.
We then set $\Sigma_0 = D^2$ and $f_0 = f' : \Sigma_0 \to M$.
If we are in the setting of part (a), then we simply set $\Sigma_0 = \Sigma$ and $f_0 = f$.
So in either setting, $f_0 (\partial \Sigma_0) \subset P' \subset P_1$ and $f_0$ restricted to only the outer circle of $\Sigma_0$ is non-contractible in $P_1$.
Moreover, $\area f_0 < \area f +1$ and $f_0$ has non-zero intersection number with $\sigma$ (since $\pi_2 (M) = 0$ in part (b)).
In the following, we will construct the loop $\gamma$ from the map $f_0 : \Sigma_0 \to M$ for part (a) and (b) at the same time.
It will then only require a short argument that the additional assertions of part (b) hold.

Let $\varepsilon > 0$ be a small constant that we will determine later ($\varepsilon$ may depend on $M$ and $g$).
We can find a small (in the $C^1$-sense) homotopic perturbation $f_1 : \Sigma_0 \to M$ of $f_0 : \Sigma_0 \to M$ with $f_1 |_{\partial \Sigma_0} = f_0 |_{\partial \Sigma_0}$ that is not more than $\varepsilon$ away from $f_0$ such that the following holds: $f_1$ is transverse to $\partial P_2$ on the interior of $\Sigma_0$ and its area is still bounded: $\area f_1 < A +1$.
Note that $f_1$ still has non-zero intersection number with $\sigma$.

\begin{figure}[t] 
\begin{center}
%\begin{picture}(0,0)%
%\hspace{16mm}\includegraphics[width=11cm]{expandingsolidtorus}%
%\end{picture}%
\setlength{\unitlength}{2863sp}%
\begingroup\makeatletter\ifx\SetFigFont\undefined%
\gdef\SetFigFont#1#2#3#4#5{%
  \reset@font\fontsize{#1}{#2pt}%
  \fontfamily{#3}\fontseries{#4}\fontshape{#5}%
  \selectfont}%
\fi\endgroup%
\begin{picture}(4500,5800)(3500,0)
\hspace{43mm}\includegraphics[width=9cm]{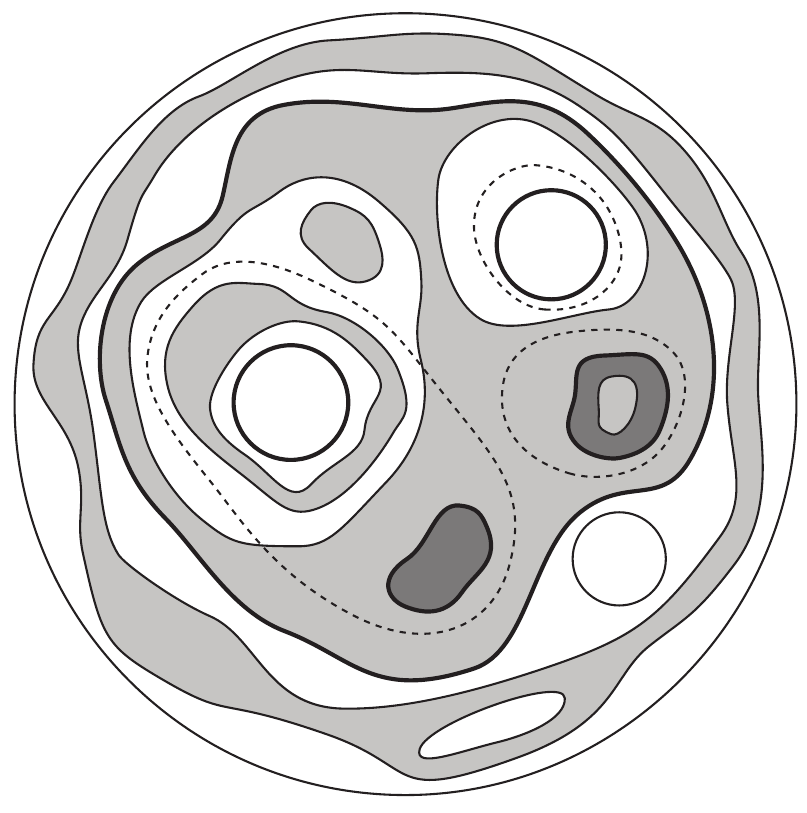}%
\put(-4040,3080){\makebox(0,0)[lb]{\smash{{\SetFigFont{12}{14.4}{\familydefault}{\mddefault}{\updefault}$C_1 \hspace{-1mm}\subset$}}}}
\put(-3990,2800){\makebox(0,0)[lb]{\smash{{\SetFigFont{12}{14.4}{\familydefault}{\mddefault}{\updefault}$\partial \Sigma_0$}}}}
\put(-4500,3400){\makebox(0,0)[lb]{\smash{{\SetFigFont{12}{14.4}{\familydefault}{\mddefault}{\updefault}$Q_3$}}}}
\put(-3130,1980){\makebox(0,0)[lb]{\smash{{\SetFigFont{12}{14.4}{\familydefault}{\mddefault}{\updefault}$C_3$}}}}
\put(-3880,1480){\makebox(0,0)[lb]{\smash{{\SetFigFont{12}{14.4}{\familydefault}{\mddefault}{\updefault}$\gamma_2$}}}}
\put(-2470,3250){\makebox(0,0)[lb]{\smash{{\SetFigFont{12}{14.4}{\familydefault}{\mddefault}{\updefault}$\gamma_3$}}}}
\put(-2050,3010){\makebox(0,0)[lb]{\smash{{\SetFigFont{12}{14.4}{\familydefault}{\mddefault}{\updefault}$C_4$}}}}
\put(-3530,4110){\makebox(0,0)[lb]{\smash{{\SetFigFont{12}{14.4}{\familydefault}{\mddefault}{\updefault}$Q_4$}}}}
\put(-4800,1210){\makebox(0,0)[lb]{\smash{{\SetFigFont{12}{14.4}{\familydefault}{\mddefault}{\updefault}$Q_1$}}}}
\put(-3870,4780){\makebox(0,0)[lb]{\smash{{\SetFigFont{12}{14.4}{\familydefault}{\mddefault}{\updefault}$Q_j = Q_2$}}}}
\put(-2480,4700){\makebox(0,0)[lb]{\smash{{\SetFigFont{12}{14.4}{\familydefault}{\mddefault}{\updefault}$\gamma_1$}}}}
\put(-1650,1800){\makebox(0,0)[lb]{\smash{{\SetFigFont{12}{14.4}{\familydefault}{\mddefault}{\updefault}$\partial \Sigma_0$}}}}
\put(-2160,4190){\makebox(0,0)[lb]{\smash{{\SetFigFont{12}{14.4}{\familydefault}{\mddefault}{\updefault}$C_2 \hspace{-1mm} \subset$}}}}
\put(-2090,3890){\makebox(0,0)[lb]{\smash{{\SetFigFont{12}{14.4}{\familydefault}{\mddefault}{\updefault}$\partial \Sigma_0$}}}}
\put(-4610,4400){\makebox(0,0)[lb]{\smash{{\SetFigFont{12}{14.4}{\familydefault}{\mddefault}{\updefault}$C_0$}}}}
\put(-6180,2300){\makebox(0,0)[lb]{\smash{{\SetFigFont{12}{14.4}{\familydefault}{\mddefault}{\updefault}$\partial \Sigma_0$}}}}%
\end{picture}%
\caption{The domains and loops used in the proof of Lemma~\ref{Lem:shortloopingeneralcase}.
The light gray domains represent $Q_1, \ldots, Q_p$.
The dark gray domains represent the components of $f^{-1}_1 (\Int S_3)$.
The loops $C_0, C_1, \ldots, C_q$ are drawn in bold and the loops $\gamma_1, \ldots, \gamma_n$ are the dashed loops.
The loop $C_0$ is the outer boundary component of $Q_j$ and the loops $C_1, \ldots, C_q$ are each either contained in $\partial \Sigma_0$ or bound components of $f^{-1}_1 (\Int S_3)$.
The domain $Q''$ is bounded by the loops $C_0, C_1, \ldots, C_q$.
\label{fig:Qjsgammas}}
\end{center}
\end{figure}
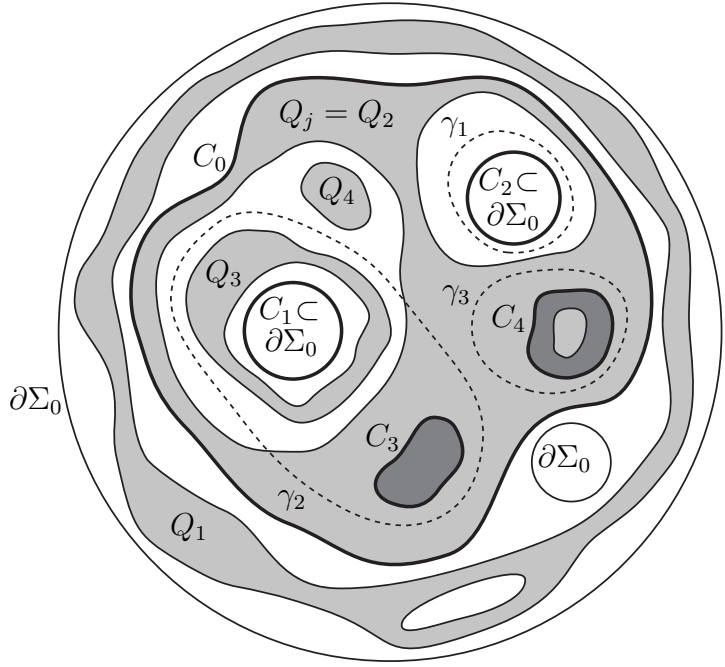
Consider all components $Q_1, \ldots, Q_p$ of $f_1^{-1} (S_2) \subset \Sigma_0 \subset \IR^2$ (see Figure \ref{fig:Qjsgammas} for an illustration).
Note that $p > 0$.
Each $Q_j$ can be extended to a disk $D_j \subset \IR^2$ by filling in its inner circles.
Let $Q'_j = D_j \cap \Sigma_0$ for each $j = 1, \ldots, p$.
Then any two disks, $D_{j_1}, D_{j_2}$ are either disjoint or one disk is contained in the other.
The same statement holds for the sets $Q'_j$.
So, after possibly changing the order of these disks, we can find a $p' \in \{ 1, \ldots, p \}$ such that the subsets $Q_1, \ldots, Q_{p'}$ are pairwise disjoint and such that $Q'_1 \cup \ldots \cup Q'_p = Q'_1 \cup \ldots \cup Q'_{p'}$.
It follows that $f_1$ restricted to $Q'_1 \cup \ldots \cup Q'_{p'}$ has the same non-zero intersection number with $\sigma$ as $f_1$.
So there are indices $j \in \{ 1, \ldots, p \}$ such that $f_1 |_{Q'_j}$ has non-zero intersection number with $\sigma$.
We can then choose an index $j \in \{ 1, \ldots, p \}$ with this property and such that $Q'_j$ is minimal, i.e. $f_1 |_{Q'_j}$ has non-zero intersection number with $\sigma$, but $f_1 |_{Q'_{j'}}$ has intersection number zero with $\sigma$ whenever $Q'_{j'} \subsetneq Q'_j$.

Let $C_0 = \partial D_j$ and observe that $f_1(C_0) \subset \partial S_2$.
Consider the set $Q'_j \setminus f_1^{-1} (\Int S_3)$ and observe that this set contains $C_0$.
Let $Q''$ be the component of $Q'_j \setminus f_1^{-1} (\Int S_3)$ that contains $C_0$.
Denote by $C_1, \ldots, C_q \subset \partial Q''$ all its other boundary circles.
These circles are either inner boundary circles of $\Sigma_0$ and are mapped into $\partial S_1$ under $f_1$ (in part (a)) or they belong to $f_1^{-1} (\partial S_3)$ and hence are mapped into $\partial S_3$ under $f_1$.
So for each $l = 1, \ldots, q$ the image $f(C_l) \subset \partial S_1 \cup \partial S_3$ has distance of at least $\frac13 L - 1$ from $f(C_0) \subset \partial S_2$.
We can hence apply Lemma \ref{Lem:separtingloopsinmultannulus} to $Q''$ equipped with the pull-back $f_1^* (g)$ and obtain pairwise disjoint, embedded loops $\gamma_1, \ldots, \gamma_n \subset \Int Q''$ such that $\gamma_1 \cup \ldots \cup \gamma_n$ separates $C_0$ from $C_1, \ldots, C_q$ and such that
\[ \ell ( f_1 |_{\gamma_1} ) + \ldots + \ell (f_1 |_{\gamma_n} ) < \frac{A + 1}{\frac13 L - 1} \leq \frac{A + 1}{\frac13 \td{L}_0 - 1} \leq \tfrac14 \alpha. \]
Every loop $\gamma_k$ bounds a topological disk $D'_k$ for $k = 1, \ldots, n$ and any two such disks, $D'_{k_1}, D'_{k_2}$, are either disjoint or one disk is contained in the other.
So, after possibly changing the order of these disks, we can find a number $n' \in \{ 1, \ldots, n \}$ such that $D'_1, \ldots, D'_n$ are pairwise disjoint and such that $D'_1 \cup \ldots \cup D'_{n'} = D'_1 \cup \ldots \cup D'_n$.
Note moreover, that $C_0$ is disjoint from this union.
So using the separation property of the $\gamma_k$ we conclude
\[ C_1, \ldots, C_q \subset D'_1 \cup \ldots \cup D'_n = D'_1 \cup \ldots \cup D'_{n'}. \]
This implies that $Q'' \cup D'_1 \cup \ldots \cup D'_{n'} = D_j$.
The intersection number of $f_1$ with $\sigma$ on $(D'_1 \cup \ldots \cup D'_{n'}) \cap \Sigma_0$ is the same as the intersection number on $D_j \cap \Sigma_0 = Q'_j$, which is non-zero.
So we can find an index $k \in \{ 1, \ldots, n' \}$ such that $f_1 |_{D'_k \cap \Sigma_0}$ has non-zero intersection number with $\sigma$.
Set $D' = D'_k$ and $\gamma' = \gamma_k$.

We now argue that $\gamma'$ has a point in common with $f_1^{-1} (S_2) = Q_1 \cup \ldots \cup Q_p$.
If not, then $D' \cap f_1^{-1} (S_2) \subsetneq Q'_j$ is the disjoint union of some of the $Q_{j'}$.
Similarly as before, we can adjust the order of these components and find a $p'' \in \{ 1, \ldots, p \}$ such that $Q'_1, \ldots, Q'_{p''}$ are pairwise disjoint and $D' \cap f_1^{-1} (S_2) = Q'_1 \cup \ldots \cup Q'_{p''}$.
By the minimality property of $Q'_j$, the intersection number of $f_1$ with $\sigma$ is zero on $Q'_1, \ldots, Q'_{p''}$.
So $f_1$ restricted to $D' \cap f_1^{-1}(S_2)$ has intersection number zero with $\sigma$, contradicting the fact that $f_1$ restricted to $D' \cap \Sigma_0$ does not.
Hence $\gamma'$ has to intersect $f_1^{-1} (S_2)$ and thus $f_1(\gamma') \cap S_2 \neq \emptyset$.
Since $\gamma' \subset Q'' = Q'_j \setminus f_1^{-1} (\Int S_3)$, we have $f_1(\gamma') \cap \Int S_3 = \emptyset$.
It follows that $f_1(\gamma') \cap P_2 \neq \emptyset$.
This implies that all points of $f_1(\gamma')$ have distance of at least $\frac13 L - 1 - \alpha > \frac13 L - 2$ from $\partial P$ and that $f_1(\gamma') \subset P$.
Since the intersection number of $f_1 |_{D' \cap \Sigma_0}$ with $\sigma$ is non-zero, $f_1 |_{\gamma'}$ has to be non-contractible in $P$, but contractible in $S$.
This establishes part (a) of the Lemma with $\gamma = f_1 |_{\gamma'}$.

Assume now for the rest of the proof that we are in the setting of part (b).
Then $Q'_j = D_j$ is a disk and $D' \subset \Sigma_0$.
Moreover, $f_1$ is an $\varepsilon$-perturbation of the (stable) area minimizing map $f_0 : \Sigma_0 \approx D^2 \to M$.
By \cite[Theorem 8.1]{Gul}, $f_0$ is an immersion on $\Int \Sigma_0$.
So we can additionally assume that the perturbation $f_1$ is a graph over $f_0$.

Consider the following regions: Let $B(P_2, 1)$ and $B(P_2,2) \subset P \setminus P'$ be the (open) $1$ and $2$-tubular neighborhoods of $P_2$ and let $\Sigma_1$ and $\Sigma_2$ be the components of $f^{-1}_0 (B(P_2,1))$ and $f_0^{-1}(B(P_2, 2))$ that contain $\gamma'$, i.e. $\gamma' \subset \Sigma_1 \subset \Sigma_2 \subset \Int \Sigma_0$.
By the results of \cite[Theorem 3]{Sch}, we obtain a bound on the second fundamental form of $f_0 (\Sigma_1)$ that only depends on $K$: $|{A_{f_0(\Sigma_1)}}| < K'(K)$.
Moreover, this bound and the bound on the curvature on $B(P_2, 1)$ gives us a curvature bound $K'' = K''(K) < \infty$ of the metric $f^*_0 (g)$ on $\Sigma_1$ that only depends on $K$.
Since $f_1$ was assumed to be a graph over $f_0 (\Sigma)$ and a $C^1$-small perturbation of $f_0$, we conclude that
\[ \ell (f_0 |_{\gamma'}) \leq 2 \ell (f_1 |_{\gamma'}) < \tfrac12 \alpha \]
if $\varepsilon$ is small enough depending on these bounds.
The loop $\gamma'$ is non-contractible in $\Sigma_1$, because otherwise $f_1 |_{\gamma'}$ would be contractible in $P$.
So we can apply Lemma \ref{Lem:1loop2d} to conclude that if $\td{\alpha}_0 < \td\varepsilon_2 (K'')$, then there is an embedded loop $\gamma'' \subset \Sigma_1$ that intersects $\gamma'$, is homotopic to $\gamma'$ in $\Sigma_1$ and that has the following properties:
$\ell( f_0 |_{\gamma''} ) \leq 2 \ell( f_1 |_{\gamma'} ) < \alpha$ and the geodesic curvature on $\gamma''$ in $(\Sigma_1, f^*_0 (g))$ is bounded by $\Gamma'(K'')$.
The loop $\gamma''$ still bounds a disk $D'' \subset \Sigma_0$ whose area under $f_0$ is bounded by $\area f +1$.
Let now $\gamma = f_0 |_{\gamma''}$.
Then the geodesic curvature on $\gamma$ in $(M, g)$ is bounded by some constant $\td\Gamma = \td\Gamma ( \Gamma'(K''(K)), K'(K)) = \td\Gamma (K) < \infty$.
This establishes assertion (b).
\end{proof}

\section{The main argument} \label{sec:mainargument}
In this section we will frequently use notions that were introduced in \cite{Bamler-LT-Perelman}.
For example, for any Ricci flow with surgery $\MM$, any time $t$, point $x \in \MM(t)$ and scale $r_0 > 0$, we set (compare with (\ref{eq:defofrho1inpartD}) and \cite[Definition \ref{Def:rhoscale}]{Bamler-LT-Perelman})
\[ \rho_{r_0} (x,t) = \sup \big\{  r \in (0, r_0] \;\; : \;\; \sec_t \geq - r^{-2} \quad \text{on} \quad B(x,t,r) \big\} \]

\subsection{The geometry on late, but short time-intervals} \label{subsec:firstcurvbounds}
In the following, we will analyze long-time existent Ricci flows with surgery $\MM$ as defined in \cite[section \ref{sec:IntroRFsurg}]{Bamler-LT-Perelman}.
Using the tools from \cite[sections \ref{sec:Perelman}, \ref{sec:maintools}]{Bamler-LT-Perelman}, we will give a bound on the curvature and the geometry in regions of the manifold on a time-interval of small, but uniform size.
This description is achieved in three steps, the last step and main result of this subsection being Proposition \ref{Prop:firstcurvboundstep3}.

In the first step, Lemma \ref{Lem:firstcurvboundstep1}, we will bound the curvature by a uniform constant $K_1 t_0^{-1}$ away from finitely many, pairwise disjoint embedded solid tori $S_1(t_0), \ldots, \linebreak[1] S_m(t_0)$ and on a time-interval of the form $[(1-\tau_1) t_0, t_0]$, where $t_0$ is large and $\tau_1$ a uniform constant (see assertions (a), (b)).
We will moreover find a curvature bound of the form $K' t_0^{-1}$ on those solid tori $S_i (t_0)$ whose normalized diameter is bounded by some given constant $L$ (see assertion (c)).
Here the normalized diameter is the diameter divided by $t_0^{1/2}$ and $K'$ is a constant that depends on the normalized diameter of $S_i (t_0)$.
We will moreover show that if a solid torus $S_i (t_0)$ has large normalized diameter at time $t_0$, then it also has large normalized diameter at all times of the interval $[(1-\tau_1) t_0, t_0]$, i.e. solid tori cannot ``grow too fast'' (see assertion (d)).
The function $\Delta_1$ will hereby serve as a lower bound for the normalized diameter at earlier times.
Note that the constant $\tau_1$, determining the size of the time-interval $[(1-\tau_1) t_0, t_0]$, does not depend on $L$ or the normalized diameter of $S_i (t_0)$.
We will also construct long collar torus structures inside each solid torus (see assertion (e)) and show that they become $\nu$-collapsed in at least one direction for any given $\nu > 0$ (see assertion (g)), but $\ov{w}$-non-collapsed in a local universal cover (see assertion (f)).
Note that $\nu$ can be chosen independently of $L$ or the diameter of $S_i$ and $\ov{w}$ is a universal constant, not depending on these quantities as well.
These independences will become important in the subsequent steps.
In the following lemma, the collar torus structures will be the closures of complements $S_i (t_0) \setminus W_i$, where $W_i \subset S_i (t_0)$ is a solid torus within $S_i (t_0)$.
Finally, we mention that, for technical reasons, we will also prove a version of these statements in which the analysis takes place in a subset $U \subset \MM (t_0)$.

\begin{Lemma}[first step] \label{Lem:firstcurvboundstep1}
There are a continuous function $\delta : [0, \infty) \to (0, \infty)$ and constants $K_1 < \infty$ and $\tau_1, \ov{w}, w^\#, \mu^\# > 0$ as well as continuous, non-decreasing functions $\Delta_1, K'_1 : (0, \infty) \to (0, \infty)$ with $\Delta_1(d) \to \infty$ as $d \to \infty$ and a non-increasing function $\tau'_1 : (0, \infty) \to (0, \infty)$ such that the following holds: \\
For every $L < \infty$ and $\nu > 0$ there are constants $T_1 = T_1(L) < \infty$, $w_1 = w_1(L, \nu) > 0$ such that: \\
Let $\MM$ be a Ricci flow with surgery on the time-interval $[0, \infty)$ with normalized initial conditions that is performed by $\delta(t)$-precise cutoff.
Consider the constant $T_0 < \infty$ and the function $w : [T_0, \infty) \to (0, \infty)$ from \cite[Proposition \ref{Prop:thickthindec}]{Bamler-LT-Perelman}.
Assume that $t_0 = r_0^2 \geq \max \{ T_1, 2 T_0 \}$ and that $w(t) < w_1$ for all $t \in [\frac12 t_0, t_0]$.
Assume moreover that all components of $\MM (t_0)$ are irreducible and not diffeomorphic to spherical space forms and that all surgeries on the time-interval $[\frac12 t_0, t_0]$ are trivial. \\
Let $U \subset \MM(t_0)$ be a subset with either $U = \MM(t_0)$ or
\begin{enumerate}[label=$-$]
\item $U$ is a smoothly embedded solid torus ($\approx S^1 \times D^2$).
\item There is a closed subset $U' \subset U$ that is diffeomorphic to $T^2 \times I$ with $\partial U \subset \partial U'$, whose boundary components have time-$t_0$ distance of at least $2 r_0$ from one another, and a $T^2$-fibration $p : U' \to I$ such that the fiber through every $x \in U'$ has time-$t_0$ diameter $< \mu^\# \rho_{r_0} (x, t_0)$.
\item All points of $\partial U$ are $w^\#$-good at scale $r_0$ and time $t_0$ (see \cite[Definition \ref{Def:goodness}]{Bamler-LT-Perelman} for the definition of the goodness property).
\end{enumerate} \vspace{2mm}
Then there are sub-Ricci flows with surgery $S_1, \ldots, S_m \subset \MM$ on the time-interval $[(1-\tau_1) t_0, t_0]$ such that their final time-slices $S_1(t_0), \ldots, S_m (t_0)$ form a collection of pairwise disjoint, incompressible solid tori ($\approx S^1 \times D^2$) in $\Int U$.
Moreover, there are subsets $W_i \subset S_i(t_0)$ ($i = 1, \ldots, m$) such that for all $i = 1, \ldots, m$
\begin{enumerate}[label=(\alph*)]
\item The pair $(S_i (t_0), W_i)$ is diffeomorphic to $(S^1 \times D^2(1), S^1 \times D^2(\frac12))$.
\item For all $x \in U$ with $\dist_{t_0} (x, U \setminus (S_1(t_0) \cup \ldots \cup S_m(t_0))) \leq 100 r_0$, the point $(x, t_0)$ survives until time $(1- \tau_1) t_0$ and for all $t \in [(1- \tau_1) t_0, t_0]$ we have
\[  |{\Rm}| (x,t) < K_1 t_0^{-1} . \]
\item We have $\diam_{t_0} S_i(t_0) > 100 r_0$ and if $\diam_{t_0} S_i(t_0) \leq L r_0$, then we have the curvature bound
\begin{multline*}
 \qquad\qquad | {\Rm_t} | < K'_1 ( r_0^{-1} \diam_{t_0} S_i (t_0)) t_0^{-1} \quad \text{on} \quad S_i(t) \quad \\
  \text{for all} \quad t \in \big[ (1-\tau'_1( r_0^{-1} \diam_{t_0} S_i(t_0)) t_0, t_0 \big].
\end{multline*}
Furthermore, $S_i$ is non-singular on the time-interval $[ (1-\tau'_1( r_0^{-1} \diam_{t_0} S_i \linebreak[1] (t_0)) t_0, \linebreak[1] t_0]$.
\item For all $t \in [(1-\tau_1) t_0, t_0]$ we have
\[  \diam_t S_i(t) > \min \big\{ \Delta_1 (r_0^{-1} \diam_{t_0} S_i(t_0)), L \big\} r_0. \]
\item At time $t_0$, the closure of $S_i(t_0) \setminus W_i$ is a torus structure of width $\leq r_0$ and length 
\[ \dist_{t_0} (\partial S_i (t_0), \partial W_i) = \min \big\{ \diam_{t_0} S_i(t_0) - 2r_0 , Lr_0 \big\}. \]
\item All points of $S_i(t_0) \setminus W_i$ are locally $\ov{w}$-good at scale $r_0$ and time $t_0$.
\item For every point $x \in S_i(t_0) \setminus W_i$, there is a loop $\sigma \subset \MM(t_0)$ that is based at $x$, that is incompressible in $\MM(t_0)$ and that has length $\ell_{t_0} (\sigma) < \nu r_0$.
\item In the case in which $U = \MM(t_0)$, we have for all $t \in [(1-\tau_1) t_0, t_0]$:
\begin{enumerate}[label=$-$]
\item There is a closed subset $U'_{i,t} \subset S_i (t)$ that is diffeomorphic to $T^2 \times I$ with $\partial S_i (t) \subset \partial U'_{i,t}$ whose boundary components have time-$t$ distance of at least $2 \sqrt{t}$ from each other and a $T^2$-fibration $p_{i,t} : U'_{i,t} \to I$ such that the fiber through every $x \in U'_{i,t}$ has time-$t$ diameter $< \mu^\# \rho_{\sqrt{t}} (x, t)$.
\item The points on $\partial S_i$ are $w^\#$-good at scale $\sqrt{t}$ and time $t$.
\end{enumerate}
\end{enumerate}
\end{Lemma}

The following proof makes use of the decomposition of the thin part $\MM_{\textnormal{thin}} (t_0)$ into subsets $V_1, V_2, V'_2$, as described in Proposition \ref{Prop:MorganTianMain}.
Using the results of the topological analysis of this decomposition from subsection \ref{subsec:topimplications}, we will identify the good components of $V_1, V_2, V'_2$, which exhibit a collapse along incompressible fibers, and we will understand their distribution.
With the help of Lemma \ref{Lem:unwrapfibration}, \cite[Proposition \ref{Prop:curvcontrolgood}]{Bamler-LT-Perelman} (``bounded curvature around good points'') and \cite[Proposition \ref{Prop:curvcontrolincompressiblecollapse}]{Bamler-LT-Perelman} (``Bounded curvature at bounded distance from sufficiently collapsed and good regions''), we will then find a uniform curvature bound on these good components.
The bad components of $V_1, V_2, V'_2$ are covered by solid tori or they form subsets that are diffeomorphic to $T^2 \times I$ or $\Klein^2 \td\times I$ and are adjacent to good components.
Using \cite[Proposition \ref{Prop:curvcontrolincompressiblecollapse}]{Bamler-LT-Perelman} (``Bounded curvature at bounded distance from sufficiently collapsed and good regions''), we will extend our uniform curvature bound to the second type of bad components, therefore establishing uniform curvature control outside of finitely many solid tori $S_i (t_0)$.
Within these solid tori, we will apply \cite[Proposition \ref{Prop:curvcontrolincompressiblecollapse}]{Bamler-LT-Perelman} (``Bounded curvature at bounded distance from sufficiently collapsed and good regions'') to obtain a distance-dependent curvature bound on the collars of the $S_i (t_0)$.
In order to show that these collars are long torus structures, we will show that for each $i$ the distance of the boundary $\partial S_i (t_0)$ to $V_2 \cap S_i (t_0)$ becomes arbitrarily large.
A comparison geometry argument will furthermore show that these collars become arbitrarily collapsed in one direction.
Finally, the statement that solid tori of large diameter also have large diameter at slightly earlier times will be a consequence of \cite[Proposition \ref{Prop:slowdiamgrowth}]{Bamler-LT-Perelman} (``Controlled diameter growth of regions whose boundary is sufficiently collapsed and good'').

\begin{proof}
Observe that it suffices to construct the functions $\Delta_1$ and $K'_1$ in such a way that they satisfy all the claimed properties except for continuity, since all properties stay true after decreasing the values of $\Delta_1$ and increasing the values of $K'_1$.

The function $\delta(t)$ will be assumed to be bounded by the corresponding functions from \cite[Corollary \ref{Cor:Perelman68} and Propositions \ref{Prop:thickthindec}, \ref{Prop:curvcontrolgood}, \ref{Prop:curvcontrolincompressiblecollapse}, \ref{Prop:curvboundinbetween} and \ref{Prop:slowdiamgrowth}]{Bamler-LT-Perelman}.
We also set
\[ \mu^\# = \min \{ w_0(\min \{ \mu_1, \tfrac1{10} \}, \ov{r}( \cdot, 1), K_2 (\cdot, 1)), \mu_1, \tfrac1{10} \}, \]
where $w_0$ is the constant from \cite[Proposition \ref{Prop:MorganTianMain}]{Bamler-LT-Perelman}, $\mu_1$ is the constant from \cite[Lemma \ref{Lem:unwrapfibration}]{Bamler-LT-Perelman} and $\ov{r}, K_2$ are the functions from \cite[Corollary \ref{Cor:Perelman68}]{Bamler-LT-Perelman}.
If $U = \MM(t_0)$, then we set $\mu^\circ = \mu^\#$ and if $U$ is a solid torus, we set $\mu^\circ = \min \{ \mu_1, \frac1{10} \}$.

Next, we make a remark on the constant $w^\#$.
This constant appears in the conditions of the Lemma in the case in which $U$ is a solid torus and in assertion (h), which holds in the case $U = \MM(t_0)$.
In the following proof, both of these cases will be dealt with simultaneously.
In the case in which $U = \MM(t_0)$, the constant $w^\#$ will be determined and will never be used.
In the case in which $U$ is a solid torus, $w^\#$ will be assumed to be given and all universal constants, that are determined in this case, may depend on it.
Note that this does not create a circular argument since one could carry out the following proof first for the case $U = \MM(t_0)$, obtaining a set of constants and functions 
\begin{equation} \label{eq:constantsfromstep1}
 K_1, \tau_1, \ov{w}, \Delta_1, K'_1, \tau'_1, T_1, w_1
\end{equation}
as well as $w^\#$ and then one could carry out the proof again in the case in which $U$ is a solid torus, obtaining another set of constants and functions as listed in (\ref{eq:constantsfromstep1}).
The final set of constants and functions will then be the minima of the two values obtained for each $\tau_1, \ov{w}, \Delta_1, \tau'_1, T_1, w_1$ in each case and the maxima of the two values obtained for each $K_1, K'_1, T_1$ in each case.

We now carry out the main argument.
Apply \cite[Proposition \ref{Prop:thickthindec}]{Bamler-LT-Perelman} to obtain a decomposition $\MM(t) = \MM_{\thick}(t) \cup \MM_{\thin}(t)$ for all $t \in [\frac12 t_0, t_0]$.
Consider for a moment the case in which $U$ is a solid torus.
Since $U$ cannot contain an incompressible torus, none of the boundary tori of $\MM_{\thick}(t_0)$ can be contained in $U$.
Let $T' \subset U'$ be a $T^2$-fiber of $p$ with $\dist_{t_0} (\partial U, T') = r_0$.
Then $\diam_{t_0} T' < \mu^\# r_0 \leq \frac1{10} r_0$.
Assuming $w(t_0) < \frac1{10}$, every component of $\partial \MM_{\thick} (t_0)$ has diameter $< \frac1{10} r_0$ (see \cite[Proposition \ref{Prop:thickthindec}(b)]{Bamler-LT-Perelman}).
This implies that if $T' \cap \MM_{\thick} (t_0) \neq \emptyset$, then $U$ is contained in the $2r_0$-tubular neighborhood of $\MM_{\thick}(t_0)$.
In this case we have a curvature bound on $U$ on a small time-interval with final time $t_0$ (see \cite[Proposition \ref{Prop:thickthindec}(c) and (d)]{Bamler-LT-Perelman}) and we are done by setting $m = 0$.
On the other hand, if $T' \cap \MM_{\thick} (t_0) = \emptyset$ then $U$ is contained in the $2r_0$-tubular neighborhood of $\MM_{\thin}(t_0)$.
We will assume from now on that in the case in which $U$ is a solid torus, $U$ is contained in a $2r_0$-tubular neighborhood of $\MM_{\thin}(t_0)$.

We will now apply Proposition \ref{Prop:MorganTianMain} with $\mu \leftarrow \mu^\circ$.
Observe for this that next to each component of $\partial \MM_{\thick} (t_0)$ there is a torus structure of width $\leq 10 w(t_0) r_0$ and length $2 r_0$ inside $\MM_{\thick} (t_0)$.
In the case in which $U = \MM (t_0)$, we apply Proposition \ref{Prop:MorganTianMain} with $M$ being the union of $\MM_{\thin} (t_0)$ with these torus structures and $g \leftarrow r_0^{-2} g(t_0)$.
If $U$ is a solid torus, we apply Proposition \ref{Prop:MorganTianMain} with $M \leftarrow U$.
So either by the assumption of the Lemma or by \cite[Proposition \ref{Prop:thickthindec}]{Bamler-LT-Perelman} for sufficiently small $w(t_0)$, condition (i) of Proposition \ref{Prop:MorganTianMain} is satisfied.
Condition (ii) follows from \cite[Proposition \ref{Prop:thickthindec}(e)]{Bamler-LT-Perelman} assuming $w(t_0) <  \min \{ \frac1{10} w_0 (\mu^\circ, \ov{r} (\cdot, 1), \linebreak[1] K_2 (\cdot, 1)), \frac1{10} \}$.
Condition (iii) is a consequence of \cite[Corollary \ref{Cor:Perelman68}]{Bamler-LT-Perelman} if  $t_0 > T_{\ref{Cor:Perelman68}} (w_0 (\mu^\circ, \ov{r} (\cdot, 1), K_2 (\cdot, \linebreak[1] 1)), \linebreak[1] 1, 2)$.
Note that we can assume that $T_{\ref{Cor:Perelman68}}$ is monotone in the first parameter.
Now look at the conclusions of Proposition \ref{Prop:MorganTianMain}.
We first consider the case in which a component $M'$ of $M$ is diffeomorphic to an infra-nilmanifold or a manifold that carries a metric of non-negative sectional curvature and we have $\diam_{t_0} M' < \mu^\circ \rho_{r_0} (x, t_0)$ for all $x \in M'$.
By the assumptions of the Lemma, $M'$ is not a spherical space form or a quotient of $S^1 \times S^2$, so it is either an infra-nilmanifold or a quotient of $T^3$.
By Lemma \ref{Lem:unwrapfibration}(v), all points in $M'$ are $w_1(\mu^\circ)$-good at scale $r_0$ and time $t_0$.
Thus for large $t_0$ we obtain a curvature bound on all of $M'$ at time $t_0$ and slightly before by \cite[Proposition \ref{Prop:curvcontrolgood}]{Bamler-LT-Perelman}.
For the rest of the proof, we can exclude these components from $M$ and assume that we have a decomposition $M = V_1 \cup V_2 \cup V'_2$ satisfying the properties (a1)--(c5) of Proposition \ref{Prop:MorganTianMain}.

Next, we apply the discussion of subsection \ref{subsec:topimplications}---in particular Proposition \ref{Prop:GGpp}---to this decomposition.
Consider the set $\mathcal{G} \subset M$ that we obtained there in Definition \ref{Def:GG}.
\begin{Claim1}
There are universal constants $w^*_1, {w'_1}^*, \alpha^*_1 > 0$ and $K^*_1, T^*_1 < \infty$ such that if $t_0 > \max \{ T^*_1, 2 T_0 \}$ and $w(t) < w^*_1$ for all $t \in [\frac12 t_0, t_0]$, then
\[ |{\Rm}| < K^*_1 t_0^{-1} \quad \text{on} \quad P(x, t_0, 2 r_0, - (\alpha^*_1 r_0)^2) \quad  \text{for all} \quad x \in \mathcal{G} \cup \MM_{\thick}(t_0) \cup \partial U,
\]
where the parabolic neighborhood $P(x, t_0, r_0, - (\alpha^*_1 r_0)^2)$ is always non-singular.
Moreover, all points of $\mathcal{G} \cup \partial U$ are ${w'_1}^*$-good at scale $r_0$ and time $t_0$.
\end{Claim1}
\begin{proof}
If $x \in \mathcal{G}$, then $x$ is $w_1(\mu^\circ)$-good at scale $r_0$ and time $t_0$ by Lemma \ref{Lem:unwrapfibration}.
If $x \in \partial U$, then $x$ is $w^\#$-good at scale $r_0$ and time $t_0$.
So in these cases, the curvature bound follows from \cite[Proposition \ref{Prop:curvcontrolincompressiblecollapse}]{Bamler-LT-Perelman} for sufficiently large $t_0$ and small $w(t_0)$.
If $x \in \MM_{\thick}(t_0)$, then the curvature bound is a direct consequence of \cite[Proposition \ref{Prop:thickthindec}(d)]{Bamler-LT-Perelman}.
\end{proof}

Now consider the set $\mathcal{G}' \supset \mathcal{G}$ as in Definition \ref{Def:GGp}.
In the next claim, we extend the curvature control onto $\mathcal{G}'$.
Recall that $\partial \mathcal{G}' \subset \partial \mathcal{G} \cup \partial U$.
\begin{Claim2}
There are constants $w^*_2, \alpha^*_2 > 0$ and $K^*_2, T^*_2 < \infty$ such that:
If $t_0 > \max \{ T^*_2, 2 T_0 \}$ and $w(t) < w^*_2$ for all $t \in [\frac12 t_0, t_0]$, then
\[ |{\Rm}| < K^*_2 t_0^{-1} \quad \text{on} \quad P(x, t_0, \alpha^*_2 r_0, - (\alpha^*_2 r_0)^2) \quad \text{for all} \quad x \in \mathcal{G}' \cup \MM_{\thick}(t_0) \cup \partial U, \]
where the parabolic neighborhood $P(x, t_0, \alpha^*_2 r_0, - (\alpha^*_2 r_0)^2)$ is always non-singular.
\end{Claim2}
\begin{proof}
We only need to consider the case in which $x \in \mathcal{G}' \setminus \mathcal{G}$ and $\dist_{t_0} (x, \partial \mathcal{G} \cup \partial U) > r_0$.
Let $N$ be the component of $\mathcal{G}' \setminus \mathcal{G}$ that contains $x$.
Then $\partial N \subset \partial \mathcal{G} \cup \partial U$ and $B(x, t_0, \rho_{r_0} (x, \linebreak[1] t_0)) \subset N$.
So we can apply Lemma \ref{Lem:unwrapfibration}(ii) and (iv) to conclude that for any $\td{x} \in \td{N}$ in the universal cover of $N$ we have $\vol_{t_0} B(\td{x}, t_0, \rho_{r_0} (x, t_0)) \linebreak[1] > w_1(\mu^\circ) \rho^3_{r_0} (x, t_0)$.
This implies that $x$ is $\td{c} w_1(\mu^\circ)$-good at any scale $r \leq r_0$ relative to $N$, for some universal constant $\td{c} > 0$ (see the remarks in \cite[subsection \ref{subsec:goodness}]{Bamler-LT-Perelman}).

Since all points in $\partial N$ survive until time $(1-(\alpha^*_1)^2) t_0$ and all surgeries on $[\frac12 t_0, t_0]$ are trivial, we can extend $N$ to a sub-Ricci flow with surgery $N' \subset \MM$ on the time-interval $[(1-(\alpha^*_1)^2) t_0, t_0]$.
We now apply \cite[Proposition \ref{Prop:curvboundinbetween}]{Bamler-LT-Perelman} for $r_0 \leftarrow \min \{ \alpha_1^*, (K_1^*)^{-1/2} \} r_0$, $U \leftarrow N'$ and $w \leftarrow \td{c} w$ to obtain the desired curvature bound for sufficiently large $t_0$.
\end{proof}

Next, we find a curvature bound in controlled distance to $\mathcal{G}'$, which however deteriorates with larger distances.
\begin{Claim3}
For every $A < \infty$ there are constants $w^*_3 = w^*_3(A), \alpha^*_3 = \alpha^*_3(A) > 0$ and $K^*_3 = K^*_3(A), T^*_3 = T^*_3(A) < \infty$ such that if $t_0 > \max \{ T^*_3, 2 T_0 \}$ and $w(t) < w^*_3$ for all $t \in [\frac12 t_0, t_0]$, then
\[ |{\Rm}| < K^*_3 t_0^{-1} \quad \text{on} \quad P(x, t_0, A r_0, - (\alpha^*_3 r_0)^2) \quad \text{for all} \quad x \in \mathcal{G}' \cup \MM_{\thick}(t_0) \cup \partial U, \]
where the parabolic neighborhoods $P(x, t_0, A r_0, - (\alpha^*_3 r_0)^2)$ are non-singular.
\end{Claim3}
\begin{proof}
The case $x \in \mathcal{G}' \cup \partial U$ can be reduced to the case $x \in \partial \mathcal{G}' \subset \partial \mathcal{G} \cup \partial U$.
In this case the claim follows from \cite[Proposition \ref{Prop:curvcontrolincompressiblecollapse}]{Bamler-LT-Perelman} for $A \leftarrow A+1$ together Claim 1 and a distance distortion estimate.
The case $x \in \MM_{\thick}(t_0)$ follows directly from \cite[Proposition \ref{Prop:thickthindec}(d)]{Bamler-LT-Perelman}.
\end{proof}

Now consider the sets $\mathcal{G}'' \subset \mathcal{G}'$ and $\mathcal{S}'' \subset M$ as introduced in Proposition \ref{Prop:GGpp}.
Recall that $\mathcal{S}''$ is a disjoint union of smoothly embedded solid tori.
The next claim is rather geometric.
It ensures that there are no components of $V_2$ outside of $\mathcal{G}''$ in controlled distance to $\mathcal{G}''$ if $w(t)$ is assumed to be sufficiently small.
\begin{Claim4}
For every $A < \infty$ there is a $w^*_4 = w^*_4(A) > 0$ and a $T^*_4 = T^*_4(A) < \infty$ such that if $t_0 > \max \{ T^*_4, 2 T_0 \}$ and $w(t) < w^*_4(A)$ for all $t \in [\frac12 t_0, t_0]$, then the following holds: \\
For every component $\CC''$ of $\SS''$ there is a component $\CC$ of $V_1$ with $\CC \subset \CC''$ and $\partial \CC'' \subset \partial \CC$.
Moreover, one of the following cases applies:
\begin{enumerate}[label=(\alph*)]
\item $\CC \approx S^1 \times D^2$ or 
\item $\CC \approx T^2 \times I$ and $\CC$ is adjacent to a component $\CC'$ of $V'_2$ that is diffeomorphic to $S^1 \times D^2$ and that is contained in $\CC''$ (this implies that $\partial \CC \setminus \partial \CC'' \subset \partial \CC'$) or
\item $\CC \approx T^2 \times I$ and the boundary components of $\CC$ have time-$t_0$ distance of at least $A r_0$ from each other.
\end{enumerate}
So, in particular, the components of $V_2$ that are not contained in $\mathcal{G}''$ have time-$t_0$ distance of at least $A r_0$ from $\mathcal{G}''$.
\end{Claim4}
\begin{proof}
By Proposition \ref{Prop:GGpp}(e), we only have to consider the case in which $\CC \approx T^2 \times I$ and $\CC$ is adjacent to a component $\CC'$ of $V_2$ on the other side.
Observe that then, again by Proposition \ref{Prop:GGpp}(e), the $S^1$-fibers of $\CC' \cap V_{2, \textnormal{reg}}$ are contractible in $\MM(t_0)$.
Assume that the boundary components of $\CC'$ have time-$t_0$ distance of less than $A r_0$ from each other.
Then we can find points $x_0 \in \partial \mathcal{G} \cup \partial U$ and $x_1 \in \CC'$ with $\dist_{t_0} (x_0, x_1) < A r_0$.
Without loss of generality, we can assume that $x_1 \in \CC' \cap V_{2, \text{reg}}$ (e.g. by assuming $x_1 \in \partial \CC$).
Let $\td{x}_0, \td{x}_1 \in \td{\MM}(t_0)$ be lifts of $x_0, x_1$ in the universal cover with $\dist_{t_0}(\td{x}_0, \td{x}_1) = \dist_{t_0}(x_0, x_1)$.
Using Claim 1, we can deduce a lower bound on $\rho_{r_0} (x_0, t_0)$ and hence find a universal constant $w_1^{**} > 0$ such that $\vol_{t_0} B(\td{x}_0, t_0, r_0) >  w^{**}_1 r_0^3$.
Using Claim 3 (applied with $A \leftarrow 2A+1$), we find a curvature bound on $B(\td{x}_1, t_0, (A+1) r_0)$ for large $t_0$.
So by volume comparison we have $\vol_{t_0} B(\td{x}_1, t_0, \rho_{r_0} (x_1, t_0)) > w^{**}_2 \rho_{r_0}^3(x_1, t_0)$ for some $w_2^{**} = w_2^{**}(A) > 0$.

We will now derive a contradiction to the local collapsedness around $x_1$ for small enough $w(t_0)$.
By Proposition \ref{Prop:MorganTianMain}(c2), there is a universal constant $0 < s = s_2 (\mu^\circ, \ov{r}( \cdot, 1), K_2(\cdot, 1)) < \frac1{10}$ and a subset $U_2$ with
\[ B(x_1, t_0, \tfrac12 s \rho_{r_0} (x_1, t_0)) \subset U_2 \subset B(x_1, t_0, s \rho_{r_0}(x_1, t_0)) \]
that is diffeomorphic to $B^2 \times S^1$ such that the $S^1$-directions are isotopic to the $S^1$-fibers in $\CC' \cap V_{2, \text{reg}}$ and hence contractible in $\MM(t_0)$.
So if $\td{U}_2 \subset \td{\MM}(t_0)$ is the lift of $U_2$ that contains $\td{x}_1$, then the universal covering projection is injective on $\td{U}_2$.
Hence
\begin{multline*}
 \vol_{t_0} B(x_1, t_0, \rho_{r_0}(x_1, t_0))  \geq \vol_{t_0} U_2 = \vol_{t_0} \td{U}_2 \geq \vol_{t_0} B(\td{x}_1, t_0, \tfrac12 s \rho_{r_0}(x_1, t_0)) \\
 \geq \tfrac18 \td{c} s^3 \vol_{t_0} B(\td{x}_1, t_0, \rho_{r_0} (x_1, t_0))
 \geq \tfrac18 \td{c} w_2^{**} s^3 \rho_{r_0}^3 (x_1, t_0).
\end{multline*}
Since $\dist_{t_0} (x_1, \MM_{\thin}(t_0)) < 2 r_0$, we obtain a contradiction if we choose $w_4^*(A) < \frac18 \td{c} w_2^{**}(A) s^3$.
This finishes the proof.
\end{proof}

Next we show that the diameter of each component of $\SS''$ cannot grow too fast on a time-interval of small, but uniform size.
\begin{Claim5}
There is a constant $\alpha^*_5 > 0$ and for every $A < \infty$ there are constants $B^*_5 = B^*_5(A), T^*_5 = T^*_5(A) < \infty$ and $w^*_5(A) > 0$ such that if $t_0 > \max \{ T^*_5, 2 T_0 \}$ and $w(t) < w^*_5$ for all $t \in [\frac12 t_0, t_0]$, then we have: \\
Let $\CC$ be a component of $\SS''$.
Then there is a unique sub-Ricci flow with surgery $N \subset \MM$ on the time-interval $[t_0 - (\alpha^*_5 r_0)^2, t_0]$ with $\CC = N(t_0)$ such that the following holds:
If $\diam_{t_0} N(t_0) > B^*_5 r_0$, then $\diam_t N(t) > A r_0$ for all $t \in [t_0 - (\alpha^*_5 r_0)^2, t_0]$.
\end{Claim5}
\begin{proof}
By Claim 1 and the fact that all surgeries on $[\frac12 t_0, t_0]$ are trivial, we can extend $\CC$ to a sub-Ricci flow with surgery $N \subset \MM$ on the time-interval $[t_0 - (\alpha_1^* r_0)^2, t_0]$.

The rest of the claim is a consequence of \cite[Proposition \ref{Prop:slowdiamgrowth}]{Bamler-LT-Perelman}.
Choose $x_0 \in \partial \CC \subset \mathcal{G} \cup \partial U$.
So $x_0$ is ${w'_1}^*$-good at scale $r_0$.
Let $\tau^* = \min \{ (\alpha^*_1)^2, (K_1^*)^{-1}, \tau_0({w'_1}^*) \}$ where $\tau_0$ is the constant from \cite[Proposition \ref{Prop:slowdiamgrowth}]{Bamler-LT-Perelman}.
Note that by Proposition \ref{Prop:MorganTianMain}(c1$\delta$) and by the fact that $\mu^\circ \leq \frac1{10}$, we have $\partial N (t_0) \subset B(x_0, t_0, \frac1{10} r_0)$.
Using Claim 1, the fact that $\tau^* \leq (K_1^*)^{-1}$ and distance distortion estimates, this implies $\partial N (t) \subset B(x_0, t, r_0)$ for all $t \in [(1- \tau^*) t_0, t_0]$.
We can now apply \cite[Proposition \ref{Prop:slowdiamgrowth}(d)]{Bamler-LT-Perelman} with $U \leftarrow N$, $r_0 \leftarrow r_0$, $x_0 \leftarrow x_0$, $w \leftarrow {w'_1}^*$, $A \leftarrow A$ to conclude that if for any $\tau \in (0, \tau^*]$ we have $N \subset B(x_0, t_0 - \tau r_0^2, A r_0)$, then $\CC = N(t_0) \subset B(x_0, t_0, A'({w'_1}^*, A) r_0)$.
This implies the claim for sufficiently large $t_0$ and small $w(t)$ (depending on $A$).
\end{proof}

Note that without loss of generality we can assume that the functions $w_3^*(A), \linebreak[1] \alpha_3^* (A), \linebreak[1] w_4^*(A), \linebreak[1] w_5^*(A)$ are non-increasing and the functions $K_3^*(A), \linebreak[1] T_3^*(A), \linebreak[1] T_4^*(A), \linebreak[1] B_5^*(A), \linebreak[1] T_5^*(A)$ are non-decre\-asing in $A$.
In the following, we will define the sub-Ricci flows with surgery $S_i$ and the sets $W_i$ and show that they satisfy the assertions (a)--(h).
In order to do this, we will denote the components of $\mathcal{S}''$ by $S''_1, \ldots, S''_{m''}$ and choose a subcollection $S^*_1, \ldots, S^*_m$ of the $S''_1, \ldots, S''_{m''}$ in the next paragraph.
The final time-slices  $S_1(t_0), \ldots, S_m(t_0)$ will arise from the sets $S^*_1, \ldots, S^*_m$ by removing a collar neighborhood of diameter $\leq 1.5 r_0$.
Fix from now on the constant $L$ and assume that $L > 102$.

Assume first that $t_0 > \max \{ T_1^*, T_3^*(L+2), 2 T_0 \}$ and $w(t) < \min \{ w_1^*, w_3^*(L+2) \}$ for all $t \in [\frac12 t_0, t_0]$.
If $d_i = r_0^{-1} \diam_{t_0} S''_i < L+2$ for some $i$, then by Claim 3, all points in $S''_i$ survive until time $t_0 - (\alpha^*_3(d_i) r_0)^2$ and we have $|{\Rm}| < K_3^*(d_i) t_0^{-1}$ on $S''_i \times [t_0 - (\alpha^*_3(d_i) r_0)^2, t_0]$.
Given the fact that the sets $S_1, \ldots, S_m$ are chosen in the way described above, this establishes the second part of assertion (c) for $K' (d) = K_3^*(d+2)$ and $\tau'_1 (d) = (\alpha^*_3 (d+2))^2$.
Moreover, assuming $\tau_1 < \tau'_1(102)$, we can remove all $S''_i$ with $\diam_{t_0} \leq 102 r_0$ and define the sets $S^*_1, \ldots, S^*_m$ to be the sets $S''_i$ with $\diam_{t_0} S''_i > 102 r_0$.
So, by a reapplication of Claim 3 assertion (b) is verified and by Claim 1 the second part of assertion (h) is true for some small but universal $\tau_1$.
Also, the first part of assertion (c) holds.
Note that by assertion (b) and the fact that the surgeries on $[\frac12 t_0, t_0]$ are trivial, we can extend every set $S^*_i$ to a sub-Ricci flow with surgery on the time-interval $[(1-\tau_1) t_0, t_0]$, which we will in the following also denote by $S^*_i \subset \MM$.

Now assume that also $t_0 > T_4^*(L+2)$ and $w(t) < w_4^*(L+2)$ for all $t \in [\frac12 t_0, t_0]$.
For each $S^*_i$ there is a component $\CC_i$ of $V_1$ that is contained in $S^*_i (t_0)$ and that shares a boundary with it.
Consider the cases (a)--(c) from Claim 4.
In cases (b), (c) we set $P_i = \CC_i$.
In case (a), we can apply Proposition \ref{Prop:MorganTianMain}(c1)($\alpha$) to find a torus structure $P_i \subset \CC_i$ such that $\partial \CC_i \subset \partial P_i$ and such that $\diam_{t_0} \CC_i \setminus P_i \leq \frac1{10} r_0$.
Observe that in all cases, the torus structure $P_i$ has width $\leq \mu^\circ r_0 \leq \frac1{10} r_0$.
In case (c) it has length $\geq (L+3) r_0$ by Claim 4 and in cases (a), (b) it has length $> \diam_{t_0} S^*_i(t_0) - \diam_{t_0} (S^*_i(t_0) \setminus P_i) - \frac1{10} r_0 > \diam_{t_0} S^*_i (t_0) - \frac2{10} r_0$ at time $t_0$.
Chop off $P_i$ on both sides such that the new boundary tori have distance of exactly $r_0$ from the corresponding boundary tori of $P_i$ and call the result $P'_i$.
Then define $S_i (t_0)$ to be the union of $P'_i$ with the component of $S^*_i(t_0) \setminus P'_i$ whose closure is diffeomorphic to a solid torus.
By assertion (b) we can extend $S_i(t_0)$ to a sub-Ricci flow with surgery $S_i \subset \MM$ on the time-interval $[(1-\tau_1) t_0, t_0]$.
Note that in all cases the torus structure $P'_i$ has length $\geq \min \{ L r_0, \diam_{t_0} S_i (t_0) - 2 r_0 \}$ at time $t_0$.
We can hence chop off $P'_i$ on the side that is contained in the interior of $S_i(t_0)$ and produce a torus structure $P''_i$ of width $\leq \frac1{10} r_0$ and length $= \min \{ L r_0, \diam_{t_0} S_i (t_0) - 2 r_0 \}$.
Let $W_i$ be the closure of $S_i (t_0) \setminus P''_i$.
Then assertion (e) holds.
Moreover, the first part of assertion (h) follows from Proposition \ref{Prop:MorganTianMain}(c1).
Assertion (a) is clear.
Observe also that $\diam_{t_0} S^*_i(t_0) - \frac{11}{10} r_0 < \diam_{t_0} S_i(t_0) \leq \diam_{t_0} S^*_i(t_0)$.

We discuss assertion (f).
Let $x \in P''_i$.
By Lemma \ref{Lem:unwrapfibration}(ii) and (iii), we conclude that $x$ is $w_1(\mu^\circ)$-good at scale $r_0$ and time $t_0$ relative to $P_i$.
Since $B(x, t_0, \linebreak[1] \rho_{r_0} (x, \linebreak[1] t_0)) \linebreak[1] \subset \linebreak[1] P_i$ this implies that $x$ is also locally $w_1(\mu^\circ)$-good at scale $r_0$ and time $t_0$.

Next we establish assertion (d).
Observe that choosing $\tau_1$ small enough, we have at least $\diam_t S_i(t) > 50 r_0$ for all $t \in [(1-\tau_1) t_0, t_0]$ by assertion (b) and the fact that $\diam_{t_0} S_i (t_0) > 100 r_0$.
Assume now that $t_0 > T_5^*(L)$ and $w(t) < w_5^*(L)$ for all $t \in [\frac12 t_0, t_0]$.
By Claim 5 for $\CC = S_i(t_0)$ we conclude that for all $A \leq L$ we have: if $\diam_{t_0} S^*_i (t_0) > B_5^*(A) r_0$, then $\diam_t S^*_i (t) > A r_0$ for all $t \in [t_0 - (\alpha^*_5 r_0)^2, t_0]$.
So assertion (d) holds for the function
\[ \Delta_1 (d) = \sup \{ A > 0 \;\; : \;\; B_5^*(A+2) < d \} \cup \{ 50 \}. \]
Note that $\Delta_1$ is monotonically non-decreasing and $\lim_{d \to \infty} \Delta_1(d) = \infty$.

Finally, we prove assertion (g).
Assume that $t_0 > T_3^* (2L+10)$ and that $w(t) < w_3^*(2L + 10)$ for all $t \in [\frac12 t_0, t_0]$.
Let $x \in P''_i$ and choose an arbitrary point $x_0 \in \partial S^*_i(t_0)$.
Let $\td{x}$, $\td{x}_0$ be lifts of $x$, $x_0$ in the universal cover $\td{\MM}(t_0)$ with $\dist_{t_0} (\td{x}, \td{x}_0) = \dist_{t_0} (x, x_0)$.
As in the proof of Claim 4 we have
\[ \vol_{t_0} B(\td{x}_0, t_0, r_0) > w_1^{**} r_0^3 \]
for some universal $w_1^{**} > 0$.
By Claim 3, we have curvature control $| {\Rm_{t_0}} | < K_3^*(2 L+10) t_0^{-1}$ on $B(x, t_0, (L+5) r_0) \subset B(x_0, t_0, (2 L+10) r_0)$.
In particular, there is a $\rho^* = \rho^*(L) > 0$ such that $\rho_{r_0}(x, t_0) > \rho^* r_0$.
Without loss of generality, we can assume that $\nu < \min \{ \rho^*, 1 \}$.
Hence, by volume comparison there is some $c^* = c^*(L) > 0$ such that
\begin{multline*}
 \vol_{t_0} B(\td{x}, t_0, \nu r_0) \geq \nu^3 c^* \vol_{t_0} B(\td{x}, t_0, (L+5)r_0) \\ > \nu^3 c^* \vol_{t_0} B(\td{x}_0, t_0, r_0) > \nu^3 c^* w_1^{**} r_0^3.
\end{multline*}
On the other hand, 
\[ \vol_{t_0} B(x, t_0, \nu r_0) < \vol_{t_0} B(x, t_0, \rho_{r_0}(x,t_0)) < w(t_0) \rho_{r_0}^3(x_0,t_0) < w(t_0) r_0^3. \]
Assume first that there is no loop based at $x$ that is non-contractible in $\MM(t_0)$ and has length $< \nu r_0$.
Then
\[ w(t_0) r_0^3 > \vol_{t_0} B(x, t_0, \nu r_0) = \vol_{t_0} B(\td{x}, t_0, \nu r_0) > \nu^3 c^* w_1^{**} r_0^3. \]
So if $w(t_0) < \nu^3 c^* w_1^{**}$, we obtain a contradiction.
We conclude that if $w(t_0)$ is sufficiently small depending on $L$ and $\nu$, there is a non-contractible loop $\sigma \subset \MM(t_0)$ based at $x$ that has length $\ell_{t_0} (\sigma) < \nu r_0$.
This implies $\sigma \subset P_i \subset S^*_i (t_0)$ and hence $\sigma$ is even incompressible in $\MM(t_0)$.
\end{proof}

In the second step, Proposition \ref{Prop:firstcurvboundstep2}, we extend the uniform curvature control from Lemma \ref{Lem:firstcurvboundstep1}(b) further into the regions $S_i(t_0) \setminus W_i(t_0)$.
The following proposition will mostly use the same notation as the previous Lemma \ref{Lem:firstcurvboundstep1} and the assertions of this proposition will roughly correspond to those of Lemma \ref{Lem:firstcurvboundstep1}.
After stating the proposition, we will explain its most important innovations.
%This fact will follow from Proposition \ref{Prop:curvboundnotnullinarea} and Lemma \ref{Lem:firstcurvboundstep1}(f).

\begin{Proposition}[second step] \label{Prop:firstcurvboundstep2}
There are a positive continuous function $\delta : [0, \infty) \to (0, \infty)$, constants $K_2 < \infty$, $\tau_2 > 0$ and functions $\Lambda_2, K'_2, \tau'_2 : (0, \infty) \to (0, \infty)$ with the property that $\tau'_2$ is non-increasing, $K'_2$ and $\Lambda_2$ are non-decreasing and $\Lambda_2(d) \to \infty$ as $d \to \infty$ such that: \\
For every $L < \infty$ and $\nu > 0$ there are constants $T_2 = T_2(L) < \infty$, $w_2 = w_2 (L, \nu) > 0$ such that: \\
Let $\MM$ be a Ricci flow with surgery on the time-interval $[0, \infty)$ with normalized initial conditions that is performed by $\delta(t)$-precise cutoff.
Consider the constant $T_0 < \infty$ and the function $w : [T_0, \infty) \to (0, \infty)$ obtained in \cite[Proposition \ref{Prop:thickthindec}]{Bamler-LT-Perelman} and assume that
\begin{enumerate}[label=(\roman*)]
\item $r_0^2 = t_0 \geq \max \{ 4 T_0, T_2 \}$,
\item $w(t) < w_2$ for all $t \in [\frac14 t_0, t_0]$,
\item all components of $\MM(t_0)$ are irreducible and not diffeomorphic to spherical space forms and all surgeries on the time-interval $[\frac14 t_0, t_0]$ are trivial.
\end{enumerate}
Then there are sub-Ricci flows with surgery $S_1, \ldots, S_m \subset \MM$ on the time-interval $[(1-\tau_2) t_0, t_0]$ such that $S_1(t_0), \ldots, \linebreak[1] S_m(t_0)$ is a collection of pairwise disjoint, incompressible solid tori in $\MM(t_0)$ and there are sub-Ricci flows with surgery $W_i \subset S_i$ ($i=1,\ldots, m$) on the time-interval $[(1-\tau_2) t_0, t_0]$ such that for all $i = 1, \ldots, m$:
\begin{enumerate}[label=(\alph*)]
\item The pair $(S_i (t_0), W_i (t_0))$ is diffeomorphic to $(S^1 \times D^2(1), S^1 \times D^2(\frac12))$.
\item The set $\MM (t_0) \setminus (W_1 (t_0) \cup \ldots \cup W_m (t_0))$ is non-singular on the time-interval $[(1 - \tau_2) t_0, t_0]$ and
\[ \qquad \qquad | {\Rm} | < K_2 t_0^{-1} \qquad \text{on} \qquad \big(\MM (t_0) \setminus (W_1(t_0) \cup \ldots \cup W_m (t_0)) \big) \times [(1-\tau_2) t_0, t_0]. \]
\item If $\diam_{t_0} S_i (t_0) \leq L r_0$, then $S_i$ is non-singular on the time-interval $[(1- \tau'_2( r_0^{-1} \diam_{t_0} S_i ) ) t_0, t_0 ]$ and we have the curvature bound
\[  \qquad\quad |{\Rm}| < K'_2 ( r_0^{-1} \diam_{t_0} S_i ) t_0^{-1}  \quad \text{on} \quad S_i (t_0) \times  [(1- \tau'_2( r_0^{-1} \diam_{t_0} S_i (t_0) ) ) t_0, t_0 ].
\]
\item The set $S_i(t_0) \setminus \Int W_i(t_0)$ is a torus structure of width $\leq r_0$ and length
\[ \dist_{t_0} (\partial S_i (t_0), \partial W_i (t_0)) = \min \big\{ \Lambda_2 ( r_0^{-1} \diam_{t_0} S_i (t_0)), L \big\} r_0. \]
\item For every point $x \in S_i (t_0) \setminus W_i (t_0)$, there is a loop $\sigma \subset \MM(t_0)$ based at $x$ that is incompressible in $\MM(t_0)$ and has length $\ell_{t_0}(\sigma) < \nu r_0$.
\end{enumerate}
\end{Proposition}
The most important statement of this proposition is the fact that the uniform curvature bound in (b) also holds on $S_i (t_0) \setminus W_i (t_0)$ and on a time-interval whose size does not depend on $r_0^{-1} \diam_{t_0} S_i(t_0)$.
Since this bound enables us to estimate the metric distortion of the regions $S_i(t_0) \setminus W_i(t_0)$ on this time-interval, we don't need to list the lower diameter estimate from Lemma \ref{Lem:firstcurvboundstep1}(d).
We have also omitted the statement from Lemma \ref{Lem:firstcurvboundstep1}(f) since we won't make use of it anymore.

Observe that we cannot only establish the curvature bound in assertion (c) on a time-interval of universal size and that the length of the torus structure in assertion (d) cannot be bounded from below by a constant depending on the diameter of $S_i$ at time $(1-\tau_2) t_0$.
The reason for this has to do with the fact that the geometry on $S_i$ could be close to that of a cigar soliton times $S^1$.
In fact, after rescaling by a proper constant, regions of the cigar soliton of large diameter can shrink rapidly under the Ricci flow.
It is the content of \cite[Proposition \ref{Prop:slowdiamgrowth}]{Bamler-LT-Perelman}, which we have applied in the proof of Lemma \ref{Lem:firstcurvboundstep1}, that however the opposite behavior cannot occur, i.e. regions of bounded diameter can not grow too fast in a short time.

The main idea behind the following proof is quite straight-forward:
We will choose the constant $\tau_2$ in such a way that we can conclude by Lemma \ref{Lem:firstcurvboundstep1}(d) that whenever a solid torus $S_i (t_0)$ has large normalized diameter at time $t_0$, then $S_i (t)$ also has large normalized diameter at all times $t \in [(1-2\tau_2) t_0, t_0]$.
We will then apply Lemma \ref{Lem:firstcurvboundstep1}(f) at all times $t \in [(1-2\tau_2) t_0, t_0]$ (with $U \leftarrow S_i (t)$) and deduce that the points in a collar of $S_i (t)$ are locally $\ov{w}$-good at scale $r_0$ or the curvature is bounded there.
The desired curvature bound on $S_i (t) \setminus W_i (t)$ for all $t \in [(1-\tau_2) t_0, t_0]$ now follows using \cite[Proposition \ref{Prop:curvboundnotnullinarea}]{Bamler-LT-Perelman} (``Curvature control in large regions that are locally good everywhere'').

However, our approach has the following caveat:
Whenever we apply Lemma \ref{Lem:firstcurvboundstep1} at some time $t \in [(1-2\tau_2) t_0, t_0]$, we obtain a new set of solid tori $S^*_1(t), \ldots, \linebreak[1] S^*_{m^*}(t)$ within $S_i (t)$.
A priori, there is no relation between these solid tori with the ambient solid torus $S_i(t)$, which we obtained when we applied Lemma \ref{Lem:firstcurvboundstep1} at time $t_0$.
For example, there could be several $S^*_i (t)$, some of those solid tori could be very far away from $\partial S_i (t)$ and have small diameter and hence short collar torus structures.

In order to gain more geometric understanding on the location of these solid tori, we will first show that $S_i (t)$ has long torus collars of width $\leq 2 r_0$ for all $t \in [(1-2\tau_2) t_0, t_0]$.
This fact will follow from a continuity argument:
We will start at time $t_0$ and observe the behavior of the long torus structure $S_i(t_0) \setminus \Int W_i$ of width $\leq r_0$ when going backwards in time.
Consider the first (i.e. maximal) time $t \in [(1-2\tau_2) t_0, t_0]$ at which this torus structure is destroyed, e.g. because a region within this torus structure develops a width that is bigger than some time-dependent threshold between $r_0$ and $2r_0$.
Then there are two cases:
In the first case this region is close to the complement of the solid tori $S^*_i (t)$ obtained when applying Lemma \ref{Lem:firstcurvboundstep1} at time $t$.
And in the second case it has to be close to one of the solid tori $W^*_i$ obtained by this application.
In the first case, we have a curvature bound around the region at which the width of the torus structure is too large.
So at some slightly later time $t' > t$ this width would be large as well, contradicting the maximal choice of $t$.
In the second case, the diameter of the corresponding $S^*_i (t)$ can be bounded and $W^*_i$ has to ``cap off'' the torus structure in question.
This fact would imply an upper diameter bound on $S_i (t)$ that contradicts the lower bound from Lemma \ref{Lem:firstcurvboundstep1}(d).

Having established the existence of long torus collars of $S_i (t)$ for all $t \in [(1-2\tau_2) t_0, t_0]$, we now have enough geometric control to conclude that whenever we apply Lemma \ref{Lem:firstcurvboundstep1} at some time $t \in [(1-2\tau_2) t_0, t_0]$, then the resulting solid tori $S^*_i (t)$ or at least the smaller solid tori $W^*_i$ stay far enough away from $\partial S_i (t)$.
So at every point of $S_i (t)$ within a certain distance of $\partial S_i (t)$ the curvature is either bounded or the $\ov{w}$-goodness condition holds at scale $r_0$.
Eventually, we can apply \cite[Proposition \ref{Prop:curvboundnotnullinarea}]{Bamler-LT-Perelman}.

\begin{proof}
We will fix several constants and functions that we will use in the course of the proof.
First we choose $\delta(t)$ such that it is bounded by the corresponding functions from Lemma \ref{Lem:firstcurvboundstep1} and \cite[Proposition \ref{Prop:curvboundnotnullinarea}]{Bamler-LT-Perelman}.
Consider moreover the functions $\Delta_1$, $K'_1$, $\tau'_1$ from Lemma \ref{Lem:firstcurvboundstep1} and choose $D^* < \infty$ such that
\[ D^* > 100 \qquad \text{and} \qquad \Delta_1(D^*) > 100. \]
Define the functions $L_1^*, \ldots, L_5^*(d) : [D^*, \infty) \to (1, \infty)$ by
\begin{alignat*}{1}
 L_1^*(d) &= \tfrac14 \Delta_1(d) - 10  \displaybreak[1] \\ 
 L_2^*(d) &= \tfrac12 \min \{ d-3, L_1^*(d) \}  \displaybreak[1] \\
 L_3^* (d) &= \min \{ L_1^*(d) - 1, L_2^* (d) \}  \displaybreak[1] \\
 L_4^* (d) &= L_3^*(d) - 1  \displaybreak[1] \\
 L_5^* (d) &= \tfrac12 L_4^*(d)
\end{alignat*}
Observe, that $L_1^*, \ldots , L_5^*$ are continuous, monotonically non-decreasing and $L^*_i (d) \to \infty$ as $d \to \infty$.
Using these functions we define
\[ \Lambda_2 (d) = \begin{cases} \min \{ L^*_5(d) - 1, d-2 \} & \text{if $d \geq D^*$} \\ 1 & \text{if $d < D^*$} \end{cases} \]
Then $\Lambda_2(d)$ is also non-decreasing and $\Lambda_2(d) \to \infty$ as $d \to \infty$.

Given the constant $L$, we pick
\[ L^\circ = L^\circ (L) > \max \{ D^*, 10L + 100 \}. \]
Finally, using the constants $w_1$ and $T_1$ from Lemma \ref{Lem:firstcurvboundstep1}, we assume
\[ w_2 (L, \nu) < w_1 (L^\circ, \nu) \qquad \text{and} \qquad T_2 (L, \nu) > 2 T_1 (L^\circ, \nu) . \]

We first apply Lemma \ref{Lem:firstcurvboundstep1} at time $t_0 \leftarrow t_0$ with $U \leftarrow \MM(t_0)$, $L \leftarrow L^\circ$ and $\nu \leftarrow \nu$.
We obtain sub-Ricci flows with surgery on the time-interval $[(1-\tau^*_0) t_0, t_0]$, which we denote by $S'_1, \ldots, S'_{m'} \subset \MM$ and subsets, which we denote by $W''_i \subset S'_i (t_0)$ for $i = 1, \ldots, m'$.
By Lemma \ref{Lem:firstcurvboundstep1}(c), if $\diam_{t_0} S'_i (t_0) \leq D^* r_0$, then $S'_i$ is non-singular on the time-interval $[(1-\tau'_1(D^*)) t_0, t_0]$ and we have a curvature bound there.
Let now $S_1, \ldots, S_m$ be the subcollection of the $S'_1, \ldots, S'_{m'}$ for which $d_i = r_0^{-1}\diam_{t_0} S'_i (t_0) > D^*$ and pick the sets $W'_i \subset S_i(t_0)$ accordingly.
Consider the torus structures $P'_i = \Int S_i(t_0) \setminus W'_i$ of width $\leq r_0$ and length 
\[ \min \{ d_i - 2, L^\circ \} r_0 \geq \min \{ \Lambda_2(d_i), L + 1\} r_0. \]
Chop off each $P'_i$ on the side that is not adjacent to $\partial S_i (t_0)$ producing torus structures $P_i$ of width $\leq r_0$ and length exactly $\min \{ \Lambda_2(d_i), L \} r_0$.
Set $W_i = \Int S_i(t_0) \setminus \Int P_i$.
We will later be able to extend $W_i$ to a sub-Ricci flow with surgery on a small, but uniform time-interval. 

\begin{Claim0}
There are universal constants $\tau^*_0, w^*_0 > 0$ and $K^*_0 < \infty$ such that: \\
For all $x \in \MM(t_0)$ with $\dist_{t_0} (x, \MM(t_0) \setminus (S_1(t_0) \cup \ldots \cup S_m(t_0))) \leq 100 r_0$ the point $(x, t_0)$ survives until time $(1-\tau^*_0) t_0$ and
\[
|{\Rm}|(x,t) < K^*_0 t_0^{-1} \qquad \text{for all} \qquad t \in [(1-\tau^*_0) t_0, t_0].
\]
Moreover, assertions (a), (c), (d) and (e) of this Proposition hold.
\end{Claim0}

\begin{proof}
The first statement is a direct consequence of Lemma \ref{Lem:firstcurvboundstep1}(b) and (c).
Here we assume that $\tau^*_0 < \min \{ \tau_1, \tau'_1(D^*) \}$ and $K^*_0 > \max \{ K_1, K'_1(D^*) \}$.

Assertion (a) is clear and assertion (c) is a consequence of Lemma \ref{Lem:firstcurvboundstep1}(c).
Assertion (d) follows by the choice of $\Lambda_2$ and assertion (e) by Lemma \ref{Lem:firstcurvboundstep1}(g) and the fact that $W'_i(t_0) \subset W_i$.
\end{proof}

So it remains to extend the curvature bound from Claim 0 to the subsets $S_i (t_0) \setminus W_i$ on a uniform time-interval.
The proof of this fact will involve the application of Lemma \ref{Lem:firstcurvboundstep1} at times $t \in [(1-\tau_0^*) t_0, t_0]$ for $U \leftarrow S_i(t)$, $L \leftarrow L^\circ$ and $\nu \leftarrow \nu$.
By assertion (h) from the previous application of Lemma \ref{Lem:firstcurvboundstep1}, the extra conditions of Lemma \ref{Lem:firstcurvboundstep1} in the solid torus case are satisfied.
The remaining conditions hold by the choice of $w_2$ and $T_2$.

The desired curvature bound is established in the following Claims 1--5.
In Claims 1--3 we will first derive a local goodness bound for points in controlled distance from $\partial S_i(t)$ for any $t$ of a uniform time-interval.
An important tool will hereby be the notion of ``torus collars of length up to'' a certain constant as introduced in Definition \ref{Def:toruscollars}.
In Claim 4 we will derive a curvature bound using this local goodness bound together with \cite[Proposition \ref{Prop:curvboundnotnullinarea}]{Bamler-LT-Perelman}.
Claim 5 will translate this result into the final form.

In the following, fix some $i = 1, \ldots, m$ and recall that $d_i = r_0^{-1} \diam_{t_0} S_i(t_0) > D^*$.

\begin{Claim1}
There are universal constants $K_1^* < \infty$ and $0 < \tau_1^* < \tau^*_0$ such that for all $t \in [(1-\tau_1^*) t_0, t_0]$ the following holds:
Consider numbers
\[ 0 < \td{L} \leq \min \{ L_1^*(d_i), 4 (L+2) \}, \qquad 1 \leq a \leq 2 \]
and assume that $S_i(t)$ does not have torus collars of width $\leq ar_0$ and length up to $\td{L} r_0$, but it has torus collars of width $\leq a r_0$ and length up to $(\td{L}-1) r_0$ if $\td{L} > 1$. \\
Then $|{\Rm}| (x, t) < K_1^* t_0^{-1}$ for all $x \in S_i(t)$ with $\dist_t (x, \partial S_i(t)) < (\td{L} + 10) r_0$.
\end{Claim1}
\begin{proof}
Observe first that in the case $\td{L} \leq 1$ we are done by Claim 0 and a sufficiently small choice of $\tau^*_1$.
So assume in the following that $\td{L} > 1$.

Assume that $\tau_1^* < \tau_0^*$ and fix some $t \in [(1-\tau_1^*) t_0, t_0]$.
So we can apply Lemma \ref{Lem:firstcurvboundstep1} at time $t$ with $U \leftarrow S_i(t)$ and $L \leftarrow L^{\circ}$ and obtain the sub-Ricci flows with surgery $S_1^*, \ldots, S^*_{m^*} \subset \MM$ and the subsets $W^*_i \subset S^*_i( t)$.
Observe that the parameter $r = \sqrt{t}$ changes only slightly, i.e. we may assume that for the right choice of $\tau_1^*$ we have $0.9 r_0 < r \leq r_0$.
Moreover, we assume that $\tau^*_1$ is chosen small enough that $\diam_t \partial S_i(t) < 2 \diam_{t_0} \partial S_i(t) \leq2 r_0$.

If $\dist_t (S^*_{i^*}(t), \partial S_i (t)) \geq (\td{L} - 30) r_0$ for all $i^* = 1, \ldots, m^*$, then we are done by Lemma \ref{Lem:firstcurvboundstep1}(b) applied at time $t$.
So all we need to do is to assume that there is an $i^* \in \{ 1, \ldots, m^* \}$ with 
\begin{equation} \label{eq:torcollarcontradiction}
 \dist_t (S_{i^*}^*(t), \partial S_i(t)) < (\td{L} - 30) r_0
\end{equation}
and derive a contradiction.

Observe that by Lemma \ref{Lem:firstcurvboundstep1}(c) and (e) applied at time $t$ we have $\dist_t (\partial S^*_{i^*} (t), \linebreak[1] \partial W^*_{i^*}) \geq \min \{ 97 r, L^\circ r \} \geq 50 r_0$.
So we can choose a point $y \in S_{i^*}^*(t) \setminus W^*_{i^*}$ such that $\dist_t( y, \partial S^*_{i^*}(t)) = 3 r_0$.
Then we have at least $\dist_t (y, \partial S_i(t)) < (\td{L} - 20)r_0$ and hence by our assumption, there is a set $P \subset S_i(t)$ that is bounded by $\partial S_i(t)$ and a torus $T \subset S_i(t)$ with $y \in T$ and $\diam_t T \leq a r_0 \leq 2 r_0$.
By the choice of $y$ we have $T \subset S_{i^*}^*(t)$.
This implies
\begin{equation} \label{eq:SpPSs}
 S_i(t) = P \cup S_{i^*}^*(t)
\end{equation}
and we conclude, using assertion (d) of Lemma \ref{Lem:firstcurvboundstep1} applied at time $t$, that
\[ \diam_t S^*_{i^*} (t) + \diam_t P \geq \diam_t S_i(t) > \min \{ \Delta_1(d_i), L^{\circ} \} r_0. \]
Observe now that by Lemma \ref{Lem:collardiameter}, we know that
\begin{equation} \label{eq:diamboundonP}
 \diam_t P < (\td{L} - 20 + 4a) r_0 < (\td{L} - 10) r_0 < \min \{ L_1^*(d_i), \linebreak[1] 4 (L+2) \} r_0.
\end{equation}
So using the fact that $r \leq r_0$, we obtain
\begin{multline*}
 \diam_t S^*_{i^*}(t) > \big( \min \{ \Delta_1(d_i), L^{\circ} \} - \min \{ L_1^* (d_i), 4(L+2) \}  \big) r_0 \\
 \geq \min \big\{ \Delta_1(d_i) - L_1^*(d_i) , L^\circ - 4 (L + 2) \big\} r .
\end{multline*}
Observe that the right hand side is larger than $10 r$.
We conclude further using Lemma \ref{Lem:firstcurvboundstep1}(e) applied at time $t$ that
\begin{multline}
 \dist_t (W^*_{i^*}, \partial S_i(t)) \geq \dist_t (\partial S^*_{i^*}(t), \partial W^*_{i^*}) = \min \{ \diam_t S^*_{i^*}(t) - 2 r, L^\circ r \} \\
\geq 0.9 \min \big\{ \Delta_1(d_i) - L_1^*(d_i) - 2, L^\circ - 4 (L + 3), L^\circ \big\} r_0 \\
> \min \big\{ L_1^* (d_i) + 1, 4(L+2) \big\} r_0 \geq \td{L} r_0. \label{eq:distSpWs}
\end{multline}

So far we have only used the first assumption, which states that $S_i (t)$ \emph{does} have torus collars of width $\leq a r_0$ and length up to $(\td{L} - 1)r_0$.
We will now show that in contradiction to the second assumption of the claim, $S_i(t)$ also has torus collars of width $\leq a r_0$ and length up to $\td{L} r_0$.
So assume that $x \in S_i(t)$ with $(\td{L}-1) r_0 < \dist_t(x, \partial S_i(t)) \leq \td{L} r_0$.
By (\ref{eq:SpPSs}), the diameter bound (\ref{eq:diamboundonP}) on $P$ and (\ref{eq:distSpWs}), we conclude $x \in S_{i^*}^*(t) \setminus W_{i^*}^*$.
So, we can find a set $P^* \subset S_{i^*}^*(t) \setminus W_{i^*}^*$ that is diffeomorphic to $T^2 \times I$ and bounded by $\partial S_{i^*}^*(t)$ and a $2$-torus $T^* \subset S^*_{i^*} (t) \setminus W^*_{i^*}$ with $x \in T^*$ and $\diam_t T^* \leq r \leq r_0 \leq a r_0$.
Again by (\ref{eq:diamboundonP}) we find $T^* \cap P = \emptyset$.
It follows from Lemma \ref{Lem:2T2timesI}, that $P \cup P^*$ is diffeomorphic to $T^2 \times I$.
This finishes the contradiction argument and shows that (\ref{eq:torcollarcontradiction}) does not hold for any $i^* = 1, \ldots, m^*$.
\end{proof}

\begin{Claim2}
There are universal constants $0 < \tau_2^* < \tau^*_0$ and $T_2^* < \infty$ such that if $t_0 > T_2^*$ then at all times $t \in [(1-\tau_2^*) t_0, t_0]$ the set $S_i(t)$ has torus collars of width $\leq 2r_0$ and length up to $\min \{ L_2^*(d_i), 2(L+2) \} r_0$.
\end{Claim2}
\begin{proof}
Choose $\tau_2^* < \tau_1^*$ such that $\exp(2K_1^* \tau_2^*) < 2$.
By Lemma \ref{Lem:firstcurvboundstep1}(e) we already know that \emph{at time $t_0$}, the set $S_i(t_0)$ has torus collars of width $\leq r_0$ and length up to $L_i r_0$ where
\[ L_i = 2\min \{ L_2^* (d_i) , 2(L+2) \} \leq \min \{ d_i-2, L^\circ \}. \]
Let $t^* \in [(1-\tau_2^*) t_0, t_0]$ be minimal with the property that for all $t \in (t^*, t_0]$ the set $S_i(t)$ has torus collars of width $\leq \exp (2K_1^* t_0^{-1} (t_0 - t)) r_0$ and length up to $\exp (-2K_1^* t_0^{-1} (t_0 - t)) L_i r_0$ at time $t$.
We are done if $t^* = (1-\tau_2^*) t_0$.
So consider the case $t^* > (1-\tau_2^*) t_0$.

Let $\varepsilon > 0$ be a small constant, which we will determine later.
It will not be a universal constant.
By the choice of $t^*$, we find times $t_1 \leq t^* \leq t_2$ with $t_2 - t_1 < \varepsilon$ such that at time $t_2$ the set $S_i(t_2)$ has torus collars of width $\leq \exp (2K_1^* t_0^{-1} (t_0 - t_2))r_0$ and length up to $\exp ( -2 K_1^* t_0^{-1} (t_0 - t_2) ) L_i r_0$, but at time $t_1$ it does not have torus collars of width $\leq \exp (2 K_1^* t_0^{-1} (t_0 - t_1))r_0$ and length up to $\exp ( - 2 K_1^* t_0^{-1} (t_0 - t_1) ) L_i r_0$.

Choose $\td{L} \leq \exp ( - 2 K_1^* t_0^{-1} (t_0 - t_1) ) L_i$ such that at time $t_1$ the set $S_i(t_1)$ does not have torus collars of width $\leq \exp (2 K_1^* t_0^{-1} (t_0 - t_1))r_0$ and length up to $\td{L}r_0$, but it does have torus collars of width $\leq \exp (2 K_1^* t_0^{-1} (t_0 - t_1))r_0$ and length only up to $(\td{L} - 1) r_0$ if $\td{L} > 1$.
Observe that $\td{L} \leq \min \{ L^*_1 (d_i), 4(L+2) \}$ and $\exp (2 K_1^* t_0^{-1} (t_0 - t_1))r_0 < 2r_0$.
So we can apply Claim 1 to conclude that
\[ |{\Rm}| (x,t_1)  < K_1^* t_0^{-1} \qquad \text{if} \qquad x \in S_i(t_1) \quad \text{and} \quad  \dist_{t_1} (x, \partial S_i(t_1)) < (\td{L} + 10) r_0. \]
Let $Q < \infty$ be an upper bound on the curvature at all surgery points in $\MM$ on the time-interval $[t_1, t_2]$.
Then all strong $\delta(t)$-necks around surgery points as described in \cite[Definition \ref{Def:precisecutoff}(4)]{Bamler-LT-Perelman} are defined on a time-interval of length $> \frac1{100} Q^{-1/2}$ and the curvature $|{\Rm}|$ there is bounded from below by $> c' \delta^{-2}(t)$ for some $t \in [t_1, t_2]$ and a universal $c' > 0$.
So if we choose $\varepsilon < \frac1{100} Q^{-1/2}$ and assume $t_0$ to be large enough, then we can exclude surgery points of the form $(x,t)$ with $\dist_{t_1} (x, \partial S_i(t_1)) < (\td{L} + 10) r_0$ and $t \in [t_1, t_2]$.
Moreover, again by choosing $\varepsilon$ sufficiently small, we can assume that curvatures at points that survive until time $t_2$ cannot grow by more than a factor of $2$ such that we have
\begin{equation} \label{eq:rm2K1star}
 | {\Rm} | (x,t) < 2 K_1^* t_0^{-1} \quad \text{if} \quad (x,t) \in S_i(t) \times [t_1, t_2] \; \text{and} \; \dist_{t_1} (x, \partial S_i(t)) < (\td{L} + 10) r_0. 
\end{equation}
(We remark, that we could have also excluded surgery points using property (2) of the canonical neighborhood assumptions in \cite[Definition \ref{Def:CNA}]{Bamler-LT-Perelman}.)

Now let $x \in S_i(t_1)$ be a point with $\dist_{t_1} (x, \partial S_i(t_1)) \leq \td{L} r_0$.
Then by the curvature bound we conclude
\[ \dist_{t_2} (x, \partial S_i(t_2)) \leq \exp(2 K_1^* t_0^{-1} (t_2 - t_1)) \td{L} r_0 \leq \exp(-2 K_1^* t_0^{-1} (t_0 - t_2)) L_i r_0 . \]
So there is a collar $P \subset S_i(t_2)$ that is bounded by $\partial S_i(t_2)$ and an embedded $2$-torus $T \subset S_i(t_2)$ with $x \in T$ and
\begin{equation} \label{eq:diamt2exp}
\diam_{t_2} T \leq \exp ( 2K_1^* t_0^{-1} (t_0 - t_2)) r_0.
\end{equation}
By Lemma \ref{Lem:collardiameter} we have under the assumption that $\varepsilon$ is so small that $\exp (2 K^*_1 \varepsilon) L_i < L_i + 1$
\begin{equation} \label{eq:diamPlessL9}
 \diam_{t_2} P \leq \exp (2 K^*_1 t_0^{-1} (t_2 - t_1)) \td{L} r_0 + 8 r_0 < (\td{L} + 9) r_0. 
\end{equation}
Again by assuming $\varepsilon$ small, we conclude that the distance distortion on $P$ for times $[t_1, t_2]$ is bounded by $r_0$ and hence $\MM$ is non-singular on $P \times [t_1, t_2]$.
So at time $t_1$ the set $P$ is still bounded by $\partial S_i(t_1)$ and $T$.
Moreover, by (\ref{eq:rm2K1star}), (\ref{eq:diamt2exp}), (\ref{eq:diamPlessL9}) and a standard distance distortion estimate, we have
\[ \diam_{t_1} T \leq \exp (2 K^*_1 t_0^{-1} (t_2 - t_1)) \diam_{t_2} T \leq \exp ( 2 K_1^* t_0^{-1} (t_0 - t_1)) r_0. \]

We have just shown that $S_i(t_1)$ does indeed have torus collars of width $\leq \exp ( 2 K_1^* t_0^{-1} (t_0 - t_1)) r_0$ and length up to $\td{L} r_0$ contradicting our assumption.
\end{proof}

\begin{Claim3}
There are constants $0 < \tau_3^* < \tau_0^*$, $w_3^* > 0$ and $K_3^*, T_3^* < \infty$ such that:
Assume that $t_0 > T_3^*$.
Then for every $t \in [(1-\tau_3^*)t_0, t_0]$ and every point $x \in S_i(t)$ with $\dist_t (\partial S_i(t), x) < \min \{ L_3^*(d_i), 2(L+2) \} r_0$ either $|{\Rm}|(x,t) < K^*_3 t_0^{-1}$ or $x$ is locally $w_3^*$-good at scale $r_0$ and time $t$.
\end{Claim3}
\begin{proof}
Assume that $\tau_3^* < \min \{ \tau^*_0, \tau^*_1, \tau_2^* \}$ and fix some time $t \in [(1-\tau_3^*) t_0, t_0]$.
We will argue as in the first part of the proof of Claim 1 with $\td{L} = \min \{ L^*_3 (d_i), 2 (L+2) \} + 1 \leq \min \{ L^*_1(d_i), 4 (L+2) \}$ and $a = 2$.
Observe hereby that by Claim 2 the set $S_i(t)$ has torus collars of width $\leq 2 r_0$ and length up to $\min \{ L^*_3 (d_i), 2 (L+2) \} \leq \min \{ L^*_2(d_i), 2(L+2) \}$.

So we apply again Lemma \ref{Lem:firstcurvboundstep1} at time $t$ with $U \leftarrow S_i(t)$ and $L \leftarrow L^\circ$ and obtain pairs of subsets $(S^*_1(t), W^*_1), \ldots, (S^*_{m^*}(t), W^*_{m^*})$ of $S_i(t)$.
If $\dist_t ( S^*_{i^*} (t), \partial S_i (t) ) \geq (\td{L} - 30) r_0$ for all $i^* = 1, \ldots, m^*$, then we obtain a curvature bound as before using Lemma \ref{Lem:firstcurvboundstep1}(b).
If not, i.e. if (\ref{eq:torcollarcontradiction}) is satisfied for some $i^* = 1, \ldots, m^*$, then we obtain from (\ref{eq:distSpWs}) that
\[ \dist_t ( W^*_{i^*}, \partial S_i(t) ) > \min \{ L^*_2 (d_i), 2 (L+2) \} r_0. \]
This implies that every $x \in S_i(t)$ with $\dist_t (\partial S_i(t), x) < \min \{ L_2^*(d_i), 2(L+2) \} r_0$ is either contained in $S_i(t) \setminus (S^*_1(t) \cup \ldots \cup S^*_{m^*} (t))$, in which case we obtain a curvature bound from Lemma \ref{Lem:firstcurvboundstep1}(b), or contained in $S^*_{i^*} (t) \setminus W^*_{i^*}$, in which case $x$ is locally $\ov{w}$-good at scale $r$ and time $t$ by Lemma \ref{Lem:firstcurvboundstep1}(f).
\end{proof}

\begin{Claim4}
There are universal constants $K_4^*, T_4^* < \infty$ and $0 < \tau_4^* < \tau_0^*$ such that if $t_0 > T_4^*$, then $|{\Rm}| (x,t) \linebreak[1] < K^*_4 t_0^{-1}$ for all $t \in [(1-\tau_4^*) t_0, t_0]$ and $x \in S_i(t)$ with 
\[ \dist_t(\partial S_i(t), x) \leq \min \{ L^*_4( d_i), 2 (L + 1 ) \} r_0 \]
and none of these points are surgery points.
\end{Claim4}
\begin{proof}
We use Claim 3 and \cite[Proposition \ref{Prop:curvboundnotnullinarea}]{Bamler-LT-Perelman} with $U \leftarrow S_i$, $r_0 \leftarrow r_0$, $r_1 \leftarrow (\tau^*_3)^{1/2} r_0$, $A \leftarrow K_0^* \tau^*_3$, $w \leftarrow w^*_3$, $b \leftarrow \min \{ L^*_3( d_i), 2 (L+2) \} r_0$ to conclude that for all $t \in [(1-\frac12 \tau^*_3) t_0, t_0]$ and $x \in S_i(t)$ with $\dist_t(\partial S_i(t), x) \leq \min \{ L^*_2 (d_i), 2(L+1) \} r_0 \leq (b -1)r_0$ we have
\[ |{\Rm}| (x,t) < K_{\ref{Prop:curvboundnotnullinarea}} (w^*_3, K_0^* \tau^*_3) \big( r_0^{-2} + (\tfrac12 \tau^*_3 t_0)^{-1} \big) .  \]
This implies the claim.
\end{proof}

\begin{Claim5}
There are constants $K_5^*, T_5^* < \infty$ and $0 < \tau_5^* < \tau^*_0$ such that if $t_0 > T^*_5$, then for all $x \in S_i(t_0)$ for which 
\[ \dist_{t_0}(\partial S_i(t), x) \leq \min \{ L^*_5 ( d_i) ,  L + 1 \} r_0 \]
the point $(x,t_0)$ survives until time $(1-\tau^*_5) t_0$ and for all $t \in [(1-\tau^*_5) t_0, t_0]$ we have $|{\Rm}|(x,t) < K^*_5 t_0^{-1}$.
\end{Claim5}
\begin{proof}
This follows by a distance distortion estimate and Claim 4.
We just need to choose $\tau_5^* < \tau_4^*$ so small that distances don't shrink by more than a factor of $2$ on a time-interval of size $\leq \tau_5^* t_0$ and in a region in which the curvature is bounded by $K_4^* t_0^{-1}$.
\end{proof}

To conclude the proof of the Proposition, we just need to use Claim 5 and observe that the torus structure $\Int S_i (t_0) \setminus \Int W_i$ has width $\leq r_0$ and length $\leq \min \{ L_5^* (d_i) - 1, L \} r_0$.
Hence, all its points satisfy the distance bound from Claim 5.
So assuming $\tau_2 < \tau^*_5$, we obtain a curvature bound on the non-singular neighborhood $(S_i (t_0) \setminus \Int W_i) \times [(1-\tau_2) t_0, t_0]$ and we can extend $W_i$ to a sub-Ricci flow with surgery $W_i \subset \MM$ on the time-interval $[(1-\tau_2) t_0, t_0]$.
\end{proof}

In the third step, Proposition \ref{Prop:firstcurvboundstep3}, we impose the additional assumption that embedded, incompressible solid tori in $\MM (t_0)$ have ``compressing planar domains'' of bounded area, as produced by \cite[Proposition \ref{Prop:maincombinatorialresult}(a)]{Bamler-LT-topology}.
With the help of Lemmas \ref{Lem:shortloopingeneralcase}, \ref{Lem:2loopstorus} and \ref{Lem:bettertorusstructure} and Proposition \ref{Prop:firstcurvboundstep2}(e), these compressing domains can be used to pick torus structures $P_i \subset S_i (t_0) \setminus W_i$ of arbitrarily good precision whenever the diameters of the corresponding solid tori $S_i (t_0)$ are large enough.
Each torus structure $P_i$ encloses a solid torus, which is a thickening of the solid torus $W_i$ and the final time-slice of a sub-Ricci flow with surgery denoted by $U_i \subset \MM$.
The remaining assertions of the following proposition are similar to those of Proposition \ref{Prop:firstcurvboundstep2}, the main difference being that the dependences of the involved parameters are rewritten in a way that is more suitable for the following subsection.

\begin{Proposition}[third step] \label{Prop:firstcurvboundstep3}
There are a positive continuous function $\delta : [0, \infty) \linebreak[1] \to (0, \infty)$ and constants $K < \infty$, $\tau > 0$ and for every $A < \infty$ there are non-increasing functions $D_A, K'_A : (0,2] \to (0, \infty)$ such that for every $\eta > 0$ there are $w_3 = w_3(\eta, A) > 0$, $T_3 = T_3(\eta, A) < \infty$ such that: \\
Let $\MM$ be a Ricci flow with surgery on the time-interval $[0, \infty)$ with normalized initial conditions that is performed by $\delta(t)$-precise cutoff.
Consider the constant $T_0 < \infty$, the function $w : [T_0, \infty) \to (0, \infty)$ as well as the decomposition $\MM(t) = \MM_{\thick} (t) \cup \MM_{\thin} (t)$ for all $t \in [T_0, \infty)$ obtained in \cite[Proposition \ref{Prop:thickthindec}]{Bamler-LT-Perelman}.
Assume that
\begin{enumerate}[label=(\roman*)]
\item $r_0^2 = t_0 \geq \max \{ 4T_0, T_3 \}$,
\item $w(t) < w_3$ for all $t \in [\frac14 t_0, t_0]$,
\item all components of $\MM(t_0)$ are irreducible and not diffeomorphic to spherical space forms and all surgeries on the time-interval $[\frac12 t_0, t_0]$ are trivial,
\item for every smoothly embedded, solid torus $S \subset \Int \MM_{\textnormal{thin}}(t_0)$, $S \approx S^1 \times D^2$ that is incompressible in $\MM(t_0)$ there is a compact, smooth domain $\Sigma \subset \IR^2$ and a smooth map $f : \Sigma \to S$ with $f(\partial \Sigma) \subset \partial S$ such that $f$ restricted to the outer boundary circle of $\Sigma$ is non-contractible in $\partial S$ and $f$ restricted to all other boundary circles of $\Sigma$ is contractible in $\partial S$ and $\area_{t_0} f < A t_0$.
\end{enumerate}
Then there are closed subsets $P_1, \ldots, P_m \subset \MM(t_0)$ and sub-Ricci flows with surgery $U_1, \ldots, U_m \subset \MM$ on the time-interval $[(1-\tau) t_0, t_0]$ as well as numbers $h_1, \linebreak[1] \ldots, \linebreak[1] h_m \in [\eta, 2]$ such that the sets $P_1 \cup U_1(t_0), \ldots, P_m \cup U_m(t_0)$ are pairwise disjoint and such that for all $i = 1, \ldots, m$
\begin{enumerate}[label=(\alph*)]
\item The set $U_i(t_0)$ is a smoothly embedded, incompressible solid torus ($\approx S^1 \times D^2$), $P_i \approx T^2 \times I$ and $P_i, U_i(t_0)$ share a torus boundary, i.e. $P_i \cap U_i(t_0) = \partial U_i(t_0)$.
So $(P_i \cup U_i(t_0), U_i(t_0)) \approx (S^1 \times D^2(1), S^1 \times D^2(\frac12))$.
\item For any $t \in [(1-\tau) t_0, t_0]$ and any $x \in \MM (t)$ with $\dist_{t} (x,  \MM(t) \setminus (U_1(t) \cup \ldots \cup U_m(t))) \leq r_0$, the point $(x,t)$ is non-singular and we have
\[ |{\Rm}| (x,t) < K t_0^{-1}. \]
In particular, this means that the set $\MM(t_0) \setminus (U_1(t_0) \cup \ldots \cup U_m(t_0))$ is non-singular on the time-interval $[(1-\tau) t_0, t_0]$ and
\[ \qquad\qquad | {\Rm} | < K t_0^{-1} \qquad \text{on} \qquad \big( \MM(t_0) \setminus (U_1(t_0) \cup \ldots \cup U_m(t_0)) \big) \times [(1-\tau) t_0, t_0]. \]
\item The set $P_i$ is an $h_i$-precise torus structure at scale $r_0$ and at any time of the time-interval $[(1-\tau) t_0, t_0]$.
\item If $h_i > \eta$, then 
\begin{multline*}
 \qquad\qquad \diam_{t_0} (P_i \cup U_i (t_0)) < D(h_i) r_0 \qquad \text{and} \\ \qquad  | {\Rm_{t_0}} | < K'(h_i) t_0^{-1} \quad \text{on} \quad P_i \cup U_i (t_0).
\end{multline*}
\end{enumerate}
\end{Proposition}

\begin{proof}
We first define the constants $K, \tau$, the functions $D_A, K'_A$ and the quantities $w_3$, $T_3$.
Consider the functions $\Lambda_2, K'_2, \tau'_2$ and the constants $K_2, \tau_2$ from Proposition \ref{Prop:firstcurvboundstep2}.
Choose $D^* < \infty $ such that $\Lambda_2 (D^*) > 1000$ and set
\[ K = \max \{ K_2, K'_2 (D^*) \} \qquad \text{and} \qquad \tau = \min \big\{ \tau_2, \tau'_2 (D^*), \tfrac1{10} K_2^{-1} \big\}. \]

Now fix the constant $A < \infty$.
Before defining $D_A$ and $K'_A$, we need to fix a few other quantities and functions, which will be important in the course of the proof.
Using the constants $\td{L}_0$ from Lemma \ref{Lem:shortloopingeneralcase}, $\td\nu$ from Lemma \ref{Lem:bettertorusstructure} and $\td\varepsilon_1$ from Lemma \ref{Lem:2loopstorus}, we set
\[ \ov{L}_{A} = \max \Big\{ \td{L}_0 \big( \min \big\{ \tfrac1{10} \td\nu ( K_2, 1, 1 ), \td\varepsilon_1 (K_2) \big\}, \; A \big), \; 10 \big(\td\nu ( K_2, 1, 1 ) \big)^{-1} \Big\}.  \]
Then we define the functions $L^{**}_A, L^*_A : (0,2] \to (0, \infty)$ by
\begin{alignat*}{1}
 L^{**}_A(h) &= \max \Big\{ \ov{L}_A, \; \td{L}_0 \big( \min \big\{ \tfrac1{10} \td\nu ( K_2, 2h^{-1}, \tfrac12 h ), \td\varepsilon_1(K_2) \big\}, \; A \big) , \\
 & \qquad\qquad \qquad\qquad\qquad\qquad 3 h^{-1} + 500, \; 10 \big(\td\nu ( K_2, 2 h^{-1}, \tfrac12 h ) \big)^{-1}, \Lambda_2 (1) \Big\}. \\
 L^*_A(h) &= \inf \big\{ L^{**}_A (h'') \;\; : \;\; 0 < h'' \leq \tfrac12 h \big\}.
\end{alignat*}
Then $L^*_A$ is non-increasing.
Using this function, we define the functions $D_A, K'_A$ by
\begin{alignat*}{1}
 D_A(h) &= \sup \{ d > 0 \;\; : \;\; \Lambda_2(d) \leq L^*_A (h) \}. \\
 K'_A(h) &= K'_2(D_A(h))
\end{alignat*}
Observe that $D_A(h)$ is well-defined since $L^*_A (h) \geq \Lambda_2(1)$ for all $h \in (0,2]$.
Moreover, note that $D_A$ and $K'_A$ are non-decreasing.

Now also fix the constant $\eta > 0$ and set
\[
 L^\circ = \max \{ L^*_A(\eta) + 1, D_A(\eta), 1000 \} \qquad \text{and} \qquad \nu^\circ = \min \Big\{ \frac{1}{L^\circ}, \td\varepsilon_1 (K_2) \Big\} .
\]
Then we define
\[ w_3(\eta, A) = w_2(L^\circ, \nu^\circ) \qquad \text{and} \qquad T_3(\eta, A) = T_2 (L^\circ). \]

By this choice of $w_3$ and $T_3$, we can apply Proposition \ref{Prop:firstcurvboundstep2} with $L \leftarrow L^\circ$, $\nu \leftarrow \nu^\circ$ and obtain sub-Ricci flows with surgery $S_1, \ldots, S_m \subset \MM$ on the time-interval $[(1-\tau) t_0, t_0]$ and subsets $W_i \subset S_i(t_0)$.
For each $i = 1, \ldots, m$ set $d_i = r_0^{-1} \diam_{t_0} S_i$.
Then $P'_i = S_i \setminus \Int W_i$ are torus structures of width $\leq r_0$ and length $L_i r_0$ for
\[ L_i = \min \{ \Lambda_2 ( d_i), L^\circ \}. \]
We can assume without loss of generality that $L_i \geq 1000$ for all $i=1, \ldots, m$: If $L_i < 1000$ for some $i = 1, \ldots, m$, then by Proposition \ref{Prop:firstcurvboundstep2}(d) $\Lambda_2(d_i) < 1000 \leq L^\circ$.
So by monotonicity of $\Lambda_2$ we have $d_i \leq D^*$ and by Proposition \ref{Prop:firstcurvboundstep2}(c) and the choice of $\tau$, $K$ the flow $S_i$ is non-singular on the time-interval $[(1-\tau) t_0, t_0]$ and we have $| {\Rm} | < K t_0^{-1}$ on $S_i(t_0) \times [(1-\tau)t_0, t_0]$.
Hence we can remove the pair $S_i$ and $W_i$ from the list.
We summarize
\[
1000 \leq L_i \leq L^\circ \qquad \text{for all} \qquad i = 1, \ldots, m.
\]

Next, we define numbers $h'_i > 0$ that will give rise to the $h_i$.
If $L_i \leq \ov{L}_A$, then we just set $h'_i = 2$.
Observe that by definition of $L^*_A$, we then immediately get $L_i \leq L^*_A(h_i)$.
In the case in which $L_i > \ov{L}_A$, we choose an $h'_i \in (0,2]$ such that the following three equations are satisfied:
\begin{alignat}{1}
L_i &> \td{L}_0 \big( \min \big\{ \tfrac1{10} \td\nu ( K_2, 2(h'_i)^{-1}, \tfrac12 h'_i ), \td\varepsilon_1(K_2) \big\}, \; A \big) \notag  \\
L_i &> 3 (h'_i)^{-1} + 500 \label{eq:Listwicehinverse} \\
L_i &>  10 (\td\nu ( K_2, 2 (h'_i)^{-1}, \tfrac12 h'_i ))^{-1} \label{eq:Li10nuinverse}
\end{alignat}
The conditions $L_i > \ov{L}_A$ and $L_i \geq 1000$ ensure that we can find such an $h'_i$ (for example, we can choose $h'_i = 2$).
We can furthermore assume that $h'_i$ is so small such that for any $h'' \leq \frac12 h'_i$, at least one of these conditions is not fulfilled.
So for any $h'' \leq \frac12 h'_i$ we have $L_i \leq L^{**}_A (h'')$ and thus $L_i \leq L^*_A (h'_i)$.
Furthermore, observe that by (\ref{eq:Li10nuinverse}) we have
\begin{equation} \label{eq:lowerboundon110nu}
 \tfrac1{10} \td\nu (K_2, 2(h'_i)^{-1}, \tfrac12 h'_i) > \frac{1}{L_i} \geq \frac1{L^\circ} \geq \nu^\circ.
\end{equation}
Lastly, we define
\[ h_i = \max \{ h'_i, \eta \}. \]
So for all $i = 1, \ldots, m$ for which $h_i > \eta$ we have $L_i \leq L^*_A(h'_i) = L^*_A (h_i)$.
Using the definition of $L_i$ and the inequality $L^\circ \geq L^*_A(\eta) \geq L^*_A (h_i)$, we then conclude that $\Lambda_2(d_i) \leq L^*_A (h_i)$.
Hence $d_i \leq D_A(h_i)$ and by Proposition \ref{Prop:firstcurvboundstep2}(c) we have the curvature bound $|{\Rm_{t_0}}| < K'(h_i) t_0^{-1}$ on $S_i(t_0)$.
Since $P_i \cup U(t_0)$ will be strictly contained in $S_i(t_0)$, this establishes assertion (d).

For the next paragraphs fix $i = 1, \ldots, m$.
We will now construct the sets $P_i$ and the sub-Ricci flows with surgery $U_i \subset \MM$.
$U_i$ and $P_i$ will be a modification of the sets $W_i$ and $P'_i$.
The torus structure $P_i$ will be a subset of $P'_i$.

First consider the case $L_i \leq \ov{L}_A$.
Then $h_i = 2$.
Recall that the closure of $P'_i$ is a torus structure of width $\leq r_0$ and length $L_i r_0 \geq 1000 r_0$ at time $t_0$.
So we can choose a torus structure $P_i \subset P'_i$ such that the pair $(P'_i, P_i)$ is diffeomorphic to $(T^2 \times [-2,2], T^2 \times [-1,1])$ of width $\leq r_0$ and length $> r_0$ at time $t_0$ such that $P_i$ has time-$t_0$ of at least $10r_0$ from $\partial P'_i$.
Note that $P_i$ is $1$-precise at scale $\sqrt{t_0}$ and time $t_0$.
Next observe that by Proposition \ref{Prop:firstcurvboundstep2}(b), $\MM$ is non-singular on $P'_i$ and we have $|{\Rm}| < K_2 t_0^{-1}$ on $P'_i \times [(1-\tau) t_0, t_0]$.
Since $P_i$ is far enough away from $\partial P'_i$ at time $t_0$, any time-$t_0$ minimizing geodesic between points in $P_i$ is contained in $P'_i$.
So we can use a distance distortion estimate to conclude that $P_i$ is $e^{K_2 \tau} h_i$-precise at scale $\sqrt{t_0}$ and at every time $t \in [(1-\tau) t_0, t_0]$.
Since by assumption $K_2 \tau < \frac1{10}$, this implies assertion (c) in the case $L_i \leq \ov{L}_A$.

Assume in the following that $L_i > \ov{L}_A$.
Then we first apply assumption (iv) for $S \leftarrow S_i(t_0)$ to obtain the domain $\Sigma$ and the smooth map $f : \Sigma \to S_i (t_0)$.
Next, we use Lemma \ref{Lem:shortloopingeneralcase}(a) with $M \leftarrow \MM(t_0)$, $S \leftarrow S_i (t_0)$, $P \leftarrow P'_i$ and $f \leftarrow f$ to obtain a loop $\gamma_i \subset P'_i$ that is non-contractible in $P'_i$, but contractible in $S_i(t_0)$, that has length
\[ \ell_{t_0} (\gamma_i) < \min \big\{ \tfrac1{10} \td\nu (K_2, 2(h'_i)^{-1}, \tfrac12 h'_i), \td\varepsilon_1 (K_2) \big\} r_0 \]
and that has time-$t_0$ distance of at least $(\frac13 L_i - 2) r_0$ from $\partial P'_i$.
Let $p_i \in \gamma_i$ be an arbitrary base point.
By Proposition \ref{Prop:firstcurvboundstep2}(e), there is a closed loop $\sigma_i \subset P'_i$ based at $p_i$ that is non-contractible in $S_i$ and has length (see (\ref{eq:lowerboundon110nu})) 
\[ \ell_{t_0} (\sigma_i) < \nu^\circ r_0 \leq \min \big\{ \tfrac1{10} \td\nu ( K_2, 2 (h'_i)^{-1}, \tfrac12 h'_i ), \td\varepsilon_1 (K_2) \big\} r_0. \]
In particular, $\gamma_i$ and $\sigma_i$ represent two linearly independent homotopy classes in $\pi_1(P'_i) \cong \IZ^2$.
By Lemma \ref{Lem:2loopstorus} there is an embedded torus $T_i \subset P'_i$ with $p_i \in T_i$ that is incompressible in $P'_i$, separates its two ends and has diameter
\[ \diam_{t_0} T_i < \td\nu ( K_2, 2 (h'_i)^{-1}, \tfrac12 h'_i ) r_0. \]
Observe that $T_i$ has distance of at least $(\frac13 L_i - 3) r_0 \geq ( (h'_i)^{-1} + 100) r_0$ from $\partial P'_i$ (see (\ref{eq:Listwicehinverse})).
We can hence apply Lemma \ref{Lem:bettertorusstructure} and obtain a torus structure $P_i \subset P'_i$ of width $\leq \frac12 h'_i r_0$ and length $> 2 (h'_i)^{-1} r_0$ such that the pair $(P'_i, P_i)$ is diffeomorphic to $(T^2 \times [-2, 2], T^2 \times [-1,1])$ and such that $P_i$ has time-$t_0$ distance of at least $10r_0$ from $\partial P'_i$.
Assertion (c) follows now similarly as in the case $L_i \leq \ov{L}_A$.

Finally, we let $U_i (t_0)$ be the closure of the component of $S_i(t_0) \setminus P_i$ that is diffeomorphic to a solid torus.
Moreover, we extend $U_i (t_0)$ to a sub-Ricci flow with surgery $U_i \subset \MM$ on the time-interval $[(1-\tau) t_0, t_0]$.
Then assertions (a), (b) hold by construction and assertions (c), (d) were established before.
\end{proof}

\subsection{The geometry on late and long time-intervals} \label{subsec:lateandlongtimei}
In this subsection, we relate the conclusions from Proposition \ref{Prop:firstcurvboundstep3} applied at each time of a larger time-interval $[t_0, t_\omega ]$, $t_\omega \leq L t_0$, $L \gg 1$ towards one another and obtain a geometric description of the flow on this time-interval.
More specifically, we will show that if the diameter of one of the solid tori $P_i \cup U_i (t_\omega)$ from Proposition \ref{Prop:firstcurvboundstep3} applied at time $t_\omega$ is large, then this solid torus persists and stays large when going backwards in time until time $t_0$.
At time $t_0$, we will find a loop $\gamma_i \subset P_i$, non-contractible in $P_i$, whose length is small and whose geodesic curvature is controlled on the time-interval $[t_0, t_\omega]$ and that bounds a disk $h_i : D^2 \to \MM ( t_0)$ of area $< (A+1) t_0$ at time $ t_0$.

\begin{Proposition} \label{Prop:structontimeinterval}
There is a positive continuous function $\delta : [0, \infty) \to (0, \infty)$ and for every $L, A < \infty$, $\alpha > 0$ there are constants $K_4 = K_4 (L, \linebreak[1] A, \linebreak[1] \alpha), \linebreak[1] \Gamma_4 = \Gamma_4 (L, A), T_4 = T_4 (L, A, \alpha) < \infty$ and $w_4 = w_4 (L, A, \alpha) > 0$ (observe that $\Gamma_4$ does not depend on $\alpha$) such that: \\
Let $\MM$ be a Ricci flow with surgery on the time-interval $[0, \infty)$ with normalized initial conditions that is performed by $\delta(t)$-precise cutoff.
Consider the constant $T_0 < \infty$, the function $w : [T_0, \infty) \to (0, \infty)$ as well as the decomposition $\MM (t) = \MM_{\thick} (t) \cup \MM_{\thin} (t)$ for all $t \in [T_0, \infty)$ obtained in \cite[Proposition \ref{Prop:thickthindec}]{Bamler-LT-Perelman} and assume that
\begin{enumerate}[label=(\roman*)]
\item $t_\omega > r_0^2 = t_0 \geq \max \{ 4T_0, T_4 \}$ and $t_\omega \leq L t_0$,
\item $w(t) < w_4$ for all $t \in [\frac14 t_0, t_\omega]$,
\item for every $t \in [\frac14 t_0, t_\omega]$ all components of $\MM(t)$ are irreducible and not diffeomorphic to spherical space forms and all surgeries on the time-interval $[\frac14 t_0, t_\omega]$ are trivial,
\item for every time $t \in [t_0, t_\omega]$ and every smoothly embedded, solid torus $S \subset \Int \MM_{\textnormal{thin}}(t)$, $S \approx S^1 \times D^2$ that is incompressible in $\MM(t)$, there is a compact smooth domain $\Sigma \subset \IR^2$ and a smooth map $f : \Sigma \to S$ with $f(\partial \Sigma) \subset \partial S$ such that $f$ restricted to the outer boundary circle of $\Sigma$ is non-contractible in $\partial S$ and $f$ restricted to all other boundary circles of $\Sigma$ is contractible in $\partial S$ and $\area_t f < A t$.
\end{enumerate}
Then there is a collection of sub-Ricci flows with surgery $U_1, \ldots, U_m \subset \MM$ on the time-interval $[t_0, t_\omega]$ such that for all $t \in [t_0, t_\omega]$, the sets $U_1 (t), \ldots, U_m(t) \subset \MM (t)$ are pairwise disjoint, incompressible, solid tori.
Moreover, for each $i = 1, \ldots, m$ there is a collar $P_i \subset \MM (t_\omega) \setminus \Int (U_1(t_\omega) \cup \ldots \cup U_m (t_\omega))$ of $U_i (t_\omega)$ that is diffeomorphic to $T^2 \times I$ and non-singular on the time-interval $[t_0, t_\omega]$ and there is a smooth map $h_i : D^2 \to \MM (t_0)$ such that the image of the boundary loop $\gamma_i = h_i |_{\partial D^2}$ is contained in $P_i$ and such that:
\begin{enumerate}[label=(\alph*)]
\item $| {\Rm_{t_\omega}} | < K_4 t_\omega^{-1}$ on $\MM(t_\omega) \setminus (U_1 (t_\omega) \cup \ldots \cup U_{m} (t_\omega))$,
\item $\ell_t (\gamma_i) < \alpha \sqrt{t}$ and $\diam_t \partial U_i(t) < \alpha \sqrt{t}$ for all $t \in [t_0, t_\omega]$ and $i = 1, \ldots, m$,
\item $\max \curv_t \gamma_i < \Gamma_4 t^{-1}$ for all $t \in [t_0, t_\omega]$ and all $i = 1, \ldots, m$,
\item $\area_{t_0} h_i < (A + 1)t_0$ for all $i = 1, \ldots, m$,
\item $\gamma_i$ is non-contractible in $P_i$.
\end{enumerate}
\end{Proposition}

\begin{proof}
Let $\tau$ be the constant from Proposition \ref{Prop:firstcurvboundstep3}, assume without loss of generality that $\tau < \frac1{10}$ and pick $N \in \IN$ minimal with the property that $(1+\tau)^N t_0 \geq t_\omega$.
For simplicity, we assume in the following that the equality case occurs in this inequality, i.e. that $t_0 = (1-\tau)^N t_\omega$ and that $N \geq 3$; if not, we decrease $\tau$ slightly.
Subdivide the time-interval $[(1-\tau) t_0, t_\omega]$ by times $t_k = (1-\tau)^{N-k} t_\omega$ for $k = -1, 0, 1, \ldots, N$.
Notice that for $k = 0$ this definition gives us the time $t_0$ since $t_0 = (1- \tau)^N t_\omega$.
Observe also that $N$ depends on $L$.

In the following proof, we will apply Proposition \ref{Prop:firstcurvboundstep3} at the times $t_0 \leftarrow t_k$ for $k = 0, \ldots, N$ with $A \leftarrow A$ and $\eta \leftarrow \eta^\circ$.
Here $\eta^\circ = \eta^\circ(L, A, \alpha) > 0$ is a constant that we are going to determine in the course of the proof.
It will be clear that $\eta^\circ$ can be chosen such that it only depends on $L$, $A$ and $\alpha$.
In order to be able to apply Proposition \ref{Prop:firstcurvboundstep3}, we assume $t_0 > T_3(\eta^\circ, A)$ and $w(t) < w_3(\eta^\circ, A)$ for all $t \in [\frac14 t_0, t_\omega]$.
Then at time $t_k$, for each $k = 0, \ldots, N$, Proposition \ref{Prop:firstcurvboundstep3} provides subsets $P_i^{(k)} \subset \MM(t_k)$, sub-Ricci flows with surgery $U_i^{(k)} \subset \MM$ on the time-interval $[t_{k-1}, t_k]$ as well as numbers $h_i^{(k)} \in [ \eta^\circ ,2]$ for $i = 1, \ldots, m^{(k)}$.
The $P_i^{(k)}$ are $h_i^{(k)}$-precise torus structures at scale $\sqrt{t_k}$ at any time of the time-interval $[t_{k-1}, t_k]$ and, using the universal constant $K$ from Proposition \ref{Prop:firstcurvboundstep3}, we have for each $k =0, \ldots, N$
\begin{equation} \label{eq:summarycurvboundgoodpart} 
| {\Rm} | < K t^{-1} \quad \text{on} \quad \big( \MM(t_k) \setminus \big( U_1^{(k)} (t_k) \cup \ldots \cup U_{m^{(k)}}^{(k)} (t_k) \big)\big) \times [t_{k-1}, t_k].
\end{equation}
Moreover, as described in Proposition \ref{Prop:firstcurvboundstep3}(b) we have the same curvature bound in a slightly larger region: Namely, for any $t \in [ t_{k-1}, t_k]$ and any $x \in \MM (t)$ with $\dist_{t} (x,  \MM(t) \setminus (U_1(t) \cup \ldots \cup U_m(t))) \leq \sqrt{t}$, the point $(x,t)$ is non-singular and we have $|{\Rm}| (x,t) < K t^{-1}$.
The purpose of this extra bound is purely technical. 
It will be needed later in a distance distortion argument, due to the existence of short geodesics that may leave the domain which is described in (\ref{eq:summarycurvboundgoodpart}).

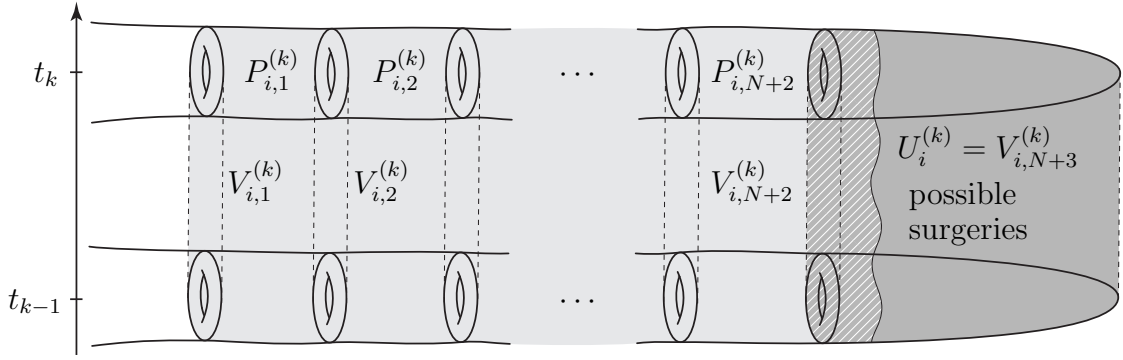
\begin{figure}[t] 
\begin{center}
\setlength{\unitlength}{2863sp}%
\begingroup\makeatletter\ifx\SetFigFont\undefined%
\gdef\SetFigFont#1#2#3#4#5{%
  \reset@font\fontsize{#1}{#2pt}%
  \fontfamily{#3}\fontseries{#4}\fontshape{#5}%
  \selectfont}%
\fi\endgroup%
\begin{picture}(4500,3000)(3500,0)
\hspace{20mm}\includegraphics[width=14cm]{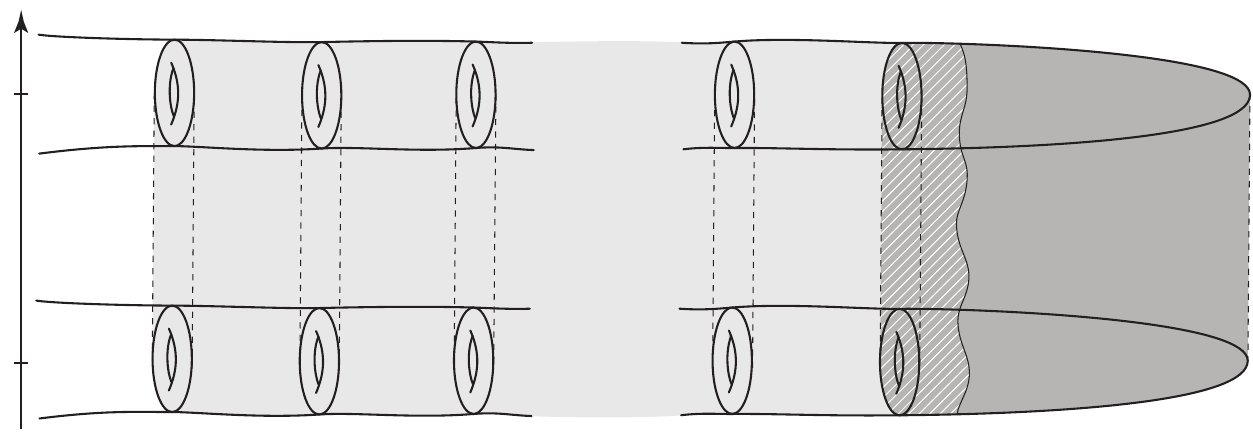}%
\put(-9330,2380){\makebox(0,0)[lb]{\smash{{\SetFigFont{12}{14.4}{\familydefault}{\mddefault}{\updefault}$t_k$}}}}
\put(-7530,2380){\makebox(0,0)[lb]{\smash{{\SetFigFont{12}{14.4}{\familydefault}{\mddefault}{\updefault}$P^{(k)}_{i,1}$}}}}
\put(-6430,2380){\makebox(0,0)[lb]{\smash{{\SetFigFont{12}{14.4}{\familydefault}{\mddefault}{\updefault}$P^{(k)}_{i,2}$}}}}
\put(-4830,2380){\makebox(0,0)[lb]{\smash{{\SetFigFont{12}{14.4}{\familydefault}{\mddefault}{\updefault}$\cdots$}}}}
\put(-3530,2380){\makebox(0,0)[lb]{\smash{{\SetFigFont{12}{14.4}{\familydefault}{\mddefault}{\updefault}$P^{(k)}_{i,N+2}$}}}}
\put(-1930,1720){\makebox(0,0)[lb]{\smash{{\SetFigFont{12}{14.4}{\familydefault}{\mddefault}{\updefault}$U^{(k)}_{i} = V^{(k)}_{i, N+3}$}}}}
\put(-9550,420){\makebox(0,0)[lb]{\smash{{\SetFigFont{12}{14.4}{\familydefault}{\mddefault}{\updefault}$t_{k-1}$}}}}
\put(-7670,1420){\makebox(0,0)[lb]{\smash{{\SetFigFont{12}{14.4}{\familydefault}{\mddefault}{\updefault}$V^{(k)}_{i,1}$}}}}
\put(-6590,1420){\makebox(0,0)[lb]{\smash{{\SetFigFont{12}{14.4}{\familydefault}{\mddefault}{\updefault}$V^{(k)}_{i,2}$}}}}
\put(-4830,420){\makebox(0,0)[lb]{\smash{{\SetFigFont{12}{14.4}{\familydefault}{\mddefault}{\updefault}$\cdots$}}}}
\put(-3530,1420){\makebox(0,0)[lb]{\smash{{\SetFigFont{12}{14.4}{\familydefault}{\mddefault}{\updefault}$V^{(k)}_{i,N+2}$}}}}
%\put(-1930,420){\makebox(0,0)[lb]{\smash{{\SetFigFont{12}{14.4}{\familydefault}{\mddefault}{\updefault}$U^{(k)}_{i} (t_{k-1})$}}}}
%
\put(-1830,1320){\makebox(0,0)[lb]{\smash{{\SetFigFont{12}{14.4}{\familydefault}{\mddefault}{\updefault}possible}}}}
\put(-1830,1020){\makebox(0,0)[lb]{\smash{{\SetFigFont{12}{14.4}{\familydefault}{\mddefault}{\updefault}surgeries}}}}
\end{picture}%
\caption{Illustration of the decomposition (\ref{eq:Pik-decomposition}) of $P^{(k)}_i$.
The subsets $P^{(k)}_{i,1}, \ldots, P^{(k)}_{i, N+2}$ (lightly shaded region) and the complement of the solid tori $P_i^{(k)} \cup U^{(k)}_i (t_k)$ (white region) are non-singular on the time-interval $[t_{k-1}, t_k]$.
Singularities may occur in the sub-Ricci flow with surgery $U^{(k)}_i$ (dark shaded region).
As explained in (\ref{eq:summarycurvboundgoodpart}), the curvature bound $|{\Rm}| < K t^{-1}$ holds on the lightly shaded and white region.
By the remark after (\ref{eq:summarycurvboundgoodpart}), this bound extends to the diagonally ruled collar region inside $U^{(k)}_i (t)$ for $t \in [t_{k-1}, t_k]$.
\label{fig:Pik-decomposition}}
\end{center}
\end{figure}
For all $k = 0, \ldots, N$ and $i = 1, \ldots, m^{(k)}$, we can divide $P_i^{(k)}$ into $N + 2$ approximately equally long torus structures (compare with Figure \ref{fig:Pik-decomposition}):
\begin{equation} \label{eq:Pik-decomposition}
 P_i^{(k)} = P_{i,1}^{(k)} \cup \ldots \cup P_{i, N+2}^{(k)} 
\end{equation}
such that if $h^{(k)}_i < \frac1{2(N+1)}$, then the $P_{i,j}^{(k)}$ are $2N h_i^{(k)}$-precise at scale $\sqrt{t_k}$ and time $t_k$ and such that $P_{i,j}^{(k)}$ and $P_{i,j+1}^{(k)}$ are adjacent for any $j < N + 2$ and $P_{i, N+2}^{(k)}$ is adjacent to $U_i^{(k)}(t_k)$.
Using this subdivision, we define the sub-Ricci flows with surgery $V^{(k)}_{i,1}, \ldots, V^{(k)}_{i, N+3} \subset \MM$ on the time-interval $[t_{k-1}, t_k]$ as follows:
$V^{(k)}_{i,j}$ is the extension of the subset
\[ P^{(k)}_{i, j} \cup \ldots \cup P^{(k)}_{i, N+2} \cup U^{(k)}_i (t_k) \subset \MM(t_k) \]
to the time-interval $[t_{k-1}, t_k]$.
Note that $V^{(k)}_{i, N+3} = U_i^{(k)}$.

A large part of the following proof will be concerned with the analysis of the spatial relations of the subsets $P^{(k)}_{i,j}$ and $V^{(k)}_{i,j} \subset \MM (t_k)$ for different $k = 0, \ldots, N$.
Notice that these subsets are contained in different time-slices and may not survive to a common time.
So, for example, it may not be possible to state that $P^{(k_1)}_{i_1, j_1}$ and $P^{(k_2)}_{i_2, j_2}$ are ``disjoint'' or that ``one is contained in the other'', because not all points in $P^{(k_1)}_{i_1, j_1} \subset \MM (t_{k_1})$ may survive up to time $t_{k_2}$ or, vice versa, not all points in $P^{(k_2)}_{i_2, j_2} \subset \MM (t_{k_2})$ may survive up to time $t_{k_1}$.
In other words, surgeries may interfere with $P^{(k_1)}_{i_1, j_1}$ or $P^{(k_2)}_{i_2, j_2}$ between the times $t_{k_1}$ and $t_{k_2}$.
Recall that there are no surgery points on the time-interval $[t_{-1}, t_N]$ outside the sub-Ricci flows with surgery $U^{(k)}_{i}$.
So if surgery points interfere with $P^{(k_1)}_{i_1,j_1}$ or $P^{(k_2)}_{i_2, j_2}$, then even some $U^{(k')}_{i'}$ interferes with it, in the sense that there is a point in $P^{(k_1)}_{i_1,j_1}$ or $P^{(k_2)}_{i_2, j_2}$ that survives up to some point inside $U^{(k')}_{i'}$.
Our main goal will be to show that if one of the torus structures $P_i^{(N)}$ is sufficiently precise, then there must be some $P^{(k^*)}_{i^*, 1}$ that is non-singular on the entire time-interval $[t_{-1}, t_\omega = t_N]$ and none of the $U^{(k')}_{i'}$ interferes with it.
%The latter statement will imply the curvature bound $|{\Rm}| < Kt^{-1}$ on $P^{(k^*)}_{i^*, 1} \times [t_{-1}, t_\omega]$ (see (\ref{eq:summarycurvboundgoodpart})).
The existence of the torus structure $P^{(k^*)}_{i^*, 1}$ will eventually be established in Claim~3 later (we will even prove a slightly more general result there).
An important objective for the proof of Claim 3 is to rule out a circular interference pattern between the $P^{(k)}_i$ and the $U^{(k')}_{i'}$.
For example, it may a priori be possible that there is a circular chain $(k_1, i_1), \ldots, (k_m, i_m), (k_{m+1} , i_{m+1})= (k_1, i_1)$ such that $U^{(k_{l+1})}_{i_{l+1}}$ interferes with $P^{(k_l)}_{i_l}$ for all $l = 1, \ldots, m$.
If this chain consisted of all pairs $(k, i)$, $k = 0, \ldots, N$, $i = 1, \ldots, m^{(k)}$, then the assertion of Claim 3 would be false.

As a first step towards Claim 3, let us now look at a more elementary case.
As explained in the last paragraph, the scenario in which some $P^{(k_1)}_{i_1, j}$ fails to be non-singular on the entire time-interval $[t_{-1}, t_\omega]$ requires special attention.
As every surgery point on the time-interval $[t_{-1}, t_\omega]$ is contained in one of the sub-Ricci flows with surgery $U^{(k')}_{i'}$ (recall that the flow is non-singular away from the $U^{(k')}_{i'}$), this scenario is accompanied by the phenomenon that some point in $P^{(k_1)}_{i_1, j}$ survives forward to some time $t_{k_2 - 1}$ or backward to some time $t_{k_2}$ and ends up in $U^{(k_2)}_{i_2} (t_{k_2 - 1})$ or $U^{(k_2)}_{i_2} (t_{k_2})$, respectively (see Figure \ref{fig:Pik1k2}).
In the following Claim 1, we will analyze this phenomenon under the additional assumptions that $j \leq N+1$ and that \emph{all} points of $P^{(k_1)}_{i_1, j} \cup P^{(k_1)}_{i_1, j+1}$ survive until time $t_{k_2 -1}$ (for $k_1 < k_2$) or $t_{k_2}$ (for $k_2 < k_1$) and don't intersect any $U^{(k')}_{i'}$ strictly between time $t_{k_1}$ and time $t_{k_2-1}$ or $t_{k_2}$.
So in this case, it makes sense to say that ``$P^{(k_1)}_{i_1, j}$, followed forward or backward in time up to time $t_{k_2 -1}$ or $t_{k_2}$, intersects $U^{(k_2)}_{i_2} (t_{k_2 - 1})$ or $U^{(k_2)}_{i_2} (t_{k_2})$''.
Our expectation is that then $P^{(k_1)}_{i_1, j+1}$, followed forward or backward in time up to time $t_{k_2 -1}$ or $t_{k_2}$, has to be almost fully contained in $U^{(k_2)}_{i_2} (t_{k_2 - 1})$ or $U^{(k_2)}_{i_2} (t_{k_2})$.
In the case $k_2 < k_1$, this will lead to a lower diameter bound of $U^{(k_2)}_{i_2} (t_{k_2})$, and therefore to an upper bound on $h^{(k_2)}_{i_2}$ in terms of $h^{(k_1)}_{i_1}$.
Furthermore, in both cases, $k_2 < k_1$ and $k_2 > k_1$, we will be able to extend $V_{i_1, j+2}^{(k_1)}$ to a sub-Ricci flow with surgery that is defined up to time $t_{k_2 -1}$ or $t_{k_2}$ such that $V_{i_1, j+2}^{(k_1)} (t_{k_2 - 1}) \subsetneq U_{i_2}^{(k_2)} (t_{k_2 - 1} )$ or $V_{i_1, j+2}^{(k_1)} (t_{k_2}) \subsetneq U_{i_2}^{(k_2)} (t_{k_2} )$.
In Claim 2, we will use this containment relationship to rule out the feared circular interference relation between the $P^{(k)}_{i}$ and the $U^{(k)}_i$.

\begin{figure}[t] 
\begin{center}
\setlength{\unitlength}{2863sp}%
\begingroup\makeatletter\ifx\SetFigFont\undefined%
\gdef\SetFigFont#1#2#3#4#5{%
  \reset@font\fontsize{#1}{#2pt}%
  \fontfamily{#3}\fontseries{#4}\fontshape{#5}%
  \selectfont}%
\fi\endgroup%
\begin{picture}(4500,5300)(3500,0)
\hspace{18mm}\includegraphics[width=14.3cm]{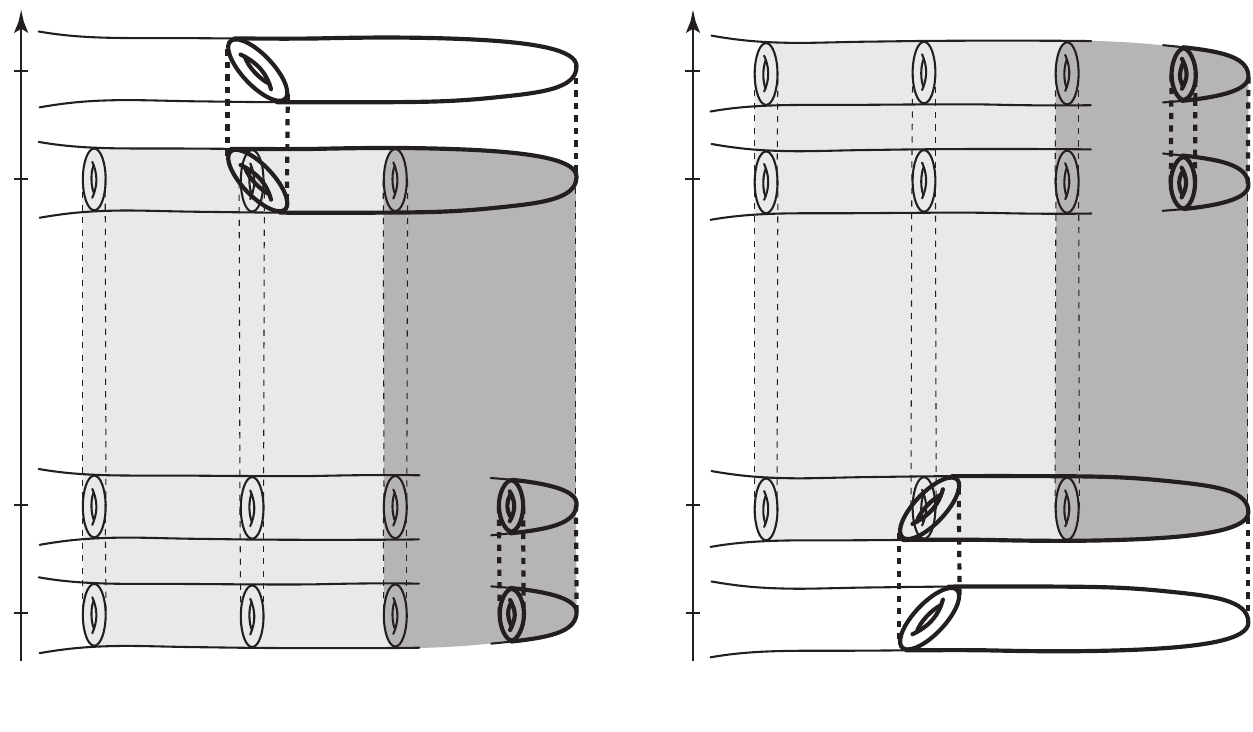}%
\put(-9570,4860){\makebox(0,0)[lb]{\smash{{\SetFigFont{12}{14.4}{\familydefault}{\mddefault}{\updefault}$t_{k_2}$}}}}
\put(-9830,4030){\makebox(0,0)[lb]{\smash{{\SetFigFont{12}{14.4}{\familydefault}{\mddefault}{\updefault}$t_{k_2-1}$}}}}
\put(-6870,4830){\makebox(0,0)[lb]{\smash{{\SetFigFont{12}{14.4}{\familydefault}{\mddefault}{\updefault}$U^{(k_2)}_{i_2}$}}}}
\put(-6020,2730){\makebox(0,0)[lb]{\smash{{\SetFigFont{12}{14.4}{\familydefault}{\mddefault}{\updefault}$V^{(k_1)}_{i_1,j+2}$}}}}
\put(-8330,1590){\makebox(0,0)[lb]{\smash{{\SetFigFont{12}{14.4}{\familydefault}{\mddefault}{\updefault}$P^{(k_1)}_{i_1,j}$}}}}
\put(-7280,1590){\makebox(0,0)[lb]{\smash{{\SetFigFont{12}{14.4}{\familydefault}{\mddefault}{\updefault}$P^{(k_1)}_{i_1,j+1}$}}}}
\put(-9570,1640){\makebox(0,0)[lb]{\smash{{\SetFigFont{12}{14.4}{\familydefault}{\mddefault}{\updefault}$t_{k_1}$}}}}
\put(-9830,810){\makebox(0,0)[lb]{\smash{{\SetFigFont{12}{14.4}{\familydefault}{\mddefault}{\updefault}$t_{k_1-1}$}}}}
\put(-6100,1590){\makebox(0,0)[lb]{\smash{{\SetFigFont{12}{14.4}{\familydefault}{\mddefault}{\updefault}$\cdots$}}}}
\put(-6100,810){\makebox(0,0)[lb]{\smash{{\SetFigFont{12}{14.4}{\familydefault}{\mddefault}{\updefault}$\cdots$}}}}
\put(-7600,0){\makebox(0,0)[lb]{\smash{{\SetFigFont{12}{14.4}{\familydefault}{\mddefault}{\updefault}Case (i): $k_1 < k_2$}}}}%
\put(-4570,4860){\makebox(0,0)[lb]{\smash{{\SetFigFont{12}{14.4}{\familydefault}{\mddefault}{\updefault}$t_{k_1}$}}}}
\put(-4830,4030){\makebox(0,0)[lb]{\smash{{\SetFigFont{12}{14.4}{\familydefault}{\mddefault}{\updefault}$t_{k_1-1}$}}}}
\put(-1670,750){\makebox(0,0)[lb]{\smash{{\SetFigFont{12}{14.4}{\familydefault}{\mddefault}{\updefault}$U^{(k_2)}_{i_2}$}}}}
\put(-1120,2730){\makebox(0,0)[lb]{\smash{{\SetFigFont{12}{14.4}{\familydefault}{\mddefault}{\updefault}$V^{(k_1)}_{i_1,j+2}$}}}}
\put(-3330,4810){\makebox(0,0)[lb]{\smash{{\SetFigFont{12}{14.4}{\familydefault}{\mddefault}{\updefault}$P^{(k_1)}_{i_1,j}$}}}}
\put(-2290,4810){\makebox(0,0)[lb]{\smash{{\SetFigFont{12}{14.4}{\familydefault}{\mddefault}{\updefault}$P^{(k_1)}_{i_1,j+1}$}}}}
\put(-4570,1640){\makebox(0,0)[lb]{\smash{{\SetFigFont{12}{14.4}{\familydefault}{\mddefault}{\updefault}$t_{k_2}$}}}}
\put(-4830,810){\makebox(0,0)[lb]{\smash{{\SetFigFont{12}{14.4}{\familydefault}{\mddefault}{\updefault}$t_{k_2-1}$}}}}
\put(-1120,4810){\makebox(0,0)[lb]{\smash{{\SetFigFont{12}{14.4}{\familydefault}{\mddefault}{\updefault}$\cdots$}}}}
\put(-1120,4010){\makebox(0,0)[lb]{\smash{{\SetFigFont{12}{14.4}{\familydefault}{\mddefault}{\updefault}$\cdots$}}}}
\put(-3000,0){\makebox(0,0)[lb]{\smash{{\SetFigFont{12}{14.4}{\familydefault}{\mddefault}{\updefault}Case (ii): $k_2 < k_1$}}}}%
\end{picture}%
\caption[caption]{Illustration of the setting in Claim~1.
The subsets $P^{(k_1)}_{i_1,j}$ and $P^{(k_1)}_{i_2,j+1}$ survive until time $t_{k_2-1}$ (if $k_1 < k_2$) or $t_{k_2}$ (if $k_2 < k_1$); see the light gray regions. 
The light gray regions do not intersect the interior of any $U^{(k')}_{i'}$.
At time $t_{k_2-1}$ (if $k_1 < k_2$) or $t_{k_2}$ (if $k_2 < k_1$), the subset $P^{(k_1)}_{k_1,j}$ intersects the initial or final time-slice of $U^{(k_2)}_{i_2}$ (bold outlined region). 
As demonstrated in the figure, this may entail that the boundary tori of $P^{(k_1)}_{i_1,j}$ and $U^{(k_2)}_{i_2}$ intersect in a complicated way. 
 \\\hspace{\textwidth}
Claim~1 asserts that the sub-Ricci flow with surgery $V^{(k_1)}_{i_1, j+2}$ can be extended to the time-interval $[t_{k_1-1}, t_{k_2-1}]$ (if $k_1 < k_2$) or $[t_{k_2}, t_{k_1}]$ (if $k_2 < k_1$), see the dark gray regions, and that $V_{i_1, j+2}^{(k_1)} (t_{k_2 - 1}) \subsetneq U_{i_2}^{(k_2)} (t_{k_2 - 1} )$ (in case $k_1 < k_2$) or $V_{i_1, j+2}^{(k_1)} (t_{k_2}) \subsetneq U_{i_2}^{(k_2)} (t_{k_2} )$ (in case $k_2 < k_1$). \\\hspace{\textwidth}
In the proof of Claim 1, the curvature bound $|{\Rm}| < Kt^{-1}$ on the light gray region is used to show that $P^{(k_1)}_{i_1, j}$ and $P^{(k_1)}_{i_1, j+1}$ remain sufficiently precise up to time $t_{k_2-1}$ or $t_{k_2}$.
\label{fig:Pik1k2}}
\end{center}
\end{figure}
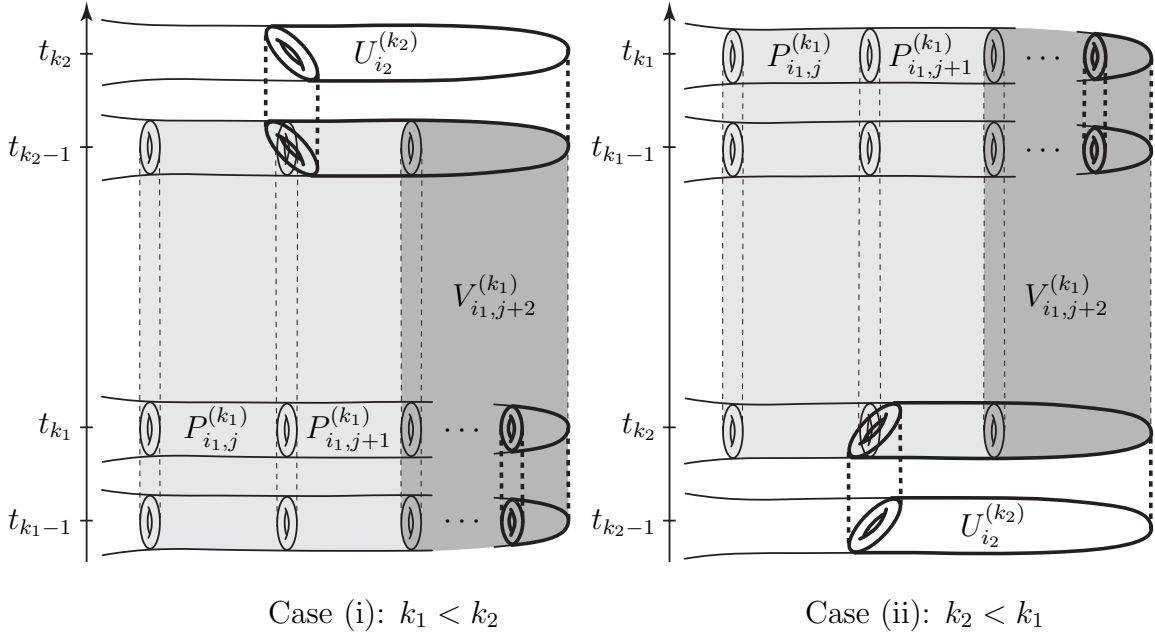
\begin{Claim1}
There are a constant $\eta^*_1 = \eta^*_1 (L, A) > 0$ and a non-decreasing function $\varphi^*_1 = \varphi^*_{1, L, A} : (0, \infty) \to (0, \infty)$ that both depend on $L$ and $A$ such that $\varphi^*_1(h) < h$ for all $h > 0$ and such that the following holds: \\
Let $k_1, k_2 \in \{1, \ldots, N \}$, $k_1 \neq k_2$, $i_1 \in \{ 1, \ldots, m^{(k_1)} \}$, $i_2 \in \{ 1, \ldots, m^{(k_2)} \}$.
Assume that there is some $j \leq N+1$ such that (see Figure~\ref{fig:Pik1k2} for an illustration)
\begin{enumerate}[label=(\roman*)]
\item if $k_1 < k_2$: all points in $P_{i_1, j}^{(k_1)} \cup P_{i_1, j+1}^{(k_1)}$ survive until time $t_{k_2-1}$ and $P_{i_1, j}^{(k_1)} \cap  U_{i_2}^{(k_2)} (t_{k_2 - 1}) \neq \emptyset$.
Moreover, we have $(P_{i_1, j}^{(k_1)} \cup P_{i_1, j+1}^{(k_1)}) \cap U_{i'}^{(k')}(t_{k'}) = \emptyset$ for all $k'$ strictly between $k_1$ and $k_2$ and all  $i' \in \{ 1, \ldots, m^{(k')} \}$.
\item if $k_2 < k_1$: all points in $P_{i_1, j}^{(k_1)} \cup P_{i_1, j+1}^{(k_1)}$ survive until time $t_{k_2}$ and $P_{i_1, j}^{(k_1)} \cap  U_{i_2}^{(k_2)} (t_{k_2}) \neq \emptyset$.
Moreover, we have $(P_{i_1, j}^{(k_1)} \cup P_{i_1, j+1}^{(k_1)}) \cap U_{i'}^{(k')}(t_{k'}) = \emptyset$ for all $k'$ strictly between $k_1$ and $k_2$ and all  $i' \in \{ 1, \ldots, m^{(k')} \}$.
\end{enumerate}
Then in case (ii) we have $\varphi_1^* (h_{i_2}^{(k_2)}) \leq \max \{ h_{i_1}^{(k_1)}, \varphi_1^*(\eta^\circ) \}$.

Moreover in cases (i) and (ii), assuming $h_{i_1}^{(k_1)} < \eta^*_1$, we can uniquely extend the sub-Ricci flow with surgery $V_{i_1, j+2}^{(k_1)}$ to the time-interval $[t_{k_1-1}, t_{k_2 - 1}]$ (in case (i)) or $[t_{k_2}, t_{k_1}]$ (in case (ii)).
These extensions satisfy $V_{i_1, j+2}^{(k_1)} (t_{k_2 - 1}) \subsetneq U_{i_2}^{(k_2)} (t_{k_2 - 1} )$ (in case (i)) or $V_{i_1, j+2}^{(k_1)} (t_{k_2}) \subsetneq U_{i_2}^{(k_2)} (t_{k_2} )$ (in case (ii)).
\end{Claim1}

The strategy for the proof of this claim will be to use the curvature bound $|{\Rm}| < K t^{-1}$ on (and near) $(P_{i_1, j}^{(k_1)} \cup P_{i_1, j+1}^{(k_1)}) \times [t_{k_1-1}, t_{k_2-1}]$ (or $(P_{i_1, j}^{(k_1)} \cup P_{i_1, j+1}^{(k_1)}) \times [t_{k_2}, t_{k_1}]$)  to show that $P_{i_1, j}^{(k_1)}$ and $P_{i_1, j+1}^{(k_1)}$ remain sufficiently precise torus structure up to time $t_{k_2-1}$ (or $t_{k_2}$, respectively).
We will then compare the boundary torus $\partial U^{(k_2)}_{i_2} (t_{k_2-1})$ (or $\partial U^{(k_2)}_{i_2} (t_{k_2})$) with the torus structures $P_{i_1, j}^{(k_1)}$ and $P_{i_1, j+1}^{(k_1)}$ at time $t_{k_2-1}$ (or $t_{k_2}$) and argue that the asserted containment relationship holds.
The possibility that the boundary torus $\partial U^{(k_2)}_{i_2} (t_{k_2-1})$ (or $\partial U^{(k_2)}_{i_2} (t_{k_2})$) may intersect one of the boundary tori of  $P_{i_1, j}^{(k_1)}$ in a non-trivial way, will present a slight technical difficulty to us here.
Finally, in the case $k_2 < k_1$, we can roughly bound the diameter of $U^{(k_2)}_{i_2} (t_{k_2})$ from below by the length of $P_{i_1, j+1}^{(k_1)}$ at time $t_{k_2}$.
This bound will in turn imply an upper bound of $h^{(k_2)}_{i_2}$ in terms of $h^{(k_1)}_{i_1}$.

\begin{proof}
By (\ref{eq:summarycurvboundgoodpart}) and the fact that $(P_{i_1, j}^{(k_1)} \cup P_{i_1, j+1}^{(k_1)}) \cap U_{i'}^{(k')}(t_{k'}) = \emptyset$ for all $k'$ strictly between $k_1$ and $k_2$ and all  $i' \in \{ 1, \ldots, m^{(k')} \}$, we have
\begin{multline} \label{eq:RmlessKPP}
|{\Rm_t}| < K t^{-1} \qquad \text{on} \qquad P_{i_1, j}^{(k_1)} \cup P_{i_1, j+1}^{(k_1)} \quad \\ \text{for all} \quad \text{$t \in [t_{k_1-1}, t_{k_2-1}]$ (in case (i))} \quad \text{or \quad $t \in [t_{k_2}, t_{k_1}]$ (in case (ii)).}
\end{multline}
Using the additional comment after (\ref{eq:summarycurvboundgoodpart}), we can extend the curvature bound in (\ref{eq:RmlessKPP}) to a slightly larger subset:
More specifically, we claim that there is a constant $c (L) > 0$ such that the following holds: whenever $x \in \MM (t_{k_1})$ with $\dist_{t_{k_1}} (x, P_{i_1, j}^{(k_1)} \cup P_{i_1, j+1}^{(k_1)}) < c  \sqrt{t_{k_1}}$, then $(x,t_{k_1})$ survives until any time of the time-interval $[t_{k_1-1}, t_{k_2-1}]$ (in case (i)) or any time of the time-interval $[t_{k_2}, t_{k_1}]$ (in case (ii)) and $|{\Rm}| (x,t) < K t^{-1}$ for all $t \in [t_{k_1-1}, t_{k_2-1}]$ or $t \in [t_{k_2}, t_{k_1}]$, respectively.
To see this, choose $c(L) < \frac12$ so small that by distance distortion, any curve of time-$t_{k_1}$ length less than $c \sqrt{t_{k_1}}$ has to have length less than $\frac12 \sqrt{t}$ at any time $t \in [t_{k_1-1}, t_{k_2-1}]$ or $t \in [t_{k_2}, t_{k_1}]$, as long as the curvature along this curve satisfies the bound $|{\Rm_{t'}}| \leq K {t'}^{-1}$ for all $t'$ between $t_{k_1}$ and $t$.
Now consider a point $x \in \MM (t_1)$ with $\dist_{t_{k_1}} (x, P_{i_1, j}^{(k_1)} \cup P_{i_1, j+1}^{(k_1)}) < c  \sqrt{t_{k_1}}$.
Let $\sigma : [0,1] \to \MM (t_{k_1})$ be the shortest time-$t_{k_1}$ minimizing geodesic between $x$ and $P_{i_1, j}^{(k_1)} \cup P_{i_1, j+1}^{(k_1)}$.
Consider some time $t^* \in [t_{k_1-1}, t_{k_2-1}]$ or $t^* \in [t_{k_2}, t_{k_1}]$ with the property that for all $s \in [0,1]$, the point $(\sigma(s), t_{k_1})$ survives until time $t^*$ and that for all $t \in [t_{k_1}, t^*]$ or $t \in [t^*, t_{k_1}]$, the time-$t$ length of $\sigma$ is less than $\sqrt{t}$.
So, using the comment after (\ref{eq:summarycurvboundgoodpart}), we find that $(\sigma(s),t)$ is non-singular and $|{\Rm}| (\sigma(s), t) < K t^{-1}$ for all $s \in [0,1]$ and $t \in [t_{k_1}, t^*]$ or $t \in [t^*, t_{k_1}]$.
By distance distortion and our choice of $c$, this implies that for all $t \in [t_{k_1}, t^*]$ or $t \in [t^*, t_{k_1}]$, the time-$t$ length of $\sigma$ is even less than $\frac12 \sqrt{t}$.
This implies that if $t^*$ was chosen maximal/minimal with the required properties, then we must have $t^* = t_{k_2-1}$ or $t^* = t_{k_2}$ or $t^* = t_{k_1-1}$, depending on which case we are in, and it proves the claim made after (\ref{eq:RmlessKPP}).

Next, we use the curvature bound in (\ref{eq:RmlessKPP}) and the remark thereafter to show that there is a function $\varphi' = \varphi'_L: (0, \infty) \to (0, \infty]$, which only depends on $L$, such that $\lim_{h \to 0} \varphi' (h) = 0$ and such that the following holds: 
If $\varphi'(h_{i_1}^{(k_1)}) < \infty$, then $P_{i_1, j}^{(k_1)}$ and $P_{i_1, j+1}^{(k_1)}$ are still $\varphi'(h_{i_1}^{(k_1)})$-precise torus structures at scale $\sqrt{t_{k_2-1}}$ or $\sqrt{t_{k_2}}$ and at time $t_{k_2-1}$ or $t_{k_2}$ (depending on whether we are in case (i) or (ii)).
We will carry out the proof only for $P_{i_1, j}^{(k_1)}$ in case (i).
The other cases follow analogously.
First, choose $d$ in such a way that the time-$t_{k_2 - 1}$ distance between the two boundary components of $P_{i_1, j}^{(k_1)}$ is equal to $d \sqrt{t_{k_2 - 1}}$.
Then there is a time-$t_{k_2 - 1}$ minimizing geodesic $\sigma : [0,1] \to \MM (t_{k_2 - 1})$ between these boundary components, whose time-$t_{k_2-1}$ length is equal to $d \sqrt{t_{k_2 - 1}}$.
By minimality, the image of this geodesic has to be contained in $P_{i_1, j}^{(k_1)}$.
So, using distance distortion and (\ref{eq:RmlessKPP}), we find a constant $C^* = C^* (L) < \infty$ such that the time-$t_{k_1}$ length of $\sigma$ is less than $C^* d \sqrt{t_{k_1}}$.
But if $h^{(k_1)}_{i_1} < \frac1{2(N+1)}$, then by construction of $P^{(k_1)}_{i_1, j}$ this distance must be larger than $(2N h^{(k_1)}_{i_1})^{-1} \sqrt{t_{k_1}}$.
Therefore, we obtain a lower bound on $d$ that goes to infinity as $h^{(k_1)}_{i_1}$ goes to zero.
Next, we bound the time-$t_{k_2 - 1}$ diameters of the cross-sectional tori of $P^{(k_1)}_{i_1, j}$ from above.
Let $x, y \in P^{(k_1)}_{i_1, j}$ be two points that lie in the same cross-sectional $2$-torus.
If $h^{(k_1)}_{i_1} < \frac1{2(N+1)}$, then by construction of $P^{(k_1)}_{i_1, j}$ we know that $\dist_{t_{k_1}} (x,y) \leq 2N h^{(k_1)}_{i_1} \sqrt{t_{k_1}}$.
Let $\sigma : [0,1] \to \MM (t_{k_1})$ be a time-$t_{k_1}$ minimizing geodesic between $x$ and $y$.
If $2N h^{(k_1)}_{i_1} < c$, then for all $s \in [0,1]$ the point $(\sigma(s), t_{k_1})$ survives until time $t_{k_2 - 1}$ and we have $|{\Rm}| (\sigma (s), t) < K t^{-1}$ for all $s \in [0,1]$ and $t \in [t_{k_1}, t_{k_2 - 1}]$.
Therefore, by distance distortion, we obtain that $\dist_{t_{k_2 - 1}} (x,y) \leq C^{**} h^{(k_2)}_{i_2} \sqrt{t_{k_2 - 1}}$ for some $C^{**} = C^{**} (L) < \infty$.
So for sufficiently small $h^{(k_1)}_{i_1}$, we can ensure that $P_{i_1, j}^{(k_1)}$ is arbitrarily precise at scale $\sqrt{t_{k_2-1}}$ and at time $t_{k_2-1}$, which proves the existence of the function $\varphi'$.

Next we show that it is possible to extend the sub-Ricci flow with surgery $V_{i_1, j+2}^{(k_1)}$, which was originally defined on the time-interval $[t_{k_1-1}, t_{k_1}]$, to the time-interval $[t_{k_1-1}, t_{k_2 - 1}]$ (in case (i)) or $[t_{k_2}, t_{k_1}]$ (in case (ii)).
As the flow is non-singular on $(P_{i_1, j}^{(k_1)} \cup P_{i_1, j+1}^{(k_1)}) \times [t_{k_1-1}, t_{k_2 - 1}]$ or $(P_{i_1, j}^{(k_1)} \cup P_{i_1, j+1}^{(k_1)}) \times [t_{k_2}, t_{k_1}]$, respectively, we know that the points in the boundary $\partial V^{(k_1)}_{i_1, j+2} (t_{k_1})$ survive on the entire time-interval $[t_{k_1-1}, t_{k_2 - 1}]$ or $[t_{k_2-1}, t_{k_1}]$.
Also, since by assumption (iii), all surgeries on $[\frac14 t_0, t_\omega]$ are trivial, the boundary $\partial V^{(k_1)}_{i_1, j+2} (t_{k_1})$ remains separating in $\MM (t)$ for all $t \in [t_{k_1-1}, t_{k_2 - 1}]$ or $t \in [t_{k_2}, t_{k_1}]$.
Define now $V^{(k_1)}_{i_1, j+2} (t)$ for each such $t$ to be closure of the component of $\MM (t) \setminus \partial V^{(k_1)}_{i_1, j+2} (t)$ that does not contain $P_{i_1, j+1}^{(k_1)}$.
Then $V_{i_1, j+2}^{(k_1)}$ becomes a sub-Ricci flow with surgery on the time-interval $[t_{k_1-1}, t_{k_2 - 1}]$ or $[t_{k_2}, t_{k_1}]$, which finishes our construction.
Observe moreover, that the topology of $V^{(k_1)}_{i_1, j+2} (t)$ does not change in time, since by assumption (iii) all surgeries on $[\frac14 t_0, t_\omega]$ are trivial and $\MM (t)$ is free of spherical components for all $t \in [\frac14 t_0, t_\omega]$.
So for all $t \in [t_{k_1-1}, t_{k_2 - 1}]$ or $t \in [t_{k_2}, t_{k_1}]$ we have $V^{(k_1)}_{i_1, j+2} (t) \approx S^1 \times D^2$.

Our next goal is to show that if $h_{i_1}^{(k_1)}$ is small enough, depending on $L$ and $A$, then we must have $V_{i_1, j+2}^{(k_1)} (t_{k_2 /- 1}) \subsetneq U_{i_2}^{(k_2)} (t_{k_2 /- 1} )$ (by ``$k_2/-1$'' we mean $k_2$ in case (i) and $k_2 -1$ in case (ii)).
We will first illustrate the idea of our proof using slightly imprecise language: By assumption, the solid tori $U_{i_2}^{(k_2)} (t_{k_2 /- 1} )$ and $V_{i_1, j}^{(k_1)} (t_{k_2 /- 1})$ (notice the ``$j$'' instead of ``$j+2$'' in the index!) intersect in a point.
The diameters of the boundaries of these solid tori are small compared to the lengths of a collar neighborhoods around these boundaries, on which the geometry is almost product-like.
So up to the addition or subtraction of a small collar, we may assume that either both solid tori cover a component of $\MM (t_{k_2 /-1})$ or one solid torus is contained in the other.
The first case can be ruled out using \cite[Lemma \ref{Lem:coverMbysth}]{Bamler-LT-topology} and assumption (iii).
So in Figure~\ref{fig:Pik1k2}, this means that both solid tori ``open up in the same direction''.
Next, we show that, up to the addition of a small collar neighborhood around its boundary, $U_{i_2}^{(k_2)} (t_{k_2 /- 1} )$ must contain $V_{i_1, j+1}^{(k_1)} (t_{k_2 /- 1})$ (notice the ``$j+1$'' instead of ``$j$'' in the index!), since a slight extension of $U_{i_2}^{(k_2)} (t_{k_2 /- 1} )$ contains a cross-sectional torus of $P^{(k_1)}_{i_1, j}$.
The add-on ``up to the addition of a small collar neighborhood around its boundary'' can be removed if we reduce the solid torus $V_{i_1, j+1}^{(k_1)} (t_{k_2 /- 1})$ to the (significantly smaller) solid torus $V_{i_1, j+2}^{(k_1)} (t_{k_2 /- 1})$.

Let us now carry out the precise proof:
As by assumption $U_{i_2}^{(k_2)} (t_{k_2 /- 1} )$ has a point in common with $P^{(k_1)}_{i_1,j}$, which is disjoint from $V_{i_1, j+2}^{(k_1)} (t_{k_2 /- 1})$, it suffices to show that $V_{i_1, j+2}^{(k_1)} (t_{k_2 /- 1}) \subset U_{i_2}^{(k_2)} (t_{k_2 /- 1} )$.

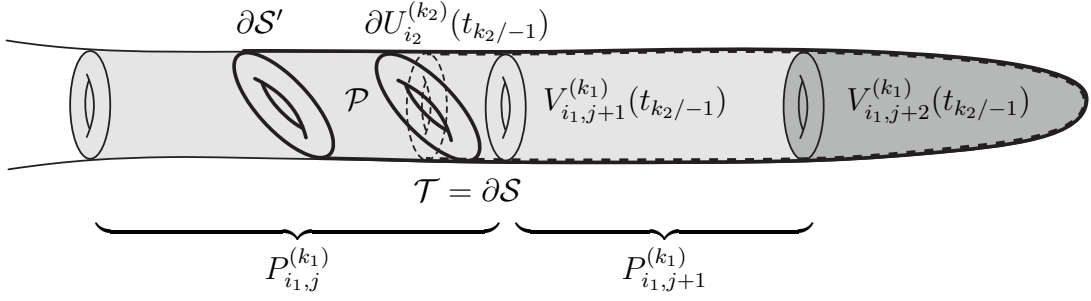
\begin{figure}[t] 
\begin{center}
\setlength{\unitlength}{2863sp}%
\begingroup\makeatletter\ifx\SetFigFont\undefined%
\gdef\SetFigFont#1#2#3#4#5{%
  \reset@font\fontsize{#1}{#2pt}%
  \fontfamily{#3}\fontseries{#4}\fontshape{#5}%
  \selectfont}%
\fi\endgroup%
\begin{picture}(4500,2250)(3500,-950)
\hspace{18mm}\includegraphics[width=14.3cm]{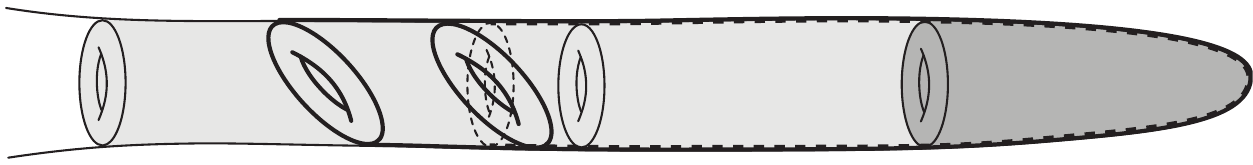}%
\put(-4670,560){\makebox(0,0)[lb]{\smash{{\SetFigFont{12}{14.4}{\familydefault}{\mddefault}{\updefault}$V^{(k_1)}_{i_1, j+1} (t_{k_2/-1})$}}}}
%\put(-4730,560){\makebox(0,0)[lb]{\smash{{\SetFigFont{12}{14.4}{\familydefault}{\mddefault}{\updefault}$\mathcal{S}' = U^{(k_2)}_{i_2} (t_{k_2 /-1})$}}}}
\put(-6370,560){\makebox(0,0)[lb]{\smash{{\SetFigFont{12}{14.4}{\familydefault}{\mddefault}{\updefault}$\mathcal{P}$}}}}
\put(-2070,560){\makebox(0,0)[lb]{\smash{{\SetFigFont{12}{14.4}{\familydefault}{\mddefault}{\updefault}$V^{(k_1)}_{i_1,j+2}(t_{k_2/-1})$}}}}
\put(-8500,-400){\makebox(0,0)[lb]{\smash{{\SetFigFont{12}{14.4}{\familydefault}{\mddefault}{\updefault}$\underbrace{\hspace{53mm}}_{\displaystyle P^{(k_1)}_{i_1, j}}$}}}}
\put(-4900,-400){\makebox(0,0)[lb]{\smash{{\SetFigFont{12}{14.4}{\familydefault}{\mddefault}{\updefault}$\underbrace{\hspace{39mm}}_{\displaystyle P^{(k_1)}_{i_1, j+1}}$}}}}
\put(-5780,-220){\makebox(0,0)[lb]{\smash{{\SetFigFont{12}{14.4}{\familydefault}{\mddefault}{\updefault}$\mathcal{T} = \partial \mathcal{S}$}}}}
\put(-6200,1230){\makebox(0,0)[lb]{\smash{{\SetFigFont{12}{14.4}{\familydefault}{\mddefault}{\updefault}$\partial U^{(k_2)}_{i_2} (t_{k_2/-1})$}}}}
\put(-7300,1230){\makebox(0,0)[lb]{\smash{{\SetFigFont{12}{14.4}{\familydefault}{\mddefault}{\updefault}$\partial \mathcal{S}'$}}}}
\end{picture}%
\caption[caption]{Picture at time $t_{k_2/-1}$ in the case in which $\partial U^{(k_2)}_{i_2} (t_{k_2/-1})$ intersects $P^{(k_1)}_{i_1,j}$.
The solid torus $U^{(k_2)}_{i_2} (t_{k_2/-1})$ and its extension $\mathcal{S}'$ by the collar extension $\mathcal{P}$ are outlined in bold.
The solid torus $\mathcal{S}$, which is bounded by a cross-sectional torus $\mathcal{T}$ of $P^{(k_1)}_{i_1,j}$, has a dashed outline.
The region  $P^{(k_1)}_{i_1,j} \cup  P^{(k_1)}_{i_1,j +1}$ is shaded in light gray and the region $V^{(k_1)}_{i_1, j+2} (t_{k_2/-1})$ is shaded in dark gray.
\label{fig:SS-SSprime1}}
\end{center}
\end{figure}
Consider first the case in which $\partial U_{i_2}^{(k_2)} (t_{k_2 /- 1} )$ intersects $P^{(k_1)}_{i_1,j}$ (see Figure~\ref{fig:SS-SSprime1} for an illustration).
Let $\mathcal{T} \subset P^{(k_1)}_{i_1, j}$ be a cross-sectional $2$-torus that intersects $\partial U_{i_2}^{(k_2)} (t_{k_2 /- 1} )$.
If $\partial U_{i_2}^{(k_2)} (t_{k_2 /- 1} )$ is not fully contained in $P^{(k_1)}_{i_1,j}$, then we can choose $\mathcal{T}$ to be one of its boundary tori.
The $2$-torus $\mathcal{T}$ bounds a solid torus $\mathcal{S} \subset \MM (t_{k_2 /-1})$, $\mathcal{T} = \partial \mathcal{S}$, $\mathcal{S} \approx S^1 \times D^2$, which is an extension of $V^{(k_1)}_{i_1,j+1} (t_{k_2/-1})$.
Therefore, we have $V_{i_1, j+2}^{(k_1)} (t_{k_2 /- 1}) \subset \mathcal{S}$.
As $P^{(k_1)}_{i_1, j}$ is $\varphi'( h^{(k_1)}_{i_1} )$-precise at scale $\sqrt{t_{k_2 /-1}}$ and at time $t_{k_2 /-1}$, we find that
\begin{equation} \label{eq:diamTTht}
 \diam_{t_{k_2/-1}} \mathcal{T} \leq \varphi'( h^{(k_1)}_{i_1} ) \sqrt{t_{k_2 /-1}} . 
\end{equation}
On the other hand, by Proposition \ref{Prop:firstcurvboundstep3}(c) we know that $P^{(k_2)}_{i_2}$ is $2$-precise at scale $\sqrt{t_{k_2}}$ and at time $t_{k_2 /-1}$.
So we can find a collar $\mathcal{P} \subset P^{(k_2)}_{i_2}$, $\mathcal{P} \approx T^2 \times I$, $\partial U_{i_2}^{(k_2)} (t_{k_2 /- 1} ) \subset \mathcal{P}$ whose two boundary components have time-$t_{k_2 /-1}$ distance of at least $\frac1{10} \sqrt{t_{k_2/-1}}$ from one another and whose time-$t_{k_2 /-1}$ diameter is less than $10 \sqrt{t_{k_2 /-1}}$.
Then $\mathcal{S}' = \mathcal{P} \cup U_{i_2}^{(k_2)} (t_{k_2 /- 1} )$ is diffeomorphic to a solid torus and by (\ref{eq:diamTTht}) we have $\partial \mathcal{S} \subset \Int \mathcal{S}'$ if $h^{(k_1)}_{i_1}$ is sufficiently small.
Assume for the remainder of this paragraph that $h^{(k_1)}_{i_1}$ is so small such that $\partial \mathcal{S} \subset \Int \mathcal{S}'$ holds.
If we also had $\partial \mathcal{S}' \subset \mathcal{S}$, then the union $\mathcal{S} \cup \mathcal{S}'$ would have no boundary.
So it would be a closed component of $\MM (t_{k_2 /-1})$.
Since $\mathcal{S}$ and $\mathcal{S}'$ are both diffeomorphic to solid tori, this would contradict \cite[Lemma \ref{Lem:coverMbysth}]{Bamler-LT-topology} and assumption (iii).
So $\partial \mathcal{S}'$ cannot be fully contained in $\mathcal{S}$.
Since $\partial \mathcal{S}$ is contained in the interior of $\mathcal{S}'$ and hence disjoint from its boundary, this implies that $\partial \mathcal{S}'$ is disjoint from $\mathcal{S}$.
Since $\mathcal{S}$ is connected, we obtain that $\mathcal{S} \subset \mathcal{S}'$.
So
\begin{equation} \label{eq:VsubsetPU}
 V_{i_1, j+2}^{(k_1)}(t_{k_2 /-1}) \subset \mathcal{S} \subset \mathcal{S}' = \mathcal{P} \cup U_{i_2}^{(k_2)} (t_{k_2 /- 1} ). 
\end{equation}
Lastly, note that the time-$t_{k_2 /-1}$ diameter of $\mathcal{P}$ is less than $10 \sqrt{t_{k_2 /-1}}$, but since $P^{(k_1)}_{i_1, j+1}$ is $\varphi'( h^{(k_1)}_{i_1} )$-precise at scale $\sqrt{t_{k_2 /-1}}$ and at time $t_{k_2 /-1}$, the boundary components of $P^{(k_1)}_{i_1,j + 1}$ have time-$t_{k_2 /-1}$ distance of more than $(\varphi'( h^{(k_1)}_{i_1} ))^{-1} \sqrt{t_{k_2 /-1}}$ from one another.
This subset separates $P^{(k_1)}_{i_1,j }$ from $V_{i_1, j+2}^{(k_1)} (t_{k_2 /- 1})$.
So, if $h_{i_1}^{(k_1)}$ is sufficiently small, then $P^{(k_1)}_{i_1,j }$ is far enough away from $V_{i_1, j+2}^{(k_1)} (t_{k_2 /- 1})$ to ensure that $\mathcal{P}$ cannot intersect both of these subsets.
Since $\mathcal{P}$ intersects $P^{(k_1)}_{i_1,j}$, it must be disjoint from $V_{i_1, j+2}^{(k_1)}(t_{k_2 /-1})$ for sufficiently small $h_{i_1}^{(k_1)}$.
So by (\ref{eq:VsubsetPU}) we must have $V_{i_1, j+2}^{(k_1)}(t_{k_2 /-1}) \subset U_{i_2}^{(k_2)} (t_{k_2 /- 1} )$ for sufficiently small $h_{i_1}^{(k_1)}$, which is what we wanted to show.

\begin{figure}[t] 
\begin{center}
\setlength{\unitlength}{2863sp}%
\begingroup\makeatletter\ifx\SetFigFont\undefined%
\gdef\SetFigFont#1#2#3#4#5{%
  \reset@font\fontsize{#1}{#2pt}%
  \fontfamily{#3}\fontseries{#4}\fontshape{#5}%
  \selectfont}%
\fi\endgroup%
\begin{picture}(4500,2050)(3500,-650)
\hspace{18mm}\includegraphics[width=14.3cm]{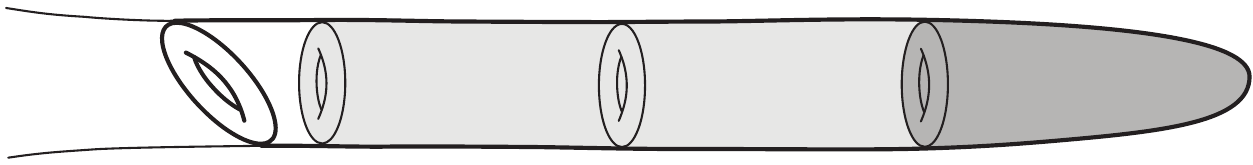}%
%\put(-4670,560){\makebox(0,0)[lb]{\smash{{\SetFigFont{12}{14.4}{\familydefault}{\mddefault}{\updefault}$V^{(k_1)}_{i_1, j+1} (t_{k_2/-1})$}}}}
%\put(-4730,560){\makebox(0,0)[lb]{\smash{{\SetFigFont{12}{14.4}{\familydefault}{\mddefault}{\updefault}$\mathcal{S}' = U^{(k_2)}_{i_2} (t_{k_2 /-1})$}}}}
\put(-6170,520){\makebox(0,0)[lb]{\smash{{\SetFigFont{12}{14.4}{\familydefault}{\mddefault}{\updefault}$P^{(k_1)}_{i_1, j}$}}}}
\put(-4170,520){\makebox(0,0)[lb]{\smash{{\SetFigFont{12}{14.4}{\familydefault}{\mddefault}{\updefault}$P^{(k_1)}_{i_1, j+1}$}}}}
\put(-2070,560){\makebox(0,0)[lb]{\smash{{\SetFigFont{12}{14.4}{\familydefault}{\mddefault}{\updefault}$V^{(k_1)}_{i_1,j+2}(t_{k_2/-1})$}}}}
\put(-6880,-150){\makebox(0,0)[lb]{\smash{{\SetFigFont{12}{14.4}{\familydefault}{\mddefault}{\updefault}$\underbrace{\hspace{103mm}}_{\displaystyle V^{(k_1)}_{i_1, j} (t_{k_2/-1})}$}}}}
%\put(-9700,530){\makebox(0,0)[lb]{\smash{{\SetFigFont{12}{14.4}{\familydefault}{\mddefault}{\updefault}$\partial U^{(k_2)}_{i_2} (t_{k_2/-1})$}}}}
\put(-7800,1230){\makebox(0,0)[lb]{\smash{{\SetFigFont{12}{14.4}{\familydefault}{\mddefault}{\updefault}$U^{(k_2)}_{i_2} (t_{k_2/-1})$}}}}
\end{picture}%
\caption[caption]{Picture at time $t_{k_2/-1}$ in the case in which $\partial U^{(k_2)}_{i_2} (t_{k_2/-1})$ does not intersect $P^{(k_1)}_{i_1,j}$.
The solid torus $U^{(k_2)}_{i_2} (t_{k_2/-1})$ is outlined in bold.
The region  $P^{(k_1)}_{i_1,j} \cup  P^{(k_1)}_{i_1,j +1}$ is shaded in light gray and the region $V^{(k_1)}_{i_1, j+2} (t_{k_2/-1})$ is shaded in dark gray.
\label{fig:SS-SSprime2}}
\end{center}
\end{figure}
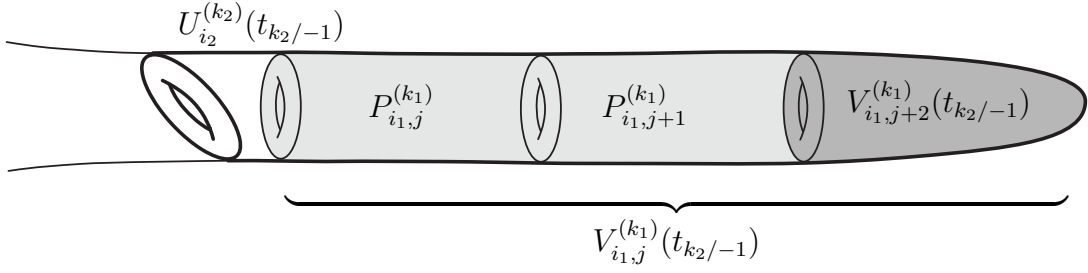
Second, consider the case in which $\partial U_{i_2}^{(k_2)} (t_{k_2 /- 1} )$ does not intersect $P^{(k_1)}_{i_1,j}$ (compare with Figure~\ref{fig:SS-SSprime2}).
Since by assumption $U_{i_2}^{(k_2)} (t_{k_2 /- 1} )$ intersects $P^{(k_1)}_{i_1,j}$, and since $P^{(k_1)}_{i_1,j}$ is connected, it follows that $P^{(k_1)}_{i_1,j} \subset \Int U_{i_2}^{(k_2)} (t_{k_2 /- 1} )$.
In particular, this implies that 
\begin{equation} \label{eq:partialVsubsetPU}
\partial V_{i_1, j}^{(k_1)} (t_{k_2 /-1}) \subset P^{(k_1)}_{i_1,j} \subset \Int U_{i_2}^{(k_2)} (t_{k_2 /- 1} ).
\end{equation}
(Note that $V_{i_1, j}^{(k_1)} (t_{k_2 /-1}) = P^{(k_1)}_{i_1,j} \cup P^{(k_1)}_{i_1,j+1} \cup V_{i_1, j+2}^{(k_1)} (t_{k_2 /-1})$.)
Since $\partial V_{i_1, j}^{(k_1)} (t_{k_2 /-1})$ and $\partial U_{i_2}^{(k_2)} (t_{k_2 /- 1} )$ are disjoint, we find moreover that either $\partial U_{i_2}^{(k_2)} (t_{k_2 /- 1} )$ is contained in $V_{i_1, j}^{(k_1)} (t_{k_2 /-1})$ or it is disjoint from $V_{i_1, j}^{(k_1)} (t_{k_2 /-1})$.
If $\partial U_{i_2}^{(k_2)} (t_{k_2 /- 1} ) \subset V_{i_1, j}^{(k_1)} (t_{k_2 /-1})$, then using (\ref{eq:partialVsubsetPU}) we conclude that $U_{i_2}^{(k_2)} (t_{k_2 /- 1} ) \cup V_{i_1, j}^{(k_1)} (t_{k_2 /-1}) $ does not have a boundary.
So it is equal to a component of $\MM (t_{k_2 / -1})$.
Since $U_{i_2}^{(k_2)} (t_{k_2 /- 1} )$ and $V_{i_1, j}^{(k_1)} (t_{k_2 /- 1})$ are both diffeomorphic to solid tori, this would again contradict \cite[Lemma \ref{Lem:coverMbysth}]{Bamler-LT-topology} and assumption (iii).
It follows that $\partial U_{i_2}^{(k_2)} (t_{k_2 /- 1} )$ is disjoint from $V_{i_1, j}^{(k_1)} (t_{k_2 /-1})$.
Since $V_{i_1, j}^{(k_1)} (t_{k_2 /-1})$ is connected, it follows from (\ref{eq:partialVsubsetPU}) that $V_{i_1, j}^{(k_1)} (t_{k_2 /-1}) \subset U_{i_2}^{(k_2)} (t_{k_2 /- 1} )$.
This implies that $V_{i_1, j+2}^{(k_1)} (t_{k_2 /-1}) \subset V_{i_1, j}^{(k_1)} (t_{k_2 /-1}) \subset U_{i_2}^{(k_2)} (t_{k_2 /- 1} )$, which is what we wanted to prove.
This finishes the proof of the second part of the claim.

For the first part of the claim note that we may always choose $\varphi^*_1$ such that $\varphi_1^*(\eta) < \eta_1^*$ for all $\eta > 0$.
So we only need to consider the case $h^{(k_1)}_{i_1} < \eta_1^*$.
Moreover, we can assume that $h^{(k_2)}_{i_2} > \eta^\circ$, because otherwise the statement is trivial.
Observe now that $U_{i_2}^{(k_2)} (t_{k_2})$ intersects both $P_{i_1, j}^{(k_1)}$ and $V_{i_1, j+2}^{(k_1)} (t_{k_2})$ which are separated by $P_{i_1, j+1}^{(k_1)}$, whose boundary components have time-$t_{k_2}$ distance of at least $(\varphi'(h_{i_1}^{(k_1)}))^{-1} \sqrt{t_{k_2}}$ from one another.
So, using Proposition \ref{Prop:firstcurvboundstep3}(d), we find
\[ \big( \varphi'(h_{i_1}^{(k_1)}) \big)^{-1} \sqrt{ t_{k_2}} \leq \diam_{t_{k_2}} U_{i_2}^{(k_2)} (t_{k_2})
 < D(h_{i_2}^{(k_2)}) \sqrt{t_{k_2}}. \]
As $\lim_{h \to 0} \varphi' (h) = 0$, this bound allows us to choose an appropriate $\varphi^*_1$.
\end{proof}

Consider now the index set $I = \{ (k, i) \; : \; 0 \leq k \leq N, 1 \leq i \leq m^{(k)} \}$.
We will write $(k_1, i_1) \prec (k_2, i_2)$ whenever we are in the situation of Claim 1, i.e. if there is a $j  \leq N + 1$ such case (i) or (ii) of this claim holds.

In the next claim, we will analyze the situation in which we have a chain $(k_1, i_1), (k_2, i_2), \ldots \in I$ of pairs of indices such that each two consecutive pairs satisfy the assumptions of Claim~1.
Under certain additional assumptions, we will then conclude that only case (ii) of Claim~1 occurs, which implies that $k_1 > k_2 > \ldots$.
Therefore, the chain under consideration cannot be circular.

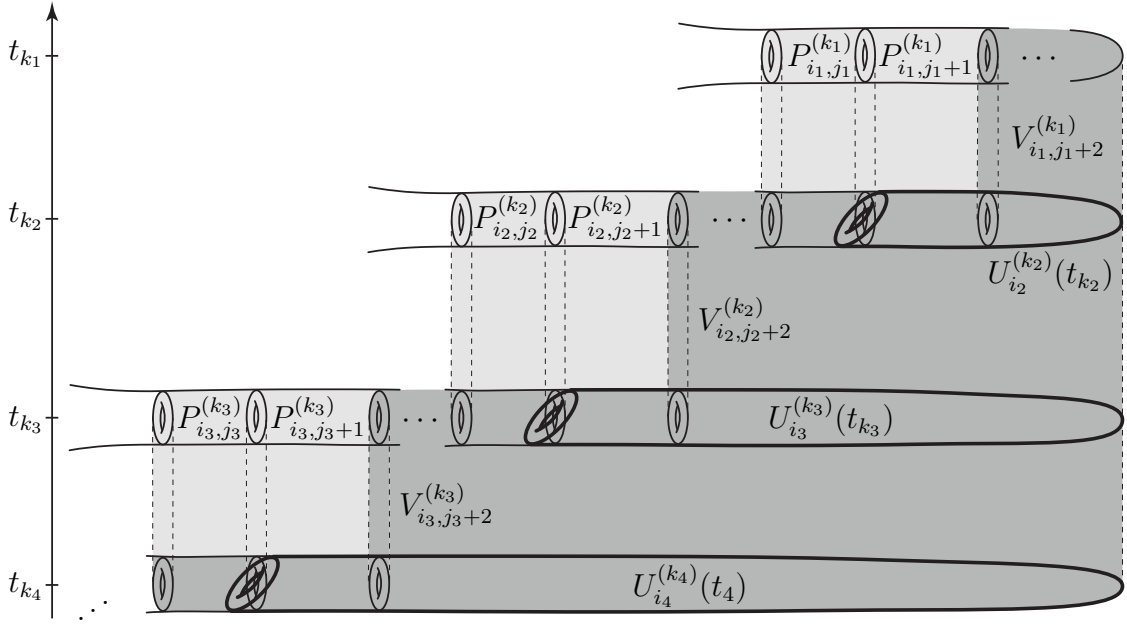
\begin{figure}[t] 
\begin{center}
\setlength{\unitlength}{2863sp}%
\begingroup\makeatletter\ifx\SetFigFont\undefined%
\gdef\SetFigFont#1#2#3#4#5{%
  \reset@font\fontsize{#1}{#2pt}%
  \fontfamily{#3}\fontseries{#4}\fontshape{#5}%
  \selectfont}%
\fi\endgroup%
\begin{picture}(4500,5300)(3500,0)
\hspace{18mm}\includegraphics[width=14.3cm]{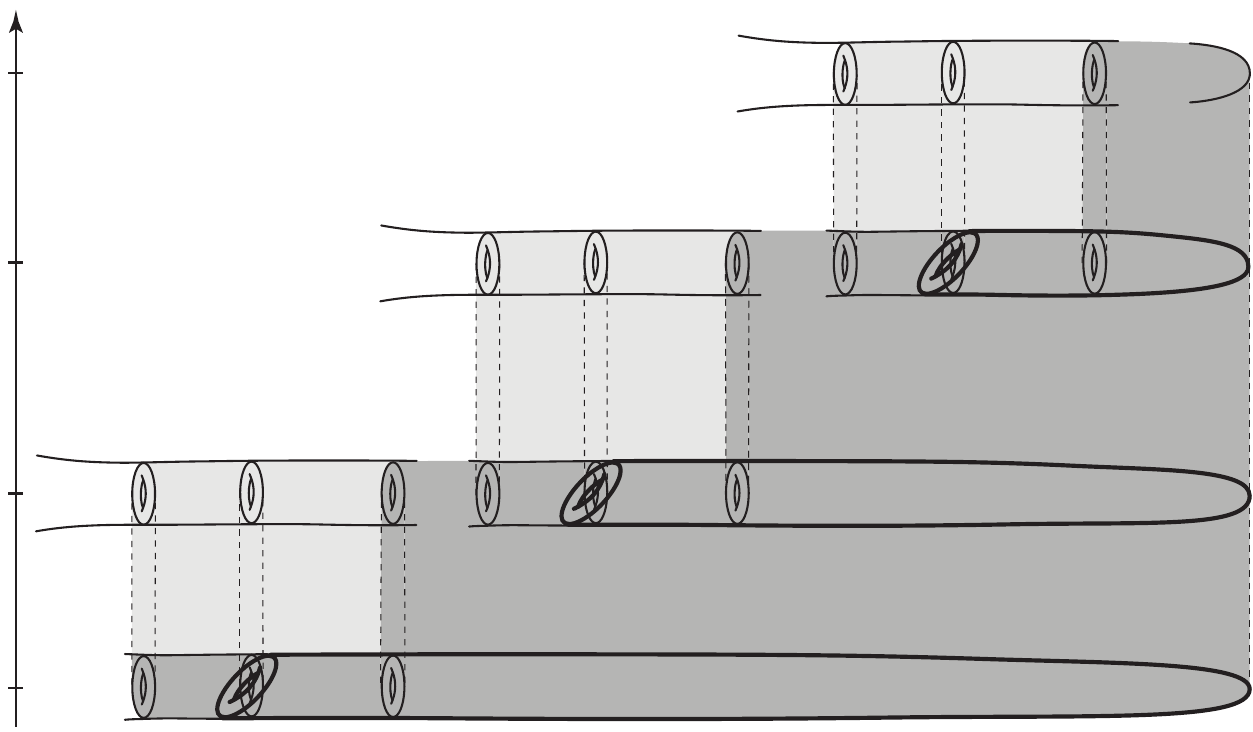}%
\put(-9570,4860){\makebox(0,0)[lb]{\smash{{\SetFigFont{12}{14.4}{\familydefault}{\mddefault}{\updefault}$t_{k_1}$}}}}
\put(-9570,3460){\makebox(0,0)[lb]{\smash{{\SetFigFont{12}{14.4}{\familydefault}{\mddefault}{\updefault}$t_{k_2}$}}}}
\put(-9570,1700){\makebox(0,0)[lb]{\smash{{\SetFigFont{12}{14.4}{\familydefault}{\mddefault}{\updefault}$t_{k_3}$}}}}
\put(-9570,260){\makebox(0,0)[lb]{\smash{{\SetFigFont{12}{14.4}{\familydefault}{\mddefault}{\updefault}$t_{k_4}$}}}}
\put(-980,4100){\makebox(0,0)[lb]{\smash{{\SetFigFont{12}{14.4}{\familydefault}{\mddefault}{\updefault}$V^{(k_1)}_{i_1,j_1+2}$}}}}
\put(-3660,2530){\makebox(0,0)[lb]{\smash{{\SetFigFont{12}{14.4}{\familydefault}{\mddefault}{\updefault}$V^{(k_2)}_{i_2,j_2+2}$}}}}
\put(-6220,930){\makebox(0,0)[lb]{\smash{{\SetFigFont{12}{14.4}{\familydefault}{\mddefault}{\updefault}$V^{(k_3)}_{i_3,j_3+2}$}}}}
\put(-1180,2930){\makebox(0,0)[lb]{\smash{{\SetFigFont{12}{14.4}{\familydefault}{\mddefault}{\updefault}$U^{(k_2)}_{i_2}(t_{k_2})$}}}}
\put(-3060,1680){\makebox(0,0)[lb]{\smash{{\SetFigFont{12}{14.4}{\familydefault}{\mddefault}{\updefault}$U^{(k_3)}_{i_3} (t_{k_3})$}}}}
\put(-4220,230){\makebox(0,0)[lb]{\smash{{\SetFigFont{12}{14.4}{\familydefault}{\mddefault}{\updefault}$U^{(k_4)}_{i_4}(t_4)$}}}}
\put(-2900,4810){\makebox(0,0)[lb]{\smash{{\SetFigFont{12}{14.4}
{\familydefault}{\mddefault}{\updefault}$P^{(k_1)}_{i_1,j_1}$}}}}
\put(-2120,4810){\makebox(0,0)[lb]{\smash{{\SetFigFont{12}{14.4}{\familydefault}{\mddefault}{\updefault}$P^{(k_1)}_{i_1,j_1+1}$}}}}
\put(-900,4810){\makebox(0,0)[lb]{\smash{{\SetFigFont{12}{14.4}{\familydefault}{\mddefault}{\updefault}$\cdots$}}}}
\put(-5590,3410){\makebox(0,0)[lb]{\smash{{\SetFigFont{12}{14.4}{\familydefault}{\mddefault}{\updefault}$P^{(k_2)}_{i_2,j_2}$}}}}
\put(-4780,3410){\makebox(0,0)[lb]{\smash{{\SetFigFont{12}{14.4}{\familydefault}{\mddefault}{\updefault}$P^{(k_2)}_{i_2,j_2+1}$}}}}
\put(-3550,3410){\makebox(0,0)[lb]{\smash{{\SetFigFont{12}{14.4}{\familydefault}{\mddefault}{\updefault}$\cdots$}}}}
\put(-9000,0){\makebox(0,0)[lb]{\smash{{\SetFigFont{12}{14.4}{\familydefault}{\mddefault}{\updefault}\scalebox{-1}[1]{$\ddots$}}}}}
\put(-8130,1680){\makebox(0,0)[lb]{\smash{{\SetFigFont{12}{14.4}{\familydefault}{\mddefault}{\updefault}$P^{(k_3)}_{i_3,j_3}$}}}}
\put(-7330,1680){\makebox(0,0)[lb]{\smash{{\SetFigFont{12}{14.4}{\familydefault}{\mddefault}{\updefault}$P^{(k_3)}_{i_3,j_3+1}$}}}}
\put(-6200,1680){\makebox(0,0)[lb]{\smash{{\SetFigFont{12}{14.4}{\familydefault}{\mddefault}{\updefault}$\cdots$}}}}
\end{picture}%
\caption[caption]{Illustration of the assertion of Claim~2.
The domains $P^{(k_1)}_{i_1, j_1}, P^{(k_2)}_{i_2, j_2}, \ldots$ and $P^{(k_2)}_{i_2, j_2}, P^{(k_2)}_{i_2, j_2+1}, \ldots$, are shaded in light gray wherever they are non-singular and don't intersect any $U^{(k')}_{i'}$.
The sub-Ricci flows with surgery $V^{(k_1)}_{i_1}, V^{(k_2)}_{i_2}, \ldots$ are shaded in dark gray.
The solid tori $U^{(k_1)}_{i_1} (t_{k_1}), U^{(k_2)}_{i_2} (t_{k_2}), \ldots$ are outlined in bold.
\label{fig:Claim2-good-case}}
\end{center}
\end{figure}
\begin{Claim2}
There are a constant $\eta_2^* = \eta_2^*(L,A)  > 0$ and a monotonically non-decreasing function $\varphi_2^* = \varphi_{2, L, A}^* : (0, \infty) \to (0, \infty)$, which both depend on $L$ and $A$, such that if $\eta^\circ < \eta_2^*$, then the following holds: \\ Whenever we have a chain
\[ (k_1, i_1) \prec (k_2, i_2) \prec \ldots \prec (k_m, i_m) \]
such that $0 \leq k_2, \ldots, k_m \leq k_1$ and $h_{i_1}^{(k_1)} < \eta_2^*$, then $m \leq N + 1$ and $k_1 > k_2 > \ldots > k_m$.
Moreover, there are indices $j_1, \ldots, j_{m-1} \in \{ 1, \ldots, N+1 \}$ such that the sub-Ricci flow with surgery $V_{i_1, j_1+2}^{(k_1)}$ can be extended to the time-interval $[t_{k_2}, t_{k_1}]$, $V_{i_2, j_2+2}^{(k_2)}$ can be extended to the time-interval $[t_{k_3}, t_{k_2}]$, \ldots, $V_{i_{m-1}, j_{m-1} + 2}^{(k_{m-1})}$ can be extended to the time-interval $[t_{k_m}, t_{k_{m-1}}]$ and such that (see Figure~\ref{fig:Claim2-good-case} for an illustration)
\begin{multline*}
 V_{i_1, j_1+2}^{(k_1)} (t_{k_2}) \subsetneq U_{i_2}^{(k_2)} (t_{k_2}), \qquad V_{i_2, j_2+2}^{(k_2)} (t_{k_3}) \subsetneq U_{i_3}^{(k_3)} (t_{k_3}),  \qquad \ldots, \\
V_{i_{m-1}, j_{m-1}+2}^{(k_{m-1})} (t_{k_m}) \subsetneq U_{i_m}^{(k_m)} (t_{k_m})
\end{multline*}
Lastly, $\varphi_2^* (h_{i_m}^{(k_m)}) \leq \max \{ h_{i_1}^{(k_1)}, \eta^\circ \}$.
\end{Claim2}

The idea of the proof of Claim~2 will be to apply Claim~1 for each consecutive pair $(k_l, i_l) \prec (k_{l+1}, i_{l+1})$.
The strict containment relationship asserted in Claim~1, together with the fact that the solid tori $U^{(k')}_{i'} (t_{k'})$ are pairwise disjoint for each $k'$, will imply that the sequence $k_1, k_2, \ldots$ cannot reverse its monotonicity.
Since $k_1 > k_2$, this means that the sequence must remain decreasing.
Note that Claim~1 can only be applied if the corresponding preciseness parameter $h^{(k_l)}_{i_l}$ is small enough.
This bound will follow from iteratively applying the asserted inequality in Claim~1 under an inductive assumption.
The inequality at the end of Claim~2 then follows similarly.

\begin{proof}
Set $\eta_2^* = {\varphi_1^*}^{(N)}( \eta_1^* )$ and $\varphi_2^* (h) = {\varphi_1^*}^{(N)} (h)$, where $\eta_1^*, \varphi_1^*$ are taken from Claim 1 and the upper index in parentheses indicates multiple application.

Without loss of generality, we can assume that $m \leq N + 2$, because otherwise we can shorten the given chain to size $N + 2$.
We will first show the claim without the last line, by induction on $m$.
Additionally, we will show that
\begin{equation} \label{eq:extraindasspt}
 h_{i_m}^{(k_m)} \leq  {\varphi_1^*}^{(N - m +1 )} (\eta_1^*),
\end{equation}
if $m \leq N + 1$.
For $m=1$ there is nothing to show.
Assume that the induction hypothesis holds for $m-1$, i.e. that we can extend the sub-Ricci flows with surgery $V_{i_1, j_1+2}^{(k_1)}, \ldots, V_{i_{m-2}, j_{m-2}+2}^{(k_1)}$ to the appropriate time-intervals so that they satisfy the inclusion property above, that $k_1 > k_2 > \ldots > k_{m-1}$ and that $h_{i_m}^{(k_{m-1})}$ satisfies inequality (\ref{eq:extraindasspt}) above with $m$ replaced by $m-1$.

So $h_{i_{m-1}}^{(k_{m-1})} \leq {\varphi^*_1}^{(N - m +2 )} (\eta_1^*) \leq \eta_1^*$ and we conclude by Claim 1 that there is a $j_{m-1}$ such that we can extend the sub-Ricci flow with surgery $V^{(k_{m-1})}_{i_{m-1}, j_{m-1} +2}$ to the time-interval $[t_{k_m}, t_{k_{m-1}}]$ or $[t_{k_{m-1} - 1}, t_{k_m - 1}]$ depending on whether $k_m < k_{m-1}$ or $k_m > k_{m-1}$ and we have
\begin{equation} \label{eq:VkminUkm}
V^{(k_{m-1})}_{i_{m-1}, j_{m-1} + 2} (t_{k_m /- 1}) \subsetneq U^{(k_m)}_{i_m}( t_{k_m /-1}).
\end{equation}
(Here we again use the notation $k_{m/-1} = k_m$ if $k_m < k_{m-1}$ and $k_{m/-1} = k_{m-1}$ if $k_{m-1} > k_{m}$.)

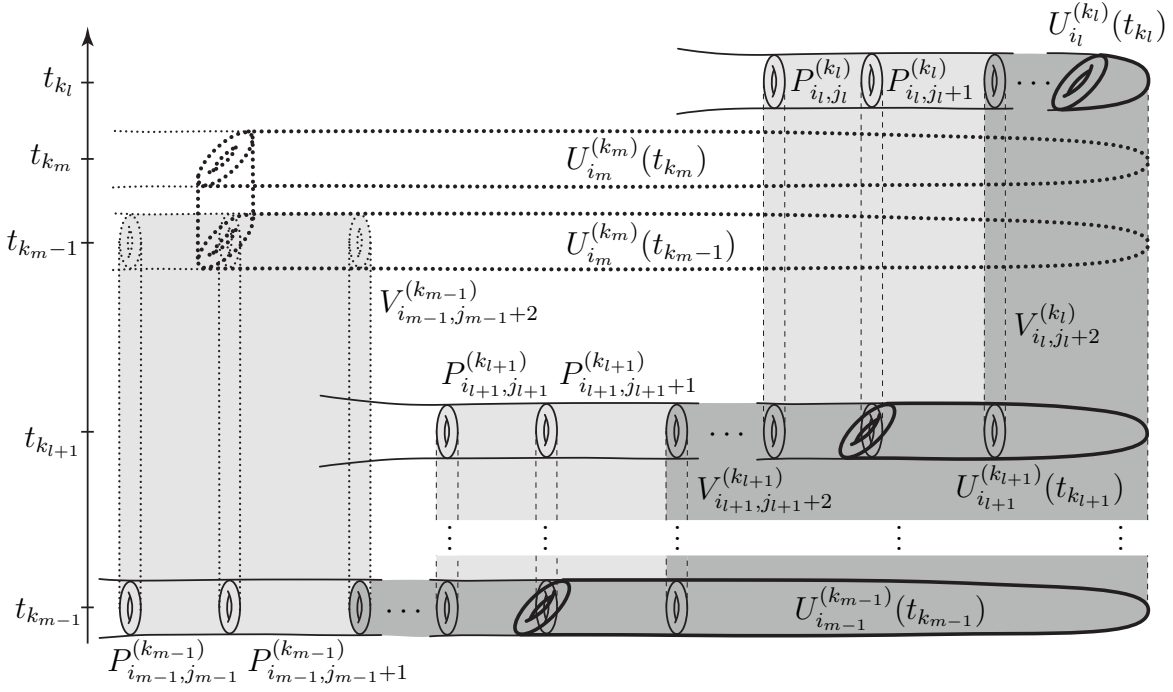
\begin{figure}[t] 
\begin{center}
\setlength{\unitlength}{2863sp}%
\begingroup\makeatletter\ifx\SetFigFont\undefined%
\gdef\SetFigFont#1#2#3#4#5{%
  \reset@font\fontsize{#1}{#2pt}%
  \fontfamily{#3}\fontseries{#4}\fontshape{#5}%
  \selectfont}%
\fi\endgroup%
\begin{picture}(4400,5900)(3400,-400)
\hspace{18mm}\includegraphics[width=14.3cm]{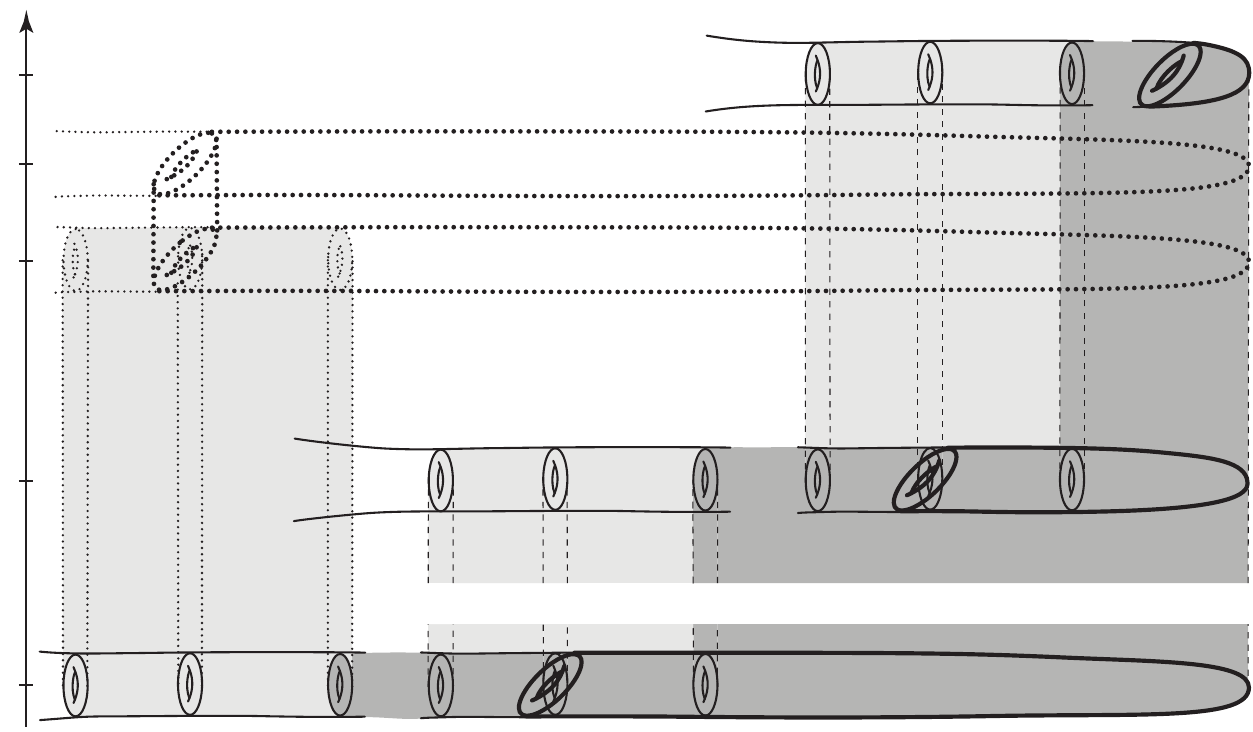}%
\put(-9500,4840){\makebox(0,0)[lb]{\smash{{\SetFigFont{12}{14.4}{\familydefault}{\mddefault}{\updefault}$t_{k_l}$}}}}
\put(-9810,3430){\makebox(0,0)[lb]{\smash{{\SetFigFont{12}{14.4}{\familydefault}{\mddefault}{\updefault}$t_{k_m-1}$}}}}
\put(-9610,4180){\makebox(0,0)[lb]{\smash{{\SetFigFont{12}{14.4}{\familydefault}{\mddefault}{\updefault}$t_{k_m}$}}}}
\put(-9660,1730){\makebox(0,0)[lb]{\smash{{\SetFigFont{12}{14.4}{\familydefault}{\mddefault}{\updefault}$t_{k_{l+1}}$}}}}
\put(-9740,260){\makebox(0,0)[lb]{\smash{{\SetFigFont{12}{14.4}{\familydefault}{\mddefault}{\updefault}$t_{k_{m-1}}$}}}}
\put(-890,5300){\makebox(0,0)[lb]{\smash{{\SetFigFont{12}{14.4}{\familydefault}{\mddefault}{\updefault}$U^{(k_l)}_{i_l} (t_{k_l})$}}}}
\put(-1180,2700){\makebox(0,0)[lb]{\smash{{\SetFigFont{12}{14.4}{\familydefault}{\mddefault}{\updefault}$V^{(k_l)}_{i_l,j_l+2}$}}}}
\put(-3900,1280){\makebox(0,0)[lb]{\smash{{\SetFigFont{12}{14.4}{\familydefault}{\mddefault}{\updefault}$V^{(k_{l+1})}_{i_{l+1},j_{l+1}+2}$}}}}
\put(-6620,2900){\makebox(0,0)[lb]{\smash{{\SetFigFont{12}{14.4}{\familydefault}{\mddefault}{\updefault}$V^{(k_{m-1})}_{i_{m-1},j_{m-1}+2}$}}}}
\put(-4100,840){\makebox(0,0)[lb]{\smash{{\SetFigFont{12}{14.4}{\familydefault}{\mddefault}{\updefault}$\vdots$}}}}
\put(-2200,840){\makebox(0,0)[lb]{\smash{{\SetFigFont{12}{14.4}{\familydefault}{\mddefault}{\updefault}$\vdots$}}}}
\put(-300,840){\makebox(0,0)[lb]{\smash{{\SetFigFont{12}{14.4}{\familydefault}{\mddefault}{\updefault}$\vdots$}}}}
\put(-5230,840){\makebox(0,0)[lb]{\smash{{\SetFigFont{12}{14.4}{\familydefault}{\mddefault}{\updefault}$\vdots$}}}}
\put(-6050,840){\makebox(0,0)[lb]{\smash{{\SetFigFont{12}{14.4}{\familydefault}{\mddefault}{\updefault}$\vdots$}}}}
\put(-1680,1300){\makebox(0,0)[lb]{\smash{{\SetFigFont{12}{14.4}{\familydefault}{\mddefault}{\updefault}$U^{(k_{l+1})}_{i_{l+1}}(t_{k_{l+1}})$}}}}
\put(-3060,250){\makebox(0,0)[lb]{\smash{{\SetFigFont{12}{14.4}{\familydefault}{\mddefault}{\updefault}$U^{(k_{m-1})}_{i_{m-1}} (t_{k_{m-1}})$}}}}
\put(-5020,4130){\makebox(0,0)[lb]{\smash{{\SetFigFont{12}{14.4}{\familydefault}{\mddefault}{\updefault}$U^{(k_m)}_{i_m}(t_{k_m})$}}}}
\put(-5020,3400){\makebox(0,0)[lb]{\smash{{\SetFigFont{12}{14.4}{\familydefault}{\mddefault}{\updefault}$U^{(k_m)}_{i_m}(t_{{k_m}-1})$}}}}
\put(-3100,4810){\makebox(0,0)[lb]{\smash{{\SetFigFont{12}{14.4}
{\familydefault}{\mddefault}{\updefault}$P^{(k_l)}_{i_l,j_l}$}}}}
\put(-2280,4810){\makebox(0,0)[lb]{\smash{{\SetFigFont{12}{14.4}{\familydefault}{\mddefault}{\updefault}$P^{(k_l)}_{i_l,j_l+1}$}}}}
\put(-1170,4810){\makebox(0,0)[lb]{\smash{{\SetFigFont{12}{14.4}{\familydefault}{\mddefault}{\updefault}$\cdots$}}}}
\put(-6090,2280){\makebox(0,0)[lb]{\smash{{\SetFigFont{12}{14.4}{\familydefault}{\mddefault}{\updefault}$P^{(k_{l+1})}_{i_{l+1},j_{l+1}}$}}}}
\put(-5080,2280){\makebox(0,0)[lb]{\smash{{\SetFigFont{12}{14.4}{\familydefault}{\mddefault}{\updefault}$P^{(k_{l+1})}_{i_{l+1},j_{l+1}+1}$}}}}
\put(-3800,1780){\makebox(0,0)[lb]{\smash{{\SetFigFont{12}{14.4}{\familydefault}{\mddefault}{\updefault}$\cdots$}}}}
\put(-8950,-200){\makebox(0,0)[lb]{\smash{{\SetFigFont{12}{14.4}{\familydefault}{\mddefault}{\updefault}$P^{(k_{m-1})}_{i_{m-1},j_{m-1}}$}}}}
\put(-7750,-200){\makebox(0,0)[lb]{\smash{{\SetFigFont{12}{14.4}{\familydefault}{\mddefault}{\updefault}$P^{(k_{m-1})}_{i_{m-1},j_{m-1}+1}$}}}}
\put(-6570,250){\makebox(0,0)[lb]{\smash{{\SetFigFont{12}{14.4}{\familydefault}{\mddefault}{\updefault}$\cdots$}}}}
\end{picture}%
\caption[caption]{Illustration of the case in which $k_m > k_{m-1}$.
The subsets $P^{(k_{l'})}_{i_{l'}, j_{l'}}$ and $P^{(k_{l'})}_{i_{l'}, j_{l'}+1}$ are shaded in light gray wherever the flow is non-singular and disjoint from the $U^{(k')}_{i'}$.
The sub-Ricci flows with surgery $V^{(k_{l'})}_{i_{l'}, j_{l'}+2}$ are shaded in dark gray, except for $V^{(k_{m-1})}_{i_{m-1}, j_{m-1}+2}$, in order to avoid overlap. 
The solid tori $U^{(k_{l'})} (t_{k_{l'}})$ are outlined in bold.
The objects involving $k_m$ are dotted, since their location in space-time is being contradicted.
Note that there may be several time-steps between $t_{k_{l+1}}$ and $t_{k_{m-1}}$.
This is indicated by the white omission and the vertical dots.
The light/dark gray regions below this omission may not necessarily correspond to the light/dark gray above the omission.
\label{fig:Claim2-bad-case}}
\end{center}
\end{figure}
We now show in the next two paragraphs that we must have $k_m < k_{m-1}$:
Assume by contradiction that $k_m > k_{m-1}$ (see Figure~\ref{fig:Claim2-bad-case} for an illustration).
Then $V^{(k_{m-1})}_{i_{m-1}, j_{m-1} +2}$ is defined on $[t_{k_{m-1} -1}, t_{k_m -1}]$ and we have $m \geq 3$ and there is an $l \in \{ 1, \ldots, m-2 \}$ such that $k_{l+1} + 1\leq k_m \leq k_l$.
We first show that
\begin{equation} \label{eq:Ulp1inVkm}
V^{(k_{l+1})}_{i_{l+1}, j_{l+1}} (t_{k_{l+1}}) \subset V^{(k_{m-1})}_{i_{m-1}, j_{m-1} + 2} (t_{k_{l+1}}).
\end{equation}
Observe here that the sub-Ricci flow with surgery $V^{(k_{m-1})}_{i_{m-1}, j_{m-1} + 2}$ is defined at time $t_{k_{l+1}}$ since $k_{l+1} < k_m$.
Choose $l+1 \leq l^* \leq m-1$ minimal with the property that $V^{(k_{l^*})}_{i_{l^*}, j_{l^*} +2} (t_{k_{l^*}})\subset V^{(k_{m-1})}_{i_{m-1}, j_{m-1} + 2} (t_{k_{l^*}})$.
This is possible, since the inclusion trivially holds for $l^* = m-1$.
Then by the induction assumption
\begin{equation} \label{eq:Vklsm1Vkls}
 V^{(k_{l^*-1})}_{i_{l^*-1}, j_{l^*-1} + 2} (t_{k_{l^*}}) \subset U^{(k_{l^*})}_{i_{l^*}} (t_{k_{l^*}})\subset V^{(k_{l^*})}_{i_{l^*}, j_{l^*} + 2} (t_{k_{l^*}}) \subset V^{(k_{m-1})}_{i_{m-1}, j_{m-1} + 2} (t_{k_{l^*}}). 
\end{equation}
Now, if $l^* > l+1$, then $k_{l^*-1} \leq k_{l+1} \leq k_{m}-1$, so both sub-Ricci flows with surgery, $V^{(k_{l^*-1})}_{i_{l^*-1}, j_{l^*-1} + 2}$ and $V^{(k_{m-1})}_{i_{m-1}, j_{m-1} + 2}$, are defined on the time-interval $[t_{k_{l^*}}, t_{k_{l^*-1}}]$.
Since at every surgery time of the time-interval $[t_{k_{l^*}}, t_{k_{l^*-1}}]$ the topology of the time-slice does not change, we conclude that the containment relationship of the left and right-hand side in (\ref{eq:Vklsm1Vkls}) is preserved on $[t_{k_{l^*}}, t_{k_{l^*-1}}]$.
This implies that if $l^* > l+1$, then
\[  V^{(k_{l^*-1})}_{i_{l^*-1}, j_{l^*-1} + 2} (t_{k_{l^*-1}}) \subset V^{(k_{m-1})}_{i_{m-1}, j_{m-1} + 2} (t_{k_{l^*-1}}), \]
in contradiction to the minimal choice of $l^*$.
So $l^* = l+1$ and (\ref{eq:Ulp1inVkm}) holds.

Similarly as in (\ref{eq:Vklsm1Vkls}), we can use (\ref{eq:Ulp1inVkm}) to conclude that
\begin{equation} \label{eq:VUVV}
 V^{(k_l)}_{i_l, j_l +2} (t_{k_{l+1}} ) \subset U^{(k_{l+1})}_{i_{l+1}} (t_{k_{l+1}}) \subset V^{(k_{l+1})}_{i_{l+1}, j_{l+1} +2} (t_{k_{l+1}}) \subset V^{(k_{m-1})}_{i_{m-1}, j_{m-1} +2} (t_{k_{l+1}}). 
\end{equation}
Again, similarly as argued after (\ref{eq:Vklsm1Vkls}), the containment relationship of the left and right-hand side of (\ref{eq:VUVV}) is preserved on the time-interval $[t_{k_{l+1}}, t_{k_{m-1}}]$.
So
\[ V^{(k_l)}_{i_l, j_l +2} (t_{k_{m-1}} ) \subset V^{(k_{m-1})}_{i_{m-1}, j_{m-1} +2} (t_{k_{m-1}}). \]
Combining this with (\ref{eq:VkminUkm}) yields
\[ V^{(k_l)}_{i_l, j_l +2} (t_{k_{m-1}} ) \subset V^{(k_{m-1})}_{i_{m-1}, j_{m-1} +2} (t_{k_{m-1}}) \subsetneq U^{(k_m)}_{i_m} (t_{k_{m-1}}). \]
Since both sub-Ricci flows with surgery, $V^{(k_l)}_{i_l, j_l +2}$ and $U^{(k_m)}_{i_m}$, are defined on the time-interval $[t_{k_{m-1}}, t_{k_m}]$, this containment relationship is preserved up to time $t_{k_m}$:
\begin{equation} \label{eq:Vkliljl2Ukmim}
 V^{(k_l)}_{i_l, j_l +2} (t_{k_{m}} ) \subsetneq U^{(k_m)}_{i_m} (t_{k_{m}}). 
\end{equation}
We will now consider the two cases $k_m = k_l$ and $k_m \leq k_l -1$ separately (recall that $k_{l+1} + 1 \leq k_m \leq k_l$).
In the case $k_m = k_l$, we obtain
\[ U^{(k_m)}_{i_l} (t_{k_{m}} ) \subset V^{(k_m)}_{i_l, j_l +2} (t_{k_{m}} ) \subsetneq U^{(k_m)}_{i_m} (t_{k_{m}}), \]
which is impossible as the solid tori $U^{(k_m)}_{1} (t_{k_{m}} ), \ldots, U^{(k_m)}_{m^{(k_m)}} (t_{k_{m}} ) \subset \MM( t_m )$ are pairwise disjoint.
It remains to consider the case in which $k_m \leq k_l - 1$.
In this case observe that (\ref{eq:Vkliljl2Ukmim}) implies that
\[ \partial V^{(k_l)}_{i_l, j_l +2} (t_{k_{m}} ) \subset V^{(k_l)}_{i_l, j_l +2} (t_{k_{m}} ) \subsetneq U^{(k_m)}_{i_m} (t_{k_{m}}). \]
As $\partial V^{(k_l)}_{i_l, j_l +2} (t_{k_{m}} )$ is contained in the non-singular space-time region $( P^{(k_l)}_{i_l, j_l} \cup P^{(k_l)}_{i_l, j_l+1} ) \times [t_{k_{l+1}}, t_{k_l}]$, this implies that $U^{(k_m)}_{i_m} (t_{k_{m}})$ intersects this region.
But since $k_m$ lies strictly between $k_{l+1}$ and $k_l$, this contradicts the definition of the relation $(k_l, i_l) \prec (k_{l+1}, i_{l+1})$.
So, indeed we have $k_m < k_{m-1}$.

To finish the induction, we can now apply case (ii) of Claim~1 to conclude from (\ref{eq:extraindasspt}) for $m-1$ that
\begin{multline*}
 \varphi_1^* ( h_{i_m}^{(k_m)} ) \leq \max \big\{ h_{i_{m-1}}^{(k_{m-1})} , \varphi_1^*(\eta^\circ) \big\} \\ \leq \max \big\{ {\varphi_1^*}^{(N - m + 2 )} (\eta_1^*), \varphi_1^*(\eta_2^*) \big\} 
 \leq  {\varphi_1^*}^{(N - m + 2 )} (\eta_1^*).
\end{multline*}
This implies (\ref{eq:extraindasspt}) for $m$, by the monotonicity of $\varphi_1^*$ and finishes the induction.
In particular, we obtain that $k_1 > k_2 > \ldots > k_m$, which implies that $m \leq N+1$.

For the last line in the claim, we can use Claim 1 to conclude
\begin{alignat*}{1}
\varphi_2^*( h^{(k_m)}_{i_m}) &= {\varphi_1^*}^{(N)}( h^{(k_m)}_{i_m} ) \leq {\varphi_1^*}^{(N-1)} \big( \max \{ h^{(k_{m-1})}_{i_{m-1}} , \varphi_1^*(\eta^\circ) \} \big) \\
&\leq \max \big\{ {\varphi_1^*}^{(N-1)} (  h^{(k_{m-1})}_{i_{m-1}} ), \eta^\circ \big\}
\\
& \qquad \leq \max \big\{ {\varphi_1^*}^{(N-2)} \big(  \max \{ h^{(k_{m-2})}_{i_{m-2}}, \varphi_1^*(\eta^\circ) \} \big), \eta^\circ \big\} \\
 &\leq \max \big\{ {\varphi_1^*}^{(N-2)} ( h^{(k_{m-2})}_{i_{m-2}} ), \eta^\circ \big\} 
 \\ &\leq \ldots \\ 
 &\leq \max \big\{ {\varphi_1^*}^{(N-m+1)} ( h^{(k_1)}_{i_1} ), \eta^\circ \big\}
 \leq  \max \big\{ h^{(k_1)}_{i_1} , \eta^\circ \big\} \\[-2.1\baselineskip]
\end{alignat*}
\end{proof}

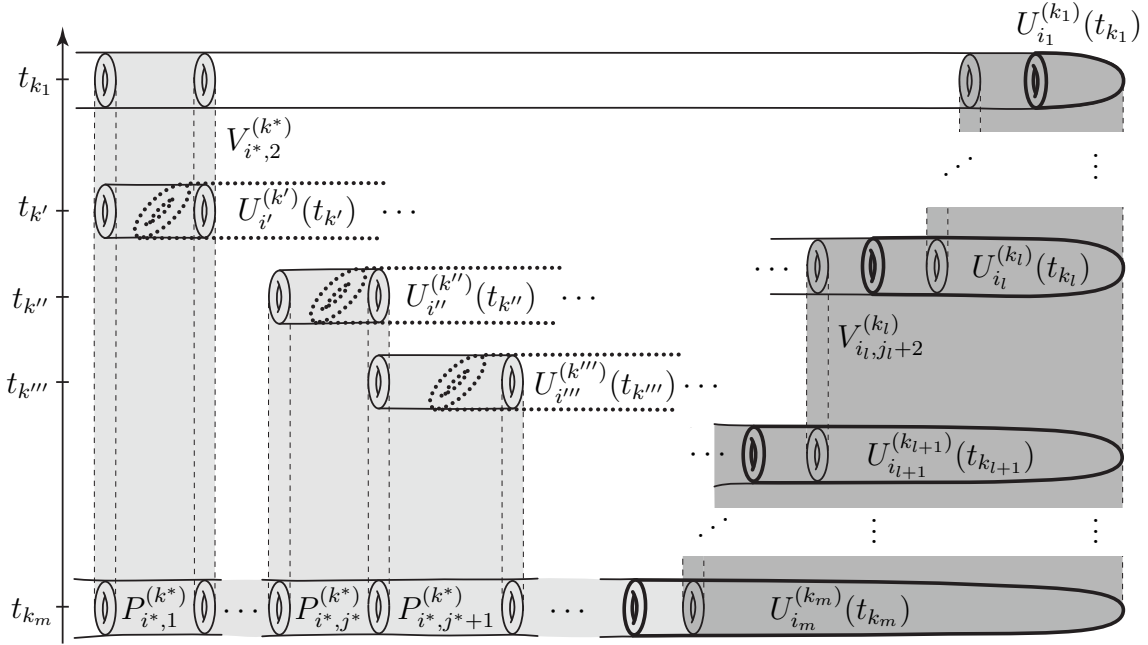
\begin{figure}[t] 
\begin{center}
\setlength{\unitlength}{2863sp}%
\begingroup\makeatletter\ifx\SetFigFont\undefined%
\gdef\SetFigFont#1#2#3#4#5{%
  \reset@font\fontsize{#1}{#2pt}%
  \fontfamily{#3}\fontseries{#4}\fontshape{#5}%
  \selectfont}%
\fi\endgroup%
\begin{picture}(4400,5900)(3400,-400)
\hspace{18mm}\includegraphics[width=14.3cm]{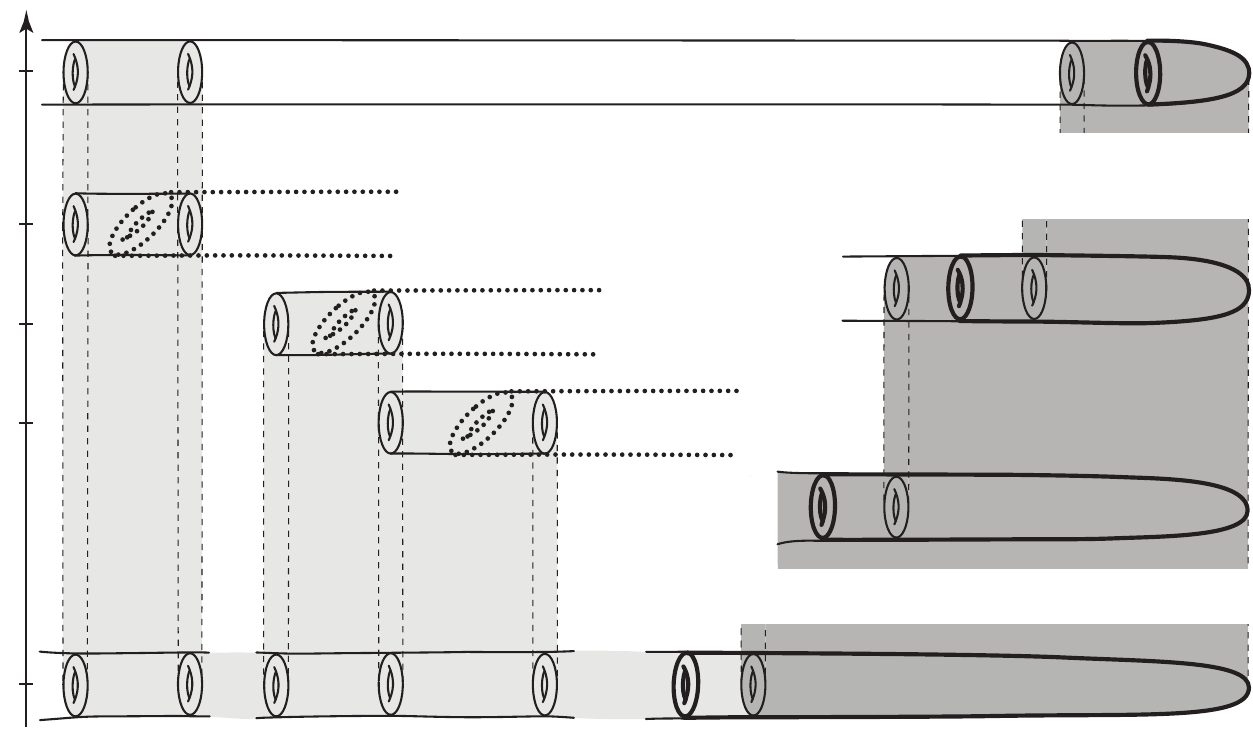}%
\put(-9500,4840){\makebox(0,0)[lb]{\smash{{\SetFigFont{12}{14.4}{\familydefault}{\mddefault}{\updefault}$t_{k_1}$}}}}
\put(-9500,3730){\makebox(0,0)[lb]{\smash{{\SetFigFont{12}{14.4}{\familydefault}{\mddefault}{\updefault}$t_{k'}$}}}}
\put(-9550,2930){\makebox(0,0)[lb]{\smash{{\SetFigFont{12}{14.4}{\familydefault}{\mddefault}{\updefault}$t_{k''}$}}}}
\put(-9600,2200){\makebox(0,0)[lb]{\smash{{\SetFigFont{12}{14.4}{\familydefault}{\mddefault}{\updefault}$t_{k'''}$}}}}
\put(-9560,260){\makebox(0,0)[lb]{\smash{{\SetFigFont{12}{14.4}{\familydefault}{\mddefault}{\updefault}$t_{k_{m}}$}}}}
\put(-950,5300){\makebox(0,0)[lb]{\smash{{\SetFigFont{12}{14.4}{\familydefault}{\mddefault}{\updefault}$U^{(k_1)}_{i_1} (t_{k_1})$}}}}
\put(-2480,2600){\makebox(0,0)[lb]{\smash{{\SetFigFont{12}{14.4}{\familydefault}{\mddefault}{\updefault}$V^{(k_l)}_{i_l,j_l+2}$}}}}
\put(-3700,930){\makebox(0,0)[lb]{\smash{{\SetFigFont{12}{14.4}{\familydefault}{\mddefault}{\updefault}\scalebox{-1}[1]{$\ddots$}}}}}
\put(-2200,900){\makebox(0,0)[lb]{\smash{{\SetFigFont{12}{14.4}{\familydefault}{\mddefault}{\updefault}$\vdots$}}}}
\put(-300,900){\makebox(0,0)[lb]{\smash{{\SetFigFont{12}{14.4}{\familydefault}{\mddefault}{\updefault}$\vdots$}}}}
\put(-300,4050){\makebox(0,0)[lb]{\smash{{\SetFigFont{12}{14.4}{\familydefault}{\mddefault}{\updefault}$\vdots$}}}}
\put(-1600,4050){\makebox(0,0)[lb]{\smash{{\SetFigFont{12}{14.4}{\familydefault}{\mddefault}{\updefault}\scalebox{-1}[1]{$\ddots$}}}}}
%\put(-5230,930){\makebox(0,0)[lb]{\smash{{\SetFigFont{12}{14.4}{\familydefault}{\mddefault}{\updefault}$\vdots$}}}}
%\put(-6050,9300){\makebox(0,0)[lb]{\smash{{\SetFigFont{12}{14.4}{\familydefault}{\mddefault}{\updefault}$\vdots$}}}}
%
\put(-2230,1580){\makebox(0,0)[lb]{\smash{{\SetFigFont{12}{14.4}{\familydefault}{\mddefault}{\updefault}$U^{(k_{l+1})}_{i_{l+1}}(t_{k_{l+1}})$}}}}
\put(-1330,3200){\makebox(0,0)[lb]{\smash{{\SetFigFont{12}{14.4}{\familydefault}{\mddefault}{\updefault}$U^{(k_{l})}_{i_{l}}(t_{k_{l}})$}}}}
\put(-3060,250){\makebox(0,0)[lb]{\smash{{\SetFigFont{12}{14.4}{\familydefault}{\mddefault}{\updefault}$U^{(k_{m})}_{i_{m}} (t_{k_{m}})$}}}}
\put(-7720,4300){\makebox(0,0)[lb]{\smash{{\SetFigFont{12}{14.4}{\familydefault}{\mddefault}{\updefault}$V^{(k^*)}_{i^*,2}$}}}}
\put(-7620,3690){\makebox(0,0)[lb]{\smash{{\SetFigFont{12}{14.4}{\familydefault}{\mddefault}{\updefault}$U^{(k')}_{i'}(t_{k'}) \;\; \cdots$}}}}
\put(-6180,2950){\makebox(0,0)[lb]{\smash{{\SetFigFont{12}{14.4}{\familydefault}{\mddefault}{\updefault}$U^{(k'')}_{i''}(t_{{k''}}) \;\; \cdots$}}}}
\put(-5080,2190){\makebox(0,0)[lb]{\smash{{\SetFigFont{12}{14.4}{\familydefault}{\mddefault}{\updefault}$U^{(k''')}_{i'''}(t_{{k'''}})  \cdots$}}}}
%
%\put(-3100,4810){\makebox(0,0)[lb]{\smash{{\SetFigFont{12}{14.4}{\familydefault}{\mddefault}{\updefault}$P^{(k_l)}_{i_l,j_l}$}}}}
%\put(-2280,4810){\makebox(0,0)[lb]{\smash{{\SetFigFont{12}{14.4}{\familydefault}{\mddefault}{\updefault}$P^{(k_l)}_{i_l,j_l+1}$}}}}
%
%\put(-6090,2280){\makebox(0,0)[lb]{\smash{{\SetFigFont{12}{14.4}{\familydefault}{\mddefault}{\updefault}$P^{(k_{l+1})}_{i_{l+1},j_{l+1}}$}}}}
%\put(-5080,2280){\makebox(0,0)[lb]{\smash{{\SetFigFont{12}{14.4}{\familydefault}{\mddefault}{\updefault}$P^{(k_{l+1})}_{i_{l+1},j_{l+1}+1}$}}}}
\put(-3750,1600){\makebox(0,0)[lb]{\smash{{\SetFigFont{12}{14.4}{\familydefault}{\mddefault}{\updefault}$\cdots$}}}}
\put(-3200,3200){\makebox(0,0)[lb]{\smash{{\SetFigFont{12}{14.4}{\familydefault}{\mddefault}{\updefault}$\cdots$}}}}
\put(-8620,250){\makebox(0,0)[lb]{\smash{{\SetFigFont{12}{14.4}{\familydefault}{\mddefault}{\updefault}$P^{(k^*)}_{i^*,1}$}}}}
\put(-7130,250){\makebox(0,0)[lb]{\smash{{\SetFigFont{12}{14.4}{\familydefault}{\mddefault}{\updefault}$P^{(k^*)}_{i^*,j^*}$}}}}
\put(-6250,250){\makebox(0,0)[lb]{\smash{{\SetFigFont{12}{14.4}{\familydefault}{\mddefault}{\updefault}$P^{(k^*)}_{i^*,j^*+1}$}}}}
\put(-7750,250){\makebox(0,0)[lb]{\smash{{\SetFigFont{12}{14.4}{\familydefault}{\mddefault}{\updefault}$\cdots$}}}}
\put(-4950,250){\makebox(0,0)[lb]{\smash{{\SetFigFont{12}{14.4}{\familydefault}{\mddefault}{\updefault}$\cdots$}}}}
\end{picture}%
\caption[caption]{Illustration of the proof of Claim~3 for the case $k' > k^* = k_m$.
The regions $P^{(k^*)}_{i^*,1}, \ldots, P^{(k^*)}_{i^*,N+2}$ are shaded in light gray.
The solid tori $U^{(k')}_{i'} (t_{k'})$, $U^{(k'')}_{i''} (t_{k''})$ and $U^{(k''')}_{i'''} (t_{k'''})$ are outlined in bold dots.
The sub-Ricci flows with surgery $V^{(k_1)}_{i_1, j_1+2}, \ldots, V^{(k_{m-1})}_{i_{m-1}, j_{m-1} +2}$ are shaded in dark gray and the solid tori $U^{(k_1)}_{i_1} (t_{k_1}), \ldots, U^{(k_m)}_{i_m} (t_{k_m})$ are outlined in bold.
\label{fig:Claim3PUV}}
\end{center}
\end{figure}

With the help of Claim~2, we can finally derive the existence of regions $P^{(k^*)}_{i^*,1}$ that are non-singular on $[t_{-1}, t_k]$.

\begin{Claim3}
Assume that $\eta^\circ < \eta_2^*$.
If there are indices $(k, i) \in I$ such that $h_i^{(k)} < \eta_2^*$, then there are indices $(k^*, i^*) \in I$ with $k^* \leq k$ such that the following holds:
The set $P^{(k^*)}_{i^*, 1}$ is non-singular on $[t_{-1}, t_k]$ and we can extend $V^{(k^*)}_{i^*, 2}$ to a sub-Ricci flow with surgery on the time-interval $[t_{-1}, t_k]$.
Moreover, $P^{(k^*)}_{i^*, 1} \cap U^{(k')}_{i'} (t_{k'}) = \emptyset$ for all $(k', i') \in I$ with $k' \leq k$ and we have $U^{(k)}_i (t_k) \subsetneq V^{(k^*)}_{i^*, 2} (t_k)$.
Lastly, $\varphi_2^* ( h_{i^*}^{(k^*)} ) \leq \max \{ h_i^{(k)}, \eta^\circ \}$.
\end{Claim3}

The idea of the proof is to look at a maximal chain $(k_1, i_1) = (k,i) \prec (k_2, i_2) \prec \ldots \prec (k_m, i_m)$, starting with $(k,i)$, and set $(k^*, i^*) = (k_m, i_m)$.
Then $(k^*, i^*) \not\prec (k', i')$ for any $(k', i') \in I$ with $k' \leq k$.
Looking back at the definition of the relation $\prec$, this means roughly that for any $j^* = 1, \ldots, N+1$ the subset $P^{(k^*)}_{i^*, j^*}$ is non-singular on $[t_{-1}, t_k]$ or $P^{(k^*)}_{i^*, j^*+1}$ becomes singular at least one time-step earlier than $P^{(k^*)}_{i^*, j^*}$.
So if $P^{(k^*)}_{i^*, 1}$ became singular at some time, then $P^{(k^*)}_{i^*, 2}$ would become singular at least one time-step earlier than $P^{(k^*)}_{i^*, 1}$, so $P^{(k^*)}_{i^*,3}$ would become singular at least one time step earlier than $P^{(k^*)}_{i^*, 2}$, and so on.
However, there are only at most $N$ time-steps, but $N+1$ subsets $P^{(k^*)}_{i^*, j^*}$ for which we can draw this conclusion.
So this process would have to terminate too early, giving us the desired contradiction.
Note that, for technical reasons, the following proof uses a minimality argument in lieu of an induction argument.
The underlying principle will, however, be the same as laid out in this explanation.

\begin{proof}
Consider a maximal chain as in Claim 2 with $(k_1, i_1) = (k, i)$ and $k_2, \ldots, k_m \linebreak[2] \leq k_1$ and set $(k^*, i^*) = (k_m, i_m)$.
By Claim 2, we have $k_1 > k_2 > \ldots > k_m$ and we obtain indices $j_1, \ldots, j_{m-1} \in \{ 1, \ldots, N+1 \}$ together with extensions of the flows $V^{(k_1)}_{i_1, j_1 +2}, \ldots, V^{(k_{m-1})}_{i_{m-1}, j_{m-1} + 2}$ that satisfy the inclusion property mentioned above.
Moreover, $\varphi^*_2( h^{(k^*)}_{i^*} ) \leq \max \{ h_i^{(k)}, \eta^\circ \}$, which establishes the last part of the claim.

Assume now that $P^{(k^*)}_{i^*, 1}$ is singular on the time-interval $[t_{-1}, t_k]$.
Then we can find some $k' \leq k$ such that $P^{(k^*)}_{i^*, 1}$ is non-singular on $[t_{k^*}, t_{k'-1}]$ (if $k' > k^*$) or on $[t_{k'}, t_{k^*}]$ (if $k' < k^*$) and there is an index $i' \in \{ 1, \ldots, m^{(k')} \}$ such that $P^{(k^*)}_{i^*, 1} \cap U^{(k')}_{i'} (t_{k' - 1}) \neq \emptyset$ (if $k' > k^*$) or $P^{(k^*)}_{i^*, 1} \cap U^{(k')}_{i'} (t_{k'}) \neq \emptyset$ (if $k' < k^*$).
Let us now consider all triples $(k'', i'', j^*)$ of indices with $(k'', i'') \in I$, $k'' \leq k$ and $j^* \in \{ 1, \ldots, N + 2 \}$ for which $P^{(k^*)}_{i^*, j^*}$ is non-singular on $[t_{k^*}, t_{k''-1}]$ (if $k'' > k^*$) or on $[t_{k''}, t_{k^*}]$ (if $k'' < k^*$) and
\[ P_{i^*, j^*}^{(k^*)} \cap U_{i''}^{(k'')} (t_{k''-1})  \not= \emptyset \quad \text{if $k'' > k^*$} \quad \text{or} \quad P_{i^*, j^*}^{(k^*)} \cap U_{i''}^{(k'')} (t_{k''})  \not= \emptyset \quad \text{if $k'' < k^*$}. \]
(Note that we have exchanged the $1$ for $j^*$. So $(k', i', 1)$ is one of these triples.)
We can assume that we have picked $(k'', i'', j^*)$ amongst all these triples of indices, such that $j^* + |k^* - k''|$ is minimal and amongst such triples of indices for which this number is the same, we can assume that $| k^* - k'' |$ is minimal.

Now observe that by maximality of $(k_m, i_m)$ with respect to $\prec$ we must have $(k^*, i^*) \not\prec (k'', i'')$.
So either $j^* = N + 2$ or $j^* \leq N+1$, and there are indices $(k''', i''') \in I$ with $k'''$ strictly between $k^*$ and $k''$ such that the following holds:
The set $P_{i^*, j^* + 1}^{(k^*)}$ is non-singular on the time-interval $[t_{k^*}, t_{k'''-1}]$ (if $k'' > k^*$) or on $[t_{k'''}, t_{k^*}]$ (if $k'' < k^*$) and we have $(P_{i^*, j^*}^{(k^*)} \cup P_{i^*, j^* + 1}^{(k^*)} ) \cap U_{i'''}^{(k''')} (t_{k''' - 1}) \neq \emptyset$ (if $k'' > k^*$) or $(P_{i^*, j^*}^{(k^*)} \cup P_{i^*, j^* + 1}^{(k^*)} ) \cap U_{i'''}^{(k''')} (t_{k'''}) \neq \emptyset$ (if $k'' < k^*$).
But the latter possibility implies that we could replace the triple $(k'', i'', j^*)$ by either $(k''', i''', j^*)$ or $(k''', i''', j^* + 1)$, contradicting its minimal choice.
So $j^* = N + 2$.
However, the triple $(k', i', 1)$ would make $j^* + |k^* - k'|$ smaller than the triple $(k'', i'', N+2)$.
This yields the desired contradiction and shows that $P^{(k^*)}_{i^*, 1}$ is non-singular on the time-interval $[t_{-1}, t_k]$ as well as the fact that $P^{(k^*)}_{i^*, 1} \cap U^{(k')}_{i'} (t_{k'}) = \emptyset$ for all $(k', i') \in I$ with $k' \leq k$.
Moreover, it follows that the sub-Ricci flow with surgery $V^{(k^*)}_{i^*, 2}$ can be extended to the time-interval $[t_{-1}, t_k]$.

We finally show by induction that $U^{(k_l)}_{i_l} (t_{k_l}) \subset V^{(k^*)}_{i^*, 2} (t_{k_l})$ for all $l = m, \ldots, 1$.
This implies the claim for $l = 1$.
The statement is clear for $l = m$, so assume that $l < m$ and that it holds for $l + 1$.
By Claim 2 we have $V^{(k_l)}_{i_l, j_l + 2} (t_{k_{l+1}}) \subsetneq U^{(k_{l+1})}_{i_{l+1}} (t_{k_{l+1}}) \subset V^{(k^*)}_{i^*, 2} (t_{k_{l+1}})$.
The relationship $V^{(k_l)}_{i_l, j_l + 2} (t_{k_{l+1}}) \linebreak[1] \subsetneq \linebreak[1] V^{(k^*)}_{i^*, 2} (t_{k_{l+1}})$ remains preserved up to time $t_{k_l}$.
So $U^{(k_l)}_{i_l} (t_{k_l}) \subset V^{(k_l)}_{i_l, j_l + 2} (t_{k_l}) \subsetneq V^{(k^*)}_{i^*, 2} (t_{k_l})$, finishing the induction.
\end{proof}

We can now finish the proof of the proposition.
Assume in the following that $\eta^\circ < \eta_2^*$.
Apply Claim 3 for $k = N$.
So for every $i \in \{ 1, \ldots, m^{(N)} \}$ for which $h_i^{(N)} < \eta_2^*$, we can find indices $(k^*_i, i^*_i) \in I$ such that the following holds: $P^{(k^*_i)}_{i^*_i, 1}$ is non-singular on the time-interval $[t_{-1}, t_\omega]$ and $P^{(k^*_i)}_{i^*_i, 1} \cap U^{(k')}_{i'} (t_{k'}) = \emptyset$ for all indices $(k', i') \in I$.
Moreover, we can extend $V^{(k^*_i)}_{i^*_i, 2}$ to the time-interval $[t_0, t_\omega]$ and we have $U^{(k)}_i (t_\omega) \subsetneq V^{(k^*_i)}_{i^*_i, 2} (t_\omega)$.
Using (\ref{eq:summarycurvboundgoodpart}), we obtain that
\begin{equation} \label{eq:RmKtinvPis1}
| {\Rm}_t | < K t^{-1} \quad \text{on} \quad P^{(k^*_i)}_{i^*_i, 1} \qquad \text{for all} \qquad t \in [t_{-1}, t_\omega]
\end{equation}
and by Shi's estimates there is a universal $K_1 > K$ such that for all $t \in [t_0, t_\omega]$ we have $ |{\nabla \Rm}| < K_1 t^{-3/2}$ at all points of $P^{(k^*_i)}_{i^*_i, 1}$ that have time-$t$ distance of at least $\sqrt{t}$ away from $\partial P^{(k^*_i)}_{i^*_i, 1}$.
Using the curvature bound (\ref{eq:RmKtinvPis1}) and a distance distortion estimate, we find a function $\varphi'' = \varphi''_{L} : (0, \infty) \to (0, \infty]$ with $\lim_{h \to 0} \varphi''(h) = 0$, which depends only on $L$, such that whenever $\varphi'' ( h^{(k^*_i)}_{i^*_i}) < \infty$ there is a subset $P'_i \subset P^{(k^*_i)}_{i^*_i, 1}$ in such a way that $(P^{(k^*_i)}_{i^*_i, 1}, P'_i)$ is diffeomorphic to $(T^2 \times [-2,2], T^2 \times [-1,1])$ and such that the following holds for any $t \in [t_0, t_\omega]$:
The subset $P'_i$ has time-$t$ distance from $\partial P^{(k^*_i)}_{i^*_i, 1}$ at least $\varphi^{\prime\prime -1} ( h^{(k^*_i)}_{i^*_i} ) \sqrt{t}$ and $P'_i$ is a $\varphi''( h^{(k^*_i)}_{i^*_i} )$-precise torus structure at scale $\sqrt{t}$ and time $t$.
Using Claim 3, we find that there is a constant $\eta_4^* = \eta_4^*(L,A, \alpha) > 0$ with $\eta_4^* < \eta_2^*$ such that, assuming $\eta^\circ < \eta_4^*$, the following holds: 
For all $i \in \{ 1, \ldots, m^{(N)} \}$ with $h_i^{(N)} < \eta_4^*$ and all $t \in [t_0, t_\omega]$, the set $P'_i$ is a $\min \{ \alpha, \frac1{10} \}$-precise torus structure at scale $\sqrt{t}$ and time $t$ and $P'_i$ has time-$t$ distance of at least $10 \sqrt{t}$ from $\partial P^{(k^*_i)}_{i^*_i, 1}$.
Moreover, at time $t_0$, the torus structure $P^{(k^*_i)}_{i^*_i, 1}$ is even $\td{L}_0^{-1} (\min \{ e^{-LK} \alpha, \td{\alpha}_0 (A, K_1) \}, A)$-precise.
Here $\td{L}_0$ and $\td{\alpha}_0$ are the constants from Lemma \ref{Lem:shortloopingeneralcase}.
Let us now fix $\eta^\circ = \eta^\circ (L, A, \alpha) > 0$ such that $\eta^\circ < \eta_4^*$ for the remainder of the proof.

Next, we address the issue that there is a small gap between each sub-Ricci flow with surgery $V^{(k^*_i)}_{i^*_i, 2}$ and $P'_i$:
For each $i \in \{ 1, \ldots, m^{(N)} \}$ with $h_i^{(N)} < \eta_4^*$ we let $U'_i$ be the union of $V^{(k^*_i)}_{i^*_i, 2}$ with the closure of the component of $P^{(k^*_i)}_{i^*_i, 1} \setminus P'_i$ that is adjacent to $V^{(k^*_i)}_{i^*_i, 2}$.

We now pick a subcollection of the $U'_1, \ldots, U'_{m^{(N)}}$ that are pairwise disjoint at time $t_\omega$.
If $U'_{i_1} (t_\omega) \cap U'_{i_2} (t_\omega) \neq \emptyset$, then by the fact that $P'_{i_1}$ and $P'_{i_2}$ are $\frac1{10}$-precise and \cite[Lemma \ref{Lem:coverMbysth}]{Bamler-LT-topology} we have $U'_{i_1} (t_\omega) \subset P'_{i_2} \cup U'_{i_2} (t_\omega)$ or $U'_{i_2} (t_\omega) \subset P'_{i_1} \cup U'_{i_1} (t_\omega)$.
In the first case we remove the index $i_1$ from the list and in the second case, we remove $i_2$ (if both cases hold, then we remove either $i_1$ or $i_2$).
We can repeat this process until we arrive at a collection $U_1, \ldots, U_m$ whose time-$t_\omega$ slices are pairwise disjoint.
This implies that the time-$t$ slices are pairwise disjoint as well for any $t \in [t_0, t_\omega]$.
Let $P_1, \ldots, P_m$ be the corresponding collection of torus structures.
Observe that at each step of this process, the set $\bigcup_i U'_i (t_\omega) \setminus \bigcup_i P'_i$ does not decrease.
Thus $U_1(t_\omega) \cup \ldots \cup U_m (t_\omega) \supset \bigcup_i U'_i (t_\omega) \setminus \bigcup_i  P'_i$.
We conclude that every point of $\MM(t_\omega) \setminus ( U_1(t_\omega) \cup \ldots \cup U_m (t_\omega))$ that does not belong to $\MM(t_\omega) \setminus ( U^{(N)}_1 (t_\omega) \cup \ldots \cup U^{(m)}_{m^{(N)}} (t_\omega))$ is either contained in some $U^{(N)}_j(t_\omega)$ for which $h_j^{(N)} \geq \eta_4^*$ or it is contained in some $U'_j (t_\omega)$, in which case it must belong to $\bigcup_i P'_i$.
So assertion (a) holds for $K_4 = \max \{ K, K'(\eta^*_4) \}$ (see Proposition \ref{Prop:firstcurvboundstep3}(c)).
The second part of assertion (b) holds by the choice of $\eta^*_4$.

It remains to construct the loops $\gamma_i \subset P_i$ and the maps $h_i : D^2 \to \MM (t_0)$ such that assertions (b)--(e) hold.
Fix $i = 1, \ldots, m$.
By the choice of $\eta_4^*$ we know that $|{\Rm}_{t_0}| < K_1 t_0^{-1}$ and $|{\nabla \Rm}_{t_0}| < K_1 t_0^{-3/2}$ in a time-$\sqrt{t_0}$ neighborhood of size $\sqrt{t_0}$ around $P_i$.
We can hence apply Lemma \ref{Lem:shortloopingeneralcase}(a), (b) with $\alpha \leftarrow \min \{ e^{-LK} \alpha, \td{\alpha}_0 (A, K_1) \}$, $A \leftarrow A$, $K \leftarrow K_1$, $P \leftarrow P_i$ and $S \leftarrow P_i \cup U_i$.
Observe hereby that $\pi_2$ of the component of $\MM(t_0)$ that contains $P^*_i$ vanishes due to assumption (iii) and \cite[Proposition \ref{Prop:pi2irred}]{Bamler-LT-topology}.
We hence obtain a loop $\gamma_i : S^1 \to P_i$ of length $\ell_{t_0} (\gamma_i) < e^{-LK} \alpha$ that is non-contractible in $P_i$, bounds by a disk $h_i : D^2 \to \MM(t_0)$ of time-$t_0$ area $\area_{t_0} h_i < A t_0$ and whose geodesic curvatures at time $t_0$ are bounded by $\td\Gamma(K_1) t_0^{-1/2}$.
So assertions (d) and (e) hold and the first part of assertion (b) follows by a distance distortion estimate.
For assertion (d) observe that $|{\Rm}| < K t^{-1}$ and $|{\nabla \Rm}| < K_1 t^{-3/2}$ on $\gamma_i$ for all $t \in [t_0, t_\omega]$.
\end{proof}

\subsection{Proof of the first main result} \label{subsec:finalargument}
We will finally show that the general picture as described in Proposition \ref{Prop:structontimeinterval} is impossible on a long time-interval, i.e. that for an appropriate choice of the parameters $L$, $A$ and $\alpha$ we must have $m = 0$.
This fact will imply a curvature bound on the final time-slice $\MM (t_\omega)$ and hence establish \cite[Theorem \ref{Thm:LT0-main-1}]{Bamler-LT-Introduction}.

As a preparation, we first prove that after some large time, all time-slices are irreducible and all surgeries are trivial (see also \cite[Proposition 18.9]{MTRicciflow}).
\begin{Proposition} \label{Prop:irreducibleafterfinitetime}
Let $\MM$ be a Ricci flow with surgery and precise cutoff whose time-slices are closed manifolds and that is defined on the time-interval $[T, \infty)$ $(T\geq0)$.
Then there is some $T_1 \in [T, \infty)$ such that all surgeries on $[T_1, \infty)$ are trivial and we can find a sub-Ricci flow with surgery $\MM' \subset \MM$ on the time-interval $[T, \infty)$ that is performed by precise cutoff (and whose time-slices have no boundary) such that:
For all $t \in [T_1, \infty)$ the complement $\MM(t) \setminus \MM'(t)$ consists of a disjoint union of spheres and all components of $\MM'(t)$ are irreducible and not diffeomorphic to spherical space forms.

Moreover, if there is a time $T^* \geq T_1$ such that if $\MM'$ is non-singular on $[T^*, \infty)$, then there is also a time $T^{**} \geq T^*$ such that $\MM$ is non-singular on $[T^{**}, \infty)$ and $\MM'(t) = \MM(t)$ for all $t \in [T^{**}, \infty)$.
\end{Proposition}
\begin{proof}
Let $\MM = (( T^i ), (M^i \times I^i), (g^i_t), (\Omega_i), (U^i_\pm))$ (see \cite[Definition \ref{Def:RFsurg}]{Bamler-LT-Perelman}).
By \cite[Definition \ref{Def:precisecutoff}]{Bamler-LT-Perelman}, for any surgery time $T^i$, the topological manifold $M^i$ can be obtained from $M^{i+1}$ by possibly adding spherical space forms or copies of $S^1 \times S^2$ to the components of $\MM(t_2)$ and then performing connected sums between some of those components.
So every component of $M^{i+1}$ that is not diffeomorphic to a sphere forms the building block of a component of $M^i$ that is also not diffeomorphic to a sphere.
This fact enables us to choose $\MM' \subset \MM$ such that for every $t \in [T, \infty)$ the set $\MM'(t)$ is the union of all components of $\MM(t)$ that are not diffeomorphic to spheres.

By the existence and uniqueness of the prime decomposition (see e.g. \cite[Theorem 1.5]{Hat}) and the conclusion above, there are only finitely many times when the topology of $\MM'(t)$ changes by more than the removal of components that are diffeomorphic to spherical space forms or $S^1 \times S^2$.
Choose $T'_1 \in [T, \infty)$ larger than those times.
So the number of components of $\MM' (t)$ is non-increasing in $t$ for $t \in [T'_1, \infty)$.
We can hence choose $T_1 \in [T'_1, \infty)$ such that the number of components of $\MM' (t)$ remains constant for $t \in [T_1, \infty)$.
This implies that the topology of $\MM' (t)$ is constant on $[T_1, \infty)$.
Note moreover that all surgeries on $[T_1, \infty)$ are trivial on $\MM'$ and hence also on $\MM$.

By finite-time extinction of spherical components (see \cite{PerelmanIII}, \cite{ColdingMinicozziextinction}), we conclude that $\MM'(t)$ cannot have components that are diffeomorphic to spherical space forms for any $t \in [T_1, \infty)$.
Next, assume that $\MM' (T_1)$ was not irreducible.
Then by \cite[Proposition \ref{Prop:pi2irred}]{Bamler-LT-topology}, we have $\pi_2(N) \not= 0$ for some component $N$ of $\MM'(T_1)$.
We can thus use \cite[Proposition \ref{Prop:evolsphere}]{Bamler-LT-simpcx} to obtain a contradiction.

The last part of the proposition follows again from finite-time extinction.
\end{proof}

We can finally finish the proof of the main result, \cite[Theorem \ref{Thm:LT0-main-1}]{Bamler-LT-Introduction}, which provides a curvature bound for large times and states that surgeries eventually stop occurring if the cutoff is performed sufficiently precise.
In the case in which the underlying manifold is not covered by a torus bundle over a circle, we will obtain this result by ruling out the existence of the loops $\gamma_i$ obtained from Proposition \ref{Prop:structontimeinterval} (with the appropriate parameters) using the minimal disk argument from \cite[Proposition \ref{Prop:evolminsurfgeneral}]{Bamler-LT-simpcx}.
In the case in which the underlying manifold is covered by a torus bundle over a circle, we will not be able to construct the ``compressing planar domains'' as needed in assumption (iv) of Proposition \ref{Prop:structontimeinterval}.
Note that \cite[Proposition \ref{Prop:maincombinatorialresult}(a)]{Bamler-LT-topology}, which produces such compressing domains, requires the underlying manifold not be covered by such a torus bundle over a circle.
In this case, however, we can use a different and more basic argument, which makes use of \cite[Proposition \ref{Prop:maincombinatorialresult}(b)]{Bamler-LT-topology}.

\makeatletter
\def\@cite#1#2{{\normalfont{\itshape[#1\if@tempswa , #2\fi]}}}
\makeatother
\begin{proof}[Proof of {\cite[Theorem \ref{Thm:LT0-main-1}]{Bamler-LT-Introduction}}]
\makeatletter
\def\@cite#1#2{{\normalfont[{#1\if@tempswa , #2\fi}]}}
\makeatother
Let the function $\delta(t)$ be the minimum of the functions given in Proposition \ref{Prop:structontimeinterval}, \cite[Proposition \ref{Prop:thickthindec}]{Bamler-LT-Perelman}, \cite[Corollary \ref{Cor:Perelman68}]{Bamler-LT-Perelman}, and \cite[Proposition \ref{Prop:curvcontrolgood}]{Bamler-LT-Perelman}.

Consider the constant $T_1 < \infty$ and the sub-Ricci flow with surgery $\MM' \subset \MM$ from Proposition \ref{Prop:irreducibleafterfinitetime} defined on the time-interval $[0, \infty)$.
Recall that all components of all time-slices of $\MM'$ at or after time $T_1$ are irreducible and not diffeomorphic to spherical space forms and that all surgeries of $\MM'$ at or after time $T_1$ are trivial.
Moreover, all time slices $\MM' (t)$ for $t \geq T_1$ are diffeomorphic to one another.
By the last statement of Proposition \ref{Prop:irreducibleafterfinitetime}, it suffices to establish the desired curvature bound and the finiteness of the surgeries on $\MM'$.
Choose now a sub-Ricci flow with surgery $\MM^* \subset \MM$ defined on the time-interval $[T_1, \infty)$ whose time-slices $\MM^*(t)$ are all connected, closed components of $\MM(t)$.
Since the choice of $\MM^*$ was arbitrary, it suffices to establish the curvature bound and the finiteness of the surgeries on $\MM^*$ instead of $\MM$.

Next, we apply \cite[Proposition \ref{Prop:thickthindec}]{Bamler-LT-Perelman} to $\MM$ and consider the time $T_0 < \infty$, the function $w : [T_0, \infty) \to (0, \infty)$ as well as the decomposition $\MM (t) = \MM_{\textnormal{thick}} (t) \cup \MM_{\textnormal{thin}}(t)$ for all $t \in [T_0, \infty)$.
Set $T_2 = \max \{ T_0, T_1 \}$.

Let now $M = \MM^* (T_2)$ and $M_{\textnormal{hyp}} = \MM_{\textnormal{thick}}(T_2) \cap \MM^*(T_2)$, $M_{\textnormal{Seif}} = \MM_{\textnormal{thin}} (T_2) \cap \MM^*(T_2)$.
So $M = M_{\textnormal{hyp}} \cup M_{\textnormal{Seif}}$.
This decomposition can be refined to a geometric decomposition of $M$ since by the results of \cite{MorganTian} or \cite{KLcollapse}, which led to the resolution of the Geometrization Conjecture (essentially their statement is Proposition \ref{Prop:MorganTianMain} plus a topological discussion), we know that $M_{\textnormal{Seif}}$ is a graph manifold (see \cite[Definition \ref{Def:geomdec}]{Bamler-LT-topology} and the subsequent discussion).
So there are pairwise disjoint, embedded, incompressible $2$-tori $T^*_1, \ldots, T^*_k \subset M_{\textnormal{Seif}}$ that cut $M$ into hyperbolic and Seifert pieces in such a way that the union of the closures of the hyperbolic pieces is exactly $M_{\textnormal{hyp}}$.
Moreover by construction, no two hyperbolic components in this decomposition are adjacent to one another.
We are hence in a position to apply \cite[Proposition \ref{Prop:maincombinatorialresult}]{Bamler-LT-topology}, which yields a simplicial complex $V$ and either a continuous map $f_0 : V \to M$ with $f_0 (\partial V) \subset \partial M_{\textnormal{Seif}} = \partial M_{\textnormal{hyp}}$ that is a smooth immersion on $\partial V$ (if $M$ is not covered by a $T^2$-bundle over a circle) or a sequence of continuous maps $f_1, f_2, \ldots : V \to M$ (if $M$ is covered by a $T^2$-bundle over a circle).

Next, we apply \cite[Proposition \ref{Prop:areaevolutioninMM}]{Bamler-LT-simpcx} to obtain a constant $A_0 < \infty$ and (not necessarily continuous) families of piecewise smooth maps $f_{0, t} : V \to \MM^*(t)$ or $f_{1,t}, f_{2,t}, \ldots : V \to \MM^*(t)$ for all $t \in [T_2, \infty)$ with $f_{0,T_2} |_{\partial V} = f_0 |_{\partial V}$ such that $f_{0, t} |_{\partial V}$ moves by the ambient isotopies of \cite[Proposition \ref{Prop:thickthindec}]{Bamler-LT-Perelman}, $f_{n, t}$ is homotopic to $f_n$ in space-time---restricting to said isotopies on $\partial V$ if $n = 0$---and such that for $n = 0$ or all $n \geq 1$
\begin{equation} \label{eq:areaoffntA0}
 \limsup_{t \to \infty} t^{-1} \area_t f_{n,t} < A_0.
\end{equation}
Note that the constant $A_0$ can be chosen independently of $n$ because the upper bound in \cite[Proposition \ref{Prop:areaevolutioninMM}]{Bamler-LT-simpcx} only depends on the topology of $V$ if $\partial V = \emptyset$, which is always true in the case in which $M$ is covered by a $T^2$-bundle over a circle.
We now distinguish the cases in which $M$ is or is not covered by a $T^2$-bundle over a circle.

\textbf{Case 1: $M$ is not covered by a $T^2$-bundle over a circle} \quad
Choose $T_3 > T_2$ such that $\area_t f_{0,t} < (A_0 + 1) t$ for all $t \geq T_3$.
It now follows from \cite[Proposition \ref{Prop:maincombinatorialresult}]{Bamler-LT-topology} that for every $t \geq T_3$ and every smoothly embedded solid torus $S \subset \Int\MM_{\textnormal{thin}}(t) \cap \MM^*(t)$ that is incompressible in $\MM^*(t)$ there is a compact smooth domain $\Sigma \subset \IR^2$ and a smooth map $h : \Sigma \to S$ such that $h( \partial \Sigma) \subset \partial S$ and such that $h$ restricted to only the exterior boundary circle of $\Sigma$ is non-contractible in $\partial S$ and such that
\[ \area_t h < C \area_t f_{0, t} < C (A_0 + 1) t. \]
Here, the constant $C$ only depends on the topology of the manifold $M$.

Next set
\[ A = C(A_0 + 1), \qquad L = \Big( 1 + \frac{A + 1}{ 4 \pi } \Big)^4 \]
and consider the constant $\Gamma_4 = \Gamma_4(L, A)$ from Proposition \ref{Prop:structontimeinterval}.
Set
\[ \alpha = \frac{\pi}{\Gamma_4} \]
and choose $T_4 = T_4(L, A, \alpha)$ and $w_4 (L, A, \alpha)$ according to this proposition.
Choose now $T^* > \max \{ 4 T_3, T_4 \}$ such that $w(t) < w_4$ for all $t \in [ \frac14 T^*, \infty)$ and consider times $t_\omega > L T^*$ and $t_0 = L^{-1} t_\omega$.
Observe that $\MM^*$ is defined on the whole time-interval $[\frac14 t_0, t_\omega]$ and that condition (iv) of Proposition \ref{Prop:structontimeinterval} holds assuming additionally that $S \subset \MM^*(t)$.
We can then apply Proposition \ref{Prop:structontimeinterval} to the sub-Ricci flow with surgery $\MM^*$ with the parameters $L, A, \alpha$ (note that this is not strictly the statement of Proposition \ref{Prop:structontimeinterval}, but the constructions in the proofs of Propositions \ref{Prop:firstcurvboundstep3} and \ref{Prop:structontimeinterval} can be carried out separately on every component of $\MM$).
We then obtain sub-Ricci flows with surgery $U_1, \ldots, U_m \subset \MM^*$, outside of which we have a curvature bound, and maps $h_1, \ldots, h_m  : D^2 \to \MM^*(t_0)$ with $\area_{t_0} h_i < (A+1) t_0$ whose boundary loops $\gamma_i = h_i |_{\partial D^2}$ have length $< \alpha \sqrt{t}$ and geodesic curvature bounded by $\Gamma_4 t^{-1}$ at all times $t \in [t_0, t_\omega]$.
Assume that $m \geq 1$.
Since $\alpha \Gamma_4 = \pi < 2\pi$, we obtain a contradiction by \cite[Proposition \ref{Prop:evolminsurfgeneral}]{Bamler-LT-simpcx}:
\[ t_\omega < \Big( 1 + \frac{A+1}{4(2\pi - \alpha \Gamma_4)} \Big)^4 t_0 = L t_0 = t_\omega. \]
So $m = 0$ and thus we have $|{\Rm_{t_\omega}}| < K_4 t_\omega^{-1}$ on $\MM(t_\omega)$.

We have shown that if $M$ is not covered by a $T^2$-bundle over a circle, then $|{\Rm_t}| < K_4 t^{-1}$ on $\MM^*(t)$ for all $t \geq L T^*$.
So in particular, $\MM^*$ does not develop any singularities past time $L T^*$.

\textbf{Case 2: $M$ is covered by a $T^2$-bundle over a circle} \quad
In this case consider the families of piecewise smooth maps $f_{1,t}, f_{2, t}, \ldots$ and observe that the constant $A_0$ in (\ref{eq:areaoffntA0}) is independent of $n$.
Also in the present case $\MM^* (t) \subset \MM_{\textnormal{thin}} (t)$ for all $t \geq T_2$ by the uniqueness of the geometric decomposition of $M$.

Let now, $\ov{r}$, $K_2$ be the functions from \cite[Corollary \ref{Cor:Perelman68}]{Bamler-LT-Perelman} and $\mu_1$ the constant from Lemma \ref{Lem:unwrapfibration}.
Set $\mu = \min \{ \mu_1, \frac1{10} \}$ and consider the constants $w_0 = w_0 (\mu, \ov{r}(\cdot, 1), K_2(\cdot, 1))$, $0 < s_2 = s_2 (\mu, \ov{r} (\cdot, 1), K_2(\cdot, 1)) \linebreak[1] < s_1 = s_1 (\mu, \ov{r} (\cdot, 1), K_2 (\cdot, 1)) \linebreak[1] < \frac1{10}$ from Proposition \ref{Prop:MorganTianMain}.
Choose $T_3 > T_2$ such that $w(t) < w_0$ for all $t \geq T_3$.
Fix such a time $t$.
We can hence apply Proposition \ref{Prop:MorganTianMain} to $\MM^*(t)$ and conclude that there are two cases: Either $\diam_t \MM^*(t) < \mu \rho_{\sqrt{t}} (x,t)$ for all $x \in \MM^*(t)$ and $\MM^*(t)$ is diffeomorphic to an infra-nilmanifold or a manifold that carries a metric of non-negative sectional curvature, or we obtain a decomposition $\MM^* (t) = V_1 \cup V_2 \cup V'_2$ satisfying assertions (a)--(c) of this Proposition.

%Note that in the first case $\MM^*(t)$ is diffeomorphic to a quotient of the $3$-torus or the nilmanifold.

Let us first consider the second case in which we have a decomposition of the form $\MM^* (t) = V_1 \cup V_2 \cup V'_2$.
We will now analyze this decomposition further, using the tools of subsection \ref{subsec:topimplications} (observe that we are in case A of this subsection).
As in Definition \ref{Def:GG} let $\mathcal{G} \subset \MM^*(t)$ be the union of all components of $V_2$ whose generic $S^1$-fiber is incompressible in $\MM^*(t)$ and all components of $V_1$ or $V'_2$ whose generic fibers are incompressible tori.
Then by Lemma \ref{Lem:bdrygoodisV2} we have $\partial \mathcal{G} \subset V_2 \cap \mathcal{G}$.
Moreover, by Lemma \ref{Lem:SStori} there is a disjoint union of finitely many embedded solid tori $\mathcal{S} \subset \MM^*(t)$ such that $\MM^*(t) = \mathcal{G} \cup \mathcal{S}$.
So we can make the following conclusion:
Either $\mathcal{G} = \MM^* (t)$ or there is a component $\CC \subset V_2$ such that the $S^1$-fibers on $\CC \cap V_{2, \textnormal{reg}}$ are incompressible in $\MM^*(t)$.

Let $x \in \mathcal{G}$ be an arbitrary point and recall the notation (see the beginning of this section \ref{sec:mainargument})
\[ \rho_{r_0}(x,t) = \sup \{ r \in (0, r_0] \;\; : \;\; \sec_t \geq -r^{-2} \quad \text{on} \quad B(x,t,r) \}. \]
Consider the universal cover $\td\MM^* (t)$ of $M^*(t)$ and choose a lift $\td{x} \in \td\MM^* (t)$ of $x$.
Then by Lemma \ref{Lem:unwrapfibration} there is a constant $w_1 = w_1(\mu) > 0$ such that
\begin{equation} \label{eq:volBw1rho-application}
 \vol B(\td{x}, \rho_{\sqrt{t}} (x,t)) > w_1 \rho_{\sqrt{t}}^3(x,t). 
\end{equation}
In other words, $x$ is $w_1$-good at scale $\sqrt{t}$ (compare with \cite[Definition \ref{Def:goodness}]{Bamler-LT-Perelman}).
Consider now the constants $T_4 = T (w_1, 1)$, $K = K(w_1)$ and $\ov\rho = \ov\rho(w_1)$ from \cite[Proposition \ref{Prop:curvcontrolgood}]{Bamler-LT-Perelman}.
Assuming that we have picked $t$ such that $t > T_4$, we conclude that
\[ |{\Rm}|(x,t) < K t^{-1} \qquad \text{and} \qquad \rho_{\sqrt{t}} (x,t) > \ov\rho \sqrt{t} \qquad \text{for all} \qquad x \in \mathcal{G}. \]

Consider first the case in which $\mathcal{G} \neq \MM^*(t)$.
In this case there is a component $\CC \subset V_2 \cap \mathcal{G}$ such that the $S^1$-fibers on $\CC \cap V_{2, \textnormal{reg}}$ are incompressible in $\MM^*(t)$.
Pick $x \in \CC \cap V_{2, \textnormal{reg}}$.
By Proposition \ref{Prop:MorganTianMain}(c2) there are an open subset $U \subset \MM^*(t)$ with
\[ B(x, t, \tfrac12 s_2 \rho_{\sqrt{t}} (x)) \subset U \subset B(x, t, s_2 \rho_{\sqrt{t}} (x)) \]
that is diffeomorphic to $B^2 \times S^1$, vector fields $X_1, X_2$ on $U$ and a smooth map $p : U \to \IR^2$ such that:
We have $dp (X_i) = \frac{\partial}{\partial x_i}$ and $|\langle X_i, X_j \rangle - \delta_{ij} | < \frac1{10}$ for all $i,j = 1,2$.
Moreover, $p : U \approx B^2 \times S^1 \to p(U)$ corresponds to the projection to the first factor and the $S^1$-fibers coming from the second factor are isotopic to the $S^1$-fibers in $\CC \cap V_{2, \textnormal{reg}}$ and hence incompressible in $\MM^*(t)$.
It then follows easily that $p$ is $2$-Lipschitz and that $B_0 = B( p(x), \frac14 s_2 \rho_{\sqrt{t}} (x)) \subset p(U)$.

Now recall the maps $f_{1, t}, f_{2, t}, \ldots : V \to \MM^*(t)$ from the beginning of the proof.
By \cite[Proposition \ref{Prop:maincombinatorialresult}]{Bamler-LT-topology}, we know that for each $n \geq 1$, the map $f_{n, t}$ intersects each $S^1$-fiber on $U$ at least $n$ times.
In other words, $f^{-1}_{n,t} (p^{-1}(y))$ contains at least $n$ elements for each $y \in B_0 \subset p(U)$.
Since $p$ is $2$-Lipschitz, we find that
\[ \area_t f_{n, t} \geq \frac{n}{4} \area_t B_0 = \frac{n \pi s_2^2 \rho^2_{\sqrt{t}} (x,t)}{16 \cdot 4} > n \cdot \frac{s_2^2 \ov\rho^2}{100} \cdot t. \]
So it follows that for
\[ n > \frac{100}{s_2^2 \ov\rho^2} (A_0 + 1) \]
we have $\area_t f_{n, t} > (A_0 + 1) t$.
This however contradicts (\ref{eq:areaoffntA0}) for large $t$.

So there is some constant $T_5 < \infty$ such that whenever $t > T_5$, then $\mathcal{G} = \MM^*(t)$ and hence $|{\Rm_t}| < K t^{-1}$ on $\MM^*(t)$.
As before, this implies that there are no surgeries on $\MM^*$ past time $T_5$.
This concludes the case in which we have a decomposition of the form $\MM^* (t) = V_1 \cup V_2 \cup V'_2$.

Lastly, it remains to discuss the case in which $\diam_t \MM^*(t) < \mu \rho_{\sqrt{t}} (x,t)$ for all $x \in \MM^*(t)$ and in which $\MM^*(t)$ is diffeomorphic to an infra-nilmanifold or a manifold that carries a metric of non-negative sectional curvature.
Since $\MM^*(t)$ is covered by a $T^2$-bundle over a circle, it must be diffeomorphic to an infra-nilmanifold or the quotient of a torus.
Similarly as before, we can now apply part (v) of Lemma~\ref{Lem:unwrapfibration} to show that (\ref{eq:volBw1rho-application}) holds for all $\td{x} \in \td{\MM}^*(t)$.
By the same reasoning as presented after (\ref{eq:volBw1rho-application}), this implies that we have $|{\Rm}_t| < K t^{-1}$ on $\MM^*(t)$ for $K = K(w_1)$ if $t > T_4 = T_4 (w_1, 1)$.
So, as before, it follows that there are no surgeries past some time $T_5 < \infty$. This finishes the proof.
\end{proof}
\makeatletter
\def\@cite#1#2{{\normalfont[{#1\if@tempswa , #2\fi}]}}
\makeatother

\subsection{Behavior of the geometry for large times} \label{subsec:behaviorlargetimesproof}
In the following, we will prove \cite[Theorem \ref{Thm:geombehavior}]{Bamler-LT-Introduction}, which describes the behavior of the geometry of a Ricci flow with surgery $\MM$ as $t \to \infty$ in more detail.
Most of the characterizations of this behavior will follow from proofs leading to \cite[Theorem \ref{Thm:LT0-main-1}]{Bamler-LT-Introduction}.
We will refer to these proofs in the following and discuss their geometric implications.

Note that, in virtue of \cite[Theorem \ref{Thm:LT0-main-1}]{Bamler-LT-Introduction}, it suffices to restrict our attention to non-singular Ricci flows $(g_t)_{t \in [0, \infty)}$ on connected, orientable manifolds $M$ that satisfy the curvature bound $|{\Rm_t}| < C t^{-1}$.
We will sometimes denote these Ricci flows by $\MM$ to stay in line with our previous notions.
Note that $\MM (t) = (M, g_t)$ for any $t \in [0, \infty)$.
For the remainder of this subsection we fix such a Ricci flow $\MM$.
By \cite[Corollary \ref{Cor:topconditionforlongtime}]{Bamler-LT-Introduction}, $M$ is irreducible and not diffeomorphic to a spherical space form.
The curvature bound implies:

\begin{Lemma} \label{Lem:curvderbounds}
There are constants $C_1, C_2, \ldots$ such that $|\nabla^m {\Rm_t}| < C_m t^{-m/2 - 1}$ for all $t \in [1, \infty)$ and all $m \geq 1$.
\end{Lemma}

\begin{proof}
This is a consequence of the curvature bound $|{\Rm_t}| < C t^{-1}$ and Shi's estimates.
\end{proof}

\begin{Lemma}
There is a constant $\ov\rho > 0$ such that $\rho_{\sqrt{t}} (x,t) > \ov\rho \sqrt{t}$ for all $(x,t) \in M \times [1, \infty)$.
\end{Lemma}

\begin{proof}
This follows from the definition of $\rho_{r_0} (x,t)$.
\end{proof}

We first apply \cite[Proposition \ref{Prop:thickthindec}]{Bamler-LT-Perelman} to $\MM$ to obtain the time $T_0 < \infty$, the function $w : [T_0, \infty) \to (0, \infty)$, with $\lim_{t \to \infty} w(t) = 0$, as well as the decomposition $M = \MM (t) = \MM_{\textnormal{thick}} (t) \cup \MM_{\textnormal{thin}}(t)$ for all $t \in [T_0, \infty)$.
Note that by part (e) of \cite[Proposition \ref{Prop:thickthindec}]{Bamler-LT-Perelman}, we have for all $t \geq T_0$ and $x  \in \MM_{\textnormal{thin}} (t)$ that
\[ \vol_t B(x,t, \rho_{\sqrt{t}} (x,t)) < w(t) \rho^3_{\sqrt{t}} (x,t) < w(t) t^{3/2}. \]
By the previous Lemma and volume comparison, we hence obtain, after possibly adjusting $w(t)$:

\begin{Lemma}
For all $t \geq T_0$ and $x \in \MM_{\textnormal{thin}} (t)$ we have
\[ \vol_t B(x,t, \sqrt{t}) < w(t) t^{3/2}. \]
\end{Lemma}

So the thin part collapses at the \emph{uniform} scale $\sqrt{t}$, with a \emph{two-sided} curvature bound.
This more controlled collapsing behavior enables us to improve and simplify the characterization of Proposition \ref{Prop:MorganTianMain}, using the theory of Cheeger-Fukaya-Gromov (\cite{CFG}) instead of the theory of Morgan-Tian (\cite{MorganTian}).
We will, however, use the language of Proposition \ref{Prop:MorganTianMain} to describe the collapsing behavior.

\begin{Proposition} \label{Prop:V1V2decompositionboundedcurvature}
Given $\MM$ there is a function $\varepsilon : [T_0, \infty) \to (0, \infty)$ with $\lim_{t \to \infty} \varepsilon (t) = 0$ and for every $\mu > 0$ there are constants $a = a(\mu) > 0$, $T_1 = T_1 (\mu) \geq T_0$ such that for any $t \in [T_1, \infty)$ there is a metric $g'_t$ such that
$g'_t$ is $(1+\mu)$-bilipschitz to $g_t$, $| \partial^m (g'_t - g_t) |_{g_t} < \mu t^{-m/2}$ for $0 \leq m < \mu^{-1}$ and such that the following properties hold:

There are two cases.
In the first case, $\MM_{\textnormal{thin}} (t) = M$, $M$ is a quotient of the $3$-torus $T^3$ or the nilmanifold, $\diam_t M < \mu \sqrt{t}$ and $g'_t$ is flat or a quotient of a left-invariant metric.
In the second case, we can find finitely many embedded $2$-tori $\Sigma_{i,t}^T \subset \Int \MM_{\textnormal{thin}} (t)$ that are pairwise disjoint, as well as closed subsets $V_{1,t}, V_{2,t}, V^*_t \subset \MM_{\textnormal{thin}} (t)$ such that:
\begin{enumerate}[label=(\alph*)]
\item $\MM_{\textnormal{thin}} (t) = V_{1,t} \cup V_{2,t} \cup V^*_t$ and the interiors of the sets $V_{1,t}, V_{2,t}, V^*_t$ are pairwise disjoint and $V_{1,t}, V_{2,t}, V^*_t$ are separated by the $\Sigma_{i,t}^T$.
\item Every component of $\partial \MM_{\textnormal{thin}} (t)$ is adjacent to a component of $V_{1,t}$.
\item The components of $V_{1,t}$ are diffeomorphic to $T^2 \times I$, $\Klein^2 \td\times I$, a $T^2$-bundle over a circle or the union of two copies of $\Klein^2 \td\times I$ along their $T^2$-boundary.
In the last two cases $V_{1,t} = M$.
\item The set $V_{2,t}$ carries a Seifert fibration $p_{V_{2,t}} : V_{2,t} \to \Sigma_{V_{2,t}}$, where $\Sigma_{V_{2,t}}$ is a possibly disconnected orbifold with cone singularities.
The fibration is compatible with the boundary tori of $V_{2,t}$.
\item Each component of $V^*_t$ is diffeomorphic to the solid torus $S^1 \times D^2$ and adjacent to a component of $V_{1,t}$.
\end{enumerate}
We can furthermore characterize the geometric properties of $V_{1,t}$ and $V_{2,t}$ as follows:
\begin{enumerate}[label=(\alph*), start=6]
\item If $\mathcal{C}$ is a component of $V_{1,t}$, then:
\begin{enumerate}[label=(f\arabic*)] 
\item If $\mathcal{C} \approx T^2 \times I$, then there is a diffeomorphism $\Phi : T^2 \times I \to \mathcal{C}$ such that $\Phi^* g'_t$ is invariant under the $T^2$-action on the first factor.
Moreover, the orbits of this action have diameter $< \mu \sqrt{t}$.
\item If $\mathcal{C} \approx \Klein^2 \td\times I$, then the assertions of item (f1) hold for the double cover $\widehat{\mathcal{C}}$ that is diffeomorphic to $T^2 \times I$.
\item If $\mathcal{C} = M$ is diffeomorphic to a $T^2$-bundle over a circle, then there is a bundle projection $p : M \to S^1$ and in a fibered neighborhood of every $T^2$-fiber there is a fiberwise $T^2$-action on $M$ that is isometric with respect to $g'_t$.
Moreover, the fibers of $p$ have diameter $< \mu \sqrt{t}$.
\item If $\mathcal{C} = M$ is diffeomorphic to the union of two copies of $\Klein^2 \td\times I$ along their $T^2$-boundary, then there is a double cover $\widehat{M}$ of $M$ that satisfies the assertions of item (f3).
\end{enumerate}
\item All components of $V_{1,t}$ that are not equal to components of $\MM_{\textnormal{thin}} (t)$ have diameter $> \mu^{-1} \sqrt{t}$.
\item The Seifert fibers on $V_{2,t}$ are orbits of an isometric $S^1$-action on $V_{2,t}$ and the map $p_{V_{2,t}} : V_{2,t} \to \Sigma_{V_{2,t}}$ is a submersion with respect to $g'_t$ onto a smooth orbifold metric $g''_t$ on $\Sigma_{V_{2,t}}$ whose curvature is bounded by $a^{-1}(\mu) t^{-1}$.
The Seifert fibers have diameter $< \varepsilon (t) \sqrt{t}$.
Moreover, on $V_{2,t}$ the metric $g'_t$ is $(1+\varepsilon (t))$-bilipschitz to $g_t$ and $|\partial^m (g'_t - g_t)|_{g_t} < \varepsilon(t) t^{-m/2}$ for $0 \leq m < \mu^{-1}$.
\item The area of every component of $(\Sigma_{V_{2,t}}, g''_t)$ is $> a(\mu) t$ and for every $x \in \Sigma_{V_{2,t}}$ for which $B_{g''_t} (x, \sqrt{t}) \subset \Int \Sigma_{V_{2,t}}$ we have as well $\area_{g''_t} B_{g''_t} (x, \sqrt{t}) > a(\mu) t$.
Furthermore, the boundary circles of $(\Sigma_{V_{2,t}}, g''_t)$ have diameter \linebreak[1] $> a(\mu) \sqrt{t}$.
\item The second fundamental form of the $T^2$-fibers on $V_{2,t}$ and the geodesic curvatures of the Seifert fibers on $V_{2,t}$, with respect to $g'_t$ and $g_t$, are bounded from above by $a^{-1} (\mu) t^{-1/2}$.
\item The components of $V^*_t$ have diameter $< \mu \sqrt{t}$.
\end{enumerate}
\end{Proposition}

\begin{proof}
We first apply \cite[Theorem 1.7]{CFG} with sufficiently small $\varepsilon = \varepsilon (\mu)$ to obtain a $(\rho \sqrt{t}, k)$-round metric $g'_t$ that is sufficiently regular.
In (1.3.3) of \cite{CFG} this regularity is expressed in terms of bounds on the derivatives of the curvature of $g'_t$.
In our case, we actually obtain a bound of the form $| \partial^m (g'_t - g_t) |_{g_t} < \mu t^{-m/2}$ for $0 \leq m < \mu^{-1} + 1$, since by Lemma \ref{Lem:curvderbounds} the metric $t^{-1} g_t$ is already ``$A$-regular'' and we can omit the application of Abresch's Theorem, \cite[Theorem 1.12]{CFG}; see also \cite[Proposition 7.21]{CFG}.
We furthermore obtain a nilpotent Killing structure $\mathfrak{N}_t$ for $g'_t$ whose orbits are compact and have diameter $< \frac1{2}\mu \sqrt{t}$.
Here $\rho = \rho(\mu) > 0, k = k(\mu) < \infty$ are uniform in $t$.
If $\mathfrak{N}_t$ has only one single orbit that fills out $M$, then $M$ is diffeomorphic to a finite quotient of a nilpotent Lie group and $\diam_t M < \mu \sqrt{t}$; so we are done.

So consider from now on the case in which the orbits of $\mathfrak{N}_t$ are $0$, $1$ or $2$-dimensional.
We first discuss the local geometry of $g'_t$ around each orbit.
Let $p \in M$ and denote by $\mathcal{O}_{p,t}$ the orbit of $\mathcal{N}_t$ through $p$.
Then there is a subset $V_{p,t} \subset M$ containing $B(p,t,\rho \sqrt{t}) \subset V_{p,t}$, a Lie group $H_{p,t}$ that acts isometrically and faithfully on a normal cover $\pi_{p,t}: (\widehat{V}_{p,t}, \widehat{g}'_t) \to (V_{p,t}, g'_t)$ such that the group of deck transformations is represented by a discrete subgroup $\Lambda_{p,t} \subset H_{p,t}$ and such that the following holds: $H_{p,t}$ consists of $\leq k$ many components, its identity component $N_{p,t}$ is nilpotent and $H_{p,t}$ is generated by $N_{p,t}$ and $\Lambda_{p,t}$.
Moreover, the injectivity radius at every lift $\widehat{p} \in \widehat{V}_{p,t}$ of $p$ is $> \rho \sqrt{t}$.
Let $\mathcal{O}_{\widehat{p}, t} \subset \widehat{V}_{p,t}$ be the orbit of such a lift $\widehat{p}$ under $H_{p,t}$.
Then $\mathcal{O}_{p,t} = \pi (\mathcal{O}_{\widehat{p}, t} )$ and $\mathcal{O}_{\widehat{p}, t} \approx  H_{p,t} / \Stab_H (\widehat{p})$.
Here $\Stab_{H_{p,t}} (\widehat{p})$ denotes the stabilizer subgroup of $\widehat{p}$, which is isomorphic to a subgroup of $SO(3)$.
Since $SO(3)$ is not nilpotent and has no $2$-dimensional subgroups, $\Stab_H (\widehat{p})$ is either $0$ or $1$-dimensional and hence $\dim \mathcal{O}_{p,t} \leq \dim H_{p,t} \leq \dim \mathcal{O}_{p,t} + 1$.
We can then describe the local geometry around $\mathcal{O}_{p,t}$, depending on its dimension, as follows:
\begin{enumerate}[label=(\arabic*), start=0]
\item If $\mathcal{O}_{p,t}$ is $0$-dimensional, then $H_{p,t} \cap \Lambda_{p,t} = \{ 1 \}$ and hence $|\Lambda_{p,t} | \leq k$.
So $B(p,t, \rho \sqrt{t} ) > c(\mu) \rho^3 t^{3/2}$ for some $c = c(\mu) > 0$.
\item If $\mathcal{O}_{p,t}$ is $1$-dimensional, then it must be a circle and $H_{p,t}$ must be $1$ or $2$-dimensional.

If $H_{p,t}$ is $1$-dimensional, then consider the finite cover $\ov\pi_{p,t} : (\ov{V}_{p,t}, \ov{g}_t) \to (V_{p,t}, g'_t)$ corresponding to the subgroup $\Lambda_{p,t} \cap N_{p,t}$.
The group $N_{p,t}$ acts isometrically on $(\ov{V}_{p,t}, \ov{g}_t)$ and the orbits of this action are the circles of a smooth $S^1$-fibration on $\ov{V}_{p,t}$.
This $S^1$-fibration induces a Seifert fibration on $V_{p,t}$ via the covering projection map $\ov\pi_{p,t}$, and the action of $N_{p,t}$ on $\ov{V}_{p,t}$ induces an isometric action of $N_{p,t}$ on $V_{p,t}$.
Any orbit $\mathcal{O}_{p',t} \subset V_{p,t}$ is a union of Seifert fibers.
The area of of any $r$-ball, $r < \sqrt{t}$, that is compactly contained in the base space of this Seifert fibration is $> a'(\mu) r^2$ for some $a' = a'(\mu) > 0$ (this area bound is a consequence of Lemma 8.5 and (8.7) in \cite{CFG}).
The area bound implies that the rotational velocity of the local Killing field $K$ corresponding to an infinitesimal generator of $N_{p,t}$ is bounded along $\mathcal{O}_{p,t}$.
In other words, $|\nabla K| < C' (\mu) t^{-1/2} K$ on $\mathcal{O}_{p,t}$ for some $C' = C' (\mu) < \infty$.
So the geodesic curvature on $\mathcal{O}_{p,t}$ is bounded by $C'' (\mu) t^{-1/2}$ and the curvature of the base space is bounded by $C'' (\mu) t^{-1}$, for some $C'' = C'' (\mu) < \infty$ (the latter bound follows by O'Neill's formula).

If $H_{p,t}$ is $2$-dimensional, then there is a neighborhood $U$ of $\mathcal{O}_{p,t}$ that is diffeomorphic to the solid torus $S^1 \times D^2$ and that has the property that $U \setminus \mathcal{O}_{p,t}$ is the union of $2$-dimensional orbits $\mathcal{O}_{p',t}$.
This neighborhood can be chosen to contain an $a'' (\mu) \sqrt{t}$-neighborhood of $\mathcal{O}_{p,t}$ for some $a'' = a'' (\mu) > 0$.
\item If $\mathcal{O}_{p,t}$ is $2$-dimensional, then it must be a $2$-torus or a Klein bottle $\Klein^2$ and $N_{p,t}$ is isomorphic to a quotient of $\IR^2$ or the Poincar\'e group of isometries of $\IR^2$.
Moreover, all orbits of $\mathfrak{N}_t$ on $V_{p,t}$ are $1$ or $2$-dimensional and all orbits of $\mathfrak{N}_t$ in a neighborhood of $\mathcal{O}_{p,t}$ are $2$-dimensional (this follows from (1.5.2) in \cite{CFG}).
If all orbits $\mathcal{O}_{p',t}$ at distance $< r < \sqrt{t}$ are $2$-dimensional, then the second fundamental form on $\mathcal{O}_{p,t}$ is $< C' (\mu) r^{-1}$ for some $C' = C' (\mu) < \infty$ (this bound is a consequence of Remark 8.6 in \cite{CFG} and the following conclusions).

If $\mathcal{O}_{p,t}$ is a $2$-torus, then so are the orbits in a neighborhood of it.

If $\mathcal{O}_{p,t}$ is a Klein bottle, then when we pass to the cover corresponding to the finite index subgroup $\Lambda_{p,t} \cap N_{p,t}$, this orbit lifts to a $2$-torus.
So there is a neighborhood $U \subset V_{p,t}$ of $\mathcal{O}_{p,t}$ that is diffeomorphic to $\Klein^2 \td\times I$ with the property that $U \setminus \mathcal{O}_{p,t}$ is the union of orbits $\mathcal{O}_{p',t}$ that are $2$-tori.
\end{enumerate}

Using this local characterization, we can now describe the global structure of $\MM_{\textnormal{thin}} (t)$.
Let $X_{1,t} \subset M$ be the union of all $1$-dimensional orbits and $X_{2,t} \subset M$ the union of all $2$-dimensional orbits.
By item (0) of the previous list, we have $\MM_{\textnormal{thin}} (t) \subset X_{1,t} \cup X_{2,t}$ for sufficiently large $t$ (depending on $\mu$).
We first consider components $\mathcal{C}$ of $X_{1,t}$ that do not contain a ball of radius $\frac1{10} \rho \sqrt{t}$.
By item (2) in the previous list, we can extend the Killing structure $\mathfrak{N}_t$ to a Killing structure $\mathfrak{N}'$ on $\mathcal{C}$, whose generic orbits are $2$-dimensional.
Using Jacobi field comparison, we find that the orbits of this new Killing structure are bounded by $2 \cdot \frac12 \mu \sqrt{t} = \mu \sqrt{t}$ in diameter (assuming $\rho$ to be sufficiently small).
So we may assume from now on that each component of $X_{1,t}$ contains a ball of radius $\frac1{10} \rho \sqrt{t}$.

Next, we consider the isolated orbits of $X_{1,t}$ that are adjacent to $X_{2,t}$.
Around these orbits, we choose pairwise disjoint solid tori of diameter $< \mu \sqrt{t}$ within $\MM_{\textnormal{thin}} (t)$ that are unions of this orbits and orbits which are $2$-tori.
We may assume that these solid tori contain an $r(\mu) \sqrt{t}$ neighborhood of each isolated $S^1$-orbit for a uniform $r = r(\mu) > 0$.
Denote the union of these solid tori by $V^{* \prime}_t$.
Choose $V'_{1,t}$ to be the closure of $X_{2,t} \setminus V^{* \prime}_t$ and $V'_{2,t}$ to be the closure of $X_{1,t} \setminus (V'_{1,t} \cup V^{* \prime}_t)$.
Then the decomposition $\MM_{\textnormal{thin}} (t) = V'_{1,t} \cup V'_{2,t} \cup V^{* \prime}_t$ satisfies assertions (a)--(f), (j), (k) and the first parts of (h) and (i), whenever $t > T'(\mu)$.
For the second part of (i), consider a $2$-torus orbit orbit $\mathcal{O}_{p,t} \subset V_{1,t}$ close to $V_{1,t} \cap V_{2,t}$.
By item (2) of the previous list, we conclude that a collar neighborhood of $V_{2,t}$ of uniform size has a $T^2$-symmetry.
The second part of (i) now follows from the lower area bound on the base space of the Seifert fibration on this collar.
Moreover, there is a constant $C_* = C_* (\mu) < \infty$ such that the Seifert fibers of $V'_{2, t}$ have diameter $< C_* w(t) \sqrt{t}$ (this bound is again a consequence of Remark 8.6 in \cite{CFG} and the following conclusions, or it follows from the earlier bound $|\nabla K| < C' t^{-1/2}$ on a local Killing field $K$).

Consider now a component $\mathcal{C}$ of $V'_{1,t}$ that is not a component of $\MM_{\textnormal{thin}} (t)$ and that has diameter $\leq \mu^{-1} \sqrt{t}$.
Then $\mathcal{C}$ is diffeomorphic to $T^2 \times I$ or $\Klein^2 \td\times I$.
By \cite[Lemmas \ref{Lem:coverMbysth}, \ref{Lem:Kleinandsolidtorus}]{Bamler-LT-topology}, $\mathcal{C}$ has to be adjacent to a component $\mathcal{C}_1$ of $V'_{2,t}$ on at least one end.
Using the Killing structure $\mathfrak{N}_t$, the Seifert fibration on $\mathcal{C}_1$ can be extended to a Seifert fibration on $\mathcal{C} \cup \mathcal{C}_1$.
If $\mathcal{C}$ is adjacent to a component of $V^{* \prime}_t$ on the other end, then this fibration induces an isometric $S^1$-action on this component.
We now establish the following

\begin{Claim}
For sufficiently large $t$, depending on $\mu$, we have the following picture: If $\mathcal{C}$ is adjacent to another component $\mathcal{C}_2$ of $V_{2,t}$, then the Seifert fibers on $\mathcal{C}$ and $\mathcal{C}_2$ coincide on their intersection.
If $\mathcal{C}$ is adjacent to a component of $V^{* \prime}_t$, then the orbits of the $S^1$-action on this component form a Seifert fibration.
Moreover, there is a constant $a = a(\mu) > 0$ such that the area bound of assertion (i) holds for the base spaces of these Seifert fibrations.
\end{Claim}

\begin{proof}
Assume that the claim was wrong.
Then we can find sequences $t_k \to \infty$, $\mathcal{C}_k$, $\mathcal{C}_{1,k}$ of counterexamples.
Choose base points $x_k \in \mathcal{C}_k$ and consider the sequence $(M, t^{-1}_k g'_{t_k}, x_k)$ of pointed Riemannian manifolds.
After passing to a subsequence, we may assume that this sequence converges to a complete pointed metric space $(X_\infty, d_\infty, x_\infty)$.
Since $x_k$ stays in bounded distance towards $\mathcal{C}_{1,k}$ with respect to the metric $t_k^{-1} g'_{t_k}$, the limit space $(X, d_\infty, x_\infty)$ has to be $2$-dimensional.
By  \cite[Theorem 1.1]{NaberTian-collapse} the limit $(X_\infty, d_\infty, x_\infty)$ is a smooth $2$-dimensional orbifold.
Since the diameters of the Seifert fibers on $V_{2,t_k}$ are $< C_* w(t)$, the sequence $(M, t^{-1}_k g'_{t_k}, x_k)$ collapses along these fibers and hence also along the newly constructed $S^1$-fibers or orbits.
This proves that the Seifert fibration on $\mathcal{C}_{k,t}$ is compatible with the Seifert fibrations on adjacent components of $V'_{2, t_k}$ or that it can be extended to a Seifert fibration on an adjacent component of $V^{*\prime}_{t_k}$.

So the only remaining possibility for the $\mathcal{C}_k, \mathcal{C}_{1,k}$ to be counterexamples to the claim is that the areas of the base spaces of these Seifert fibrations converge to zero locally somewhere on $B(x_k, t_k, 2 \mu^{-1} \sqrt{t_k})$.
This is however impossible since the orbifold $(X_\infty, d_\infty)$ is non-collapsed on $B(x_\infty, 2 \mu^{-1}) \subset X_\infty$
\end{proof}

Let now $V_{1,t}$ be equal to $V'_{1,t}$ minus the components of diameter $\leq \mu^{-1} \sqrt{t}$ and $V_{2,t}$ be equal to the union of $V_{2, t}$ with these components and the components of $V^{* \prime}_t$ that are adjacent to such components.
Finally, let $V^*_t$ be the union of all components of $V^{* \prime}_t$ that are adjacent to $V_{1,t}$.
So we have established all assertions, except for the better bound on $g'_t - g_t$ in assertion (h).
To obtain this bound we need to replace $g'_t$ locally by the average of the metric $g_t$ under the $S^1$-action on $V_{2,t}$ and interpolate using a cutoff function.
The bound on the derivatives follows from the fact that $| \partial^m (g'_t - g_t) |_{g_t} < \mu t^{-1/2}$ for $0 \leq m < \mu^{-1} + 1$.
\end{proof}

We can finally prove the second main result, \cite[Theorem \ref{Thm:geombehavior}]{Bamler-LT-Introduction}.

\makeatletter
\def\@cite#1#2{{\normalfont{\itshape[#1\if@tempswa , #2\fi]}}}
\makeatother
\begin{proof}[Proof of {\cite[Theorem \ref{Thm:geombehavior}(a)]{Bamler-LT-Introduction}}]
\makeatletter
\def\@cite#1#2{{\normalfont[{#1\if@tempswa , #2\fi}]}}
\makeatother
By the resolution of the Geometrization Conjecture we have $\MM_{\textnormal{thick}} (t) = M_0$ for large $t$.
So by \cite[Proposition \ref{Prop:thickthindec}]{Bamler-LT-Perelman} the metric $\tfrac14 t^{-1} g_t$ can be approximated by a hyperbolic metric with better and better precision as $t \to \infty$.
The pointwise convergence of $\tfrac14 t^{-1} g_t$ follows from the stability of compact hyperbolic metrics (see \cite[Theorem 4]{Ye} or \cite{Bamler-hypcusps}).
\end{proof}
\makeatletter
\def\@cite#1#2{{\normalfont[{#1\if@tempswa , #2\fi}]}}
\makeatother

\makeatletter
\def\@cite#1#2{{\normalfont{\itshape[#1\if@tempswa , #2\fi]}}}
\makeatother
\begin{proof}[Proof of {\cite[Theorem \ref{Thm:geombehavior}(b)--(d)]{Bamler-LT-Introduction}}]
\makeatletter
\def\@cite#1#2{{\normalfont[{#1\if@tempswa , #2\fi}]}}
\makeatother
Assume first that $t^{-1/2} \diam_t M_0$ is bound\-ed for large $t$.
Then we can use \cite[Theorem 1.2]{LottDimRed} to deduce that the universal covers $(\td{M}_0, t^{-1} \td{g}_t, \td{x})$ of $(M_0, t^{-1} g_t, x)$ smoothly converge to the flat metric or to a left-invariant metric on the nilmanifold or solmanifold, depending on whether we are in case (b), (c) or (d).
In case (d) we are then done.
In case (b), we argue as follows:
Since $(\td{M}_0, t^{-1} \td{g}_t, \td{x})$ smoothly converges to a flat metric, we find that $\lim_{t \to \infty} t \Vert {\Rm}(\cdot, t) \Vert_{L^\infty(M_0)} = 0$.
So $\lim_{t \to \infty} t^{-1/2} \diam_t M_0 = 0$.
Hence by Proposition \ref{Prop:V1V2decompositionboundedcurvature}, $g_t$ is $\varepsilon (t)$-bilipschitz close to a flat metric $g'_t$ on $M_0$ for large $t$, where $\lim_{t \to \infty} \varepsilon(t) = 0$.
By \cite[Theorem 1.2]{SSS} there is some positive time $T_0 < \infty$ such that for any $t \geq t_0 \geq T_0$, there is a diffeomorphism $\Phi_{t_0, t} : M_0 \to M_0$ with the property that $g_t$ is $\varepsilon' (t_0)$-bilipschitz to $\Phi^*_{t_0, t} g_{t_0}$, where $\lim_{t \to \infty} \varepsilon' (t) = 0$.
So $(M_0, g_t)$ converges to a unique flat torus in the smooth Cheeger-Gromov sense as $t \to \infty$.
The pointwise convergence of the metric $g_t$ follows now by the stability results of \cite[Theorem 3.7]{Guenter-Isenberg-Knopf-2002} or \cite{KochLamm}.
The diameter bound in case (c) follows from the fact that in the left-invariant Ricci flow on the nilmanifold distances grow like $O(t^{1/6})$.

So assume for the rest of the proof that $t^{-1/2} \diam_t M_0$ becomes unbounded as $t \to \infty$.
We now argue similarly as in the second case of the proof of \cite[Theorem \ref{Thm:LT0-main-1}]{Bamler-LT-Introduction}.
By the resolution of the Geometrization Conjecture we have $\MM_{\textnormal{thin}} (t) = M_0$ for large $t$.
Since $M_0$ is the quotient of a torus bundle over a circle, we can apply \cite[Proposition \ref{Prop:maincombinatorialresult}(b)]{Bamler-LT-topology} and obtain a simplicial complex $V$ and a sequence of continuous maps $f_1, f_2, \ldots : V \to M_0$ with the properties indicated there.
Next, we apply \cite[Proposition \ref{Prop:areaevolutioninMM}]{Bamler-LT-simpcx} to obtain families of piecewise smooth maps $f_{1,t}, f_{2,t}, \ldots : V \to M_0$, homotopic to the maps $f_1, f_2, \ldots$, such that
\begin{equation} \label{eq:limsupofareainsecondresult}
 \limsup_{t \to \infty} t^{-1} \area_t f_{n,t} < A(V),
\end{equation}
where $A(V)$ only depends on $V$.
Now fix some small and arbitrary constant $\mu > 0$ and consider the decomposition $M_0 = V_{1,t} \cup V_{2,t} \cup V^*_t$ and the metric $g'_t$ from Proposition \ref{Prop:V1V2decompositionboundedcurvature} for large $t$.
Assume that $V_{2,t}$ is non-empty.
Note that $V_1 \leftarrow V_{1,t} \cup V^*_t$ and $V_2 \leftarrow V_{2,t}$ satisfy the topological characterizations of Proposition \ref{Prop:MorganTianMain}.
So we can apply the discussion of subsection \ref{subsec:topimplications}, in particular Lemma \ref{Lem:SStori}, and find a good component $\mathcal{C}$ of $V_{2,t}$ (i.e. a component whose Seifert fibers are incompressible in $M_0$).
Consider the submersion $p_{V_{2, t}} |_{\mathcal{C}} : (\mathcal{C}, g'_t) \to (p_{V_{2,t}} (\mathcal{C}), g''_t) \subset (\Sigma_{V_2, t}, g''_t)$ restricted to this component.
By Proposition \ref{Prop:V1V2decompositionboundedcurvature}(i), the area of the base space $(p_{V_{2,t}} (\mathcal{C}), g''_t)$ is $> a(\mu) t$.
So since for every $n$, the image of $f_{n,t}$ intersects each Seifert fiber of $\mathcal{C}$ at least $n$ times, we obtain
\[  \area_t f_{n,t} > \tfrac12 \area_{g'_t} f_{n,t} > \tfrac12 n a(\mu) t. \]
So if we pick $n > 2(A + 1) a^{-1}(\mu)$, then we obtain that $\area_t f_{n,t} > (A+1)t$ for large $t$, in contradiction to (\ref{eq:limsupofareainsecondresult}).
It follows that $V_{2,t} = \emptyset$ and hence $V_{1,t} = M_0$ for $t \geq T(\mu)$.
This proves the desired result for some function $\varepsilon(t)$ satisfying $t \geq T(\varepsilon (t))$ and $\lim_{t \to \infty} \varepsilon (t) = 0$.
\end{proof}
\makeatletter
\def\@cite#1#2{{\normalfont[{#1\if@tempswa , #2\fi}]}}
\makeatother

\makeatletter
\def\@cite#1#2{{\normalfont{\itshape[#1\if@tempswa , #2\fi]}}}
\makeatother
\begin{proof}[Proof of {\cite[Theorem \ref{Thm:geombehavior}(e)]{Bamler-LT-Introduction}}]
\makeatletter
\def\@cite#1#2{{\normalfont[{#1\if@tempswa , #2\fi}]}}
\makeatother
We first apply \cite[Proposition \ref{Prop:maincombinatorialresult}(a)]{Bamler-LT-topology} to $(M_0, \MM_{\textnormal{thin}} (T_0))$ and obtain a simplicial complex $V$ and a continuous map $f_0 : V \to M_0$ with $f_0 (\partial V) \subset \partial \MM_{\textnormal{thin}} (t)$.
Next, \cite[Proposition \ref{Prop:areaevolutioninMM}]{Bamler-LT-simpcx} yields a family of piecewise smooth maps $f_t : V \to M_0$, $t \in [T_0, \infty)$ with the following properties:
$f_{T_0} |_{\partial V} = f_0 |_{\partial V}$, $f_{T_0}$ is homotopic to $f_0$ relative $\partial V$, $f_t |_{\partial V}$ moves by the ambient isotopies of \cite[Proposition \ref{Prop:thickthindec}]{Bamler-LT-Perelman}, $f_t$ is homotopic to $f_0$ in space-time and
\[ \limsup_{t \to \infty} t^{-1} \area_t f_t < A < \infty. \]
For the rest of the proof we will always assume $t$ to be large enough in order to guarantee $t^{-1} \area_t f_t < A+1$.

Let $\mu^* > 0$ be some small constant, whose value we will determine in the course of the proof, depending on $\mu$.
Apply Proposition \ref{Prop:V1V2decompositionboundedcurvature} for large $t$ with $\mu \leftarrow \mu^*$ and consider the decomposition $\MM_{\textnormal{thin}} (t) = V_{1,t} \cup V_{2,t} \cup V^*_t$.
By the discussion of subsection \ref{subsec:topimplications}, see especially Lemma \ref{Lem:SStori}, (setting $V_1 \leftarrow V_{1,t} \cup V^*_t$ and $V_2 \leftarrow V_{2,t}$), we find that there are two cases: either all components of $V_{1,t}$ or $V_{2,t}$ are good and $V^*_t = \emptyset$ or there is a component $P \approx T^2 \times I$ of $V_{1,t}$ each of whose boundary component bounds a solid torus on one side.
We will now show that the latter case cannot occur.
The argument, which we will carry out in the next two paragraphs, will be similar to the first case of the proof of \cite[Theorem \ref{Thm:LT0-main-1}]{Bamler-LT-Introduction}: If there was such a $P$, then we could find a short geodesic loop inside $P$ of controlled geodesic curvature that bounds a disk of bounded area.
By the curvature bound, this loop stays geometrically controlled on a longer time-interval, contradicting Hamilton's minimal disk argument.

Consider such a component $P \approx T^2 \times I$ and choose a solid torus $S \subset \MM_{\textnormal{thin}} (t)$, such that $(S, S \setminus \Int P)$ is diffeomorphic to $(S^1 \times D^2 (1), S^1 \times D^2 (\frac12))$.
By \cite[Proposition \ref{Prop:maincombinatorialresult}(a)]{Bamler-LT-topology} there is a compact, smooth domain $\Sigma \subset \IR^2$ and a smooth map $h : \Sigma \to S$ with $h(\partial \Sigma) \subset \partial S$ such that $h$ restricted to only the outer boundary circle of $\Sigma$ is non-contractible in $\partial S$ and such that $\area_t h < C(A+1)t$.
Here $C < \infty$ is a constant that only depends on the topology of $M_0$.

Next, choose $K = \max \{ C_0, C_1 \}$ such that $|{\Rm_t}| < K t^{-1}$ and $|\nabla {\Rm_t}|< K t^{-3/2}$ on $M_0$ and consider the constants $\td\alpha_0 = \td\alpha_0 ( C(A+2), K)$, $\td\Gamma = \td\Gamma (K)$ from Lemma \ref{Lem:shortloopingeneralcase}.
By the uniform bounds on the curvature and its derivatives (see Lemma \ref{Lem:curvderbounds}), we can pick constants $\Gamma' = \Gamma' (\td\Gamma (K), A) < \infty$, $B = B(K, A) < \infty$ with the following properties:
If $\gamma : S^1 \to M_0$ is a loop whose geodesic curvature is bounded by $\td\Gamma (K) t^{-1/2}$ at time $t$, then it is bounded by $\Gamma' t^{\prime -1}$ at all times $t' \in [t, (1+ C(A+2)/4\pi )^4 t]$.
Moreover, $\ell_{t'} (\gamma) \leq B \ell_t (\gamma)$ for all such times $t'$.
Choose 
\[ \alpha = \min \Big\{ \frac{\pi}{\Gamma' B}, \td\alpha_0 \Big\} \] 
and observe that none of the constants in this paragraph depend on $\mu^*$.
So we may assume that $\mu^* < \frac12 \td{L}_0^{-1} (\alpha, C(A+2))$, where $\td{L}_0$ is the constant from Lemma \ref{Lem:shortloopingeneralcase}.
Then by Proposition \ref{Prop:V1V2decompositionboundedcurvature}(f), (g) the subset $P$ is a torus structure of width $\leq \td{L}_0^{-1} \sqrt{t}$ and length $\geq \td{L}_0 \sqrt{t}$.
By Lemma \ref{Lem:shortloopingeneralcase}(b) there is a map $h_0 : D^2 \to M_0$ with $\area h_0 < (C(A+1)+1) t < C(A+2) t$, whose boundary loop $\gamma_0 = h_0 |_{\partial D^2}$ has time-$t$ length $< \alpha \sqrt{t}$ and time-$t$ geodesic curvature $< \Gamma t^{-1/2}$.
So on the time-interval $[t, (1+ C(A+2)/4\pi )^4 t]$ the geodesic curvature of $\gamma_0$ is bounded by $\Gamma' t^{-1/2}$ and its length is bounded by $B \alpha t^{1/2} \leq \frac{\pi}{\Gamma'} t^{1/2}$.
Using \cite[Proposition \ref{Prop:evolminsurfgeneral}]{Bamler-LT-simpcx}, we now obtain a contradiction.

We have shown that $V^*_t = \emptyset$ and that all components of $V_{1,t}$ and $V_{2,t}$ are good.
We now denote the components of $V_{1,t}$ by $E_{1,t}, \ldots, E_{m_t,t}$, the components of $\MM_{\textnormal{thick}} (t)$ by $M_{1,t}, \ldots, M_{p,t}$ and the components of $V_{2,t}$ by $M_{p+1,t}, \ldots, M_{k_t, t}$.
Furthermore, let $\Sigma_{p+1, t}, \ldots, \Sigma_{k_t, t}$ be the corresponding components of $\Sigma_{V_{2,t}}$ and $p_{p+1,t}, \ldots, p_{k_t, t}$ the projections of $M_{p+1, t}, \ldots, M_{k_t, t}$ onto those components.
Then assertions (e1)--(e3), (e5)--(e7) and the first part of (e8) hold.

Since the components of $V_{2,t}$ are good, we can use \cite[Proposition \ref{Prop:maincombinatorialresult}(b)]{Bamler-LT-topology} to deduce that every Seifert fiber of $V_{2,t}$ intersects the image of $f_t$.
This implies that the area of the base of the Seifert fibration, $(\Sigma_{V_{2,t}}, g''_t)$, is bounded by $(A+1)t$.
So also the second part of assertion (e7) holds.
The diameter bound in assertion (e4) is a consequence of the lower area bound of balls in $(\Sigma_{V_{2,t}}, g''_t)$, Proposition \ref{Prop:V1V2decompositionboundedcurvature}(i), and assertion (e7).

Assertion (e9) is a consequence of the proof of \cite{Bamler-certain-topologies}.
\end{proof}
\makeatletter
\def\@cite#1#2{{\normalfont[{#1\if@tempswa , #2\fi}]}}
\makeatother


\begin{thebibliography}{xxxxx}
%\bibitem[Asa]{Asa}{K. Asano, ``Homeomorphisms of Prism Manifolds'', Yokohama Mathematical Journal 26, no. 1 (1978): 19-25.}
%\bibitem[Bam1]{Bamler-diploma}{R. Bamler, ``Ricci flow with surgery'', diploma thesis, Ludwig-Maximilians-Universit\"at Munich (2007)}
\bibitem[Bam2]{Bamler-hypcusps}{R. Bamler, ``Stability of hyperbolic manifolds with cusps under Ricci flow'', Adv. Math. 263 (2014), 412-467}
\bibitem[Bam3]{Bamler-certain-topologies}{R. Bamler, ``The long-time behavior of 3 dimensional Ricci flow on certain topologies'', to appear in Zeitschrift f\"ur die reine und angewandte Mathematik}
%\bibitem[Bam3]{Bamler-longtime-I}{R. Bamler, ``Long-time analysis of 3 dimensional Ricci flow I'', arXiv:{\linebreak[1]}math/{\linebreak[1]}1112.5125v1, (December 21, 2011), http://arxiv.org/abs/1112.5125v1}
%\bibitem[Bam4]{Bamler-longtime-II}{R. Bamler, ``Long-time analysis of 3 dimensional Ricci flow II'', arXiv:{\linebreak[1]}math/{\linebreak[1]}1210.1845, (October 5, 2012), http://arxiv.org/abs/1210.1845}
%\bibitem[Bam5]{Bamler-longtime-III}{R. Bamler, ``Long-time analysis of 3 dimensional Ricci flow III'', arXiv:{\linebreak[1]}math/{\linebreak[1]}1310.4483, (October 16, 2013), http://arxiv.org/abs/1310.4483}
\bibitem[Bam0]{Bamler-LT-Introduction}{R. Bamler, ``Long-time behavior of 3 dimensional Ricci flow, Introduction''}
\bibitem[BamA]{Bamler-LT-Perelman}{R. Bamler, ``Long-time behavior of 3 dimensional Ricci flow, A: Generalizations of Perelman's long-time estimates''}
\bibitem[BamB]{Bamler-LT-simpcx}{R. Bamler, ``Long-time behavior of 3 dimensional Ricci flow, B: Evolution of the minimal area of simplicial complexes under Ricci flow''}
\bibitem[BamC]{Bamler-LT-topology}{R. Bamler, ``Long-time behavior of 3 dimensional Ricci flow, C: 3-manifold topology and combinatorics of simplicial complexes in $3$-manifolds''}
%\bibitem[BBBMP1]{BBBMP}{L. Bessi\`eres, G. Besson, M. Boileau, S. Maillot, J. Porti, ``Geometrisation of 3-Manifolds'', Zuerich, Switzerland: European Mathematical Society Publishing House, (2010), http://www.ems-ph.org/books/book.php?proj\_nr=122}
\bibitem[BBBMP2]{BBMP2}{L. Bessi\`eres, G. Besson, M. Boileau, S. Maillot, J. Porti, ``Collapsing irreducible 3-manifolds with nontrivial fundamental group'', Invent. Math., 179(2):435-460, (2010)}
\bibitem[BBI]{BBI}{D. Burago, Y. Burago, S. Ivanov, ``A course in metric geometry'', Vol. 33. Graduate Studies in Mathematics. Providence, RI: American Mathematical Society, (2001)}
\bibitem[BGP]{BGP}{Y. Burago, M. Gromov, G. Perelman. ``AD Alexandrov spaces with curvature bounded below'', Russian Mathematical Surveys 47 (1992): 1-58}
\bibitem[CaG]{CG}{J. Cao, J. Ge, ``A Simple Proof of PerelmanÕs Collapsing Theorem for 3-manifolds'', J. Geom. Anal. 21, 807-869 (2011)}
\bibitem[CFG]{CFG}{J. Cheeger, K. Fukaya, M. Gromov, ``Nilpotent Structures and Invariant Metrics on Collapsed Manifolds'', Journal of the American Mathematical Society 5, no. 2 (1992): 327-372.}
%\bibitem[ChG1]{CheegerGromov-collapse1}{J. Cheeger, M. Gromov, ``Collapsing Riemannian manifolds while keeping their curvature bounded. II'', J. Differential Geom. 23 (1986), no. 3, 309-346}
%\bibitem[ChG2]{CheegerGromov-collapse2}{J. Cheeger, M. Gromov, ``Collapsing Riemannian manifolds while keeping their curvature bounded. I'', J. Differential Geom. 32 (1990), no. 1, 269-298}
%\bibitem[Cho]{Choe}{Choe, J., ``On the existence and regularity of fundamental domains with least boundary area'', Journal of Differential Geometry, 29(3) (1989): 623-663}
\bibitem[CM]{ColdingMinicozziextinction}{T. H. Colding, W. P. Minicozzi, ``Width and finite extinction time of Ricci flow'', Geom. Topol. 12, 2537-2586 (2008)}
\bibitem[Fae]{Faessler}{D. Faessler, ``On the topology of locally volume collapsed Riemannian 3-orbifolds'', arXiv:1101.3644 (January 19, 2011), http://arxiv.org/abs/1101.3644}
\bibitem[FY]{FY}{K. Fukaya, T. Yamaguchi, ``The fundamental groups of almost non-negatively curved manifolds'', Annals of Mathematics. Second Series 136, no. 2 (1992): 253-333}
%\bibitem[Fuk1]{Fukaya-collapse1}{K. Fukaya, ``Collapsing Riemannian manifolds to ones of lower dimensions'', J. Differential Geom. 25 (1987), no. 1, 139-156}
%\bibitem[Fuk2]{Fukaya-collapse2}{K. Fukaya, ``Collapsing Riemannian manifolds to ones with lower dimension. II'', J. Math. Soc. Japan 41 (1989), no. 2, 333-356}
\bibitem[GIK]{Guenter-Isenberg-Knopf-2002}{C., Guenther, J. Isenberg, D. Knopf, ``Stability of the Ricci flow at Ricci-flat metrics. Communications in Analysis and Geometry'', (2002), 10(4), 741-777.}
\bibitem[GS]{Grove-Shiohama-1977}{K. Grove, K. Shiohama, ``A Generalized Sphere Theorem'', Ann. of Math. (1977) 106, no. 2: 201-211}
\bibitem[Gul]{Gul}{R. D. Gulliver, ``Regularity of minimizing surfaces of prescribed mean curvature'', Annals of Mathematics. Second Series 97 (1973): 275-305}
%\bibitem[Ham]{Ham}{R. Hamilton, ``Non-singular solutions of the Ricci flow on three-manifolds'', Communications in Analysis and Geometry 7, no. 4 (1999): 695-729}
%\bibitem[Has1]{Hass}{J. Hass, ``Minimal Surfaces in Manifolds with $S^1$ Actions and the Simple Loop Conjecture for Seifert Fibered Spaces'', Proceedings of the American Mathematical Society 99, no. 2, (1987): 383-388.}
%\bibitem[Has2]{Hass-Plateau}{J. Hass, ``Singular curves and the Plateau problem'', International Journal of Mathematics, 2(1), (1991): 1-16}
\bibitem[Hat]{Hat}{A. Hatcher, ``Notes on basic 3-manifold topology'', in preparation, available on the web at http://www.math.cornell.edu/$\sim$hatcher}
%\bibitem[Hei]{Hei}{E. Heinz, ``\"Uber das Randverhalten quasilinearer elliptischer Systeme mit isothermen Parametern'', Mathematische Zeitschrift 113 (1970): 99-105}
%\bibitem[HH]{HH}{E. Heinz, S. Hildebrandt, ``Some remarks on minimal surfaces in Riemannian manifolds'', Communications on Pure and Applied Mathematics 23 (1970): 371-377}
%\bibitem[Kin]{Kin}{D. Kinderlehrer, ``The boundary regularity of minimal surfaces'', Ann. Scuola Norm. Sup. Pisa (3), 23, D. (1969): 711-744}
%\bibitem[KL1]{KLnotes}{B. Kleiner, J. Lott, ``Notes on Perelman's papers'', arXiv:math/0605667 (May 25, 2006), http://arxiv.org/abs/math/0605667}
\bibitem[KL2]{KLcollapse}{B. Kleiner, J. Lott, ``Locally Collapsed 3-Manifolds'', to appear in Ast\'erisque, arXiv:math/1005.5106v2  (May 27, 2010), http://arxiv.org/abs/1005.5106v2}
\bibitem[KoL]{KochLamm}{H. Koch, T. Lamm, ``Geometric flows with rough initial data'', Asian J. Math., Vol. 16, No. 2, pp. 209-236 (2012)}
%\bibitem[Lot1]{LottTypeIII}{J. Lott, ``On the long-time behavior of type-III Ricci flow solutions'', Mathematische Annalen 339, no. 3 (2007): 627-666}
\bibitem[Lot2]{LottDimRed}{J. Lott, ``Dimensional reduction and the long-time behavior of Ricci flow'', Commentarii Mathematici Helvetici. A Journal of the Swiss Mathematical Society 85, no. 3 (2010): 485-534}
%\bibitem[LS]{LottSesum}{J. Lott, N. Sesum, ``Ricci flow on three-dimensional manifolds with symmetry'', arXiv.org, February 21, 2011}
%\bibitem[LW]{Luecke-Wu-93}{J. Luecke, Y.-Q. Wu, ``Relative Euler number and finite covers of graph manifolds'', Geometric topology (Athens, GA, 1993) (Vol. 2, pp. 80-103). Providence, RI: Amer. Math. Soc., (1997)}
\bibitem[Mey]{Meyer-1989}{W. Meyer, ``Toponogov's Theorem and Applications'', Lecture Notes, College on Differential Geometry, Trieste (1989), http://wwwmath.uni-muenster.\linebreak[1]de/\linebreak[1]u/\linebreak[1]meyer/\linebreak[1]publications/\linebreak[1]topo.ps}
%\bibitem[MY]{MY}{W. Meeks, S.-T. Yau. ``Topology of three-dimensional manifolds and the embedding problems in minimal surface theory'', Annals of Mathematics. Second Series 112, no. 3 (1980): 441-484}
%\bibitem[Mor]{Mor}{C. B. Morrey, ``The problem of Plateau on a Riemannian manifold'', Annals of Mathematics. Second Series 49 (1948): 807-851}
\bibitem[MT1]{MTRicciflow}{J. W. Morgan, G. Tian, ``Ricci Flow and the Poincare Conjecture'', American Mathematical Society, (2007)}
\bibitem[MT2]{MorganTian}{J. Morgan, G. Tian, ``Completion of the Proof of the Geometrization Conjecture'', arXiv:0809.4040 (September 23, 2008), http://arxiv.org/abs/0809.4040}
%\bibitem[Neu]{Neumann}{W. D. Neumann, ``Immersed and Virtually Embedded $\pi_1$-Injective Surfaces in Graph Manifolds'', Algebraic \& Geometric Topology 1: 411-426 (electronic).}
\bibitem[NT]{NaberTian-collapse}{A. Naber, G. Tian, ``Geometric structures of collapsing Riemannian manifolds I. (English summary) Surveys in geometric analysis and relativity, 439-466, 
Adv. Lect. Math. (ALM), 20, Int. Press, Somerville, MA (2011)}
%\bibitem[Per1]{PerelmanI}{G. Perelman, ``The entropy formula for the Ricci flow and its geometric applications'', math/0211159 (November 2002): 39}
\bibitem[Per2]{PerelmanII}{G. Perelman, ``Ricci flow with surgery on three-manifolds'', arXiv:math/0303109 (March 10, 2003), http://arxiv.org/abs/math/0303109}
\bibitem[Per3]{PerelmanIII}{G. Perelman, ``Finite extinction time for the solutions to the Ricci flow on certain three-manifolds'', arXiv:math/0307245 (July 17, 2003), http://arxiv.org/abs/math/0307245}
\bibitem[Pet]{Petersen}{P. Petersen, ``Riemannian geometry'', Second edition. Graduate Texts in Mathematics, 171. Springer, New York (2006)}
\bibitem[Sch]{Sch}{R. Schoen, ``Estimates for Stable Minimal Surfaces in Three-Dimensional Manifolds'', in Seminar on Minimal Submanifolds, 103:111-126. Princeton, NJ: Princeton Univ. Press, 1983.}
\bibitem[SY]{ShioyaYamaguchi}{T. Shioya, T. Yamaguchi, ``Volume collapsed three-manifolds with a lower curvature bound'', Mathematische Annalen 333, no. 1 (2005): 131-155}
%\bibitem[SU1]{SU81}{J. Sacks, K. Uhlenbeck, ``The Existence of Minimal Immersions of 2-Spheres'', Annals of Mathematics, vol. 113, no. 1, pp. 1-24 (1981)}
%\bibitem[SU2]{Sacks-Uhlenbeck-1982}{J. Sacks, K. Uhlenbeck, ``Minimal Immersions of Closed Riemann Surfaces'', Transactions of the American Mathematical Society 271, no. 2 (1982): 639-652.}
%\bibitem[ScY]{Schoen-Yau-1979}{R. Schoen, S.-T. Yau, ``Existence of Incompressible Minimal Surfaces and the Topology of Three Dimensional Manifolds with Non-Negative Scalar Curvature'', The Annals of Mathematics 110, no. 1 (July 1979): 127.}
\bibitem[SSS]{SSS}{O. Schn\"urer, F. Schulze, M. Simon, ``Stability of Euclidean space under Ricci flow'', Comm. Anal. Geom., 16 (2008), pp. 127-158}
%\bibitem[WY]{WangYu}{S. Wang, F. Yu, ``Graph Manifolds with Non-Empty Boundary Are Covered by Surface Bundles'', Mathematical Proceedings of the Cambridge Philosophical Society 122, no. 3: 447-455.}
\bibitem[Ye]{Ye}{R. Ye, ``Ricci flow, Einstein metrics and space forms'', Transactions of the American Mathematical Society 338, no. 2 (1993): 871-896.}
\end{thebibliography}
\end{document}